\documentclass[11pt]{article}
\usepackage[utf8]{inputenc}
\usepackage[dvipsnames]{xcolor}
\usepackage{graphicx}
\usepackage{amsmath}
\usepackage{amsfonts}
\usepackage{amssymb}
\usepackage{amsthm}
\usepackage{thmtools}
\usepackage[margin=25mm]{geometry}

\usepackage{caption}
\usepackage{enumitem,calc}
\usepackage[longnamesfirst,numbers,sort&compress]{natbib}

\usepackage[unicode=true]{hyperref}
\hypersetup{
colorlinks,
linkcolor={black},
citecolor={black},
urlcolor={blue!60!black}
}

\usepackage[noabbrev,capitalise,nameinlink]{cleveref}
\crefname{lemma}{Lemma}{Lemmas}
\crefname{theorem}{Theorem}{Theorems}
\crefname{corollary}{Corollary}{Corollaries}
\crefname{proposition}{Proposition}{Propositions}
\crefname{figure}{Figure}{Figures}

\theoremstyle{plain}
\newtheorem{theorem}{Theorem}

\newtheorem{lemma}[theorem]{Lemma}
\newtheorem{corollary}[theorem]{Corollary}

\newtheorem{observation}[theorem]{Observation}
\numberwithin{theorem}{subsection}
\numberwithin{lemma}{subsection}
\numberwithin{proposition}{subsection}
\numberwithin{corollary}{subsection}
\numberwithin{observation}{subsection}

\newenvironment{proofofclaim}[1][Proof.]{\proof[#1]}{\endproof}

\newtheorem{claim}{Claim}
\usepackage{etoolbox}
\AtEndEnvironment{proof}{\setcounter{claim}{0}}

\theoremstyle{definition}

\DeclareMathOperator{\dist}{dist}

\DeclareMathOperator{\Attachable}{\scr{AC}}
\DeclareMathOperator{\FT}{F_3}
\DeclareMathOperator{\AEC}{\scr{EC}}
\DeclareMathOperator{\VC}{\scr{VC}}

\newcommand{\scr}[1]{\mathcal{#1}}

\newcommand{\floor}[1]{\lfloor#1\rfloor}
\newcommand{\ds}[1]{\mathbb{#1}}

\newcommand{\defn}[1]{\textcolor{Maroon}{\emph{#1}}}

\renewcommand{\geq}{\geqslant}
\renewcommand{\leq}{\leqslant}

\title{\bf 
Graphs Excluding a Minor in\\ 
Blowups of Treewidth 3 Graphs}
\author{Marc Distel}

\begin{document}

\maketitle

\setlength{\parindent}{0pt}
\setlength{\parskip}{10pt}

\begin{abstract}
    Alon, Seymour, and Thomas [J.~Amer.~Math.~Soc.~1990] famously showed that every $n$-vertex $K_h$-minor-free graph has treewidth $O_h(\sqrt{n})$. Recently, Distel, Dujmovi\'c, Eppstein, Hickingbotham, Joret, Micek, Morin, Seweryn, and Wood [SIAM J.~Discrete Math.~2024] refined this by showing that these graphs are $O_h(\sqrt{n})$-blowups of treewidth $4$ graphs. We improve this by showing that these graphs are $O_h(\sqrt{n})$-blowups of treewidth $3$ graphs.
\end{abstract}
\newpage
\tableofcontents
\newpage

\section{Introduction}

    Treewidth is a measure of how far a graph is from a tree. \citet{Lipton1979} showed that every $n$-vertex planar graph has treewidth $O(\sqrt{n})$, and \citet{Alon1990} extended this to show that each $K_h$-minor-free graph has treewidth $O_h(\sqrt{n})$. Both of these results are the best possible, up to the constant factor. However, recently these results have been refined, using the language of graph `blowups'.

    A \defn{strong product} of a graph $H$ and the complete graph $K_w$ is the simple graph with vertex set $V(H)\times V(K_w)$, where vertices are adjacent if and only if they are adjacent or equal in each co-ordinate. A graph $G$ is a \defn{$w$-blowup} of a graph $H$ if, after deleting parallel edges and loops, it is isomorphic to a subgraph of a strong product of $H$ and $K_w$.

    The following were shown by \citet{Distel2024}.

    \begin{theorem}
        \label{tw4Result}
        Let $h,n\in \ds{N}$. Then every $n$-vertex $K_h$-minor-free graph $G$ is a $O_h(\sqrt{n})$-blowup of a treewidth $4$ graph.
    \end{theorem}

    \begin{theorem}
        \label{tw2Surfaces}
        Let $g,n\in \ds{N}$. Then every $n$-vertex graph $G$ embedded on a surface of genus at most $g$ is a $O_g(\sqrt{n})$-blowup of a treewidth $2$ graph.
    \end{theorem}

    If a graph $G$ of treewidth $b$ is a $w$-blowup of a graph $H$ of treewidth $t$, then $b\leq (t+1)w-1$. So \cref{tw4Result} directly implies every $n$-vertex $K_h$-minor-free graphs has treewidth at most $O_h(\sqrt{n})$. Thus, up to the constant factor, this is a strengthening of the result of \citet{Alon1990}.

    We build on the ideas and techniques used in \citet{Distel2024} to show the following.

    \begin{theorem}
        \label{main}
        Let $h,n\in \ds{N}$. Then every $n$-vertex $K_h$-minor-free graph is a $O_h(\sqrt{n})$-blowup of a treewidth $3$ graph.
    \end{theorem}

    \citet{Distel2024} observed that a result of \citet{Linial2008} implied that for each $h\in \ds{N}$ with $h\geq 5$, there is no $f_h\in O(\sqrt{n})$ such that every $n$-vertex $K_h$-minor-free graph is a $(f(n))$-blowup of a tree. Thus, the `$3$' in \cref{main} is only $1$ away from the theoretical best possible of `$2$'. Further, our method is not highly specific to the treewidth $3$ case.

    We remark that \citet{Distel2025} showed that $K_h$-minor-free graphs are $O_h(\sqrt{n}\log(n)^2)$-blowups of fans. \cref{main} is an improvement on this in terms of the blowup factor, but also weaker in terms of the treewidth, as fans have treewidth $2$.
    
    The paper is written in a way to maximise the possibility that the proof can be recycled to possibly reduce the $3$ to a $2$. We also present an algorithm to `reduce' a `$(g,p,k,a)$-almost-embedding' (see \cref{SecAE}) to something `close to' a `$(0,1,k,0)$-almost-embedding' by only deleting a constant number of vertices, which may be of independent interest. See \cref{SecReductions} and \cref{findReductionModify} for details.

    \subsection{Preliminaries}

    In this paper, we allow graphs to have parallel edges and loops. We say that a graph is \defn{simple} if it has no parallel edges or loops. We only consider graphs with finite vertex and edge sets. We remark that several referenced papers, such as \citet{Distel2024} only consider simple graphs, however it is often easy to translate the stated results to allow for parallel edges and loops (as we have done).

    See \citet{Diestel2018} for standard graph-theoretic definitions.
    
    We do not consider edges to be labelled. So for the purposes of taking subgraphs, we only consider the number of parallel edges/loops. We consider an isolated vertex to be a degree $0$ vertex (have no loops). We consider the empty set to be a clique in a graph and the empty graph to be connected.

    Let \defn{$\ds{R}^+$} be the set of all strictly positive real numbers, let \defn{$\ds{R}_0^+=\ds{R}^+\cup \{0\}$}, and let \defn{$\ds{N}$}$:=\{0,1,\dots\}$. For $a,b\in \ds{N}$, we define $\binom{a}{b}:=0$ whenever $a<b$.

    Whenever we refer to a disc in this paper, it is closed.

    To \defn{contract} an edge $uv$ in a simple graph $G$ is to delete $u,v$ and add a new vertex $x$ whose neighbourhood is precisely $N_G(u)\cup N_G(v)$. We say that a simple graph $H$ is a \defn{minor} of a graph $G$ if a simple graph isomorphic to $H$ can be obtained from $G$ by deleting all parallel edges and loops, then performing any sequence of contractions, vertex deletions, and edge deletions.

    The \defn{degeneracy} of a graph $G$ is the smallest $d\in \ds{N}$ such that there exists an ordering $\prec$ of $V(G)$ such that each $v\in V(G)$ is adjacent in $G$ to at most $d$ vertices $x\in V(G)\setminus \{v\}$ with $v\prec x$.

    For a set of vertex $S$ of a graph $G$, we use \defn{$N_G[S]$} to denote the set of vertices either in $S$ or adjacent in $G$ to a vertex in $S$. We define \defn{$N_G(S)$}$:=N_G[S]\setminus S$.

    \subsection{Graph-decompositions}

    For graphs $G,H$, a $H$-decomposition of $G$ is a collection $\scr{J}=(J_v:v\in V(H))$ of subsets of $V(G)$ such that (1) for each $x\in V(G)$, the subgraph of $H$ induced by the vertices $v\in V(H)$ with $x\in J_v$ induce a connected and nonempty subgraph of $H$, and (2) for each pair of distinct vertices $x,y\in V(G)$ that are adjacent in $G$, there exists $v\in V(H)$ such that $x,y\in J_v$. We remark that these properties are not affected by the absence or presence of parallel edges or loops in $G$ or in $H$. The sets $J_v$, $v\in V(H)$, are called \defn{bags}. The \defn{width} of $\scr{J}$ is the maximum size of a bag minus $1$. The \defn{adhesion} is the maximum intersection between two bags. We remark that this maximum is achieved by bags $J_t,J_{t'}$ with $t,t'$ adjacent in $T$. For $t\in V(T)$, the \defn{torso} of $G$ at $t$ (with respect to $T,\scr{J}$), denoted \defn{$G\langle J_t\rangle$}, is the graph obtained from $G[J_t]$ by, for each $t'\in V(T)$ adjacent to $t$ in $T$ and each pair of distinct non-adjacent vertices $u,v\in J_t\cap J_{t'}$ in $G[J_t]$, adding a new edge with endpoints $u,v$. So $J_t\cap J_{t'}$ is a (possibly empty) clique in $G\langle J_t\rangle$.

    A \defn{tree-decomposition} of a graph $G$ is a $T$-decomposition of $G$ for some tree $T$. The \defn{treewidth} of $G$ is the minimum width of a tree-decomposition of $G$. It is well-known that graphs of treewidth $b$ have degeneracy at most $b$, and that every clique in such a graph has size at most $b+1$. We also define \defn{path-decompositions} analogously.
    
    We say that a triple $\scr{GD}=(H,U,\scr{J})$ is a \defn{graph-decomposition} if $H,U$ are graphs and $\scr{J}$ is a $U$-decomposition of $H$. Say that the \defn{width} of $(H,U,\scr{J})$ is the width of $\scr{J}$. Let $(J_x:x\in V(U)):=\scr{J}$. For each $v\in V(H)$, let \defn{$C_v(\scr{GD})$}$:=C_v$ be subgraph of $U$ induced by the vertices $x\in V(H)$ such that $v\in J_x$. By definition of $\scr{J}$, $C_v$ is nonempty and connected, and intersects $C_u$ whenever $u\in V(H)$ is adjacent to $v$ in $G$. We say that $\scr{GD}$ is \defn{planted} if $V(U)\subseteq V(H)$ and, for each $x\in V(U)$, $x\in J_x$. We say that $\scr{GD}$ is \defn{smooth} in a graph $G$ if $V(U)\subseteq V(G)$ and, for each $v\in V(H)$, there exists a spanning tree $T_v$ for $C_v$ such that for each subtree $T$ of $T_v$, $G[V(T)]$ is connected. We remark that since $V(T)\subseteq V(T_v')=V(T_v)\subseteq V(U)\subseteq V(G)$, $G[V(T)]$ is always well-defined.

    Graph-decompositions are `well-behaved' under taking spanning supergraphs, as given by the following observation.

    \begin{observation}
        \label{spanningSmooth}
        Let $\scr{GD}=(H,U,\scr{J})$ be a graph-decomposition, and let $U'$ be a spanning supergraph of $U$. Then $\scr{GD}':=(H,U',\scr{J})$ is a graph-decomposition of the same width as $(H,U,\scr{J})$. Further:
        \begin{enumerate}
            \item if $\scr{GD}$ is planted, then so is $\scr{GD}'$, and,
            \item if $\scr{GD}$ is smooth in a graph $G$, then $\scr{GD}'$ is smooth in any supergraph $G'$ of $G$.
        \end{enumerate}
    \end{observation}

    \begin{proof}
        Let $(J_x:x\in V(U)):=\scr{J}$. Note that $\scr{J}=(J_x:x\in V(U'))$, and that $\scr{J}\subseteq 2^{V(G)}$. For each $v\in V(G)$, let $C_v$ and $C_v'$ be the subgraphs of $U$ and $U'$ respectively induced by the vertices $x\in V(U)=V(U')$ such that $v\in J_x$. Since $U'$ is a (spanning) supergraph of $U$, $C_v'$ is a spanning supergraph of $C_v=C_v(\scr{GD})$, which is nonempty and connected. So $C_v'$ is nonempty and connected. Further, for any pair of adjacent $u,v\in V(H)$, $C_u'$ and $C_v'$ intersect as $C_u$ and $C_v$ intersect. It follows that $\scr{GD}'$ is a graph-decomposition.
        
        Since the width is determined only by $\scr{J}$, $\scr{GD}'$ has the same width as $\scr{GD}$.

        If $\scr{GD}$ is planted, then $V(U')=V(U)\subseteq V(H)$, and for each $x\in V(U)=V(U')$, $x\in J_x$. Thus, $\scr{GD}'$ is planted.

        If $\scr{GD}$ is smooth in a graph $G\subseteq G'$, then $V(U)=V(U')\subseteq V(G)\subseteq V(G')$, and for each $v\in V(H)$, there exists a spanning tree $T_v$ of $C_v$ such that for each subtree $T$ of $T_v$, $G[V(T)]\subseteq G'[V(T)]$ is connected. Since $C_v'$ is a spanning supergraph of $C_v$, $T_v$ is also a spanning tree of $C_v'=C_v(\scr{GD}')$. It follows that $\scr{GD}'$ is smooth in $G$.
    \end{proof}

    \subsection{Partitions and Almost Partitions}

    A \defn{partition} $\scr{P}$ of a graph $G$ is a collection of disjoint subsets of $V(G)$ whose union is $G$. The elements of $\scr{P}$ are called \defn{parts}. For technical reasons, we allow empty parts. We also allow duplicate parts, so $\scr{P}$ is not a set. However, we remark that the only duplicate parts that can exist are empty parts.

    The \defn{width} of a partition $\scr{P}$ of a graph $G$ is the maximum size of a part. The \defn{quotient} of $G$ by $\scr{P}$, denoted \defn{$G/\scr{P}$}, is the simple graph with vertex set $\scr{P}$ (allowing for multiple vertices to have the duplicate label $\emptyset$), where distinct parts $P_1,P_2\in \scr{P}$ are adjacent if and only if $P_2$ intersects $N_G(P_1)$. We also say that the \defn{treewidth} of $\scr{P}$ is the treewidth of $G/\scr{P}$. $\scr{P}$ is \defn{connected} if each $P\in \scr{P}$ induces a connected subgraph of $G$. Note that if $\scr{P}$ is connected, then $G/\scr{P}$ is a minor of $G$.

    It is well-known that blowups can be expressed naturally in terms of partitions. Specifically, a graph $G$ is $b$-blowup of a graph $H$ if and only if $G$ admits a partition $\scr{P}$ of width at most $b$ whose quotient is a subgraph of $H$. Since treewidth doesn't increase under taking subgraphs, \cref{tw4Result} and \cref{main} can be rephrased to `Every $n$-vertex $K_h$-minor-free graph $G$ admits a partition of width at most $O_h(\sqrt{n})$ and treewidth at most 4/3' respectively.

    However, the following improvement to \cref{tw4Result} is implicit in \citet{Distel2024}.

    \begin{theorem}
        \label{tw4ResultApex}
        Let $h,n\in \ds{N}$. Then every $n$-vertex $K_h$-minor-free graph $G$ is a $O_h(\sqrt{n})$-blowup of a graph $H$ which contains a vertex $v$ such that $H-v$ has treewidth at most $3$.
    \end{theorem}

    This implies \cref{tw4Result}, as adding a vertex increases the treewidth by at most $1$.

    We can also match this result.

    \begin{theorem}
        \label{mainApex}
        Let $h,n\in \ds{N}$. Then every $n$-vertex $K_h$-minor-free graph $G$ is a $O_h(\sqrt{n})$-blowup of a graph $H$ which contains a vertex $v$ such that $H-v$ has treewidth at most $2$.
    \end{theorem}

    \cref{mainApex} directly implies \cref{main}.

    These theorems are a bit awkward to state, even when translated to partitions. Thus, we introduce the concept of `almost-partitions'.

    An \defn{almost-partition} of a graph $G$ is a pair $(X,\scr{P})$, where $X\subseteq V(G)$ and $\scr{P}$ is a partition of $G-X$. We call $|X|$ the \defn{loss set}. It is immediate that $(X)\sqcup \scr{P}$ is a partition of $G$. So an almost-partition is just a special type of partition. However, we favour this notation, because it better allows us to isolate and focus on key properties. Say that the \defn{width} and \defn{treewidth} of $(X,\scr{P})$ are the width and treewidth of $\scr{P}$ (as a partition of $G-X$), and that the \defn{loss} of $(X,\scr{P})$ is $|X|$. So if $(X,\scr{P})$ has width and loss at most $w$, then $(X)\sqcup \scr{P}$ has width at most $w$. However, we prefer to decouple these parameters, because typically we want the loss to be substantially smaller than the width. This is because, at a later stage, we will `merge' multiple almost-partitions together, by taking the (disjoint) union of the partitions (which does not increase the width) and the union of the loss sets (which does increase the loss).

    \cref{tw4ResultApex} and \cref{mainApex} can be restated as `Every $n$-vertex $K_h$-minor-free graph $G$ admits an almost-partition of width and loss at most $O_h(\sqrt{n})$ and treewidth at most 3/2' respectively. We actually show a slightly stronger result.

    \begin{restatable}{theorem}{mainApexFlex}
        \label{mainApexFlex}
        For each $h\in \ds{N}$ with $h\geq 2$, there exists $c',q'\in \ds{N}$ and $\alpha\in \ds{R}^+$ such that for each $n\in \ds{N}$ and every $d\in \ds{R}^+$ with $d\geq \alpha\sqrt{n}$, each $n$-vertex $K_h$-minor-free graph $G$, admits an almost-partition of treewidth at most $2$, width at most $d$, and loss at most $c'n/d+q'$.
    \end{restatable}

    For each $h,n\in \ds{N}$, picking $d:=\alpha\sqrt{\max(n,1)}>0$ (as $\alpha>0$) in \cref{mainApexFlex} then implies \cref{mainApex}. \cref{mainApexFlex} is the main result that we will prove in this paper.

    \subsection{Almost-embeddings}
    \label{SecAE}

    The `Graph Minor Structure Theorem' (GMST) of \citet{Robertson2003} gives a structural description of $K_h$-minor-free graphs, which will be the backbone of our results. We now work towards stating this result. We warn that we give slightly different definitions to what is `standard', however it can be easily verified (and we explain why) that these definitions are more general and thus `account for' the standard definition.

    For a surface $\Sigma$, we use \defn{$g(\Sigma)$} to denote the (Euler) genus of $\Sigma$. A graph \defn{embedded} in $\Sigma$ is a graph $G$ along with an embedding $\sigma$ of the vertices and edges of $G$ into $\Sigma$ (such that the edges do not cross). We consider edges to be open intervals that converge to their endpoints (as opposed to closed intervals containing the endpoints). By abuse of notation, we usually neglect the embedding $\sigma$ and refer to $G$ as a graph embedded in $\Sigma$. Further, we frequently do not wish to specify the surface, and refer to an \defn{embedded graph} $G$, which is a graph embedded in some surface. A \defn{plane graph} is a graph embedded in the plane $\ds{R}^2$.

    An \defn{embedded subgraph} of an embedded graph $G$ is a subgraph of $G$ that has inherited the embedding of $G$ (the embedding is the restriction). We define \defn{embedded supergraph}, \defn{plane subgraph}, and \defn{plane supergraph} similarly.

    See \citet{Mohar2001} for more details about embedded graphs and other standard definitions.

    Given an embedded graph $G$, a (closed) disc (in the surface $G$ is embedded in) is \defn{$G$-clean} if the interior of $D$ is disjoint to $G$, and the boundary of $D$ intersects $G$ only at vertices of $G$. So $D\cap G=\partial D\cap V(G)$, where we use \defn{$\partial D$} to denote the boundary of $D$.

    For an embedded graph $G$ and a $G$-clean disc $D$, let \defn{$B(D,G)$}$:=D\cap V(G)$. Note that $D\cap G=\partial D\cap G=\partial D\cap V(G)=B(D,G)$. We call $B(D,G)$ the \defn{boundary vertices} of $D$ in $G$. We call the connected components of $\partial D\setminus B(D,G)$ the \defn{arcs} of $D$ in $G$. Observe that each arc $C$ is disjoint to $G$. Further, observe that the closure of $C$ intersects $G$ at at most two points. These points are vertices in $B(D,G)$, and we call these vertices the \defn{endpoints} of $C$. If $|B(D,G)|\geq 2$, observe that each arc has exactly two endpoints. 
    
    For an embedded graph $G$ and a $G$-clean disc $D$, let \defn{$U(D,G)$} be the simple graph with vertex set $B(D,G)$, where distinct vertices $u,v\in B(D,G)$ are adjacent if and only if there exists an arc of $D$ in $G$ whose endpoints are $u,v$. We call $U(D,G)$ the \defn{underlying cycle} of $D$ in $G$. Note that if $|B(D,G)|\geq 3$, then that $U(D,G)$ is a cycle (hence the name). Further, observe that (up to reversing the direction) the cyclic ordering of the vertices in $U(D,G)$ is precisely the cyclic ordering obtained from following $\partial D$. If $|B(D,G)|\leq 2$, then $U(D,G)$ is instead a path. Note that by definition, $V(U(D,G))=B(D,G)$.

    We will need to consider multiple $G$-clean discs at the same time. Usually (and in other papers), we would only need to consider pairwise disjoint $G$-clean discs. However, for one specific purpose, we need to allow the discs to touch at vertices of $G$. This leads to the following definition.

    For an embedded graph $G$, we say that a set of discs $\scr{D}$ is \defn{$G$-pristine} if each $D\in \scr{D}$ is $G$-clean and, for each pair of distinct $D,D'\in \scr{D}$, $D\cap D'\subseteq V(G)$. Note that this implies that the discs in $\scr{D}$ are internally disjoint.

    As mentioned before, most of the time it suffices to consider only pairwise disjoint discs, so the following observation is useful.
    
    \begin{observation}
        \label{disjointStronglyClean}
        Let $G$ be an embedded graph, and let $\scr{D}$ be a set of pairwise disjoint $G$-clean discs. Then $\scr{D}$ is $G$-strongly clean.
    \end{observation}

    \begin{proof}
        For each pair of distinct $D,D'\in \scr{D}$, $D\cap D'=\emptyset\subseteq V(G)$.
    \end{proof}

    We often use \cref{disjointStronglyClean} implicitly.

    For an embedded graph $G$ and a $G$-pristine set of discs $\scr{D}$, set \defn{$B(\scr{D},G)$}$:=\bigcup_{D\in \scr{D}}B(D,G)$ and let \defn{$U(\scr{D},G)$}$:=\bigcup_{D\in \scr{D}}U(D,G)$. We call $B(\scr{D},G)$ the \defn{boundary vertices} of $\scr{D}$ in $G$, and $U(\scr{D},G)$ the \defn{underlying cycles} of $\scr{D}$ in $G$. Note that $V(U(\scr{D},G))=B(\scr{D},G)$. By \cref{disjointStronglyClean}, $B(\scr{D},G)$ and $U(\scr{D},G)$ are well-defined for a set of pairwise disjoint $G$-clean discs. In this case, note that $U(\scr{D},G)$ is a union of disjoint (underlying) cycles.

    An \defn{almost-embedding} $\Gamma$ is a tuple $(G,\Sigma,G_0,\scr{D},H,\scr{J},A)$ such that:
    \begin{enumerate}
        \item $G$ is a graph,
        \item $A\subseteq V(G)$,
        \item $G-A=G_0\cup H$,
        \item $\Sigma$ is a surface,
        \item $G_0$ is embedded in $\Sigma$,
        \item $\scr{D}$ is a $G_0$-pristine set of discs,
        \item $V(G_0\cap H)=B(\scr{D},G_0)$, and,
        \item $(H,U(\scr{D},G_0),\scr{J})$ is a planted graph-decomposition.
    \end{enumerate}

    We call $G$ the \defn{graph} (of $\Gamma$), $\Sigma$ the \defn{surface}, $G_0$ the \defn{embedded subgraph}, $\scr{D}$ the \defn{discs}, $H$ the \defn{vortex}, $\scr{J}$ the \defn{decomposition}, and $A$ the \defn{apices}. Say that the \defn{genus} of $\Gamma$ is $g(\Sigma)$, the \defn{disc-count} of $\Gamma$ is $|\scr{D}|$, the \defn{vortex-width} of $\Gamma$ is the width of $\scr{J}$, and the \defn{apex-count} of $\Gamma$ is $|A|$. Say that $\Gamma$ is a \defn{$(g,p,k,a)$-almost-embedding} if $\Gamma$ has genus at most $g$, disc-count at most $p$, vortex-width at most $k$, and apex-count at most $a$. Set \defn{$B(\Gamma)$}$:=V(G_0\cap H)=B(\scr{D},G)$ and \defn{$U(\Gamma)$}$:=U(\scr{D},G)$. We call $B(\Gamma)$ the \defn{boundary vertices} of $\Gamma$. Say that $\Gamma$ is \defn{nontrivial} if $V(G)\neq V(A)$. Say that $\Gamma$ has \defn{disjoint discs} if the discs in $\scr{D}$ are pairwise disjoint.

    An \defn{almost-embedding} for a graph $G$ is an almost-embedding whose graph is $G$.
    
    We remark that we prefer to talk about `almost-embeddings' rather than `almost-embeddable graphs' (the graph $G$) because several important properties are dependent on the almost-embedding itself (so we cannot substitute an almost-embedding for another with the same graph $G$). We also include $G$ within the tuple (as opposed to simply saying `an almost-embedding of $G$') because we will need to perform `natural' modifications to an almost-embedding, which also change $G$ in an `uninteresting' way. Trying to discuss and relate two almost-embeddings of different graphs is cumbersome and strange, even if one was derived naturally from the other, so we prefer to hide these slight modifications within the tuple $\Gamma$.

    Our definition is also slightly different to the `standard' definition of `$(g,p,k,a)$-almost-embeddable graphs' (such as seen in \citet{Distel2024}). The most significant difference is that we allow two discs in $\scr{D}$ to intersect, provided that they only intersect at vertices in $G$. As mentioned above, we allow this for technical reasons. However, if the almost-embedding is forced to have disjoint discs, then this difference disappears.
    
    Secondly, instead of having $s:=|\scr{D}|$ graphs $G_1,\dots,G_s$, each of which admit a planted path-decomposition indexed by $B(D_i,G)$ (in the natural order), we instead just have one graph $H$ and a planted $U$-decomposition of $H$. Here, $H$ is really the union of the graphs $G_1,\dots,G_s$. The path-decompositions can be viewed as planted $U(D_i,G)$-decompositions of the same width. Then, since the discs are disjoint (under the standard definition), $U(\scr{D},G)$ is the disjoint union of $U(D_1,G),\dots,U(D_s,G)$, so we can translate the $U(D_i,G)$-decompositions to a $U(\scr{D},G)$ of the same width. So our definition of a $(g,p,k,a)$-almost-embedding generalises the standard definition.

    With these comments in mind, we can now state the GMST \citep{Robertson2003}.
    \begin{theorem}
        For each $h\in \ds{N}$ with $h\geq 2$, there exists $g,p,k,a,j\in \ds{N}$ such that each $K_h$-minor-free graph $G$ admits a tree-decomposition of adhesion at most $j$ where each torso admits a $(g,p,k,a)$-almost-embedding with disjoint discs.
    \end{theorem}

    The GMST is a very powerful tool. However, for our purposes, it is actually not quite strong enough. We need to know where the intersection between bags lies in the almost-embedding. To do this, we employ a variant of the GMST shown by \citet{Diestel2012} (Theorem 4 in \citep{Diestel2012}). This theorem is a bit hard to state, so we introduce some new terminology to assist.

    Observe that the boundary $F$ of an embedded graph $G$ is the embedding of a subgraph of $G$. We refer to the vertices and edges of this subgraph as the \defn{boundary vertices} and \defn{boundary edges} of $F$ respectively. $F$ is \defn{$2$-cell} if it is homeomorphic to a disc. In this case, following the boundary of $F$ gives a closed walk in $G$ whose vertices and edges are exactly the boundary vertices and boundary edges of $F$. We call this closed walk the \defn{boundary walk}, and write it as an alternating sequence of edges and vertices, starting with an edge. The \defn{vertex boundary walk} is the subsequence of the boundary walk consisting of only the vertices (the sequence of vertices in the walk).
    
    Say that a face $F$ of an embedded graph is a \defn{triangle-face} (respectively, an \defn{oval-face} or a \defn{loop-face}) if $F$ is $2$-cell and the boundary walk of $F$ has length exactly $3$ (respectively $2$ or $1$) with no repeat vertices. So the boundary walk of a triangle-face is a cycle of size exactly $3$, the boundary walk of an oval-face consists of two parallel edges between two distinct vertices, and the boundary walk of a loop-face consists of a single vertex and a loop. Observe that the boundary vertices in a triangle-face form a clique of size exactly $3$ in $G$. Say that a clique $K$ of size $3$ in a graph $G$ is a \defn{facial triangle} if there is a triangle-face of $G$ whose boundary vertices are exactly $K$. We use \defn{$\FT(G)$} to refer to the set of facial triangles of $G$.

    Let $\Gamma=(G,\Sigma,G_0,\scr{D},H,\scr{J},A)$ be an almost-embedding. We define the \defn{facial triangles} of $\Gamma$, denoted \defn{$\FT(\Gamma)$}, to be $\FT(G_0)$. Define the \defn{vortex cliques} of $\Gamma$, denoted \defn{$\VC(\Gamma)$}, to be the set of cliques in $H$. Define the \defn{attachable embedded cliques} of $\Gamma$, denoted \defn{$\AEC(\Gamma)$}, to be the set of cliques in $G_0$ that are either in $\FT(G_0)$ or have size at most $2$. So $\AEC(\Gamma)\supseteq \FT(\Gamma)$. Define the \defn{attachable cliques} of $\Gamma$, denoted \defn{$\Attachable(\Gamma)$}, to be the set of cliques $K$ in $G$ such that $K\setminus A\in \AEC(\Gamma)\cup \VC(\Gamma)$. Observe that $\VC(\Gamma),\AEC(\Gamma)$, and $\Attachable(\Gamma)$ are closed under taking subsets, and are all subsets of $2^{V(G)}$.

    A \defn{rooted tree} is a pair $(T,r)$ where $T$ is a tree and $r\in V(T)$. $r$ is called the \defn{root} of $T$. For distinct vertices $t,t'\in V(G)$, if $t$ is on the unique path from $t'$ to $r$ in $T$, then $t$ is an \defn{ancestor} of $t'$ (in $T$) and $t'$ is a \defn{descendant} of $t$ (in $T$). If $t$ is either an ancestor or descendant of $t'$, then $t$ and $t'$ are \defn{related} (in $T$), otherwise they are \defn{unrelated} (in $T)$. Observe that for each $t\in V(T)\setminus \{r\}$, there is unique vertex that is both adjacent to $t$ and an ancestor of $t$, which is the \defn{parent} of $t$ (in $T$), and denoted \defn{$p(t)$}. For each $t\in V(T)$, the sets of descendants of $t$ adjacent to $t$ are the \defn{children} of $t$ (in $T$). So for each child $t'$ of $t$, $t$ is the parent of $t'$. For a nonempty subtree $T'$ of $T$, we say that the \defn{induced root} of $T'$ (in $T$) is the unique vertex of $V(T')$ closest to $r$. We say that a set $Z\subseteq V(T)$ is \defn{rooted} if $r\in Z$.

    We can now state the variant of the GMST shown by \citet{Diestel2012}.

    \begin{theorem}
        \label{GMSTVariant}
        For each $h\in \ds{N}$ with $h\geq 2$, there exists $g,p,k,a\in \ds{N}$ such that for each $K_h$-minor-free graph $G$, there exists a rooted tree $(T,r)$ and a collection $((J_t,\Gamma_t):t\in V(T))$ such that:
        \begin{enumerate}
            \item $(J_t:t\in V(T))$ is a tree-decomposition of $G$,
            \item for each $t\in V(T)$, $\Gamma_t$ is a $(g,p,k,a)$-almost-embedding of $G\langle J_t\rangle$ with disjoint discs, and,
            \item for each edge $tp(t)\in E(T)$: 
            \begin{enumerate}
                \item if $A_t$ is the apices of $\Gamma_t$, then $J_t\cap J_{p(t)}\subseteq A_t$, and,
                \item $J_t\cap J_{p(t)}\in \Attachable(\Gamma_{p(t)})$.
            \end{enumerate}
        \end{enumerate}
    \end{theorem}

    We remark that it is not explicitly stated in \citet{Diestel2012} that the discs are pairwise disjoint, but it can be inferred from the fact that the `vortices' (which they label $V_1,\dots,V_{\alpha'}$) are pairwise disjoint. This then forces the discs to have disjoint boundary vertices, and thus be disjoint (as they are clean).

    \cref{GMSTVariant} is the backbone of our proof.

    \subsection{Structure of this paper}

    Our general strategy is similar to that used by \citet{Distel2024}, and to a lesser extent, \citet{Distel2025}. It is helpful to be familiar with the proof of \cref{tw4Result}, although we also outline the key ideas of that proof in this paper, as they are important to understanding the new ideas of this paper.

    Our proof can loosely be broken into two stages. The first is `almost-partitioning a graph that admits an almost-embedding'. The second is `breaking up tree-decomposition', in which we loosely reduce the problem of almost-partitioning a $K_h$-minor-free graph into almost-partitioning the torsos (which admit almost-embeddings). Both of these steps need to be further broken down, although we remark that the former is far more involved. These steps mirror the approach used in \citet{Distel2024}, however we remark that both steps require new ideas.

    In terms of a formal lemma statement, these ideas are linked by a technical `decomposability' parameter that we define later (see \cref{SecDecompos}). However, loosely speaking, this parameter just means `the graph admits a tree-decomposition under which every torso admits a ``good" almost-partition'. In our case, the tree-decomposition is from \cref{GMSTVariant}.

    The first step, along with \cref{GMSTVariant}, leads to the following.


    \begin{restatable}{theorem}{AEMain}
        \label{AEMain}
        For each $h\in \ds{N}$ with $h\geq 2$, there exists $m,c,q,j\in \ds{N}$ and $\alpha \in \ds{R}^+$ such that for each $n\in \ds{N}$ and every $d\in \ds{R}^+$ with $d\geq \alpha\sqrt{n}$, each $n$-vortex $K_h$-minor-free graph $G$ is $(2,m,d/2,c,q,j)$-decomposable.
    \end{restatable}

    The second step gives the following.

    \begin{restatable}{theorem}{TDReduced}
        \label{TDReduced}
        Let $b,m,c,q,j\in \ds{N}$. Then there exists $c',q'\in \ds{N}$ such that for each each $n\in \ds{N}$ and every $d\in \ds{R}^+$, each $n$-vertex $(b,m,d/2,c,q,j)$-decomposable graph $G$ admits an almost-partition of treewidth at most $b$, width at most $d$, and loss at most $c'n/d+q'$.
    \end{restatable}

    Even without defining what it means to be `$(b,m,d/2,c,q,j)$-decomposable', we can use these results to prove \cref{mainApexFlex} (which we restate for convenience).

    \mainApexFlex*

    \begin{proof}
        Let $m,c,q,j,\alpha$ be from \cref{AEMain} (with $h$), and then let $c',q'$ be from \cref{TDReduced} (with $b:=2$ and $m,c,q,j$). By \cref{AEMain}, $G$ is $(2,m,d/2,c,q,j)$-decomposable. By \cref{TDReduced}, $G$ admits an almost-partition of treewidth at most $2$, width at most $d$, and loss at most $c'n/d+q'$.
    \end{proof}

    As mentioned before, \cref{TDReduced} is the far easier result. So we cover it first. Proving it is the subject of \cref{sectionTD}. The remainder of the paper is devoted to the proof of \cref{AEMain}. We go into more detail into the structure of the proof of \cref{AEMain} and the remainder of the paper in \cref{SecHandleAE}. Specifically, in \cref{SecRemainingStructure}, we identify which sections of the paper are associated with which results.

    \section{Handling the tree-decomposition}
    \label{sectionTD}

    \subsection{Proof idea}
    The goal of this section is to reduce the problem of `handling an entire tree-decomposition' to `handling a torso of a tree-decomposition'. We remind the reader that, in our case, the tree-decomposition in question is from \cref{GMSTVariant}, and the latter problem becomes `handling an almost-embedding'.

    We start by employing the same algorithm used in \citet{Distel2024} and \citet{Distel2025}, which we now describe.

    We start with a tree-decomposition $(J_t:t\in V(T))$ of the graph $G$. We can find a small set $Z\subseteq V(T)$ whose deletion splits $T$ into subtrees $T'$ such that the total number of vertices contained in $\bigcup_{t\in V(T')}J_t$ is small; see \cref{treeDeletions} and \cref{findSplit}. For each $z\in Z\setminus \{r\}$, delete the bounded number of vertices in $J_z\cap J_{p(z)}$ (putting these vertices in the loss-set). The total number of vertices deleted is small, but it has the effect of breaking the graph and tree-decomposition into manageable chunks. Specifically, for each connected component, we can delete one torso, and each connected component of the remainder has small size. So the goal is to almost-partition this torso, which admits an almost-embedding, and then attach the connected components of the remainder.
    
    Assuming we can almost-partition the torso, we focus on how to re-attach the remaining components. This is the first place where our proof diverges from \citet{Distel2024} in a major way. In \citet{Distel2024}, this is a relatively easy exercise, as the nature of the partition found ensures that each component is adjacent to at most three parts (which form a clique in the quotient of the torso). So the component can be added as a new part without increasing the treewidth of the quotient above $3$ (which, when after adding the loss-set, gives treewidth at most $4$). Since we want treewidth $2$ (so that adding the loss-set gives treewidth $3$), the equivalent would be ensuring that each component is adjacent to at most two parts. But we cannot guarantee this. Specifically, we have issues when the component attaches to a facial triangle. To combat this, we introduce `concentrated' almost-partitions.

    Let $(X,\scr{P})$ be an almost-partition of a graph $G$, let $b,m\in \ds{N}$, and let $\scr{K}\subseteq 2^{V(G)}$. We say that $(X,\scr{P})$ is \defn{$(b,m,\scr{K})$-concentrated} if, for each clique $Q$ in $(G-X)/\scr{P}$, there exists $S_Q\subseteq V(G)$ with $|S_Q|\leq m$ such that, for each $K\in \scr{K}$, $K\setminus S_Q$ intersects at most $b$ parts in $Q$.

    As a guide to the reader, the set $\scr{K}$ will always be a set of cliques in $G$ (although it will not always be immediately obvious).
    
    The idea is that the sets in $\scr{K}$ don't `usually' intersect a large number of common parts. The exception occurs within a small `concentrated' region of the graph. Specifically, if we take a given clique of `common parts' we find that all sets (cliques) in $\scr{K}$ that intersect all of these parts all intersect a small region of the graph. The deletion ensures that the sets in $\scr{K}$ no longer share a large common intersection.

    The following is a useful fact.

    \begin{restatable}{observation}{concentratedSpanning}
        \label{concentratedSpanning}
        Let $G$ be a graph, let $G'$ be a spanning subgraph of $G$, and let $(X,\scr{P})$ be an almost-partition of $G$. Then $(X,\scr{P})$ is an almost-partition of $G'$ of the same width and loss and at most the same treewidth. Further, if $b,m\in \ds{N}$, and $\scr{K}\subseteq 2^{V(G)}=2^{V(G')}$ are such that $(X,\scr{P})$ is a $(b,m,\scr{K})$-concentrated almost-partition of $G$, then $(X,\scr{P})$ is also a $(b,m,\scr{K})$-concentrated almost-partition of $G'$.
    \end{restatable}
    \begin{proof}
        Since $V(G)=V(G')$ it is immediate that $(X,\scr{P})$ is an almost-partition of $G'$ of the same width and loss. Further, observe that for each $P\in \scr{P}$, $N_{G'-X}(P)\subseteq N_{G-X}(P)$, and thus $N_{G'-X}(P)$ intersects $P'\in \scr{P}$ distinct to $P$ only if $N_{G-X}(P)$ does. It follows that $(G'-X)/\scr{P}$ is a spanning subgraph of $(G-X)/\scr{P}$, and thus has at most the same treewidth. Further, every clique $Q$ in $(G'-X)/\scr{P}$ is a clique in $(G-X)/\scr{P}$. So if $b,m\in \ds{N}$, and $\scr{K}\subseteq 2^{V(G)}=2^{V(G')}$ are such that $(X,\scr{P})$ is a $(b,m,\scr{K})$-concentrated almost-partition of $G$, then there exists $S\subseteq V(G)=V(G')$ with $|S|\leq m$ such that for each $K\in \scr{K}$, $K\setminus S$ intersects at most $b$ elements in $Q$. It follows that $(X,\scr{P})$ is a $(b,m,\scr{K})$-concentrated almost-partition of $G'$, as desired.
    \end{proof}

    Concentrated partitions address the previous issue. Specifically, we require that $\scr{K}$ includes the `adhesion sets', where the remaining components attach to. These adhesion sets are always cliques in the graph, and so the parts they attach to form cliques in the quotient.

    For each clique in the quotient, we ask `how many vertices attach to exactly that clique' (belong to a component of the remainder that neighbours exactly that set of parts). Only a small number of cliques $Q$ can have a large number of attaching vertices. Call these the `large cliques'. For each large clique $Q$, we delete (put in the loss-set) the set $S_Q$ from the definition of `concentrated'. Since the sets $S_Q$ have constant size, and there are only a small number of large cliques, only a small number of vertices are deleted.

    For each component we want to attach, we ask `how many parts does it attach to (are adjacent to)' after the deletions have been performed. Since these parts form a clique in a graph of treewidth at most $b$ (in our case, $b=2$), there are at most $b+1$ of them. If there are $b$ or fewer, we can attach this entire component as a part without increasing the treewidth, as in \citet{Distel2024}. So presume we are attaching onto exactly $b+1$ parts. Call these parts $Q$.

    Note that for every large clique $Q'$, since we deleted $S_{Q'}$, every component of the remainder attaches to at most $b$ parts in $Q'$. So $Q$ is not a large clique. So if we take all the vertices attaching onto $Q$ and squish them into a single pre-existing part, that part would still be small afterwards. Using the degeneracy of bounded treewidth graphs, we can ensure that each part only receives the vertices attaching to one clique $Q$. So no part is made large. We also find that adding these extra vertices do not create any new edges in the quotient. Thus, the final quotient still has treewidth at most $b$, as desired.

    Recall that this is occurring in each `manageable chunk' of the tree-decomposition. The final partition is just the disjoint union of all the partitions we have found for each chunk. The quotient is then the disjoint union of the quotients of each chunk, which still has treewidth at most $b$. The catch is that we need to ensure that the total number of vertices across all chunks is small. This is why we separate the `loss' and `width' parameters.

    Instead of asking for an almost-partition of the torsos (and corresponding attaching components) with width and loss scaling proportional to the square root of the size of the torso, we instead ask for an almost-partition where the width is some fixed value (the square root of the size of the original graph) and the loss is inversely proportional to the width. This way, the total number of vertices deleted is still small, giving us the desired almost-partition.
    
    This does mean that when we tackle the problem of partitioning almost-embeddings, we will need to be able to arbitrarily decrease the loss, provided that we can increase the width proportionally.

    The formal proof itself is split into two parts. The first finds the set $Z$ that splits the graph up, takes the almost-partitions of each component, and shows that the appropriate union is the desired component. The latter handles the more complicated job of actually finding the almost-partition of each component, using the concentrated almost-partition. The latter stage is handled by \cref{concentratedOnStar}. The former stage is included within the proof of \cref{TDReduced}.

    \subsection{Rooted tree-decompositions}

    We find it helpful to introduce some notation to assist with the proof. Most of this is just creating formal definitions of basic ideas and ideas that existed in \citet{Distel2024} and \citet{Distel2025}. Because we will be frequently working with a tree-decomposition where the tree has a root, we formally define a tuple for it, along with many useful variables related to it.
    
    We say that a \defn{rooted tree-decomposition} $\scr{V}$ is a collection $(G,T,r,\scr{J})$ where $G$ is a graph, $T$ is a tree rooted at $r\in V(T)$, and $\scr{J}$ is a $T$-decomposition of $G$. We call $G$ the \defn{graph}, $T$ the \defn{tree}, $r$ the \defn{root}, and $\scr{J}$ the \defn{decomposition}. We say that $\scr{V}$ is a \defn{rooted tree-decomposition} for $G$, and we say that the \defn{width} of $\scr{V}$ is the width of $\scr{J}$. We always consider the pair $(T,r)$ to be a rooted tree, and use terms like `ancestor' and `induced root' accordingly.
    
    Let $\scr{V}=(G,T,r,\scr{J})$ be a rooted decomposition, and let $(J_t:t\in V(T)):=\scr{J}$. Set \defn{$K_r(\scr{V})$}$:=\emptyset$, and for $t\in V(T)\setminus r$, set \defn{$K_t(\scr{V})$}$:=J_t\cap J_{p(t)}$. We call $K_t(\scr{V})$ the \defn{parent-adhesion} at $t$. For each $t\in V(T)$, set \defn{$J_t^-(\scr{V})$}$:=J_t\setminus K_t(\scr{V})$, \defn{$H_t(\scr{V})$}$:=G\langle J_t\rangle-K_t(\scr{V})$. We call $J_t^-(\scr{V})$ the \defn{reduced bag} at $t$ and $H_t(\scr{V})$ the \defn{reduced torso} at $t$. Note that $V(H_t(\scr{V}))=J_t^-(\scr{V})\subseteq V(G)$. Further, since $K_r=\emptyset$, note that $H_r(\scr{V})$ is the torso at $r$. Thus, for each child $s$ of $r$, observe that $K_s(\scr{V})$ is a (possibly empty) clique in $H_r(\scr{V})$. For each $v\in V(G)$, let \defn{$T_v(\scr{V})$} denote the subgraph of $T$ induced by the vertices $t\in V(T)$ such that $v\in J_t$. Since $\scr{J}$ is a $T$-decomposition of $G$, $T_v(\scr{V})$ is a nonempty subtree of $T$. Let \defn{$r_v(\scr{V})$} denote the induced root of $T_v(\scr{V})$.

    The following observations are useful.

    \begin{observation}
        \label{doubleAdhesionRemove}
        Let $\scr{V}=(G,T,r,\scr{J})$ be a rooted tree-decomposition with $\scr{J}=(J_t:t\in V(T))$, and let $t,t',t''\in V(T)$ be such that $t''$ is an ancestor of or equal to $t'$, and $t'$ is an ancestor of or equal to $t$. Then $J_t\setminus (K_{t'}(\scr{V})\cup K_{t''}(\scr{V}))=J_t\setminus K_{t'}(\scr{V})$.
    \end{observation}

    \begin{proof}
        For each $t\in V(T)$, set $K_t:=K_t(\scr{V})$.
    
        Presume otherwise. Since $J_t\setminus (K_{t'}\cup K_{t''})\subseteq J_t\setminus K_{t'}$ trivially, we find that $J_t\setminus K_{t'}\nsubseteq J_t\setminus (K_{t'}\cup K_{t''})$. So there exists $v\in (J_t\setminus K_{t'})\setminus (J_t\setminus (K_{t'}\cup K_{t''}))=(J_t\cap K_{t''})\setminus K_{t'}$. Note that this implies $t'\neq t''$. Thus, $t''$ is an ancestor of $t'$, and thus $t'\neq r$. So $K_{t'}=J_{t'}\cap J_{p(t')}$. Observe also that $K_{t''}\subseteq J_{t''}$. Thus, $v\in (J_t\cap J_{t''})\setminus (J_{t'}\cap J_{p(t')})$. Set $T_v:=T_v(\scr{V})$. Since $v\in J_t\cap J_{t''}$, observe that $t,t''\in V(T_v)$. Since $T_v$ is a subtree of $T$, since $t'$ is an ancestor of or equal to $t$, and since $t''$ is an ancestor of $t'$, observe that $t',p(t')\in V(T_v)$. Thus, $v\in J_{t'}\cap J_{p(t')}$, a contradiction. The observation follows.
    \end{proof}

    \begin{restatable}{observation}{reducedBagsExact}
        \label{reducedBagsExact}
        Let $\scr{V}=(G,T,r,\scr{J})$ be a rooted tree-decomposition, and let $v\in V(G)$. Then for each $t\in V(T)$, $v\in J_t^-(\scr{V})$ if and only if $t=r_v(\scr{V})$.
    \end{restatable}
    \begin{proof}
        Let $(J_t:t\in V(T)):=\scr{J}$. Note that $J_r^-:=J_r^-(\scr{V})=J_r$, and for each $t\in V(T)\setminus \{r\}$, $J_t^-:=J_t^-(\scr{V})=J_t\setminus J_{p(t)}$. Recall that $T_v:=T_v(\scr{V})$ is the nonempty subtree of $T$ induced the vertices $t\in V(T)$ such that $v\in J_t$, and that $r_v:=r_v(\scr{V})$ is the induced rooted of $T_v$. So $v\notin J_t\supseteq J_t^-$ for each $t\in V(T)\setminus V(T_v)$.
        
        If $r_v=r$, then $v\in J_{r_v}=J_{r_v}^-$, and if $r_v\neq r$, then $v\notin J_{p(r_v)}$ (by definition of $r_v$ and $T_v$), and thus $v\in J_{r_v}\setminus J_{p(r_v)}=J_{r_v}^-$. So in either case, $v\in J_{r_v}^-$.

        For each $t\in V(T_v)\setminus \{r_v\}$, observe that $t\neq r$ and $p(t)\in V(T_v)$. Thus, $v\in J_{p(t)}$ and $v\notin J_t\setminus J_{p(t)}=J_t^-$.
        
        It follows that for each $t\in V(T)$, $v\in J_t^-$ if and only if $t=r_v$, as desired.
    \end{proof}

    Rooted tree-decompositions naturally yield partitions, as given by the following observation.

    \begin{restatable}{observation}{rtdPartition}
        \label{rtdPartition}
        If $\scr{V}=(G,T,r,\scr{J})$ is a rooted tree-decomposition, then $\scr{P}:=(J_t^-(\scr{V}):t\in V(T))$ is a partition of $G$.
    \end{restatable}

    \begin{proof}
        Let $(J_t:t\in V(T)):=\scr{J}$. Note that $J_r^-:=J_r^-(\scr{V})=J_r$, and for each $t\in V(T)\setminus \{r\}$, $J_t^-:=J_t^-(\scr{V})=J_t\setminus J_{p(t)}$. Since $\scr{J}\subseteq 2^{V(G)}$, observe that $\scr{P}\subseteq 2^{V(G)}$. By \cref{reducedBagsExact}, each $v\in V(G)$ is contained in exactly one element of $\scr{P}$ (being $J_{r_v(\scr{V})}^-$). It follows that $\scr{P}$ is a partition of $G$, as desired.
    \end{proof}

    The next thing is that we want to be able to take about the subgraph `associated' with a subtree, along with the `natural' tree-decomposition and root. This leads to the following definition.
    
    Given a rooted tree-decomposition $\scr{V}=(G,T,r,\scr{J})$ and a subtree $T'$ of $T$, let $r'$ be the induced root of $T'$ and let $G':=G[\bigcup_{t\in V(T')}J_t^-(\scr{V})]$. Let $(J_t:t\in V(T)):=\scr{J}$, and for each $t\in V(T)$, let $J_t':=J_t\setminus K_{r'}(\scr{V})$. Let $\scr{J}':=(J_t':t\in V(T'))$, and define \defn{$\scr{V}[T']$}$:=(G',T',r',\scr{J}')$. We call $\scr{V}[T']$ the rooted tree-decomposition \defn{induced} by $T'$.

    \begin{restatable}{lemma}{rtdSub}
        \label{rtdSub}
        Let $\scr{V}=(G,T,r,\scr{J})$ be a rooted tree-decomposition, and let $T'$ be a subtree of $T$ with induced root $r'$. Then $\scr{V}[T']$ is a rooted tree-decomposition. Further, if $G'$ is the graph of $\scr{V}[T']$, then for each $v\in V(G')$ and each $t\in V(T')$:
        \begin{enumerate}
            \item $T_v(\scr{V}[T'])=T_v(\scr{V})$,
            \item $r_v(\scr{V}[T'])=r_v(\scr{V})$,
            \item $K_t(\scr{V}[T'])=K_t(\scr{V})\setminus K_{r'}(\scr{V})$,
            \item $J_t^-(\scr{V}[T'])=J_t^-(\scr{V})$, and
            \item $H_t(\scr{V}[T'])$ is a spanning subgraph of $H_t(\scr{V})$.
        \end{enumerate}
    \end{restatable}

    \begin{proof}
        Let $(J_t:t\in V(T)):=\scr{J}$. For each $v\in V(G)$, recall that $T_v:=T_v(\scr{V})$ is the nonempty subtree of $T$ induced the vertices $t\in V(T)$ such that $v\in J_t$, and that $r_v:=r_v(\scr{V})$ is the induced rooted of $T_v$.

        For each $t\in V(T')$, let $J_t':=J_t\setminus K_{r'}(\scr{V})$. Let $\scr{J}':=(J_t':t\in V(T'))$, and let $G':=G[\bigcup_{t\in V(T')}J_t^-(\scr{V})]$. Recall that $\scr{V}[T']=(G',T',r',\scr{J}')$. By \cref{reducedBagsExact}, observe that $v\in V(G)$ is in $V(G')$ if and only if $r_v\in V(T')$ (in which case, we have $v\in J_{r_v}^-$).
        
        We first show that $\scr{V}[T']$ is a rooted tree-decomposition. By definition, $T'$ is a tree and $r'\in V(T')$. So it remains only to show that $\scr{J}'$ is a $T'$-decomposition of $G'$.
        
        We first need to show that $\scr{J}'\subseteq 2^{V(G')}$. Fix $t\in V(T')$ and $v\in J_t'$. Note that $J_t'\subseteq J_t\subseteq V(G)$, so $v\in V(G)$, and thus $T_v$ is well-defined. Since $v\in J_t'\subseteq J_t$ (and since $t\in V(T')\subseteq V(T)$), note that $t\in V(T_v)$. Let $r_v$ be the induced root of $T_v$. Since $t\in V(T_v)$, $r_v$ is either $t$ or an ancestor of $t$.
        
        Presume, for a contradiction, that $r_v$ is an ancestor of $r'$. We then find that $r'\neq r$ and that $p(r')\in V(T_v)$. Thus, $v\in J_{r'}\cap J_{p(r')}=K_{r'}(\scr{V})$. This is a contradiction, as $v\in J_t'=J_t\setminus K_{r'}(\scr{V})$. So $r_v$ is not an ancestor of $r'$. Thus and since $r'$ is the induced root of $T'$, since $t\in V(T')$, and since $r_v$ is an ancestor or equal to $t$, we find that $r_v\in V(T')$, and thus $v\in V(G')$. Thus, $J_t'\subseteq V(G')$, and hence $\scr{J}'\subseteq 2^{V(G')}$, as desired.

        Next, for each $v\in V(G')$, let $T_v'$ be the subgraph of $T'$ induced by the vertices $t\in V(T)$ such that $v\in J_t'$. Since $v\in V(G')$, recall that $r_v\in V(T')$ and $v\in J_{r_v}^-(\scr{V})$. If $r'=r$, then observe that $v\notin \emptyset=K_{r'}(\scr{V})$, and if $r'\neq r$, then observe that $p(r')\notin V(T_v)$ (by definition of $r_v,r'$ and since $r_v\in V(T')$) and thus $v\notin J_{p(r')}\supseteq K_{r'}(\scr{V})$. So in either case, we have $v\notin K_{r'}(\scr{V})$. Thus, observe that for each $t\in V(T)$, $v\in J_t'=J_t\setminus K_{r'}(\scr{V})$ if and only if $v\in J_t$. Thus, $V(T_v')=V(T_v)$, and $T_v'=T_v$ (as both subgraphs are induced). In particular, $T_v'$ is a nonempty and connected subgraph (subtree) of $T$, and the induced root $r_v'$ of $T_v'$ is $r_v$.

        Let $u,v\in V(G')$ be adjacent in $G'$. Since $G'\subseteq G$, $u,v$ are adjacent in $G$. So there exists $t\in V(T)$ such that $u,v\in J_t$. So $t\in V(T_u)\cap V(T_v)$. Since $u,v\in V(G')$, we have $T_u=T_u'$ and $T_v=T_v'$. Thus, $t\in V(T_u')\cap V(T_v')\subseteq V(T')$. So $t\in V(T')$ and $u,v\in J_t'$. It follows that $\scr{J}'$ is a $T'$-decomposition of $G'$, and thus $(G',T',r',\scr{J}')=\scr{V}[T']$ is a rooted tree-decomposition. Further, observe that for each $v\in V(G')$, $T_v(\scr{V}[T'])=T_v'=T_v$ and $r_v(\scr{V}[T'])=r_v'=r_v$.

        Set $\scr{V}':=\scr{V}[T']$ for convenience.

        Consider any $t\in V(T')$. If $t=r'$, observe that $K_t(\scr{V}')=\emptyset=K_{r'}(\scr{V})\setminus K_{r'}(\scr{V})=K_t(\scr{V})\setminus K_{r'}(\scr{V})$. Otherwise, $t\neq r'$, then observe that $t\neq r$, and that the parent $p(t)$ of $t$ in $T$ is also the parent of $t$ in $T'$. We then find that $K_t(\scr{V}')=J_t'\setminus J_{p(t)}'=(J_t\setminus K_{r'}(\scr{V}))\cap (J_{p(t)}\setminus K_{r'}(\scr{V}))=(J_t\cap J_{p(t})\setminus K_{r'}(\scr{V})=K_t(\scr{V})\setminus K_{r'}(\scr{V})$. So in either scenario, $K_t(\scr{V}')=K_t(\scr{V})\setminus K_{r'}(\scr{V})$.

        Observe that $J_t^-(\scr{V}')=J_t'\setminus K_t(\scr{V}')=(J_t\setminus K_{r'}(\scr{V}))\setminus (K_t(\scr{V})\setminus K_{r'}(\scr{V}))=(J_t\setminus K_t(\scr{V}))\setminus K_{r'}(\scr{V})=J_t\setminus (K_t(\scr{V})\cup K_{r'}(\scr{V}))$. Observe that $r'$ is an ancestor of $t$ as $t\in V(T')$. Thus, by \cref{doubleAdhesionRemove}, $J_t^-(\scr{V}')=J_t\setminus (K_t(\scr{V})\cup K_{r'}(\scr{V}))=J_t\setminus K_t(\scr{V})=J_t^-(\scr{V})$. Thus, $J_t^-(\scr{V}')=J_t^-(\scr{V})$.

        Recall that $H_t(\scr{V}')=G'\langle J_t'\rangle - K_t(\scr{V}')$, where $G'\langle J_t'\rangle$ is the graph obtained from $G'[J_t']$ by, for each $t'\in V(T')$ adjacent to $t$ in $T'$ and each pair of distinct non-adjacent vertices $u,v\in J_t'\cap J_{t'}'$ in $G'[J_t']$, adding a new edge with endpoints $u,v$. In particular, recall that $V(H_t(\scr{V}'))=J_t^-(\scr{V}')=J_t^-(\scr{V})=V(H_t(\scr{V}))$.

        Since $G'$ is an induced subgraph of $G$, observe that $G'[J_t']=G[J_t']=G[J_t\setminus K_{r'}(\scr{V})]$. Thus, $G'[J_t]$ is an induced subgraph of $G[J_t]$. Note that for each $t'\in V(T')$ adjacent to $t$ in $T'$, $t'$ is adjacent to $t$ in $T$. Further, for each pair of non-adjacent vertices $u,v\in J_t'\cap J_{t'}'=(J_t\cap J_{t'})\setminus K_{r'}(\scr{V})$ in $G'[J_t']$, we have that $u,v\in J_t\cap J_{t'}$ and are non-adjacent in $G[J_t]$. It follows that each edge of $G\langle J_t\rangle$ is also an edge of $G'\langle J_t'\rangle$. Observe that the edges of $H_t(\scr{V})$ and $H_t(\scr{V}')$ are precisely the edges of $G\langle J_t\rangle$ and $G'\langle J_t'\rangle$ respectively with both endpoints in $V(H_t(\scr{V}))=V(H_t(\scr{V}'))=J_t^-(\scr{V})$. It follows that $E(H_t(\scr{V}'))\subseteq E(H_t(\scr{V}))$. Thus, $H_t(\scr{V}')$ is a spanning subgraph of $H_t(\scr{V})$, as desired.

        This completes the proof.
    \end{proof}

    We use the fact that $\scr{V}[T']$ is a rooted tree-decomposition implicitly from now on.

    Induced rooted tree-decompositions are well-behaved under composition, as given by the following observation.

    \begin{restatable}{observation}{rtdSubSub}
        \label{rtdSubSub}
        Let $\scr{V}=(G,T,r,\scr{J})$ be a rooted tree-decomposition, let $T'$ be a subtree of $T$, and let $T''$ be a subtree of $T'$. Then $\scr{V}[T'']=(\scr{V}[T'])[T'']$.
    \end{restatable}

    \begin{proof}
        Let $(J_t:t\in V(T)):=\scr{J}$, and let $r',r''$ be the induced roots of $T',T''$ in $T$ respectively. Observe that $r''$ is also the induced root of $T''$ in $T'$.
        
        For each $t\in V(T')$, let $J_t':=J_t\setminus K_{r'}(\scr{V})$. For each $t\in V(T'')$, let $J_t'':=J_t\setminus K_{r''}(\scr{V})$, and let $J_t^*:=J_t'\setminus K_{r''}(\scr{V}[T'])$. Let $\scr{J}':=(J_t':t\in V(T'))$, $\scr{J}'':=(J_t':t\in V(T''))$, and $\scr{J}^*:=(J_t^*:t\in V(T''))$. Let $G':=G[\bigcup_{t\in V(T')}J_t^-(\scr{V})]$, $G'':=G[\bigcup_{t\in V(T'')}J_t^-(\scr{V})]$, and $G^*:=G'[\bigcup_{t\in V(T'')}J_t^-(\scr{V}[T'])]$.
        
        Recall that $\scr{V}[T']=(G',T',r',\scr{J}')$, $\scr{V}[T'']=(G'',T'',r'',\scr{J}'')$, and $(\scr{V}[T'])[T'']=\linebreak(G^*,T'',r'',\scr{J}^*)$. So to prove the lemma, we must show that $G''=G^*$ and $\scr{J}''=\scr{J}^*$.

        By \cref{rtdSub}, $K_{r''}(\scr{V}[T'])=K_{r''}(\scr{V})\setminus K_{r'}(\scr{V})$, and for each $t\in V(T'')$, $J_t^-(\scr{V}[T'])=J_t^-(\scr{V})$. Since $G',G^*$ are induced subgraphs of $G$ and $G'$ respectively, observe that $G^*:=G'[\bigcup_{t\in V(T'')}J_t^-(\scr{V}[T'])]=G[\bigcup_{t\in V(T'')}J_t^-(\scr{V}[T'])]=G[\bigcup_{t\in V(T'')}J_t^-(\scr{V})]=G''$. So $G^*=G''$
        
        Further, for each $t\in V(T'')$, $J_t^*=J_t'\setminus K_{r''}(\scr{V}[T'])=(J_t\setminus K_{r'}(\scr{V}))\setminus (K_{r''}(\scr{V})\setminus K_{r'}(\scr{V}))=(J_t\setminus K_{r''}(\scr{V}))\setminus K_{r'}(\scr{V})=J_t\setminus (K_{r''}(\scr{V})\cup K_{r'}(\scr{V}))$. Since $t\in V(T'')$ and $T''\subseteq T'$, observe that $r'$ is an ancestor or equal to $r''$, and $r''$ is an ancestor or equal to $t$. Thus, by \cref{doubleAdhesionRemove}, $J_t^*=J_t\setminus (K_{r''}(\scr{V})\cup K_{r'}(\scr{V}))=J_t\setminus K_{r''}(\scr{V})=J_t''$. Thus, $\scr{J}^*=\scr{J}''$. This completes the proof.
    \end{proof}

    Finally, the number of vertices associated with a given subtree is a very useful parameter, because it relates to the idea of how big the components we are `reattaching' are. Remember that we are actually reattaching after deleting one torso (which will be the one for the induced root), so it is also useful to consider how big the subgraphs are after that torso is removed. This leads to the following definitions.
    
    Let $\scr{V}$ be a rooted tree-decomposition $\scr{V}$ with tree $T$, let $T'$ be a subtree of $T$ with induced root $r'$, and let $G'$ be the graph of $\scr{V}[T']$. Say that the \defn{$\scr{V}$-weight} of $T'$ is $|V(G')|$. Say that the \defn{near $\scr{V}$-weight} of $T'$ is the maximum $\scr{V}$-weight of a subtree of $T'-r'$, and say that the \defn{near-weight} of $\scr{V}$ is the near $\scr{V}$-weight of $T$.

    \begin{observation}
        \label{weightSubSub}
        Let $\scr{V}=(G,T,r,\scr{J})$ be a rooted tree-decomposition, let $T'$ be a subtree of $T$, and let $T''$ be a subtree of $T'$. Then the $\scr{V}$-weight of $T''$ equals the $\scr{V}[T']$-weight of $T''$.
    \end{observation}

    \begin{proof}
        By \cref{rtdSubSub}, $(\scr{V}[T'])[T'']=\scr{V}[T'']=(G'',T'',r'',\scr{J}'')$. Thus, observe that the $\scr{V}$-weight of $T''$ and the $\scr{V}[T']$-weight of $T''$ are both $|V(G'')|$.
    \end{proof}

    \begin{corollary}
        \label{nearWeightSub}
        Let $\scr{V}=(G,T,r,\scr{J})$ be a rooted tree-decomposition, and let $T'$ be a subtree of $T$. Then the near $\scr{V}$-weight of $T'$ equals the near-weight of $\scr{V}[T']$.
    \end{corollary}

    \begin{proof}
        Let $r'$ be the induced root of $T'$. Recall that the near $\scr{V}$-weight of $T'$ is the maximum $\scr{V}$ weight of a subtree $T''$ of $T'-r'$. Recall also that $r'$ is the root of $\scr{V}[T']$, and that the near-weight of $\scr{V}[T']$ is the near $\scr{V}[T']$-weight of $T'$. Thus, the near-weight of $\scr{V}[T']$ is the maximum $\scr{V}[T']$-weight of a subtree $T''$ of $T'-r'$.
        
        By \cref{weightSubSub}, for any subtree $T''$ of $T'-r'$, the $\scr{V}[T']$-weight of $T''$ is the $\scr{V}$-weight of $T''$. Thus, the maximum $\scr{V}[T']$-weight of a subtree $T''$ of $T'-r'$ is the maximum $\scr{V}$-weight of a subtree $T''$ of $T'-r'$. Hence, the $\scr{V}$-weight of $T'$ equals the near-weight of $\scr{V}[T']$, as desired.
    \end{proof}
    
    \subsection{Splits}

    Let $\scr{V}$ be a rooted tree-decomposition and let $Z\subseteq V(T(\scr{V}))$ be a rooted set (with respect to $r(\scr{V})$). For each $z\in Z$, let $T_z$ be the maximal subtree of $T(\scr{V})$ with induced root $z$ whose vertices are disjoint to $Z\setminus \{z\}$, and let $\scr{V}_z:=\scr{V}[T_z]$. We say that the \defn{$Z$}-split of $\scr{V}$ is $(\scr{V}_z:z\in Z)$. Observe that $(V(T_z):z\in Z)$ is a partition of $T(\scr{V})$ (as $r(\scr{V})\in Z$). Recalling that $(J_t^-:t\in T(\scr{V})$ is a partition of $G(\scr{V})$, it follows that $(V(G(\scr{V}_z)):z\in Z)$ is also partition of $G$. Recalling the properties of $(J_t^-:t\in T(\scr{V})$), observe that for distinct $z,z'\in Z$, if there is an edge in $G(\scr{V})$ whose endpoints are in $V(G(\scr{V}_z))$ and $V(G(\scr{V}_{z'}))$ then $z$ and $z'$ are related. Further, if $z'$ is an ancestor of $z$, then the endpoint in $V(G(\scr{V}_{z'}))$ is contained in $K_z(\scr{V})$.

    Let $\scr{V}$ be a rooted tree-decomposition with tree $T$, and let $Z\subseteq V(T)$ be a rooted set (contain the root). For each $z\in Z$, let $T_z$ be the maximal subtree of $T$ with induced root $z$ whose vertices are disjoint to $Z\setminus \{z\}$. Then, let $\scr{V}_z:=\scr{V}[T_z]$. We say that the \defn{$Z$-split} of $\scr{V}$ is $(\scr{V}_z:z\in Z)$. A \defn{split} of $\scr{V}$ is a $Z$-split of $\scr{V}$ for some rooted set $Z\subseteq V(T)$.

    The following observation parallels \cref{reducedBagsExact}.

    \begin{restatable}{observation}{splitExact}
        \label{splitExact}
        Let $\scr{V}=(G,T,r,\scr{J})$ be a rooted tree-decomposition, let $(\scr{V}_z:z\in Z)$ be a split of $\scr{V}$. Then for each $z\in Z$, if $V(G_z)$ is the graph of $\scr{V}_z$, then $v\in V(G)$ is in $V(G_z)$ if and only if $z$ is the first vertex in $Z$ on the path from $r_v(\scr{V})$ to $r$.
    \end{restatable}

    \begin{proof}
        For each $z\in Z$, recall that $\scr{V}_z=\scr{V}[T_z]$, where $T_z$ is the maximal subtree of $T$ with induced root $z$ whose vertices are disjoint to $Z\setminus \{z\}$. Since $\scr{V}_z=\scr{V}[T_z]$, recall that $G_z=G[\bigcup_{t\in V(T_z)}J_t^-{\scr{V}}]$. So for each $v\in V(G)$, $v\in V(G_z)$ if and only if there exists $t\in V(T_z)$ such that $v\in J_t^-{\scr{V}}$.
        
        For each $v\in V(G)$ and $t\in V(T)$, by \cref{reducedBagsExact}, $v\in J_t^-$ if and only if $t=r_v(\scr{V})$. Observe that for each $z\in Z$, $r_v\in V(T_z)$ if and only if $z$ is the first vertex in $Z$ on the path from $r_v$ to $z$ (including endpoints). The result follows.
    \end{proof}

    From \cref{splitExact}, we can precisely describe how the graphs $G_z$ in the split interact with each other.

    \begin{restatable}{observation}{splitBasicPartition}
        \label{splitBasicPartition}
        Let $\scr{V}=(G,T,r,\scr{J})$ be a rooted tree-decomposition, and let $(\scr{V}_z:z\in Z)$ be a split of $\scr{V}$. For each $z\in Z$, let $G_z$ be the graph of $\scr{V}_z$. Then $\scr{P}:=(V(G_z):z\in Z)$ is a partition of $G$.
    \end{restatable}

    \begin{proof}
        For each $z\in Z$, recall that $\scr{V}_z=\scr{V}[T_z]$, where $T_z$ is the maximal subtree of $T$ with induced root $z$ whose vertices are disjoint to $Z\setminus \{z\}$. Since $\scr{V}_z=\scr{V}[T_z]$, recall that $G_z=G[\bigcup_{t\in V(T_z)}J_t^-{\scr{V}}]$. Thus, $V(G_z)\subseteq V(G)$. So $\scr{P}\subseteq 2^{V(G)}$.

        For each $v\in V(G)$ and $z\in Z$, by \cref{splitExact}, $v\in V(G_z)$ if and only if $z$ is the first vertex in $Z$ on the path from $r_v(\scr{V})$ to $r$. So $v$ is contained in at most one element of $\scr{P}$. Further, since $Z$ is rooted, $r\in Z$, and thus there is a first vertex in $Z$ on the path from $r_v(\scr{V})$ to $r$. So $v$ is contained exactly one element of $\scr{P}$. Thus, $\scr{P}$ is a partition of $G$, as desired.
    \end{proof}

    \begin{restatable}{observation}{splitEdges}
        \label{splitEdges}
        Let $\scr{V}=(G,T,r,\scr{J})$ be a rooted tree-decomposition of a graph $G$, and let $(\scr{V}_z:z\in Z)$ be a split of $\scr{V}$. Let $z,z'\in Z$ be distinct, and let $G_z,G_{z'}$ be the graphs of $\scr{V}_z$ and $\scr{V}_{z'}$ respectively. If $u\in V(G_z)$ and $v\in V(G_{z'})$ are adjacent in $G$, then $u,v$ are distinct, $z,z'$ are related, and $\{u,v\}$ intersects $K_z(\scr{V})\cup K_{z'}(\scr{V})$. Further, if $z$ is an ancestor of $z'$, then $u\in K_{z'}(\scr{V})$.
    \end{restatable}

    \begin{proof}
        The fact that $u,v$ are distinct comes from \cref{splitBasicPartition}, as $z,z'$ are distinct.

        For each $x\in \{u,v\}$, set $r_x:=r_x(\scr{V})$ and $T_x:=T_x(\scr{V})$.

        By \cref{splitExact}, $z$ is the first vertex in $Z$ on the path from $r_u$ to $r$, and $z'$ is the first vertex in $Z$ on the path from $r_v$ to $r$. In particular, $z$ is an ancestor or equal to $r_u$, and $z'$ is an ancestor or equal to $r_v$.
        
        Let $(J_t:t\in V(T)):=\scr{J}$. Since $u,v$ are adjacent in $G$, there exists $t\in V(T)$ such that $u,v\in J_t$. So $T_u$ and $T_v$ intersect at $t$. Thus, both $r_u$ and $r_v$ are either equal to or an ancestor of $t$. It follows that both $z$ and $z'$ are either equal to or an ancestor of $t$. Hence and since $z\neq z'$, we find that $z$ and $z'$ are related.
        
        Presume further that $z$ is an ancestor of $z'$. Note that $z'\neq r$. Since $z\neq z'$ is the first vertex in $Z$ on the path from $r_u$ to $r$, we find that $r_u$ is neither a descendant of nor equal to $z'$. Since $z'$ is an ancestor or equal to $t$, and since $t,r_u\in V(T_u)$, we find that $z',p(z')\in V(T_u)$. So $u\in J_{z'}\cap J_{p(z')}=K_{z'}(\scr{V})$, as desired.
    \end{proof}

    Let $\scr{V}$ be a rooted tree-decomposition, and let $(\scr{V}_z:z\in Z)$ be a split of $\scr{V}$. We say that the \defn{size} of $(\scr{V}_z:z\in Z)$ is $|Z|$, and the \defn{near-weight} of $(\scr{V}_z:z\in Z)$ is the maximum near-weight of some $\scr{V}_z$ with $z\in Z$.

    The first step of the algorithm is essentially to find a split of small size and near-weight. To do this, we employ the following result from \citet{Distel2024}.

    \begin{lemma}
        \label{treeDeletions}
        For each $q\in \ds{N}$ and every $n\in \ds{R}_0^+$, every weighted tree $T$ with weight at most $n$ has a set $Z'$ of at most $q$ vertices such that each component of $T-Z$ has weight at most $n/(q+1)$. 
    \end{lemma}

    We can convert \cref{treeDeletions} into the language of splits.

    \begin{corollary}
        \label{findSplit}
        Let $d\in \ds{R}^+$, and let $\scr{V}$ be a rooted tree-decomposition of a graph $G$ on at most $n$ vertices. Then there exists a split of $\scr{V}$ of size at most $n/d+1$ and near-weight at most $d$.
    \end{corollary}

    \begin{proof}
        Set $(G,T,r,\scr{J}):=\scr{V}$, and for each $t\in V(T)$, set $J_t^-:=J_t^-(\scr{V})$. Set $q:=\floor{n/d}$. Note that $q\leq n/d$ and $n/(q+1)\leq d$.
        
        Weight each $t\in V(T)$ by $|J_t^-|$. By \cref{rtdPartition}, $(J_t^-:t\in V(T))$ is a partition of $G$. Thus, the total weight on $T$ is exactly $|V(G)|\neq n$. By \cref{treeDeletions}, there exists $Z'\subseteq V(T)$ with $|Z'|\leq q\leq n/d$ such that each connected component of $T-Z'$ has total weight at most $n/(q+1)\leq d$. Let $Z:=Z'\cup \{r\}$. So $Z$ is rooted in $T$, and $|Z|\leq n/d+1$. Let $(\scr{V}_z:z\in Z)$ be the $Z$-split of $\scr{V}$. So $(\scr{V}_z:z\in Z)$ has size at most $n/d+1$.
        
        Fix $z\in Z$, and let $T_z$ be the tree of $\scr{V}_z$. Recall that $\scr{V}_z=\scr{V}[T_z]$. Recall that $T_z$ is the maximal of subtree of $T$ with induced root $z$ whose vertices are disjoint to $Z\setminus \{z\}$. By \cref{weightSubSub}, the near-weight of $\scr{V}_z$ is the near $\scr{V}$-weight of $T_z$ in $\scr{V}$. Recall that the near $\scr{V}$-weight of $T_z$ in $\scr{V}$ is the maximum $\scr{V}$-weight of a subtree $T'$ of $T_z-z$ (as $z$ is the induced root of $T_z$).
        
        Fix a subtree $T'$ of $T_z-z$. Let $G_{T'}$ be the graph of $\scr{V}[T']$. Recall that the $\scr{V}$-weight of $T'$ is $|V(G_{T'})|$. Further, recall that $V(G_{T'})=\bigcup_{t\in V(T')}J_t^-$. Thus, the $\scr{V}$-weight of $T'$ is exactly the total weight on $T'$. Since $V(T_z)$ is disjoint to $Z\setminus \{z\}$, observe that $V(T')$ is disjoint to $Z$. Since $Z'\subseteq Z$, observe that $T'$ is contained in a connected component of $T-Z'$. Hence, $T'$ has total weight at most $d$. Thus, $T'$ has $\scr{V}$-weight at most $d$. Hence, $T_z$ has $\scr{V}_z$ near-weight at most $d$, and $\scr{V}_z$ has near-weight at most $d$, as desired.
    \end{proof}

    \subsection{Decomposability}
    \label{SecDecompos}

    We now formally define our `decomposability' parameter.

    Let $\scr{V}$ be a rooted tree-decomposition with tree $T$. For $b,m,c,q\in \ds{N}$ and $w\in \ds{R}^+$, a triple $(X,\scr{P},\scr{K})$ is a \defn{$(b,m,w,c,q)$-decomposition} for $\scr{V}$ at $t$ if: 
    \begin{enumerate}
        \item $\scr{K}\subseteq 2^{J_t^-(\scr{V})}$, 
        \item $\scr{K}$ is closed under subsets,
        \item for each child $t'$ of $t$, $K_{t'}(\scr{V})\setminus K_t(\scr{V})\in \scr{K}$, and,
        \item $(X,\scr{P})$ is a $(b,m,\scr{K})$-concentrated almost-partition of $H_t(\scr{V})$ of treewidth at most $b$, width at most $w$, and loss at most $c|J_t^-(\scr{V})|/w + q$.
    \end{enumerate}

    An important property of this decomposability parameter is that it is `preserved' under taking induced rooted tree-decompositions.

    \begin{lemma}
        \label{decompositionStays}
        Let $b,m,c,q\in \ds{N}$ and $w\in \ds{R}^+$, let $\scr{V}$ be a rooted tree-decomposition with tree $T$, let $T'$ be a subtree of $T$, and let $t\in V(T')$. If $(X,\scr{P},\scr{K})$ is a $(b,m,w,c,q)$-decomposition for $\scr{V}$ at $t$, then $(X,\scr{P},\scr{K})$ is also a $(b,m,w,c,q)$-decomposition for $\scr{V}[T']$ at $t$.
    \end{lemma}

    \begin{proof}
        Let $r'$ be the induced root of $T'$. By \cref{rtdSub}, $J_t^-(\scr{V}[T'])=J_t^-(\scr{V})$. So $\scr{K}\subseteq 2^{J_t^-(\scr{V})}=2^{J_t^-(\scr{V}[T'])}$. Recall that $\scr{K}$ is closed under subsets. For each child $t'$ of $t$ in $T'$, observe that $t'$ is a child of $t$ in $T$.
        
        By \cref{rtdSub}, $K_{t'}(\scr{V}[T'])\setminus K_t(\scr{V}[T'])=(K_{t'}(\scr{V})\setminus K_{r'}(\scr{V}))\setminus (K_t(\scr{V})\setminus K_{r'}(\scr{V}))=(K_{t'}(\scr{V})\setminus K_t(\scr{V}))\setminus K_{r'}(\scr{V})$. Since $\scr{K}$ is closed under subsets and since $K_{t'}(\scr{V})\setminus K_t(\scr{V})\in \scr{K}$, we obtain $K_{t'}(\scr{V}[T'])\setminus K_t(\scr{V}[T'])\in \scr{K}$.
        
        Also by \cref{rtdSub}, $H_t(\scr{V}[T'])$ is a spanning subgraph of $H_t(\scr{V})$. By \cref{concentratedSpanning}, $(X,\scr{P})$ is a $(b,m,\scr{K})$-concentrated almost-partition of $H_t(\scr{V}[T'])$ of treewidth at most $b$, width at most $w$, and loss at most $c|J_t^-(\scr{V})|/w + q=c|J_t^-(\scr{V}[T'])|/w + q$. This completes the proof.
    \end{proof}
    
    For $b,m,c,q,j\in \ds{N}$ and $w\in \ds{R}^+$, say that a graph $G$ is \defn{$(b,m,w,c,q,j)$-decomposable} if there exists a rooted tree-decomposition $\scr{V}=(G,T,r,\scr{J})$ such that: 
    \begin{enumerate}
        \item $\scr{J}$ has adhesion at most $j$, and,
        \item for each $t\in V(T)$, there exists a $(b,m,w,c,q)$-decomposition for $\scr{V}$ at $t$.
    \end{enumerate}

    \subsection{Reattaching to the torso}

    We can now finally proceed to the proof. We start by proving an intermediate lemma, regarding re-attaching the small components to the almost-partition of a torso, as explained earlier. We first need the following observation.

    \begin{restatable}{observation}{APIncX}
        \label{APIncX}
        Let $(X',\scr{P}')$ be an almost-partition of a graph $G$, let $X\subseteq V(G)$ be such that $X'\subseteq X$, and let $\scr{P}:=(P\setminus X:P\in \scr{P}')$. Further, the treewidth and width of $(X,\scr{P})$ are at most the treewidth and width respectively of $(X',\scr{P}')$.
    \end{restatable}

    \begin{proof}
        Since $(X',\scr{P}')$ is an almost-partition of $G$, $\scr{P}'\subseteq 2^{V(G-X')}$, and each $v\in V(G-X')$ is contained in exactly one $P\in \scr{P}$. Since $X'\subseteq X$, $V(G-X)=V(G-X')\setminus X$. Thus, $\scr{P}\subseteq 2^{V(G-X')\setminus X}=2^{V(G-X)}$, and each $v\in V(G-X')\setminus X=V(G-X)$ is contained in $P'\setminus X$ for exactly one $P'\in \scr{P}'$. Hence, $v\in P$ for exactly one $P\in \scr{P}$. Thus, $\scr{P}$ is a partition of $G-X$, and  $(X,\scr{P})$ is an almost-partition of $G$, as desired.

        For each $P\in \scr{P}$, $P=P'\setminus X$ for some $P'\in \scr{P}'$. Thus, $|P|\leq |P'|$. Thus, if $\scr{P}'$ has width $w$, then $|P|\leq |P'|\leq w$. So $\scr{P}$ has width at most $w$. Hence, the width of $(X,\scr{P})$ is at most the width of $(X',\scr{P}')$.

        Observe that we can define a bijection $\psi:\scr{P}\mapsto \scr{P}'$ such that $P=\psi(P)\setminus X$ for each $P\in \scr{P}$. For distinct $P_1,P_2\in \scr{P}$, observe that $N_{G-X}(P_1)\subseteq N_{G-X'}(\psi(P_1))$ (as $G-X\subseteq G_X'$ since $X'\subseteq X$) and $P_2\subseteq \psi(P_2)$. Thus, $N_{G-X}(P_1)$ intersects $P_2$ only if $N_{G-X'}(\psi(P_1))$ intersects $\psi(P_2)$. It follows that $(G-X)/\scr{P}$ can be obtained from $(G-X')/\scr{P}'$ by relabelling according to $\psi^{-1}$ and then deleting edges. This does not increase treewidth. Thus, the treewidth of $(X,\scr{P})$ is at most the treewidth of $(X',\scr{P}')$. This completes the proof.
    \end{proof}

    \begin{lemma}
        \label{concentratedOnStar}
        Let $b,m,c,q,n\in \ds{N}$ and $d\in \ds{R}^+$, and let $\scr{V}=(G,T,r,\scr{J})$ be a rooted tree-decomposition of an $n$-vertex graph $G$ that has near-weight at most $d$ and admits a $(b,m,d/2,c,q)$-decomposition at $r$. Then $G$ admits an almost-partition $(X,\scr{P})$ of treewidth at most $b$, width at most $d$, and loss at most $2(c+m)|V(G)|/d+q$.
    \end{lemma}

    \begin{proof}
        Let $(J_t:t\in V(T)):=\scr{J}$. For each $t\in V(T)$, set $K_t:=K_t(\scr{V})$ and $J_t^-:=J_t^-(\scr{V})$.
        
        Let $S$ denote the set of children of $r$. Note that for each $s\in S$, $K_s\subseteq K_r=K_s$ as $K_r=\emptyset$.
        
        Let $H_r:=H_r(\scr{V})=G\langle J_r\rangle$. Recall that $V(H_r)=J_r^-=J_r$. Note that for each $s\in S$, $K_s=J_s\cap J_r$ is a (possibly empty) clique in $H_r$. Since $\scr{V}$ is $(b,m,d/2,c,q)$-decomposable at $r$, there exists a collection $(X',\scr{P}',\scr{K})$ such that:

        \begin{enumerate}
            \item $\scr{K}\subseteq 2^{J_r^-}$,
            \item for each $s\in S$, $K_s=K_s\setminus K_r\in \scr{K}$, and,
            \item $(X',\scr{P}')$ is a $(b,m,\scr{K})$-concentrated almost-partition of $H_r$ of treewidth at most $b$, width at most $d/2$, and loss at most $c|V(H_r)|/(d/2)+q=2c|V(H_r)|/d+q$.
        \end{enumerate}

        Recall that $V(H_r)=J_r\subseteq V(G)$. Thus $(X_r,\scr{P}_r)$ has loss at most $2cn/d+q$. So $|X_r|\leq 2cn/d+q$

        Let $\scr{Q}$ be the set of cliques in $(H_r-X_r)/\scr{P}_r$. Since $(X,\scr{P})$ has treewidth at most $b$, $(H_r-X_r)/\scr{P}_r$ has treewidth at most $b$, and thus each $Q\in \scr{Q}$ has size at most $b+1$. Since $(X_r,\scr{P}_r)$ is $(b,m,\scr{K})$-concentrated, for each $Q\in \scr{Q}$, there exists $X_Q\subseteq V(H_r)$ with $|X_Q|\leq m$ such that for each $K\in \scr{K}$, $K\setminus X_Q$ intersects at most $b$ parts in $Q$.
        
        Let $Z:=\{r\}\cup S$. Observe that $Z$ is rooted. Let $(\scr{V}_z:z\in Z)$ be the $Z$-split of $T$. Recall that for each $z\in Z$, $\scr{V}_z=\scr{V}[T_z]$, where $T_z$ is the maximal subtree of $T$ with induced root $z$ whose vertices are disjoint to $Z\setminus \{z\}$. So $\scr{V}_z$ is of the form $(G_z,T_z,z,\scr{J}_z)$, where $G_z=\bigcup_{t\in V(T_z)}J_t^-$. For each $s\in S\subseteq Z$, set $V_s:=V(G_s)$.
        
        Since $S$ is exactly the children of $r$, it follows that $T_r$ is the subtree containing only the vertex $r$. Thus, $V(G_r)=J_r^-=V(H_r)$. Since $J_r^-=J_r$ (as $K_r=\emptyset$), we also have that $V(H_r)=J_r$.

        Observe that for each $s\in S$, $T_s$ is a subtree of $T-r$. Since $\scr{V}$ has near-weight at most $d$, $\scr{V}[T_s]=\scr{V}_s$ has weight at most $d$. So $|V_s|=|V(G_s)|\leq d$.
        
        By \cref{splitBasicPartition}, $\scr{P}':=(V(G_z):z\in Z)=(V(H_r))\sqcup (V_s:s\in S)$ is a partition of $G$, and if distinct $z,z'\in Z$ are such that $V(G_z)$ is adjacent to $V(G_{z'})$ in $G/\scr{P}$, then $z$ and $z'$ are related. Since no two vertices in $S$ are related, it follows that for each pair of distinct $s,s'\in S$, $V_s$ and $V_{s'}$ are non-adjacent in $G/\scr{P}'$. So $N_G(V_s)$ does not intersect $V_{s'}$.

        For each $s\in S$, let $Q_s$ be the set of parts in $\scr{P}'$ that intersect $K_s$. Since $K_s$ is a clique in $H_r$, observe that $Q_s\in \scr{Q}$. Thus, $|Q_s|\leq b+1$. Let $X_s:=X_{Q_s}$. As $K_s\in \scr{K}$, $K_s\setminus X_s$ intersects at most $b$ parts in $Q_s$.
        
        For each $Q\in \scr{Q}$, let $S_Q:=\{s\in S:Q_s=Q\}$, 
        $V_Q:=\bigcup_{s\in S_Q}V_s$, and let $n_Q:=|V_Q|$. Observe that for each pair of distinct $Q,Q'\in \scr{Q}$, $S_Q$ is disjoint to $S_{Q'}$, and thus $V_Q$ is disjoint to $V_{Q'}$. Therefore, $\sum_{Q\in \scr{Q}}n_Q\leq n$. Thus, if $\scr{Q}^+\subseteq \scr{Q}$ consists of the $Q\in \scr{Q}$ with $n_Q\geq d/2$, then $|\scr{Q}^+|\leq 2n/d$.
        
        Let $S^+:=\{s\in S:Q_s\in \scr{Q}^+\}$, and let $X':=\bigcup_{s\in S^+}X_s$. Note that $X'\subseteq \bigcup_{Q\in \scr{Q}^+}X_Q$, and thus $|X'|\leq m|\scr{Q}^+|\leq 2mn/d$. Note also that $X'\subseteq V(H_r)$. Let $X:=X_r\cup X'$. Observe that $|X|\leq |X_r|+|X'|\leq 2cn/d+q+2mn/d=2(c+m)n/d+q$. Note that $X\subseteq V(H_r)$.
        
        Let $\scr{P}_r':=(P\setminus X:P\in \scr{P}_r)$. By \cref{APIncX}, $(X,\scr{P}'_r)$ is an almost-partition of $H_r$ of treewidth at most $b$ and width at most $d/2$. Let $\psi:\scr{P}_r'\mapsto \scr{P}_r$ be the bijection that satisfies $\psi^{-1}(P)=P\setminus X$ for each $P\in \scr{P}_r$.
        
        Let $\scr{Q}'$ be the set of cliques of $(H_r-X)/\scr{P}_r'$. Note that each $Q\in \scr{Q}'$ has size at most $b+1$ as $(H_r-X)/\scr{P}_r'$ has treewidth at most $b$. Note also that for each $Q\in \scr{Q}'$, $|\psi(Q')|=|Q'|$ as $\psi$ is a bijection.
        
        For each $s\in S$, let $Q_s'$ be the set of parts in $\scr{P}_r'$ that intersect $K_s$. Since the parts of $\scr{P}_r'$ are disjoint to $X$, observe that each part in $Q_s'$ intersects $K_s\setminus X$. Since $K_s$ is a clique in $H_r$, observe that $Q_s'\in \scr{Q}'$. Thus, $|Q_s'|\leq b+1$. Also, observe that $\psi(Q_s')\subseteq Q_s$. Since $|\psi(Q_s')|=Q_s'$, it follows that $|Q_s'|\leq |Q_s'|$, with equality if and only if $\psi(Q_s')=Q_s$. In particular, observe that since $|Q_s|\leq b+1$, if $|Q_s'|=b+1$, then $\psi(Q_s')=Q_s$.
        
        Let $U:=\{s\in S:|Q_s'|\leq b$ and $W:=\{s\in S:|Q_s'|=b+1$. Observe that $S=U\sqcup W$ (as $b\in \ds{N}$).
        
        For each $s\in W$, since $|Q_s'|=b+1$, recall that $\psi(Q_s')=Q_s$. Also, recall that $K_s\setminus X_s$ intersects at most $b$ parts in $Q_s=\phi(Q_s')$. Thus, $K_s\setminus X_s$ intersects at most $b$ parts in $Q_s'$, as $P\subseteq \psi(P)$ for each $P\in \scr{P}_r'$. However, since $|Q_s'|=b+1$ and since $K_s\setminus X$ intersects each part in $Q_s'$, we obtain $X_s\nsubseteq X$. Observe that if $s\in S^+$, we would have $X_s\subseteq X'\subseteq X$, a contradiction. Thus, $s\notin S^+$. Hence, $S^+\cap W=\emptyset$.

        As $(H_r-X)/\scr{P}_r'$ has treewidth at most $b$, it has degeneracy at most $b$. Thus, we can find an ordering $\prec$ of $\scr{P}_r'$ such that each $P\in \scr{P}_r'$ is adjacent in $(H_r-X)/\scr{P}_r'$ to at most $b$ parts preceding $P$ in $\prec$. 

        For each $s\in W$, let $P_s$ be the part in $Q_s'$ that is preceded by every other part in $Q_s'$ (which exists since $|Q_s'|=b+1\geq 1$). Since $P_s\in Q_s'$, observe that $P_s$ intersects $K_s$. In particular, $P_s$ is nonempty. Since $Q_s'$ is a clique in $(H_r-X)/\scr{P}_r'$ of size exactly $b+1$, observe that $Q_s'\setminus P_s$ are exactly the parts of $\scr{P}_r'$ that precede $P_s$ (under $\prec$) and are adjacent to $P_s$ in $(H_r-X)/\scr{P}_r'$. Thus, if $s,s'\in W$ are such that $P_s=P_{s'}$, then $Q_s'=Q_{s'}'$.

        For each $P\in \scr{P}_r'$, let $S_P:=\{s\in W:P_s=P\}$, and let $V_P:=P\cup \bigcup_{s\in S_P}V_s$. If $S_P=\emptyset$, $V_P=P$, and thus $|V_P|=|P|\leq d/2$ as $\scr{P}'_r$ has width at most $d/2$. Otherwise, fix $s\in S_P$, and let $Q_P:=Q_s$.

        \begin{claim}
            \label{claimSPSQP}
            If $P\in \scr{P}_r'$ is such that $S_P\neq \emptyset$, then $S_P\subseteq S_{Q_P}$.
        \end{claim}

        \begin{proofofclaim}
            Recall that $Q_P=Q_s$ for some $s\in S_P$. For each $s'\in S_P$, we have $P_s=P_{s'}=P$, so $Q_P=Q_s'=Q_{s'}$. Since $s,s'\in W$, $\psi(Q_s')=Q_s$ and $\psi(Q_{s'}')=Q_{s'}$. Since $Q_s'=Q_{s'}'$, we obtain $Q_{s'}=Q_s=Q_P$. Thus, $s'\in S_{Q_P}$. So $S_P\subseteq S_{Q_P}$, as desired.
        \end{proofofclaim}
        
        For each $P\in \scr{P}_r'$, recall that $S_P\subseteq W$, which is disjoint to $S^+$. Thus, if $s\in S_P$, then $Q_s\notin \scr{Q}^+$. In particular, using the $s\in S_P$ such that $Q_P=Q_s$, we find that $Q_P\notin \scr{Q}^+$. So $|V_{Q_P}|\leq d/2$. By \cref{claimSPSQP}, $\bigcup_{s\in S_P}V_s\subseteq \bigcup_{s\in S_{Q_P}}V_s=V_{Q_P}$. So $|V_P|\leq |P|+|V_{Q_P}|\leq d/2+d/2=d$ as $\scr{P}_r'$ has width at most $d$.
        
        Let $\scr{P}:=(V_P:P\in \scr{P}_r')\sqcup (V_s:s\in U)$.

        \begin{claim}
            \label{claimDisjointSP}
            If $P,P'\in \scr{P}$ are distinct, then $S_P$ is disjoint to $S_{P'}$.
        \end{claim}

        \begin{proofofclaim}
            Presume, for a contradiction, that there exists $s\in S_P\cap S_{P'}$. Then $P_s=P=P'$. Since $P,P'$ are distinct, this is only possible if $P_s=P=P'=\emptyset$. However, we recall that $K_s$ intersects $P_s$. So $P_s$ is nonempty, a contradiction.
        \end{proofofclaim}

        We now show that $(X,\scr{P})$ is an almost-partition of $G$. Recall that $P\subseteq V(H_r-X)=J_r^-\setminus X\subseteq V(G-X)$ for each $P\in \scr{P}_r'$. Further, recall that $V_s=V(G_s)=\bigcup_{t\in V(T_z)}J_t^-\subseteq V(G)$, and that $V_s$ is disjoint to $V(H_r)\supseteq X$. Thus, $\scr{P}\subseteq 2^{V(G-X)}$.
            
        Next, we show that the elements of $\scr{P}$ are pairwise disjoint. Recall that $(V(H_r))\sqcup (V_s:s\in S)$ is a partition of $G$. Thus, for each pair of distinct $s,s'\in U$, $V_s\cap V_{s'}=\emptyset$. For each $P\in \scr{P}'_r$, and $s\in U$, recall that $S_P\subseteq W$, which is disjoint to $U$. Thus, for each $s'\in S_P$, $s\neq s'$. So $\bigcup_{s'\in S_P}V_{s'}$ is disjoint to $V_s$. Further, since $P\subseteq V(H_r)$, $P$ is disjoint to $V_s$. So $V_P$ is disjoint to $V_s$. Finally, consider distinct $P,P'\in \scr{P}'_r$. Recall that $P$ is disjoint to $P'$. Since $P,P'\subseteq V(H_r)$, observe that $P$ is disjoint to $\bigcup_{s\in S_{P'}}V_s$ and $P'$ is disjoint to $\bigcup_{s\in S_P}V_s$. Thus, if $V_P,V_{P'}$ intersect, then $\bigcup_{s\in S_P}V_s$ and $\bigcup_{s\in S_{P'}}V_s$ intersect, which occurs only if $S_P$ and $S_{P'}$ intersect. But by \cref{claimDisjointSP}, this does not occur. So we can conclude that $V_P$ and $V_{P'}$ are disjoint. Thus, we have shown that the elements of $\scr{P}$ are pairwise disjoint. 

        Therefore, it remains only to show that each $v\in V(G-X)$ is contained in at least one element of $\scr{P}$. Recall that $(V(G_z):z\in Z)=(V(H_r))\sqcup (V_s:s\in S)$ is a partition of $G$. So there exists exactly one $z_v\in Z$ such that $z_v\in V(G_{z_v})$. If $z_v=r$, then $v\in V(G_r)=V(H_r)$. As $v\notin X$, this gives $v\in V(H_r-X)$. Thus, there exists exactly one $P_v\in \scr{P}'_r$ such that $v\in P_v\subseteq V_{P_v}\in \scr{P}$. Otherwise, if $z_v\neq r$, then $z=s\in S$ and $v\in V_s$. Recall that $S=U\cup W$. If $s\in U$, then $v\in V_s\in \scr{P}$. If $s\in W$, then $s\in S_{P_s}$, and thus $v\in V_s\subseteq \bigcup_{s'\in S_{P_s}}V_{s'}\subseteq V_{P_s}\in \scr{P}$. So $v$ is contained in at least one element of $\scr{P}$, as desired. Thus, we have shown that $(X,\scr{P})$ is an almost-partition of $G$.

        Recall that $|X|\leq 2(c+m)n/d+q$, that for each $s\in S\supseteq U$, $|V_s|\leq d$, and that for each $P\in \scr{P}_r'$, $|V_P|\leq d$. So $(X,\scr{P})$ has width at most $d$ and loss at most $2(c+m)n/d+q$ (where $n=|V(G)|$). Thus, it remains only to show that $(X,\scr{P})$ has treewidth at most $b$.

        Recall that for distinct $s,s'\in S$, $N_G(V_s)$ does not intersect $V_{s'}$. Thus, if $s,s'\in U$ are distinct, then $V_s,V_{s'}\in \scr{P}$ are non-adjacent in $(G-X)/\scr{P}$. Further, if $P\in \scr{P}'_r$ and $s\in U$, then $N_G(\bigcup_{s'\in S_P}V_{s'})$ is disjoint to $V_s$ as $S_P\subseteq W$, which is disjoint to $U$.

        \begin{claim}
            \label{claimPartNeighbourVs}
            If $P\in \scr{P}'_r$ and $s\in S$ are such that $N_G(P)$ intersects $V_s$, then $P\in Q_s'$.
        \end{claim}

        \begin{proofofclaim}
            Fix $u\in P$ and $v\in V_s$ that are adjacent in $G-X$. Note that $u\in P\subseteq V(H_r)=V(G_r)$ and $v=V_s=V(G_s)$. Since $r$ is an ancestor of $s$, by \cref{splitEdges} (with $z=r$ and $z'=s$), $u\in K_s$. So $P\in \scr{P}'_r$ intersects $K_s$. Hence, $P\in Q_s'$, as desired.
        \end{proofofclaim}

        Recall that if $P\in \scr{P}'_r$ and $s\in U$, then $N_G(\bigcup_{s'\in S_P}V_{s'})$ is disjoint to $V_s$ as $S_P\subseteq W$. Thus, if $V_P$ is adjacent to $V_s$ in $(G-X)/\scr{P}$, then $N_{G-X}(P)\subseteq N_G(P)$ intersects $V_P$. So by \cref{claimPartNeighbourVs}, $P\in Q_s'$.

        Now, consider distinct $P,P'\in \scr{P}'_r$ such that $V_P,V_{P'}\in \scr{P}$ are adjacent in $(G-X)/\scr{P}$. So $N_{G-X}(V_P)$ intersects $V_{P'}=P'\cup \bigcup_{s\in S_{P'}}V_s$. Observe that $N_{G-X}(V_P)\subseteq N_G(P)\cup \bigcup_{s\in S_P}N_G(V_s)$. By \cref{claimDisjointSP}, $S_P$ and $S_{P'}$ are disjoint. Recalling that $N_G(V_s)$ does not intersect $V_{s'}$ whenever $s,s'\in S$ are distinct, we find that $\bigcup_{s\in S_P}N_G(V_s)$ does not intersect $\bigcup_{s\in S_{P'}}V_s$. Thus, either $N_G(P)$ intersects $V_s$ for some $s\in S_{P'}$, or $N_G(V_s)$ intersects $P'$ for some $s\in S_P$. However, if $N_G(V_s)$ intersects $P'$ for some $s\in S_P$, then $N_G(P')$ intersects $V_s$. So these cases are symmetric. Without loss of generality, assume that $N_G(P)$ intersects $V_s$ for some $s\in S_{P'}$. Applying \cref{claimPartNeighbourVs}, we find that $P\in Q_s'$. Recall also that $P'=P_{s'}\in Q_s'$. Thus, since $Q_s'$ is a clique in $(H_r-X)/\scr{P}'_r$, we find that $P$ and $P'$ are adjacent in $(H_r-X)/\scr{P}'_r$.

        Therefore, we can construct $(G-X)/\scr{P}$ from $(H_r-X)/\scr{P}'_r$ as follows. For each $s\in U$, and a vertex $V_s$ whose neighbourhood is contained $Q_s'$, which is a clique. Further, by definition of $U$, this clique is of size at most $b$. Then, for each $P\in \scr{P}$, relabel $P$ as $V_P$. Finally, delete some edges. After the first operation, observe that the treewidth is the bounded by maximum of the original treewidth and $b$. The second and third operations do not increase the treewidth. Thus, $(G-X)/\scr{P}$, and hence $(X,\scr{P})$, has treewidth at most $b$.

        This completes the proof of the lemma.
    \end{proof}

    \subsection{Proof}

    We can now prove \cref{TDReduced}, which we recall for convenience.

    \TDReduced*


    \begin{proof}
        Set $c':=2(q+c+m+j)$ and $q':=q$.
    
        Since $G$ is $(b,m,d/2,c,q,j)$-decomposable, there exists a rooted tree-decomposition $\scr{V}=\linebreak(G,T,r,\scr{J})$ such that: 
        \begin{enumerate}
            \item $\scr{J}$ has adhesion at most $j$, and,
            \item for each $t\in V(T)$, there exists a $(b,m,d/2,c,q)$-decomposition $\scr{D}_t$ for $\scr{V}$ at $t$.
        \end{enumerate}
        
        For each $t\in V(T)$, set $K_t:=K_t(\scr{V})$ and $J_t^-:=J_t^-(\scr{V})$. Since $\scr{V}$ has adhesion at most $j$, observe that $|K_t|\leq j$.

        By \cref{findSplit} (with $d/2$ as $d$), there exists a split $(\scr{V}_z:z\in Z)$ of $\scr{V}$ of size at most $2n/d + 1$ and near-weight at most $d/2$. So $|Z|\leq 2n/d+1$, and for each $z\in Z$, $\scr{V}_z$ has near-weight at most $d/2$. Set $X':=\bigcup_{z\in Z}K_z$. Recall that $r\in Z$ (as $Z$ is rooted), and that $K_r=\emptyset$. Thus, observe that $|X'|\leq j|Z-1|\leq 2jn/d$.

        For each $z\in Z$, let $T_z$ be the tree of $\scr{V}_z$ and $G_z$ be the graph of $\scr{V}_z$. Set $n_z:=|V(G_z)|$. Recall that the induced root of $T_z$ is $z$, and that $\scr{V}_z=\scr{V}[T_z]$. So the root of $\scr{V}_z$ is $z$, and $G_z=\bigcup_{t\in V(T_z)}J_t^-$. In particular, $G_z$ is an induced subgraph of $G$. By \cref{decompositionStays}, $\scr{D}_z$ is a $(b,m,d/2,c,q,j)$-decomposition for $\scr{V}_z$ at $z$. By definition, $\scr{V}_z$ has near-weight at most $d/2$. Thus, by \cref{concentratedOnStar}, $G_z$ admits an almost-partition $(X_z,\scr{P}_z')$ of treewidth at most $b$, width at most $d$, and loss at most $2(c+m)n_z/d+q$.

        Let $X:=X'\cup \bigcup_{z\in Z}X_z$. Observe that $|X|\leq |X'|+\sum_{z\in Z}|X_z|\leq 2jn/d+\sum_{z\in Z}2(c+m)n_z/d+q$. By \cref{splitBasicPartition}, $(V(G_z):z\in Z)$ is a partition of $G$. Thus, $\sum_{z\in Z}n_z=|V(G)|=n$. So $|X|\leq 2jn/d+2(c+m)n/d+q|Z|\leq 2(q+c+m+j)n/d + q=c'n/d+q'$.

        For each $z\in Z$, let $\scr{P}_z:=(P\setminus X:P\in \scr{P}_z')$. Let $\psi_z:\scr{P}_z\mapsto \scr{P}_z'$ be the natural bijection such that $P=\psi_z(P)\setminus X$ for each $P\in \scr{P}_z$. So $P\subseteq \psi_z(P)$ for each $P\in \scr{P}_z$.
        
        Let $\scr{P}:=\bigsqcup_{z\in Z}\scr{P}_z$. We first show that $(X,\scr{P})$ is an almost-partition of $G$.
        
        Recall that for each $z\in Z$ and each $P\in \scr{P}_z$, $P=\psi_z(P)\setminus X$. So $P$ is disjoint to $X$. Further, recall that $\psi_z(P)\in \scr{P}_z'$, which is a partition of $G_z-X_z\subseteq G$. So $P\subseteq V(G)$. Thus, $P\subseteq V(G-X)$. So $\scr{P}\subseteq 2^{V(G-X)}$.

        Next, we show that the elements of $\scr{P}$ are disjoint. Fix distinct $P,P'\in \scr{P}$. There exists $z,z'\in Z$ such that $P\in \scr{P}_z$ and $P'\in \scr{P}_{z'}$. Recall that $P\subseteq \psi_z(P)\in \scr{P}'_z$ and $P'\subseteq \psi_{z'}(P')\in \scr{P}'_{z'}$. Since $\scr{P}'_z$ and $\scr{P}'_{z'}$ are partitions of $G_z-X_z$ and $G_{z'}-X_{z'}$ respectively, we have $P\subseteq \psi_z(P)\subseteq V(G_z)$ and $P'\subseteq \psi_{z'}(P')\subseteq V(G_{z'})$. Recall that $(V(G_s):s\in Z)$ is a partition of $G$. Thus, if $z\neq z'$, then $P$ is disjoint to $P'$. Otherwise, observe that $P,P'$ are distinct elements of $\scr{P}_z=\scr{P}_{z'}$. As $\psi_z=\psi_{z'}$ is a bijection, $\psi_z(P)$ and $\psi_z(P')$ are distinct elements of $\scr{P}'_z$. So $P\subseteq \psi_z(P)$ is disjoint to $\psi_z(P')\supseteq P'$. Thus, in all cases, $P$ is disjoint to $P'$. So the elements of $\scr{P}$ are pairwise disjoint.

        Finally, we show that each $v\in V(G-X)$ is contained in at least one $P\in \scr{P}$. Since $(V(G_z):z\in Z)$ is a partition of $G$, there exists $z\in Z$ such that $v\in V(G_z)$. Since $v\notin X\supseteq X_z$, $v\in V(G_z-X_z)$. Since $\scr{P}_z'$ is a partition of $G_z-X_z$, there exists $P'\in \scr{P}'_z$ such that $v\in P'$. Let $P:=P'\setminus X\in \scr{P}_z\subseteq \scr{P}$. Since $v\notin X$, $v\in P\in \scr{P}$, as desired. So $(X,\scr{P})$ is an almost-partition of $G$.

        For each $P\in \scr{P}$, there exists $z\in \scr{P}_z$ such that $P\in \scr{P}_z$. We then have $P\subseteq \psi_z(P)\in \scr{P}'_z$. Since $\scr{P}'_z$ has width at most $d$, $|P|\leq |\psi_z(P)|\leq d$. So $\scr{P}$ has width at most $d$. Since $|X|\leq c'n/d + q'$, $(X,\scr{P})$ has width at most $d$ and loss at most $c'n/d+q'$. So it remains only to show that $(X,\scr{P})$ has treewidth at most $b$.

        \begin{claim}
            \label{claimSplitPartsNeighbourhoods}
            If $P,P'\in \scr{P}$ are adjacent in $(G-X)/\scr{P}$, then there exists $z\in Z$ such that $P,P'\in \scr{P}_z$. Further, $\psi_z(P)$ is adjacent to $\psi_z(P')$ in $(G_z-X_z)/\scr{P}_z'$.
        \end{claim}

        \begin{proofofclaim}
            Since $P,P'$ are adjacent in $(G-X)/\scr{P}$, $N_G(P)$ intersects $P'$. Fix $u\in P$ and $v\in P'$ that are adjacent in $G$.
            
            There exists $z,z'\in Z$ such that $P\in \scr{P}_z$ and $P'\in \scr{P}_{z'}$. Observe that $u\in P\subseteq \psi_z(P)\subseteq V(G_z)$ and $v\in P'\subseteq \psi_{z'}(P')\subseteq V(G_{z'})$. So $u\in V(G_z)$ and $v\in V(G_{z'})$.
            
            Presume, for a contradiction, that $z\neq z'$. Then by \cref{splitEdges}, $\{u,v\}$ intersects $K_z\cup K_{z'}\subseteq X$. However, $u\in P$ and $v\in P'$, both of which are disjoint from $X$ (as subsets of $V(G-X)$). So this is a contradiction. Hence, $z=z'$.

            So $u,v\in V(G_z)=V(G_{z'})$. Since $G_z$ is an induced subgraph of $G$, $u,v$ are adjacent in $G_z$. Recall that $u\in \psi_z(P)\subseteq V(G_z-X_z)$ and $v\in \psi_{z'}(P')=\psi_z(P')\subseteq V(G_z-X_z)$. Thus, $N_{G_z-X_z}(\psi_z(P))$ contains $v\in \psi_z(P')$. So $N_{G_z-X_z}(\psi_z(P))$ intersects $\psi_z(P')$, and $\psi_z(P)$ is adjacent to $\psi_z(P')$ in $(G_z-X_z)/\scr{P}_z'$, as desired.
        \end{proofofclaim}
        
        By \cref{claimSplitPartsNeighbourhoods}, we find that $(G-X)/\scr{P}$ can be obtained from $\bigsqcup_{z\in Z}(G_z-X_z)/\scr{P}_z'$ by deleting some edges and relabelling. Since $(G_z-X_z)/\scr{P}_z'$ has treewidth at most $b$ for each $z\in Z$, so does the disjoint union $\bigsqcup_{z\in Z}(G_z-X_z)/\scr{P}_z'$. Further, deleting edges and relabelling does not increase treewidth. Thus, $(G-X)/\scr{P}$ has treewidth at most $b$. So $(X,\scr{P})$ has treewidth at most $b$. This completes the proof.
    \end{proof}
    
    \section{Handling almost-embeddings}
    \label{SecHandleAE}

    \subsection{Rephrasing the problem}

    We now proceed to the more difficult step, finding almost-partitions of graphs that admit $(g,p,k,a)$-almost-embeddings (with disjoint discs). Specifically, we show the following.

    \begin{restatable}{theorem}{AEReduced}
        \label{AEReduced}
        Let $g,p,k,a,n\in \ds{N}$ and let $d\in \ds{R}^+$ with $d\geq 12(k+1)\sqrt{3n}$. If $\Gamma$ is a $(g,p,k,a)$-almost-embedding with disjoint discs of a graph $G$ with $n$ vertices, then $G$ admits a $(2,6k+8,\Attachable(\Gamma))$-concentrated almost-partition of treewidth at most $2$, width at most $d$, and loss at most $(28(k+1)^2+(2g+5p+1))n/d+(4g+9p)(k+1)+a$.
    \end{restatable}

    From this, \cref{GMSTVariant}, and the following observation, we can derive \cref{AEMain}.

    \begin{observation}
        \label{removeApices}
        Let $g,p,k,a\in \ds{N}$, let $\Gamma=(G,\Sigma,G_0,\scr{D},H,\scr{J},A)$ be a $(g,p,k,a)$-almost-embedding with disjoint discs, and let $S\subseteq A$. Then $\Gamma':=(G-S,\Sigma,G_0,\scr{D},H,\scr{J},A\setminus S)$ is a $(g,p,k,a)$-almost-embedding with disjoint discs, and for each $K\in \Attachable(\Gamma)$, $K\setminus S\in \Attachable(\Gamma')$.
    \end{observation}

    \begin{proof}
        Since $S\subseteq A\subseteq V(G)$, $G-S$ is a well-defined graph, $A\setminus S\subseteq V(G-S)$, and $(G-S)-(A\setminus S)=G-A$. It follows directly from this and the other properties of $\Gamma$ that $\Gamma'$ is an almost-embedding. Further, the genus, disc-count, and vortex-width are preserved, and the apex-count is $|A\setminus S|\leq |A|\leq a$. So $\Gamma'$ is a $(g,p,k,a)$-almost-embedding. Since $\Gamma$ has disjoint discs, the discs in $\scr{D}$ are pairwise disjoint, and thus $\Gamma'$ has disjoint discs.
        
        Observe that $\VC(\Gamma)=\VC(\Gamma')$ and $\AEC(\Gamma)=\AEC(\Gamma')$, and these depend only on $G_0$ and $H$ respectively. For each $K\in \Attachable(\Gamma)$, observe that $K\setminus S$ is a clique in $G-S$, and that $(K\setminus S)\setminus (A\setminus S)=K\setminus A\in \AEC(\Gamma)\cup \VC(\Gamma)=\AEC(\Gamma')\cup \VC(\Gamma')$. So $K\setminus S\in \Attachable(\Gamma')$. This completes the proof.
    \end{proof}

    We now recall \cref{AEMain}.

    \AEMain*

    \begin{proof}
        Let $g,p,k,a\in \ds{N}$ be from \cref{GMSTVariant}. Set $m:=6k+8$, $c:=28(k+1)^2+(2g+5p+1)$, $q:=(4g+9p)(k+1)+a$, and $\alpha:=24(k+1)\sqrt{3}$.

        Note that $d/2\geq \alpha\sqrt{n}/2=12(k+1)\sqrt{3n}$.
        
        By \cref{GMSTVariant}, there exists a rooted tree $(T,r)$ and a collection $((J_t,\Gamma_t):t\in V(T))$ such that:
        \begin{enumerate}
            \item $\scr{J}:=(J_t:t\in V(T))$ is a tree-decomposition of $G$,
            \item for each $t\in V(T)$, $\Gamma_t$ is a $(g,p,k,a)$-almost-embedding of $G\langle J_t\rangle$ with disjoint discs, and,
            \item for each edge $tp(t)\in E(T)$: 
            \begin{enumerate}
                \item if $A_t$ is the apices of $\Gamma_t$, then $J_t\cap J_{p(t)}\subseteq A_t$, and,
                \item $J_t\cap J_{p(t)}\in \Attachable(\Gamma_{p(t)})$.
            \end{enumerate}
        \end{enumerate}

        Let $\scr{V}:=(G,T,r,\scr{J})$. So $\scr{V}$ is a rooted tree-decomposition of $G$.
        
        For each $t\in V(T)$, let $K_t:=K_t(\scr{V})$, $J_t^-:=J_t^-(\scr{V})$, and $H_t:=H_t(\scr{V})$.

        If $t=r$, observe that $K_r=\emptyset\subseteq A_r$ trivially. Otherwise, observe that $K_t=J_t\cap J_{p(t)}\subseteq A_t$. So in either case, $K_t\subseteq A_t$.

        Let $t'$ be a child of $t$. Observe that $t'\neq r$, and thus $K_{t'}=J_t\cap J_{t'}=J_{t'}\in \Attachable(\Gamma_t)$.
        
        Recall that $H_t=G\langle J_t\rangle - K_t$. Let $\Gamma_t'$ be obtained from $\Gamma_t$ by replacing $A_t$ with $A_t\setminus K_t$. By \cref{removeApices} (with $S:=K_t$), $\Gamma_t'$ is a $(g,p,k,a)$-almost-embedding of $H_t$ with disjoint discs. Further, for each child $t'$ of $t$, since $K_{t'}\in \Attachable(\Gamma_t)$, we have $K_{t'}\setminus K_t\in \Attachable(\Gamma_t')$.
        
        Since $V(H_t)=J_t^-$, observe that $\Attachable(\Gamma_t')\subseteq 2^{J_t^-}$. Recall also that $\Attachable(\Gamma_t')$ is closed under subsets.
        
        By \cref{AEReduced} with ($d:=d/2\geq 12(k+1)\sqrt{3n}$), $H_t$ admits a $(2,m,\Attachable(\Gamma_t'))$-concentrated almost-partition $(X_t,\scr{P}_t)$ of treewidth at most $2$, width at most $d$, and loss at most $c|V(H_t)|/(d/2)+q$. It follows that $(X_t,\scr{P}_t,\Attachable(\Gamma_t'))$ is a $(2,m,d/2,c,q)$-decomposition for $\scr{V}$ at $t$.

        It remains to check that $\scr{J}$ has adhesion at most $j=a$. Recall that the maximum intersection between two bags is achieved by two bags $J_t,J_{t'}$ with $t,t'$ adjacent in $T$. Without loss of generality, say $t'$ is the parent of $t$. So $J_t\cap J_{t'}=J_t\cap J_{p(t)}\in A_t$. Since $\Gamma_t$ is a $(g,p,k,a)$-almost-embedding, $|A_t|\leq a$. Thus, $|J_t\cap J_{t'}|\leq |A_t|\leq a=j$. So $\scr{J}$ has adhesion at most $j$. This completes the proof.
    \end{proof}

    We want to emphasise that our method for finding these almost-partitions is not particularly specific to the treewidth $2$ case. Instead, if a technical condition relating to the ability to find `good' almost-partitions of plane graphs is satisfied, we show that we can also find `good' partitions of graphs that admit $(g,p,k,a)$-almost-embeddings.

    Specifically, we have the following theorem.
    
    \begin{restatable}{theorem}{AETechinical}
        \label{AETechinical}
        Let $b\in \ds{N}$, and let $f,g:\ds{N}\mapsto \ds{R}_0^+$ be such that $(f,g)$ generate adjustments of treewidth $b$. Then for all $g,p,k,a,n\in \ds{N}$, every $(g,p,k,a)$-almost-embedding $\Gamma$ with disjoint discs of an $n$-vertex graph $G$, and every $d\in \ds{R}^+$ with $d\geq (k+1)g(n)$, $G$ admits an almost-partition of treewidth at most $b$, width at most $d$, and loss at most $((k+1)^2(28+f(n))+(2g+5p+1))n/d+(4g+9p)(k+1)+a$ that is $(1,b(b+1)(k+1),\VC(\Gamma))$-concentrated and $(2,(6k+8)\binom{b+1}{3},\Attachable(\Gamma))$-concentrated.
    \end{restatable}

    We define what it means for functions to `generate adjustments' later (see \cref{SecPReductions}).
    
    We do not make use of the fact that the almost-partition is $(1,b(b+1)(k+1),\VC(\Gamma))$-concentrated in this paper, but we believe that it could be useful for potentially showing that $K_h$-minor-free graphs admit treewidth $1$ almost-partitions of small width/loss, and it is not much extra effort to show.

    \cref{AETechinical} is then combined with the following lemma.
    \begin{restatable}{lemma}{planePartitions}
        \label{planePartitions}
        Let $f,g:\ds{N}\mapsto \ds{R}_0^+$ be defined by $f(n):=0$ and $g(n):=12\sqrt{3n}$ for each $n\in \ds{N}$. Then $(f,g)$ generate adjustments of treewidth 2.
    \end{restatable}

    We can then use these results to prove \cref{AEReduced}, which we restate for convenience.

    \AEReduced*


    \begin{proof}
        Let $f,g:\ds{N}\mapsto \ds{R}_0^+$ be defined by $f(m):=0$ and $g(m):=12\sqrt{3m}$ for each $m\in \ds{N}$. By \cref{planePartitions}, $(f,g)$ generate adjustments of treewidth 2. Note that $d\geq (k+1)g(n)$.
        
        By \cref{AETechinical} (with $b:=2$), $G$ admits a $(2,(6k+8),\Attachable(\Gamma))$-concentrated almost-partition of treewidth at most $2$, width at most $d$, and loss at most $(28(k+1)^2+(2g+5p+1))n/d+(4g+9p)(k+1)+a$, as desired.
    \end{proof}

    The proof of \cref{planePartitions} is, at its core, just a slight variant of the proof used by \citet{Distel2024} to prove \cref{tw2Surfaces}. Indeed, we end up repeating the same proof, but controlling a few extra properties. As a consequence, a majority of the proof of \cref{planePartitions} is left to the appendix (specifically \cref{SecQT}). So the main focus of the remainder of this paper is the proof of \cref{AEReduced}.

    \subsection{Proof idea}

    \label{SecMainIdea}

    We now discuss how to almost-partition a graph that admits an almost-embedding (with disjoint discs). We start by assuming that we have some method of almost-partitioning the embedded graph. The question is `how can we take the almost-partition of the embedded graph, and use it to find an almost-partition of the almost-embedding'? We can safely delete (put in the loss set) the apices, as there are only a small number of them. So this is really a question of how to almost-partition the vortex.

    One idea (which is loosely what is done in \citet{Distel2024}) is to partition the vortex as follows. Along each disc, we can add edges so that the underlying cycles are a subgraph of the embedded graph. Then, find a small number of boundary vertices along the disc such that each interval (of the underlying cycle) between these points is small. We refer to these points as the `breakpoints'. All vertices in bags indexed by breakpoints will be deleted. The remaining intervals get contracted, and we find a partition of the resulting embedded graph. We then uncontract the interval to get the final partition.

    This method works for \citet{Distel2024} as (using a result from \citet{Distel2022Surfaces}) they are able to find a treewidth 3 partition of graphs on surfaces of width $O(\sqrt{n})$ for which each part only contains a fixed number of the contracted vertices. Thus, when the contraction is undone, the total number of vertices in each part is small. However, we cannot use this partition, as we need a treewidth $2$ partition.

    As mentioned before, \citet{Distel2024} showed that every planar graph is a $O(\sqrt{n})$-blowup of a treewidth $2$ graph. So can we use this partition instead? The immediate answer is no, because unlike before, this partition cannot control how many `contracted vertices' are in a single part.

    So what if, instead of contracting at the start, we instead find the partition first, and choose the breakpoints with this partition in mind, and then `merge' parts according to the intervals? The difficulty now is that we would have to ensure that each part only meets one interval. Some parts might go between vertices far apart in the cycle, or even touch the cycle many times. However, if the part was connected, and goes between two intervals, then now we could contract both intervals together, even if they were far apart in the cycle. So, provided that we can work with connected partitions, we can instead focus on contracting groups of intervals, connected by common parts that they intersect.

    Another annoyance is the fact that, even if the parts are connected, they can `cross' in the cycle. That is, we can find four vertices in cyclic order such that the first and third are contained in one part, and the second and fourth are contained in a different part. However, this cannot occur if the embedded graph was a plane graph (and the partition is connected). In a plane graph, these parts would `block' each other, separating the graph. We prove this.

    For a cycle $U$, use \defn{$\prec_U$} to denote the cyclic ordering of the vertices of $U$ (in some fixed direction). By considering a path $U$ to be contained in a cycle, we do the same for a path.

    \begin{lemma}
        \label{nonCrossing}
        Let $G$ be a plane graph, let $D$ be a $G$-clean disc, and let $U:=U(D,G)$. Then for any pair of disjoint connected subgraphs $H_1,H_2$ of $G$, there does not exist pairwise distinct $x,y,a,b\in V(U)$ with $x\prec_U a\prec_U y\prec_U b$ and $x,y\in V(H_1)$, $a,b\in V(H_2)$.
    \end{lemma}

    \begin{proof}
        Presume otherwise. So such an $x,y,a,b\in V(U)$ exist. We can add (possibly parallel) edges along the boundary of $D$ to find a spanning plane supergraph $G'$ of $G$ with $U\subseteq G'$ (ignoring edge labels) such that $D\cap G'=D\cap G=\partial D\cap V(G)$ (as $D$ is $G$-clean). So $D$ is $G'$-clean, and it is easy to check that $U(D,G')=U$ (see also \cref{underlyingExact}).
        
        Let $G^*$ be the plane graph obtained from $G'$ by embedding a new vertex $q$ inside $D$, and, for each $v\in V(U)$, embedding a new edge with endpoints $q$ and $v$ (inside $D$). We will construct a $K_5$-minor in $G^*$, contradicting the fact that $G^*$ is planar.
        
        As $H_1$ is connected, there exists a path $P_1$ from $x$ to $y$ in $H_1$. Let $y'$ be the first vertex on this path such that $y=v_{k'}\in B$ with $a\prec_U y'\prec_U b$, and let $x'$ be the last vertex on this path before $y'$ such that $b\prec_U<x'\prec_U a$. Since $x\prec_U a\prec_U y\prec_U b$, note that $x'$ and $y'$ are well-defined and distinct. Notice also $x',y'\in V(H_1)$, and that $x'\prec_U a\prec_U y'\prec_U b$. Let $P_{x',y'}$ be the subpath of $P_1$ from $x'$ to $y'$ (inclusive). By choice of $x'$ and $y'$, observe that $P_1'$ is internally disjoint to $U$.
        
        By a symmetric argument, we can find vertices $a',b'\in V(H_2)$ such that $x'\prec_U a'\prec_U y'\prec_U b'$, and a path $P_{a',b'}$ from $a'$ to $b'$ in $H_2$ whose interior is disjoint to $U$. Since $x'\prec_U a'\prec_U y'\prec_U b'$, we can find internally disjoint paths $P_{x',a'},P_{a',y'},P_{y',b'},P_{b',x'}$ in $U$ such that, for each $P_{u,v}\in \{P_{x',a'},P_{a',y'},P_{y',b'},P_{b',x'}\}$, the endpoints of $P_{u,v}$ are $u$ and $v$. Thus and since $H_1$ and $H_2$ are disjoint, observe that the paths $P_{x',a'},P_{a',y'},P_{y',b'},P_{b',x'},P_{x',y'},P_{a',b'}$ are pairwise internally disjoint, contained in $G^*-\{q\}$, and that each path $P_{u,v}$ in this list has endpoints $u$ and $v$. Thus and since $q$ is adjacent in $G^*$ to each vertex in $V(U)$, which includes $x',y',a',$ and $b'$, we can find a $K_5$ minor in $G^*$ (with $x',y',a',b'$, and $q$ all being in distinct parts), a contradiction. The lemma follows.
    \end{proof}

    While \cref{nonCrossing} might seem like a relatively mundane result, it is absolutely critical to the proof in multiple places. As such, we name it the \defn{non-crossing property}. Because of the non-crossing property, in the case when the embedded graph is plane, we can make progress on finding these groups of intervals. We withhold on explaining the details until \cref{SectionMerging}, when we actually tackle that part of the proof.

    Conveniently and for unrelated reasons, in the case of plane graphs and in a `triangulation-like' setting, the partition produced using the algorithm behind \cref{tw2Surfaces} (or at least something very close to it) actually gives a connected partition. This gives us exactly what we need to make use of the non-crossing property.
    
    Now, to make use of any of this, we need the embedded graph to be plane. Of course, this is not usually the case. However, via a tedious and technical process, we can reduce the $(g,p,k,a)$-almost-embedding (with disjoint discs) to a `very similar' graph that admits something that is `very close' to a $(0,1,k,0)$-almost-embedding (with the same $k$). Note that we have also reduced the number of discs to just one. This is very useful, as we only need to partition one vortex, rather than multiple vortices at the same time.

    To reduce the genus of the surface, we use the same strategy as employed in \citet{rtwltwMCC} (which was then repeated by \citet{Distel2022Surfaces}). Buried within the proof of key results in these papers, the authors describe an algorithm to `cut' open a specific subgraph (which we call the `cutting subgraph') of an embedded graph to obtain a planar graph. We use this same strategy, but now with an almost-embedding. The `cut' version of this subgraph (which we call the `cut subgraph') is also the `error' keeping the result from having a $(0,1,k,0)$-almost-embedding.

    We now work towards formally describing this reduction.

    \subsection{Plane+quasi-vortex embeddings}

    First, we need to formally describe what we want to `reduce' to. We would like to reduce to a $(0,1,k,0)$-almost-embedding (for an appropriate $k\in \ds{N}$), but there is a small subgraph acting as an `obstruction'. Specifically, the vertices of this subgraph lies within the interior of the disc we would like to attach the vortex to. If we were to delete the vertices of this subgraph, we would obtain a $(0,1,k,0)$-almost-embedding.
    
    This motivates the following definitions, which parallel the definitions of clean-discs and $(0,1,k,0)$-almost-embeddings.

    Let $G$ be a graph embedded in a surface $\Sigma$, and let $S\subseteq V(G)$. Say that a disc in $\Sigma$ is $(G,S)$-clean if:
    \begin{enumerate}
        \item $D$ is $(G-S)$-clean, and
        \item the interior of $D$ contains $S$.
    \end{enumerate}
    We remark that a $G$-clean disc is also $(G,\emptyset)$-clean.

    We define \defn{$B(D,G,S)$}$:=B(D,G-S)$ and \defn{$U(D,G,S)$}$:=U(D,G,S)$. So $V(U(D,G,S))=B(D,G,S)$.
    
    A \defn{plane+quasi-vortex embedding} $\Lambda$ is a tuple $(G,G_0^+,W,D,H,\scr{J})$ such that:
    \begin{enumerate}
        \item $G$ is a graph,
        \item $G_0^+$ is a plane graph,
        \item $W$ is a connected and nonempty plane subgraph of $G$,
        \item $D$ is a $(G_0^+,V(W))$-clean disc,
        \item $G_0^+,H$ are subgraphs of $G$,
        \item $G=G_0^+\cup H$,
        \item $W$ is disjoint to $H$,
        \item $B(D,G_0^+,V(W))=V(G_0^+\cap H)$, and,
        \item $(H,U(D,G_0^+,V(W)),\scr{J})$ is a planted graph decomposition.
    \end{enumerate}

    Call $G$ the \defn{graph} (of $\Lambda$), $G_0^+$ the \defn{plane subgraph}, $W$ the \defn{obstruction subgraph}, $D$ the \defn{disc}, $H$ the \defn{quasi-vortex}, and $\scr{J}$ the \defn{decomposition}. Let \defn{$G_0(\Lambda)$}$:=G_0$ be the plane graph $G_0^+-V(W)$. Since $D$ is $(G_0^+,V(W))$-clean, observe that $D$ is $G_0$-clean. Set \defn{$B(\Lambda)$}$:=B(D,G_0^+,V(W))=B(D,G_0)=V(G_0^+\cap H)$, and \defn{$U(\Lambda)$}$:=U(D,G_0^+,V(W))=U(D,G_0)$. So $V(U(\Lambda))=B(\Lambda)=V(G_0^+\cap H)$. Say that the \defn{vortex-width} of $\Lambda$ is the width of $\scr{J}$.

    As mentioned before, we can delete the vortex-obstruction to obtain a $(0,1,k,0)$-almost-embedding. Specifically, we have the following observation, which we do not use but include for intuition.

    \begin{observation}
        \label{quasiPlaneToAE}
        Let $\Lambda=(G,G_0^+,W,D,H,\scr{J}')$ be a plane+quasi-vortex embedding of vortex-width at most $k\in \ds{N}$. Then $(G-V(W),\ds{R}^2,G_0(\Lambda),\{D\},H,\scr{J},\emptyset)$ is a $(0,1,k,0)$-almost-embedding.
    \end{observation}

    \begin{proof}
        Recall that $W$ is a connected subgraph of $G_0^+\subseteq G'$. Thus $G':=G-V(W)$ is well-defined. 
        
        Recall that $G_0:=G_0(\Lambda)$ is the plane graph $G_0^+-V(W)$. Since $G=G_0^+\cup H$, observe that $G'=(G_0^+-V(W \cap G_0^+))\cup H-V(G_0^+-(V(W \cap G_0^+)))$. Since $W\subseteq G_0^+$, since $G_0=G_0^+-V(W)$, and since $H$ is disjoint to $W$, we obtain $G'=G_0\cup H$. In particular, we find that $G_0,H\subseteq G'$

        Recall that $D$ is $G_0$-clean. Thus, $\{D\}$ is trivially $G_0$-pristine. Further, $B(\{D\},G_0)=B(D,G_0)=B(D,G_0^+,V(W))=V(G_0^+\cap H)$, and $U(\{D\},G_0)=U(D,G_0)=U(D,G_0^+,V(W))$. So $\scr{J}$ is a planted $U(\{D\},G_0)$-decomposition of $H$, and $B(\{D\},G_0)=V(G_0^+\cap H)$. Since $W$ is disjoint to $H$, $V(G_0^+\cap H)=V(G_0\cap H)$. So $B(\{D\},G_0)=V(G_0\cap H)$.

        This completes the proof.
    \end{proof}

    There are a few extra properties that we want the plane+quasi-vortex embeddings to satisfy. Firstly, we want that all the boundary vertices of the disc are adjacent to a vertex in $W$. This because we will be performing some `slight' modifications to parts that touch the disc, where a part that already touches the disc receives a new vertex on the disc. This condition ensures that if $N_{G_0^+}(V(W))$ intersects a part afterwards, it also did beforehand.

    Next, ideally, we want adjacent vertices in the underlying cycle to be adjacent in $G_0(\Lambda)$. However, because $W$ is allowed to send edges to vertices not in $B(\Lambda)$, we actually can't always ensure this. The compromise is that we only need these edges to deal with some overlap between the bags. The difficulty with the quasi-vortex is dealing with vertices that bleed far across the disc, so any natural gaps caused by an empty intersection are a non-issue. As such, we are satisfied if the decomposition $(H,U(\Lambda),\scr{J})$ is smooth in $G_0(\Lambda)$. This leads to the following definition.

    Say that a plane+quasi-vortex embedding $\Lambda=(G,G_0^+,W,D,H,\scr{J})$ is \defn{standardised} if:
    \begin{enumerate}
        \item $(H,U(\Lambda),\scr{J})$ is smooth in $G_0(\Lambda)$, and
        \item $V(G_0^+\cap H)\subseteq N_{G_0^+}(V(W))$.
    \end{enumerate}

    We remark that the condition that $V(U(\Lambda))\subseteq V(G_0(\Lambda))$ in the definition of `smooth' is automatically satisfied, as $V(U(\Lambda))=V(G_0^+\cap H)=V(G_0(\Lambda)\cap H)\subseteq V(G_0(\Lambda))$ as $W$ is disjoint to $H$.

    \subsection{Reductions}
    \label{SecReductions}

    We now give a formal definition describing how an almost-embedding can `reduce' to a plane+quasi-vortex embedding.

    Let $\Gamma=(G,\Sigma,G_0,\scr{D},H,\scr{J},A)$ be an almost-embedding. A \defn{reduction} $\scr{R}$ for $\Gamma$ is a tuple $(V_M,X,G',\linebreak G_0^+,W,D,H',\scr{J}',\phi)$ such that:
    \begin{enumerate}
        \item $V_M\subseteq V(G_0)$,
        \item $A\subseteq X\subseteq V(G)$,
        \item $(G',G_0^+,W,D,H',\scr{J}')$ is a standardised plane+quasi-vortex embedding,
        \item $V_M$ is disjoint to $V(G')$,
        \item $H-(X\cap V(H))\subseteq H'\subseteq H$, and,
        \item $\phi$ is a map from $V(G')$ to $V(G)$ such that:
        \begin{enumerate}
            \item for each $v\in V(G-X)$, $N_{G-X}[v]\subseteq \phi(N_{G'}[\phi^{-1}(v)])$,
            \item $\phi$ is the identity on $V(G')\setminus V(W)$,
            \item $\phi(V(W))=V_M$, and,
            \item for each $K\in \FT(\Gamma)=\FT(G_0)$ disjoint to $X$, there exists $K'\in \FT(G_0^+)$ with $\phi(K')=K$.
        \end{enumerate}
    \end{enumerate}

    Call $V_M$ the \defn{split-vertices}, $X$ the \defn{loss-set}, and $\phi$ the \defn{projection}. $G'$ is called the \defn{modified graph}, $D'$ the \defn{modified disc}, $H'$ the \defn{modified quasi-vortex}, and $\scr{J}'$ the \defn{modified decomposition}. $G_0^+$ and $W$ are called the \defn{plane subgraph} and \defn{obstruction subgraph} respectively, as before. Say that the \defn{loss} of $\scr{R}$ is $|X|$, the \defn{vortex-width} of $\scr{R}$ is the width of $\scr{J}'$.

    Set \defn{$\Lambda(\scr{R})$}$:=(G',G_0^+,W,D,H',\scr{J}')$. So $\Lambda(\scr{R})$ is a standardised plane+quasi-vortex embedding. Let \defn{$G_0'(\scr{R})$}$:=G_0(\Lambda(\scr{R}))=G_0^+-V(W)$, let \defn{$B'(\scr{R})$}$:=B(\Lambda(\scr{R}))=B(D,G_0^+,V(W))=B(D,G_0^+-V(W))=V(G_0^+\cap H')$, and let \defn{$U'(\scr{R})$}$:=U(\Lambda(\scr{R}))=U(D,G_0^+,V(W))=U(D,G_0^+-V(W))$. So $\scr{J}'$ is a planted $U'(\scr{R})$-decomposition of $H'$, and $D$ is $G_0'(\scr{R})$-clean. We call $B'(\scr{R})$ the \defn{modified boundary} and $U'(\scr{R})$ the \defn{modified underlying cycle}. 

    The definition of $\scr{R}$ is so that the structure of $G'$ loosely matches the structure of $G$. Particularly, most vertices are preserved, most adjacencies are preserved, most facial triangles are preserved, and most of the vortex is preserved. As a guide to the reader, we remark that $V_M$ will be the vertices of the aforementioned `cutting subgraph', and $\phi$ will describe how to `undo' the cut to recover the original graph.
    
    $\phi^{-1}$ is also very well-behaved, as given by the following lemma.

    \begin{lemma}
        \label{mostlyIdenInv}
        Let $\Gamma$ be an almost-embedding of a graph $G$, and let $\scr{R}=(V_M,X,G',G_0^+,W,D,H',\linebreak\scr{J}',\phi)$ be a reduction of $\Gamma$. Then: 
        \begin{enumerate}
            \item $V(G)\setminus (V_M\cup X)\subseteq V(G')\setminus V(W)\subseteq V(G)\setminus V_M$,
            \item $\phi^{-1}(V_M)=V(W)$,
            \item for each $v\in V(G-X)$, $\phi^{-1}(v)$ is nonempty, and
            \item for each $Z\subseteq V(G)\setminus V_M$, $\phi^{-1}(Z)\subseteq Z$. Further, if $Z\subseteq V(G')\setminus V(W)$, then $\phi^{-1}(Z)=Z$.
        \end{enumerate}
    \end{lemma}

    \begin{proof}
        Since $\phi$ is the identity on $V(G')\setminus V(W)$, we have $V(G')\setminus V(W)\subseteq V(G)$. Since $V_M$ is disjoint to $V(G')$, it follows that $V(G')\setminus V(W)\subseteq V(G)\setminus V_M$.
        
        Since $\phi(V(W))\subseteq V_M$, observe that $V(W)\subseteq \phi^{-1}(V_M)$. Presume, for a contradiction, that there exists $x\in \phi^{-1}(V_M)\setminus V(W)$. So $\phi(x)\in V_M$ and $x\in V(G')\setminus V(W)$. Thus, $x=\phi(x)$. We obtain $x=\phi(x)\in V(G')\cap V_M$, which contracts the fact that $V(G')$ and $V_M$ are disjoint. So $\phi^{-1}(V_M)=V(W)$.

        Fix $v\in V(G-X)$. Presume, for a contradiction, that $\phi^{-1}(v)$ is empty. Then $\phi(N_{G'}[\phi^{-1}(v)])$ is empty. However, observe that $v\in N_{G-X}[v]$. So $N_{G-X}[v]\nsubseteq \phi(N_{G'}[\phi^{-1}(v)])$, a contradiction. Thus, $\phi^{-1}(v)$ is nonempty.

        Let $Z\subseteq V(G)\setminus V_M$. Presume, for a contradiction, that there exists $x\in \phi^{-1}(Z)\setminus Z$. So $\phi(x)\in Z$ and $x\notin Z$. If $x\in V(G')\setminus V(W)$, we would have $x=\phi(x)\in Z$, contradicting the fact that $x\notin Z$. So $x\in V(W)$, and thus $\phi(x)\subseteq \phi(V(W))=V_M$. We obtain $\phi(x)\in V_M\cap Z$, contradicting the choice of $Z$. So $\phi^{-1}(Z)\subseteq Z$, as desired.

        If $Z\subseteq V(G')\setminus V(W)\subseteq V(G)\setminus V_M$, note that $\phi^{-1}(Z)$ is well-defined. Since $Z\subseteq V(G')\setminus V(W)$, observe that $\phi(Z)$ is well-defined and $\phi(Z)=Z$. Thus, $Z\subseteq \phi^{-1}(Z)$. Combined with the previous paragraph, this gives $\phi^{-1}(Z)=Z$, as desired.
        
        Fix $v\in V(G)\setminus (V_M\cup X)$. Since $v\in V(G-X)$, there exists $x\in \phi^{-1}(v)$. Since $\{v\}\subseteq V(G)\setminus V(M)$, $\phi^{-1}(v)\subseteq \{v\}$. Thus, $\{v\}=\{x\}=\phi^{-1}(v)$. So $v\in \phi^{-1}(v)\subseteq V(G')$. Since $\phi(V(W))=V_M$ and since $\phi(v)=v\notin V_M$ and $v\notin V(W)$. So $v\in V(G')\setminus V(W)$. Thus, $V(G)\setminus (V_M\cup X)\subseteq V(G')\setminus V(W)$. Combining this with our earlier result, we obtain $V(G)\setminus (V_M\cup X)\subseteq V(G')\setminus V(W)\subseteq V(G)\setminus V_M$, as desired.

        This completes the proof of the lemma.
    \end{proof}
    
    There is one more important parameter of $\scr{R}$ we need to consider, but it requires some explanation. When we `cut' the embedded graph, we find that every cutting vertex gets turned into multiple `cut' vertices. This is captured by the projection $\phi$. We want to assign a part to the cutting vertex, using what we know about partitioning $G'$. Normally, we would just use the inverse of $\phi$ to decide. However, if the `cut' vertices belong to multiple parts, it is unclear as to which part we should place the cutting vertex. The easiest solution is to just force all the vertices of the cut subgraph into a single part. This corresponds to putting all the split-vertices into a single part. However, this creates a new problem, because the `cutting subgraph' might be large. The solution is to find vertices whose deletion breaks the graph into components which each only intersect a `small' number of split-vertices. We include these vertices in $X$. We then partition each component according to how we would have partitioned the entire graph. Now, the part corresponding to the split-vertices is small for each component.

    This leads to the following definition.

    Let $\Gamma$ be an almost-embedding of a graph $G$, and let $\scr{R}=(V_M,X,G',G_0^+,W,D,H',\scr{J}',\phi)$ be a reduction of $\Gamma$. Let \defn{$\scr{C}(\scr{R})$} be the set of connected components of $G-X$. Define the \defn{remainder-width} to be the maximum intersection between $V(C)$ and $V_M$ for some $C\in \scr{C}(\scr{R})$.
    
    It is worth noting that the loss-set $X$ serves two roles. Firstly, it contains any vertices that cannot be described using the plane+quasi-vortex embedding. Second, it controls the remainder-width. We don't distinguish between these uses because $X$ will be the loss set of the almost-partition.

    One of the main technical achievements of this paper is an algorithm for finding a reduction. Specifically, we show the following.

    \begin{restatable}{theorem}{findReductionModify}
        \label{findReductionModify}
        Let $g,p,k,a\in \ds{N}$, and let $\Gamma$ be a nontrivial $(g,p,k,a)$ almost-embedding with disjoint discs of a graph $G$. Then there exists a layering $\scr{L}$ of $G$ and a reduction $\scr{R}=(V_M,X,G',G_0^+,W,\linebreak D,H',\scr{J}',\phi)$ of $\Gamma$ of vortex-width at most $k$ and loss at most $(4g+9p)(k+1)+a$ such that $|V_M\cap L|\leq 2g+5p+1$ for each $L\in \scr{L}$.
    \end{restatable}
    We will define a `layering' momentarily.
    
    We remark that the nontrivial condition just exists to avoid some annoying and uninteresting edge cases.

    We believe that \cref{findReductionModify} may be of independent interest, as a general strategy to simplify an almost-embedding.

    \subsection{Layerings}

    A \defn{layering} of a graph $G$ is a partition $\scr{L}=(L_0,L_1,\dots,L_m)$ (where $m:=|\scr{L}|-1$) of $G$ such that, for each $i\in \{0,\dots,m\}$, $N_G[L_i]\subseteq L_{i-1}\cup L_i\cup L_{i+1}$ (where $L_{-1}:=L_{m+1}:=\emptyset$). Equivalently, $G/\scr{L}$ is a subgraph of a path. The sets $L_i$, $i\in \{0,\dots,m\}$, are called the \defn{layers}. The \defn{breath-first search layering} (\defn{BFS layering}) from some $r\in V(G)$ is defined by setting $L_i:=\{v\in V(G):\dist_G(r,v)=i\}$. It is well-known that the BFS layering is a layering.

    Layerings provide natural ways to `separate' a graph, as seen with the following observation. 

    \begin{observation}
        \label{layeringSeperation}
        Let $b,n\in \ds{N}$, let $\scr{L}$ be a layering of an $n$-vertex graph $G$, and let $S\subseteq V(G)$ be such that for each $L\in \scr{L}$, $|L\cap S|\leq b$. Then for any $d\in \ds{R}^+$, there exists $X\subseteq V(G)$ with $|X|\leq bn/d$ such that each connected component of $G-X$ contains at most $d$ vertices in $S$.
    \end{observation}
    \begin{proof}
        Let $\scr{L}=(L_0,L_1,\dots,L_m)$ (where $m:=|\scr{L}|-1$), and let $q:=\floor{d/b}+1$. Note that $n/q\leq bn/d$ and that $b(q-1)\leq d$.
        
        For each $i\in \{0,\dots,q-1\}$, let $L_i':=\bigcup_{\{i\in \{0,\dots,m\}:j\equiv i\pmod{q}\}}L_j$. Observe that $(L_0',\dots,L_{q-1}')$ is a partition of $G$ into exactly $q$ parts. Thus, there exists $i\in \{0,\dots,q-1\}$ such that $|L_i'|\leq n/q\leq bn/d$. Set $X:=L_i'$.

        For each connected component $C$ of $G-X$, observe that there exists $j\in \ds{Z}$ such that $V(C)\subseteq L_{qj+i+1}\cup \dots \cup L_{q(j+1)+i-1}$ (where $L_k:=\emptyset$ whenever $k\notin \{0,\dots,m\}$). In particular, $C$ intersects at most $q-1$ layers of $\scr{L}$. Thus, $|V(C)\cap S|\leq b(q-1)\leq d$, as desired.
    \end{proof}

    One can always add more vertices to the loss-set of a reduction (at the cost of increasing the loss).

    \begin{restatable}{observation}{extendReduction}
        \label{extendReduction}
        Let $\Gamma$ be an almost-embedding of a graph $G$, let $\scr{R}=(V_M,X,G',G_0^+,W,D,H',\linebreak\scr{J}',\phi)$ be a reduction of $\Gamma$, and let $X'\subseteq V(G)$. Then $\scr{R}':=(V_M,X\cup X',G',G_0^+,W,D,H',\scr{J}',\phi)$ is a reduction of $\Gamma$ of the same vortex-width and at most the same remainder-width. Further, if $\scr{R}$ has loss at most $q\in \ds{R}_0^+$ and $|X'|\leq q'\in \ds{R}_0^+$, then $\scr{R}'$ has loss at most $q+q'$.
    \end{restatable}

    \begin{proof}
        Let $\Gamma:=(G,\Sigma,G_0,\scr{D},H,\scr{J},A)$. As $\scr{R}$ is a reduction of $\Gamma$ and $X'\subseteq V(G)$, observe that:
        \begin{enumerate}
            \item $A\subseteq X\subseteq X\cup X'\subseteq V(G)$,
            \item for each $v\in V(G-(X\cup X'))\subseteq V(G-X)$, $N_{G-(X\cup X')}[v]\subseteq N_{G-X}[v]\subseteq \phi(N_{G'}[\phi^{-1}(v)])$,
            \item $\phi$ is the identity on $V(G')\setminus V(W)$, and,
            \item for each $K\in \FT(\Gamma)=\FT(G_0)\subseteq \FT(G_0')$ disjoint to $X\cup X'\supseteq X$, there exists $K'\in \FT(G_0^+)$ with $\phi(K')=K$.
        \end{enumerate}
        $\scr{R}'$ satisfies all other properties of a reduction trivially (given that only the loss-set is different to $\scr{R}$). Thus, $\scr{R}'$ is a reduction for $\Gamma$.

        The vortex-width of $\scr{R}'$ is the width of $\scr{J}'$, which is the vortex-width of $\scr{R}$. So the vortex-width of $\scr{R}'$ is the vortex-width of $\scr{R}$.

        Let $w$ be the remainder-width of $\scr{R}$. Observe that each connected component $C'$ of $G-(X\cup X')$ is contained in a connected component $C$ of $G-X$. Note that $|V_M\cap V(C')|\leq |V_M\cap V(C)|\leq w$. So the remainder-width of $\scr{R}'$ is at most $w$, which is the remainder-width of $\scr{R}$.

        If $\scr{R}$ has loss at most $q\in \ds{R}_0^+$, then $|X|\leq q$. If also $|X'|\leq q'\in \ds{R}_0^+$, then $|X\cup X'|\leq q+q'$. So $\scr{R}'$ has loss at most $q+q'$. This completes the proof.
    \end{proof}

    So we can make use of \cref{layeringSeperation} to convert the reduction from \cref{findReductionModify} to one with bounded remainder-width.

    \begin{restatable}{lemma}{findReduction}
        \label{findReduction}
        Let $g,p,k,a,n\in \ds{N}$, and let $\Gamma$ be a nontrivial $(g,p,k,a)$ almost-embedding with disjoint discs of an $n$-vertex graph $G$. Then for any $d\in \ds{R}^+$, there exists a reduction $\scr{R}$ of $\Gamma$ of vortex-width at most $k$, remainder-width at most $d$, and loss at most $(2g+5p+1)n/d+(4g+9p)(k+1)+a$.
    \end{restatable}

    \begin{proof}
        By \cref{findReductionModify}, there exists a layering $\scr{L}$ of $G$ and a reduction $\scr{R}'=(V_M,X,G',G_0^+,W,D,H',\linebreak\scr{J}',\phi)$ of $\Gamma$ of vortex-width at most $k$ and loss at most $(4g+9p)(k+1)+a$ such $|V_M\cap L|\leq 2g+p+1$ for each $L\in \scr{L}$. By \cref{layeringSeperation} (with $S:=V_M$ and $b:=2g+5p+1$), there exists $X'\subseteq V(G)$ with $|X'|\leq (2g+5p+1)n/d$ such that each connected component of $G-X$ contains at most $d$ vertices in $V_M$. By \cref{extendExtend}, $\scr{R}:=(V_M,X\cup X',G',G_0^+,W,D,H',\scr{J}',\phi)$ is a reduction of $\Gamma$ of vortex-width at most $k$ and loss at most $(2g+5p+1)n/d+(4g+9p)(k+1)+a$. Further, each connected component of $G-(X\cup X')$ is contained in a connected component of $G-X'$, and thus has at most $d$ vertices of $V_M$. Hence, $\scr{R}$ has remainder-width at most $d$, as desired.
    \end{proof}

    \subsection{Adjustments}

    It is important that we know how to almost-partition the plane subgraph $G_0^+$ in a `good' way, as this will be the starting point for how we find the partition for the graph $G$. We cannot just say `find an appropriate almost-partition of $G_0^+$', because we only know how to find `appropriate almost-partitions' for very specific plane graphs (resembling a triangulation). So we have to `adjust' $G_0^+$ so that it becomes one of these specific plane graphs.
    
    This `adjustment' needs to be mindful of the disc $D$ and obstruction subgraph $W$, but we can strip away all other information from the reduction. This leads to the following definition.

    An \defn{adjustable triple} is a triple $(G,W,D)$ such that:
    \begin{enumerate}
        \item $G$ is a plane graph,
        \item $W$ is a connected and nonempty plane subgraph of $G$, and,
        \item $D$ is a $(G,V(W))$-clean disc.
    \end{enumerate}

    \begin{observation}
        \label{reductionAdjustment}
        Let $\Gamma$ be an almost-embedding, and let $\scr{R}=(V_M,X,G',G_0^+,W,D,H',\scr{J}',\phi)$ be a reduction of $\Gamma$. Then $(G_0^+,W,D)$ is an adjustable triple.
    \end{observation}
    \begin{proof}
        Follows directly from the definition of a plane+quasi-vortex embedding.
    \end{proof}

    We use \cref{reductionAdjustment} implicitly from now on.

    As for `how' we adjust $G_0^+$, we allow for embedding of new edges and even deleting of vertices. The main restrictions is that we don't allow deleting vertices from the  obstruction (as we need to keep it connected), and that the edges should be added in a way that doesn't `destroy' facial triangles (we are not worried about facial triangles lost by deleting vertices). This leads to the following definition.

    An \defn{adjustment} for an adjustable triple $(G,W,D)$ is a pair $\scr{A}=(G',S)$ such that:
    \begin{enumerate}
        \item $S\subseteq V(G)\setminus V(W)$,
        \item $G'$ is a spanning plane supergraph of $G-S$,
        \item for each $K\in \FT(G)$ disjoint to $S$, $K\in \FT(G')$, and,
        \item $D$ is $(G',V(W))$-clean.
    \end{enumerate}
    We remark that the first two conditions ensure that $W$ is a (connected and nonempty) subgraph of $G'$, so the fourth condition is well-defined. Say that the \defn{loss} of $\scr{A}$ is $|S|$.

    We remark that, within the scope of this paper, we will only produce adjustments of loss $0$ ($S=\emptyset$). However, we allow for $S$ to be nonempty to maximise the chance that our lemmas can be reused in the future (to try and improve \cref{main}).

    Say that an \defn{adjustment} for a reduction $(V_M,X,G',G_0^+,W,D,H',\scr{J}',\phi)$ is an adjustment for\linebreak$(G_0^+,W,D)$.

    We have a method of `applying' an adjustment to a reduction to obtain a similar reduction with many of the same properties.

    \begin{restatable}{lemma}{applyAdjustment}
        \label{applyAdjustment}
        Let $\scr{R}$ be a reduction of an almost-embedding $\Gamma$ of vortex-width $k\in \ds{N}$, remainder-width at most $w'\in \ds{R}_0^+$, and loss at most $q\in \ds{R}_0^+$, and let $\scr{A}=(G_0^-,S)$ be an adjustment for $\scr{R}$ of loss at most $q'\in \ds{R}_0^+$. Then there exists a reduction $\scr{R}'$ for $\Gamma$ of vortex-width at most $k$, remainder-width at most $w'$, and loss at most $(k+1)q'+q$ such that:
        \begin{enumerate}
            \item the plane subgraph of $\scr{R}'$ is $G_0^-$, and,
            \item the split vertices $V_M$, obstruction subgraph $W$, and disc $D$ of $\scr{R}'$ are the split vertices, obstruction subgraph, and disc of $\scr{R}$ respectively.
        \end{enumerate}
    \end{restatable}
    
    \subsection{Partitioned reductions}
    \label{SecPReductions}

    We mentioned that we wanted a `good' partition of the plane subgraph $G_0^+$. By what exactly do we mean by `good'. There are a couple of things. Firstly, and most importantly, we want the partition to be connected. This is so that we can exploit the non-crossing property (\cref{nonCrossing}). The second is that we want $V(W)$ (where $W$ is the obstruction subgraph) to be a part of the partition. This is because we want to handle it separately to all other parts (using the remainder-width). Thirdly, we want the partition to have small treewidth. Finally, it would be ideal if the partition had bounded width. However, the requirement that $V(W)$ is a part will normally make this impossible, as $W$ has unbounded size. The compromise is that we require that $V(W)$ is the only `large' part.

    For any possible reduction, we want to be able to generate an adjustment so that the new plane subgraph admits a partition with all of these properties. This leads to the following definition.

    Say that a pair of functions $f,g:\ds{N}\mapsto \ds{R}_0^+$ \defn{generate adjustments of treewidth $b$} if for each $n\in \ds{N}$, every $d\in \ds{R}^+$ with $d\geq g(n)$, and each adjustable triple $(G,W,D)$ with $|V(G)|-|V(W)|+1\leq n$, there exists a triple $(G',S,\scr{P})$ such that:
    \begin{enumerate}
        \item $(G',S)$ is an adjustment of $(G,W,D)$ of loss at most $f(n)/d$, and,
        \item $\scr{P}$ is a connected partition of $G'$ of treewidth at most $b$ with $V(W)\in \scr{P}$ such that each $P\in \scr{P}$ distinct from $V(W)$ has size at most $d$.
    \end{enumerate}
    
    We remark that since $W$ is nonempty, a part is distinct from $V(W)$ if and only if it is not equal to $V(W)$. We only allow partitions of $G'$, rather than almost-partitions, as any would-be loss vertices can instead be added to $S$ (although we do not show this in this paper).

    The condition on the number of vertices of $G$ should be thought of as `after contracting $V(W)$ into a single vertex, we want the result to have at most $n$ vertices'. Since $W$ is nonempty, this reduces the number of vertices by $|V(W)|-1$. Since we don't care how many vertices are in $V(W)$, in terms of finding a desired partition, it is effectively equivalent to find a partition of the graph when $V(W)$ is contracted. This also allows us to reduce the number of vertices in the graph, which is ideal, given that we expect the sizes of the parts to depend on the number of vertices. In particular, this is nice because $G_0^+$ can have more vertices than $G$, but only because of $V(W)$ being larger than $V_M$ (by \cref{mostlyIdenInv}). Adding this detail allows us to bound the size of the parts in terms of $|V(G)|$.

    It is convenient to come up with a term for a reduction for which we have found a `desired' partition of the plane subgraph. This leads to the following definition.

    A \defn{partitioned reduction} of an almost-embedding $\Gamma$ is a collection $\scr{PR}=(V_M,X,G',G_0^+,W,D,H',\linebreak\scr{J}',\phi,\scr{P}_0^+)$ such that:
    \begin{enumerate}
        \item $(V_M,X,G',G_0^+,W,D,H',\scr{J}',\phi)$ is a reduction of $\Gamma$, and,
        \item $\scr{P}_0^+$ is a connected partition of $G_0^+$ containing $V(W)$.
    \end{enumerate}
    As before, we call $V_M,X,G',G_0^+,W,D,H',\scr{J}',\phi$ the \defn{split-vertices}, \defn{loss-set}, \defn{modified graph}, \defn{plane subgraph}, \defn{obstruction subgraph}, \defn{modified quasi-vortex}, \defn{decomposition}, and \defn{projection} of $\scr{PR}$ respectively. We call $\scr{P}_0^+$ the \defn{plane partition} of $\scr{PR}$. We set \defn{$\scr{R}(\scr{PR})$}$:=\scr{R}:=(V_M,X,G',G_0^+,W,D,H',\scr{J}',\phi)$, \defn{$B'(\scr{PR})$}$:=B'(\scr{R})$, \defn{$U'(\scr{PR})$}$:=U'(\scr{R})$, \defn{$G_0'(\scr{PR})$}$:=G_0'(\scr{R})$, \defn{$\scr{C}(\scr{PR})$}$:=\scr{C}(\scr{R})$, and \defn{$\Lambda(\scr{PR})$}$:=\Lambda(\scr{R})$. Say that the \defn{vortex-width}, \defn{remainder-width}, and \defn{loss} of $\scr{PR}$ are the vortex-width, remainder-width, and loss of $\scr{R}$, the \defn{treewidth} of $\scr{PR}$ is the treewidth of $\scr{P}_0^+$, and the \defn{near-width} of $\scr{PR}$ to be the maximum size of a part $P\in \scr{P}_0^+$ distinct from $V(W)$.

    Given functions that generate adjustments, we can use \cref{applyAdjustment} to turn reductions into partitioned reductions.

    \begin{lemma}
        \label{adjustmentToPReduction}
        Let $b,n\in \ds{N}$, let $f,g:\ds{N}\mapsto \ds{R}_0^+$ be such that $(f,g)$ generate adjustments of treewidth $b$, and let $d\in \ds{R}^+$ be such that $d\geq g(n)$. Let $\Gamma$ be an almost-embedding of an $n$-vertex graph $G$, and let $\scr{R}$ be a reduction of $\Gamma$ of vortex-width $k\in \ds{N}$, remainder-width at most $w'\in \scr{R}_0^+$, and loss at most $q\in \ds{R}_0^+$. Then there exists a partitioned reduction $\scr{PR}$ for $\Gamma$ of treewidth at most $b$, vortex-width at most $k$, near-width at most $d$, remainder-width at most $w'$, and loss at most $(k+1)f(n)/d + q$.
    \end{lemma}

    \begin{proof}
        Let $(G,\Sigma,G_0,\scr{D},H,\scr{J},A):=\Gamma$ and $(V_M,X,G',G_0^+,W,D,H',\scr{J}',\phi):=\scr{R}$. By \cref{mostlyIdenInv}, $V(G')\setminus V(W)\subseteq V(G)\setminus V_M$. So $|V(G')|-|V(W)|+1\leq |V(G)|-|V_M|+1$. Since $W$ is nonempty and since $\phi(V(W))=V_M$, $V_M$ is nonempty. Thus, $|V(G)|-|V_M|+1\leq |V(G)|=n$. So $|V(G')|-|V(W)|+1\leq n$. Recalling that $G_0^+\subseteq G'$, this gives $|V(G_0^+)|-|V(W)|+1\leq n$
        
        By \cref{reductionAdjustment}, $(G_0^+,W,D)$ is an adjustable triple. By definition of $f,g$ (and since $d\geq g(n)$ and $|V(G_0^+)|-|V(W)|+1\leq n$), there exists a triple $(G_0^-,S,\scr{P}_0^+)$ such that:
        \begin{enumerate}
            \item $(G_0^-,S)$ is an adjustment of $G_0^+$ of loss at most $f(n)/d$, and,
            \item $\scr{P}_0^+$ is a connected partition of $G_0^-$ of treewidth at most $b$ with $V(W)\in \scr{P}_0^+$ such that each $P\in \scr{P}_0^+$ distinct from $V(W)$ has size at most $d$.
        \end{enumerate}

        By \cref{applyAdjustment} (with $q':=f(n)/d$), there exists a reduction $\scr{R}'$ for $\Gamma$ of vortex-width at most $k$, remainder-width at most $w'$, and loss at most $(k+1)f(n)/d+q$ such that:
        \begin{enumerate}
            \item the plane subgraph of $\scr{R}'$ is $G_0^-$, and,
            \item the split vertices $V_M$, obstruction subgraph $W$, and disc $D$ of $\scr{R}'$ are the split vertices, obstruction subgraph, and disc of $\scr{R}$ respectively.
        \end{enumerate}
        So $\scr{R}'$ is of the form $(V_M,X^+,G^-,G_0^-,W,D,H^-,\scr{J}^-,\phi^-)$ for some unknown (and unimportant) $X^+,G^-,H^-,\scr{J}^-,\phi^-$.
        
        Let $\scr{PR}:=(V_M,X^+,G^-,G_0^-,W,D,H^-,\scr{J}^-,\phi^-,\scr{P}_0^+)$. It follows directly from the definitions of $\scr{R}',\scr{P}_0^+$ that $\scr{PR}$ is a partitioned reduction of $\Gamma$ of treewidth at most $b$, vortex-width $k$, near-width at most $d$, remainder-width at most $w'$, and loss at most $(k+1)f(n)/d+q$.
    \end{proof}

    \subsection{Raises}
    \label{SecRaises}

    As stated earlier, the key difficulty is finding an `appropriate' way to merge parts, with `appropriate' breakpoints, so that we can then almost-partition the vortex. Specifically, it would be ideal to ensure that each interval that avoids the breakpoints only intersects one part. Then, we can put all quasi-vortex vertices that lie in a bag along that interval into the corresponding part. However, we do not achieve this. In fact, we suggest that this goal is not realistic, as one part could intersect the disc in many places, creating many small intervals between intersections. We would then have to merge across all these small intervals into the part, or add breakpoints for each of these intervals, or some combination of these two methods. However, regardless of the method, either the width or the loss would be high.

    Thus, instead of requiring that each `interval' only intersects one part, we instead search for a clear way of assigning vertices $v$ of the quasi-vortex to a part. We remark that the assignment method will be completely determined by the interval formed by the bags that contain the vortex. The part we assign to $v$ should contain a boundary vertex $x$ whose bag contains $v$. We think of $v$ as using $x$ to decide its part. We allow boundary vertices to be assigned to a part different from the part containing them. We think of non-boundary vertices as being assigned to their original parts.
    
    This `assignment' describes a partition of the modified graph. We then use the inverse of the projection, along with the loss-set, to almost-partition the original graph. We want to ensure that this creates no `new' edges in the quotient. We achieve this by ensuring that adjacent vertices in the modified graph (that do not map to vertices in the loss set) are assigned to adjacent parts. This leads to the following definition.

    Let $\scr{PR}=(V_M,X,G',G_0^+,W,D,H',\scr{J}',\phi,\scr{P}_0^+)$ be a partitioned reduction of an almost-embedding $\Gamma$. Set $B':=B'(\scr{R})$ and set $(J_x':x\in B'):=\scr{J}'$. A \defn{raise} for $\scr{PR}$ is a map $\tau:V(G')\mapsto \scr{P}_0^+$ such that:

    \begin{enumerate}
        \item for each $v\in V(G_0^+)\setminus B'$, $v\in \tau(v)$,
        \item for each $v\in V(H')$, there exists $x\in B'\cap \tau(v)$ such that $v\in J_x'$, and,
        \item for each pair of adjacent vertices $u,v$ in $G'$ with $\phi(u),\phi(v)\in V(G-X)$, $N_{G_0^+}[\tau(u)]$ intersects $\tau(v)$.
    \end{enumerate}

    For each $C\in \scr{C}(\scr{R})$, let \defn{$\scr{P}_C(\tau)$}$:=\scr{P}_C:=(\phi(\tau^{-1}(P))\cap V(C):P\in \scr{P}_0^+)$, and let \defn{$\psi_C(\tau)$}$:=\psi_C$ be the obvious bijection from $\scr{P}_C$ to $\scr{P}_0^+$ such that $\phi(\tau^{-1}(\psi_C(P)))\cap V(C)=P$ for each $P\in \scr{P}_C$. Let \defn{$\scr{P}(\tau)$}$:=\bigsqcup_{C\in \scr{C}}\scr{P}_C$.

    As a guide to the reader, the pair $(X,\scr{P}(\tau))$ will be the almost-partition.

    We have two important results regarding raises. The first uses the raise to find the almost-partition, and the second is an `existence' result, in which we merge parts and delete some vertices (add to the loss-set) until we can find a raise.

    Specifically, we have the following.

    \begin{restatable}{lemma}{raiseToPartition}
        \label{raiseToPartition}
        Let $\Gamma$ be an almost-embedding of a graph $G$, let $\scr{PR}$ be a partitioned reduction of $\Gamma$ of treewidth at most $b\in \ds{N}$, vortex-width at most $k\in \ds{N}$, near-width at most $w\in \ds{R}_0^+$, remainder-width at most $w'\in \ds{R}_0^+$, and loss at most $q\in \ds{R}_0^+$. If $\scr{PR}$ admits a raise $\tau$, then $G$ admits an almost-partition of treewidth at most $b$, width at most $\max((k+1)w,w')$, and loss at most $q$ that is $(1,b(b+1)(k+1),\VC(\Gamma))$-concentrated and $(2,(6k+8)\binom{b+1}{3},\Attachable(\Gamma))$-concentrated.
    \end{restatable}

    \begin{restatable}{lemma}{findRaise}
        \label{findRaise}
        Let $n\in \ds{N}$, let $\Gamma$ be an almost-embedding of an $n$-vertex graph $G$, and let $\scr{PR}$ be a partitioned reduction of an almost-embedding $\Gamma$ of treewidth at most $b\in \ds{N}$, vortex-width at most $k\in \ds{N}$, near-width at most $w\in \ds{R}_0^+$, remainder-width at most $w'\in \ds{R}_0^+$, and loss at most $q\in \ds{R}_0^+$. Then for every $d\in \ds{R}^+$ with $d\geq w$, there exists a partitioned reduction $\scr{PR}'$ of $\Gamma$ of treewidth at most $b$, vortex-width at most $k$, near-width at most $d$, remainder-width at most $w'$, and loss at most $28(k+1)n/d + q$ that admits a raise.
    \end{restatable}

    Using all the results we have listed thus far, we can prove \cref{AETechinical}.

    \AETechinical*
    

    \begin{proof}
        Set $q_1:=(2g+5p+1)n/d+(4g+9p)(k+1)+a$, $d':=d/(k+1)$, $q_2:=(k+1)f(n)/d'+q_1$, and $q_3:=28(k+1)n/d' + q_2$. Observe that $q_3=((k+1)^2(28+f(n))+(2g+5p+1))n/d+(4g+9p)(k+1)+a$. So we must show that almost-partition has loss at most $q_3$. Note that $d'\geq g(n)$ and $a\leq q_3$.
    
        Let $(G,\Sigma,G_0,\scr{D},H,\scr{J},A):=\Gamma$. If $\Gamma$ is trivial, then $|V(G)|=|V(A)|\leq a$, and $(V(G),\emptyset)$ is trivially an almost-partition of $G$ of treewidth $-1<b$, width $0<d$, and loss $a\leq q_3$. Further, $(V(G),\emptyset)$ is trivially $(1,b(b+1)(k+1),\VC(\Gamma))$-concentrated and $(2,(6k+8)\binom{b+1}{3},\Attachable(\Gamma))$-concentrated (in fact, it is $(0,0,\VC(\Gamma))$-concentrated and $(0,0,\Attachable(\Gamma))$-concentrated). So we may assume $\Gamma$ is nontrivial.
    
        By \cref{findReduction}, there exists a reduction $\scr{R}$ for $\Gamma$ of vortex-width at most $k$, remainder-width at most $d$, and loss at most $(2g+5p+1)n/d+(4g+9p)(k+1)+a=q_1$. By \cref{adjustmentToPReduction} (with $d:=d'\geq g(n)$ and $q:=q_1$), there exists a partitioned reduction $\scr{PR}$ for $\Gamma$ of treewidth at most $b$, vortex-width at most $k$, near-width at most $d'$, remainder-width at most $w'$, and loss at most $(k+1)f(n)/d + q_1=q_2$.

        By \cref{findRaise} (with $d:=d'\geq d'$), there exists a partitioned reduction $\scr{PR}'$ for $\Gamma$ of treewidth at most $b$, vortex-width at most $k$, near-width at most $d'$, remainder-width at most $w'$, and loss at most $28(k+1)n/d' + q_2=q_3$ that admits a raise. By \cref{raiseToPartition} (with $w:=d'$, $w':=d$, and $q:=q_3$), $G$ admits an almost-partition of treewidth at most $b$, width at most $\max((k+1)d',d)=\max(d,d)=d$ and loss at most $q_3$ that is $(1,b(b+1)(k+1),\VC(\Gamma))$-concentrated and $(2,(6k+8)\binom{b+1}{3},\Attachable(\Gamma))$-concentrated. This completes the proof.
    \end{proof}

    \subsection{Structure of the remainder of this paper}
    \label{SecRemainingStructure}
    
    Five results remain to be proved, namely \cref{planePartitions}, \cref{findReductionModify}, \cref{applyAdjustment}, \cref{raiseToPartition}, and \cref{findRaise}. \cref{planePartitions} is not too long (after repeating the proof of \cref{tw2Surfaces}), and the proof of \cref{applyAdjustment} is relatively simple once the correct definition has been introduced. \cref{secAdjustmentLemmas} is devoted to proving these two results. \cref{findReductionModify} has by far the most involved proof, but it is largely just a technical process. Its proof is the subject of \cref{SecFindReduction}. We consider \cref{raiseToPartition} and \cref{findRaise} to be the most interesting results, and we consider \cref{findRaise} in particular to be the heart of this paper. The proof of the former is the subject of \cref{SecPartitionWithRaise}, and the proof of the latter is the subject of \cref{SectionMerging}.

    We have chosen to order our proofs based on what we think is the most motivating order. However, they can be read in almost any order, with the only limitations being that \cref{applyAdjustment} and \cref{planePartitions} each borrow a definition and lemma from \cref{SecFindReduction} (regarding modifying a vortex and a general technique for embedding edges respectively).

    \section{Partitioning using a raise}
    \label{SecPartitionWithRaise}

    In this section we prove \cref{raiseToPartition}, in which we use a raise to construct the almost-partition. As stated before, the pair $(X,\scr{P}(\tau))$ will be the almost-partition.
    
    \subsection{Confirming the partition}

    We first verify that this is an almost-partition.

    \begin{lemma}
        \label{raiseAP}
        Let $\Gamma$ be an almost-embedding of a graph $G$, let $\scr{PR}$ be a partitioned reduction of $\Gamma$, let $X$ be the loss set of $\scr{PR}$, and let $\tau$ be a raise for $\scr{PR}$. Then $(X,\scr{P}(\tau))$ is an almost-partition of $G$.
    \end{lemma}

    \begin{proof}
        Let $(V_M,X,G',G_0^+,W,D,H',\scr{J}',\phi,\scr{P}_0^+):=\scr{PR}$, $B':=B'(\scr{PR})$, and $\scr{C}:=\scr{C}(\scr{PR})$. For each $C\in \scr{C}$, let $\psi_C:=\psi_C(\tau)$ and $\scr{P}_C:=\scr{P}_C(\tau)=(\phi(\tau^{-1}(P))\cap V(C):P\in \scr{P}_0^+)$. Recall that $\scr{P}:=\scr{P}(\tau)=\bigsqcup_{C\in \scr{C}}\scr{P}_C$. Also, for each $C\in \scr{C}$, recall that $\psi_C$ is a bijection from $\scr{P}_C$ to $\scr{P}_0^+$ such that $\phi(\tau^{-1}(\psi_C(P)))\cap V(C)=P$ for each $P\in \scr{P}_C$.

        Since each $C\in \scr{C}$ is a subgraph of $G-X$ and since each $P\in \scr{P}_C$ is contained in $V(C)$, it is immediate that $\scr{P}\subseteq 2^{V(G-X)}$.

        Consider any $v\in V(G-X)$. Since $\scr{C}$ is the connected components of $G-X$, observe that $(V(C):C\in \scr{C})$ is a partition of $G-X$. Thus, there exists $C\in \scr{C}$ such that $v\in V(C)$. By \cref{mostlyIdenInv}, there exists $x\in \phi^{-1}(v)$. So $\phi(x)=v$. Let $P':=\tau(x)$ and $P:=\psi^{-1}_C(P')=\phi(\tau^{-1}(P'))\cap V(C)$. Note that $P\in \scr{P}_C\subseteq \scr{P}$. Observe that $x\in \tau^{-1}(P')$, and thus $v\in \phi(\tau^{-1}(P'))$. Since $v\in V(C)$, we obtain $v\in P\in \scr{P}$. So each vertex of $G-X$ is contained in at least one element of $\scr{P}$.

        We first show that the elements of $\scr{P}$ are pairwise disjoint. Fix distinct $P,P'\in \scr{P}$. There exists $C,C'\in \scr{C}$ such that $P\in \scr{P}_C$ and $P'\in \scr{P}_{C'}$. Recall that $P\subseteq V(C)$ and $P'\subseteq V(C')$. Recall that $(V(C^*):C^*\in \scr{C})$ is a partition of $G-X$. Thus, if $C\neq C'$, then $P\subseteq V(C)$ is disjoint to $P'\subseteq V(C')$. So we may assume that $C=C'$. Then since $P,P'$ are distinct, $P$ and $P'$ are distinct elements of $\scr{P}_C$. Since $\psi_C$ is a bijection, $\psi_C(P)$ and $\psi_C(P')$ are distinct elements of $\scr{P}_0^+$. Thus, $\psi_C(P)$ is disjoint to $\psi_C(P')$. Hence, $\tau^{-1}(\psi_C(P))$ is disjoint to $\tau^{-1}(\psi_C(P'))$.

        Presume, for a contradiction, that $P$ intersects $P'$. Since $P\subseteq \phi(\tau^{-1}(\psi_C(P)))$ and $P'\subseteq\linebreak \phi(\tau^{-1}(\psi_C(P)))$, we find that $\phi(\tau^{-1}(\psi_C(P)))$ intersects $\phi(\tau^{-1}(\psi_C(P)))$. So there exists $u\in \linebreak\tau^{-1}(\psi_C(P))$ and $v\in \tau^{-1}(\psi_C(P))$ such that $\phi(u)=\phi(v)$. Since $\tau^{-1}(\psi_C(P))$ is disjoint to $\tau^{-1}(\psi_C(P'))$, $u\neq v$. So $\phi(x)\neq x$ for at least one $x\in \{u,v\}$. Without loss of generality, say $x=u$.

        Since $\phi(u)\neq u$, $u\in V(W)$. Thus, $\phi(u)\in V_M$. Since $V_M$ is disjoint to $G'$ (and since $v\in G'$), $v\neq \phi(u)=\phi(v)$. So $v\in V(W)$. So $u,v\in V(W)\in \scr{P}_0^+$. Recall that $B'=V(G_0^+\cap H)\subseteq V(H)$. Thus, since $W$ is disjoint to $H$ (and since $W\subseteq V(G_0^+)$), $u,v\in V(G_0^+)\setminus B'$. So $u\in \tau(u)\in \scr{P}_0^+$ and $v\in \tau(v)\in \scr{P}_0^+$. Since $\scr{P}_0^+$ is a partition of $G_0^+$, this gives $\tau(u)=\tau(v)=V(W)\neq \emptyset$. So $\tau(u)$ is the same part of $\scr{P}_0^+$ as $\tau(v)$. Further, since $u\in \tau^{-1}(\psi_C(P))$ and $v\in \tau^{-1}(\psi_C(P))$, we obtain $\tau(u)=\psi_C(P)$ and $\tau(v)=\psi_C(P')$. So $\psi_C(P)$ and $\psi_C(P')$ are the same element of $\scr{P}_0^+$. However, we recall that $\psi_C(P)$ and $\psi_C(P')$ are the different elements of $\scr{P}_0^+$, as $P,P'$ are distinct elements of $\scr{P}_C$ and since $\psi_C$ is a bijection. So we have a contradiction. Thus, $P$ is disjoint to $P'$, as desired.

        This completes the proof.
    \end{proof}

    We use \cref{raiseAP} implictly from now on.

    \subsection{Bounding the treewidth}

    Next, we bound the treewidth. We do this by showing that the raise doesn't create `new' edges in the quotient. More specifically, we have the following.

    \begin{lemma}
        \label{raiseEdges}
        Let $\Gamma$ be an almost-embedding of a graph $G$, let $\scr{PR}=(V_M,X,G',G_0^+,W,D,H',\scr{J}',\linebreak\phi,\scr{P}_0^+)$ be a partitioned reduction of $\Gamma$ and let $\tau$ be a raise for $\scr{PR}$. If $P,P'\in \scr{P}(\tau)$ are adjacent in $(G-X)/\scr{P}(\tau)$, there exists $C\in \scr{C}(\scr{R})$ such that $P,P'\in \scr{P}_C(\tau)$. Further, $(\psi_C(\tau))(P)$ is adjacent to $(\psi_C(\tau))(P')$ in $G_0^+/\scr{P}_0^+$.
    \end{lemma}

    \begin{proof}
        Let $\scr{C}:=\scr{C}(\scr{PR})$. For each $C\in \scr{C}$, let $\psi_C:=\psi_C(\tau)$ and $\scr{P}_C:=\scr{P}_C(\tau)=(\phi(\tau^{-1}(P))\cap V(C):P\in \scr{P}_0^+)$. Recall that $\scr{P}:=\scr{P}(\tau)=\bigsqcup_{C\in \scr{C}}\scr{P}_C$. Also, for each $C\in \scr{C}$, recall that $\psi_C$ is a bijection from $\scr{P}_C$ to $\scr{P}_0^+$ such that $\phi(\tau^{-1}(\psi_C(P)))\cap V(C)=P$ for each $P\in \scr{P}_C$.

        Let $P,P'\in \scr{P}$ be adjacent in $(G-X)/\scr{P}$. So $N_{G-X}(P)$ intersects $P'$.

        There exists $C,C'\in \scr{C}$ such that $P\in \scr{P}_C$ and $P'\in \scr{P}_C$. Since $\scr{C}$ is the connected components of $G-X$, observe that $N_{G-X}(P)\subseteq V(C)$. Since $P'$ intersects $N_{G-X}(P)$, we find that $C=C'$. Thus, observe that $P,P'\subseteq V(C)$.

        Let $u\in P$ and $v\in P'$ be adjacent in $G-X$. So $v\in N_{G-X}(u)\subseteq \phi(N_{G'}(\phi^{-1}(u))$ (as $v\in V(G-X)$). So there exists adjacent $u',v'\in V(G')$ such that $\phi(u')=u$ and $\phi(v')=v$. Since $\phi(u')=u,\phi(v')=v\in V(G-X)$, $N_{G_0^+}[\tau(u')]$ intersects $\tau(v')$.
        
        Presume, for a contradiction, that $\tau(u')\neq \psi_C(P)$. Set $P^*:=\psi_C^{-1}(\tau(u'))$. Since $\psi_C$ is a bijection, $P^*$ is distinct from $P$. Since $\scr{P}$ is a partition of $G-X$ (by \cref{raiseAP}) and since $P,P^*\in \scr{P}_C\subseteq \scr{P}$, $P$ is disjoint to $P^*$. Recall that $P^*=\phi(\tau^{-1}(\tau(u')))\cap V(C)$. Observe that $u'\in \tau^{-1}(\tau(u'))$, and thus $\phi(u')=u\in \phi(\tau^{-1}(\tau(u')))$. Since $u\in P\subseteq V(C)$, this gives $u\in P^*$. Thus, $u\in P\cap P^*$, contradicting the fact that $P$ and $P^*$ are disjoint. So $\tau(u')=\psi_C(P)$. An identical argument (replacing $u',u,P$ with $v,v',P'$) gives $\tau(v')=\psi_C(P')$.

        Thus, $N_{G_0^+}[\psi_C(P)]$ intersects $\psi_C(P')$. Since $P,P'$ are distinct, $\psi_C(P)$, $\psi_C(P')$ are distinct and thus disjoint. So $N_{G_0^+}(\psi_C(P))$ intersects $\psi_C(P')$. Thus, $\psi_C(P)$ is adjacent to $\psi_C(P')$ in $G_0^+/\scr{P}_0^+$, as desired.
    \end{proof}

    \cref{raiseEdges} allows us to bound the treewidth.

    \begin{corollary}
        \label{raiseTw}
        Let $\Gamma$ be an almost-embedding of a graph $G$, let $\scr{PR}$ be a partitioned reduction of $\Gamma$ of treewidth at most $b\in \ds{N}$, and let $\tau$ be a raise for $\scr{PR}$. Then $\scr{P}(\tau)$ has treewidth at most $b$.
    \end{corollary}

    \begin{proof}
        Let $(V_M,X,G',G_0^+,W,D,H',\scr{J}',\phi,\scr{P}_0^+):=\scr{PR}$, and let $\scr{C}:=\scr{C}(\scr{PR})$. Set $\scr{P}:=\scr{P}(\tau)$, and for each $C\in \scr{C}$, let $\psi_C:=\psi_C(\tau)$.
        
        By \cref{raiseEdges}, $(G-X)/\scr{P}$ can be obtained as follows. For each $C\in \scr{C}$, take a disjoint copy of $G_0^+/\scr{P}_0^+$. Relabel each $P\in \scr{P}_0^+$ to $\psi_C^{-1}(P)$, and then delete some edges. Take the disjoint union of the resulting graphs. Since $\scr{P}_0^+$ has treewidth at most $b$, $G_0^+/\scr{P}_0^+$ has treewidth at most $b$. Relabelling and deleting edges does not increase this, nor does taking the disjoint union. So $(G-X)/\scr{P}$ has treewidth at most $b$. Thus, $\scr{P}$ has treewidth at most $b$, as desired.
    \end{proof}

    \subsection{Bounding the width}

    Next, we bound the width of the partition. This comes down to proving two things. One is that $\tau$ and $\phi$ map the part $V(W)$ to parts of the form $V_M\cap V(C)$, where $C$ is a connected component of $G-X$. The remainder-width controls the size of these parts. The second is that the size of every other part is only increased by a factor of $(k+1)$ by $\tau^{-1}$, since $\tau$ is either the identity, or, in the case of quasi-vortex vertices, sends the vertex $v$ to a vertex $x$ on the boundary indexing a bag containing $v$.

    For the formal proof, we require a strong understanding of how the inverse of $\tau$ acts. So we take a detour to prove some related results.

    \begin{lemma}
        \label{raiseInv}
        Let $\Gamma$ be an almost-embedding of a graph $G$, let $\scr{PR}=(V_M,X,G',G_0^+,W,D,H',\scr{J}',\linebreak\phi,\scr{P}_0^+)$ be a partitioned reduction of $\Gamma$, let $B':=B'(\scr{PR})$, and let $(J_x':x\in B'):=\scr{J}'$. If $\tau$ is a raise for $\scr{PR}$, then for each $P\in \scr{P}_0^+$, $(P\setminus B')\subseteq \tau^{-1}(P)\subseteq (P\setminus B')\cup \bigcup_{x\in P\cap B'}J_x'$.
    \end{lemma}

    \begin{proof}
        Fix $P\in \scr{P}_0^+$. Recall that $P\subseteq V(G_0^+)$. Thus, $P\setminus B'\subseteq V(G_0^+)\setminus B$. Hence, for each $v\in P\setminus B'$, $v\in \tau(v)$. Since $P,\tau(v)\in \scr{P}$ and since $\scr{P}$ is a partition of $(G-X)$, $P$ and $\tau(v)$ are either disjoint or the same element of $\scr{P}$. Thus, since $v\in P\cap \tau(v)$, $\tau(v)=P$. So $v\in \tau^{-1}(P)$. Thus, $P\setminus B'\subseteq \tau^{-1}(P)$.

        Now, consider any $v\in V(G_0^+)\setminus B'$ such that $v\in \tau^{-1}(P)$. So $\tau(v)=P$. Since $v\in V(G_0^+)\setminus B'$, $v\in \tau(v)=P$. Since $v\notin B'$, $v\in P\setminus B'$. Thus, $\tau^{-1}(v)\cap (V(G_0^+)\setminus B')\subseteq P\setminus B'$.

        Next, consider any $v\in V(H')\cap \tau^{-1}(P)$. So there exists $x\in \tau(v)\cap B'=P\cap B'$ such that $v\in J_x'$. So $v\in \bigcup_{x\in P\cap B'}J_x'$. Thus, $\tau^{-1}(P)\cap V(H')\subseteq \bigcup_{x\in P\cap B'}J_x'$.

        Recall that $V(G')=V(G_0^+\cup H')=(V(G_0^+)\setminus B')\cup V(H')$. Thus, $\tau^{-1}(P)\subseteq (P\setminus B')\cup \bigcup_{x\in P\cap B'}J_x'$. So $(P\setminus B')\subseteq \tau^{-1}(P)\subseteq (P\setminus B')\cup \bigcup_{x\in P\cap B'}J_x'$, as desired.
    \end{proof}

    \begin{corollary}
        \label{raiseInvW}
        Let $\Gamma$ be an almost-embedding of a graph $G$, let $\scr{PR}=(V_M,X,G',G_0^+,W,D,H',\linebreak\scr{J}',\phi,\scr{P}_0^+)$ be a partitioned reduction of $\Gamma$, and let $\tau$ be a raise for $\scr{PR}$. Then $V(W)\in \scr{P}_0^+$ and $\tau^{-1}(V(W))=V(W)$.
    \end{corollary}

    \begin{proof}
        Let $B':=V(G_0^+\cap H')=B'(\scr{PR})$. By definition of $\scr{P}_0^+$, $V(W)\in \scr{P}_0^+$. Recall that $W$ is disjoint to $H'$, and thus $V(W)$ is disjoint to $B'$. Thus, $V(W)\setminus B'=V(W)$ and $V(W)\cap B'=\emptyset$. So by \cref{raiseInv}, $V(W)\subseteq \tau^{-1}(V(W))\subseteq V(W)\cup \emptyset$. Thus, $\tau^{-1}(V(W))=V(W)$, as desired.
    \end{proof}

    \begin{corollary}
        \label{raiseInvWidth}
        Let $\Gamma$ be an almost-embedding of a graph $G$, let $\scr{PR}=(V_M,X,G',G_0^+,W,D,H',\linebreak\scr{J}',\phi,\scr{P}_0^+)$ be a partitioned reduction of $\Gamma$ of vortex-width at most $k\in \ds{N}$, and let $\tau$ be a raise for $\scr{PR}$. Then for each $P\in \scr{P}_0^+$, $|\tau^{-1}(P)|\leq (k+1)|P|$.
    \end{corollary}

    \begin{proof}
        Let $B':=V(G_0^+\cap H')=B'(\scr{PR})$, and let $(J_x':x\in B'):=\scr{J}'$. 
        
        Since $\scr{PR}$ has vortex-width at most $k$, $\scr{J}'$ has width at most $k$. Thus, for each $x\in B'$, $|J_v'|\leq k+1$. 

        Fix $P\in \scr{P}_0^+$. By \cref{raiseInv}, $\tau^{-1}(P)\subseteq (P\setminus B')\cup \bigcup_{x\in P\cap B'}J_x'$. Thus, $|\tau^{-1}(P)|\leq |(P\setminus B')|+(k+1)|P\cap B|\leq (k+1)|P|$ (since $k\geq 0$), as desired.
    \end{proof}

    We can now bound the width.

    \begin{lemma}
        \label{raiseWidth}
        Let $\Gamma$ be an almost-embedding of a graph $G$, and let $\scr{PR}$ be a partitioned reduction of $\Gamma$ of vortex-width at most $k\in \ds{N}$, near-width at most $w\in \ds{R}_0^+$, and remainder-width at most $w'\in \ds{R}_0^+$, and let $\tau$ be a raise for $\scr{PR}$. Then $\scr{P}(\tau)$ has width at most $\max(w',(k+1)w)$.
    \end{lemma}

    \begin{proof}
        Let $(V_M,X,G',G_0^+,W,D,H',\scr{J}',\phi,\scr{P}_0^+):=\scr{PR}$, let $B':=V(G_0^+\cap H')=B(G_0^+-V(W),D)$, and let $\scr{C}:=\scr{C}(\scr{PR})$. For each $C\in \scr{C}$, let $\psi_C:=\psi_C(\tau)$ and $\scr{P}_C:=\scr{P}_C(\tau)=(\phi(\tau^{-1}(P))\cap V(C):P\in \scr{P}_0^+)$. Let $\scr{P}:=\scr{P}(\tau)=\bigsqcup_{C\in \scr{C}}\scr{P}_C$.

        For each $P\in \scr{P}$, there exists $C\in \scr{C}$ such that $P\in \scr{P}_C$. Let $P':=\psi_C(P)$. Recall $P=\phi(\tau^{-1}(P'))\cap V(C)$.
        
        We consider two cases.

        Case 1: $P'=V(W)$.

        By \cref{raiseInvW}, $\tau^{-1}(P')=\tau^{-1}(V(W))=V(W)$. Recall that $\phi(V(W))=W_M$. So $P=V_M\cap V(C)$. As $\scr{PR}$ has remainder-width at most $w'$, $\scr{R}$ has remainder-width at most $w'$. Thus, $|P|=|V_M\cap V(C)|\leq w'\leq \max(w',(k+1)w)$.
        
        Case 2: $P'\neq V(W)$.

        Then since $\scr{PR}$ has near-width at most $w$, $|P|\leq w$. By \cref{raiseInvWidth}, $|\tau^{-1}(P')|\leq (k+1)|P|\leq (k+1)w$. So $|P|\leq |\phi(\tau^{-1}(P'))|\leq |\tau^{-1}(P')|\leq (k+1)w\leq \max(w',(k+1)w)$, as desired.

        This completes the proof of the lemma.
    \end{proof}

    \subsection{Showing that the partition is concentrated}

    The only thing left is to show that the almost-partition is concentrated. This is the most difficult step, and requires some more explaining.

    Recall that we start with a connected partition of a plane graph. We then use the raise to partition $V(G')$ (note that $(\tau^{-1}(P):P\in \scr{P}_0^+)$ is a partition of $V(G')$), before transforming these parts with $\phi$ and dividing them based on the connected components of $G-X$. The reduction (in particular, the definition of $\phi$) ensures that there is a `strong relationship' between $\VC(\Gamma)$ and $\Attachable(\Gamma)$ and the cliques in the modified quasi-vortex $H'$, the cliques of size at most $2$ in $G_0^+$, and the facial triangles of $G_0^+$. So, loosely speaking, this problem converts into a problem of ensuring that $(\tau^{-1}(P):P\in \scr{P}_0^+)$ is `concentrated'.

    To do this, we want to show that (1) the initial partition of $G_0^+$ is `concentrated', and (2) the raise does not `ruin' the concentration `too much'.

    For the first of these two points, we have the following lemma.

    \begin{lemma}
        \label{facialConcentrated}
        Let $G$ be a plane graph, and let $\scr{P}$ be a connected partition for $G$. Then for each clique $Q$ in $G/\scr{P}$ of size $3$, there is at most two cliques in $\FT(G)$ that intersect each part in $Q$.
    \end{lemma}

    \begin{proof}
        Presume otherwise. So there exists a triangle $Q=\{P_1,P_2,P_3\}$ in $G/\scr{P}$ and three distinct facial triangles $K_1,K_2,K_3$ that each intersect all of $P_1,P_2,P_3$. So there are distinct faces $F_1,F_2,F_3$ of $G$ whose boundary vertices are $K_1,K_2$, and $K_3$ respectively. Let $G'$ be the plane graph obtained from $G$ by embedding, for each $i\in \{1,2,3\}$, a vertex $v_i$ inside $F_i$, along with edges inside $F_i$ to make the neighbourhood of $v_i$ $K_i$. Note that this is possible since the faces $F_1,F_2,F_3$ are distinct. Since $\scr{P}$ is connected partition of $G\subseteq G'$, we obtain a $K_{3,3}$-minor in $G'$ with $\{P_1,P_2,P_3\}$ on one side, and $\{v_1,v_2,v_3\}$ on the other side, a contradiction.
    \end{proof}

    An easy corollary/application of \cref{facialConcentrated} (which we don't prove, because we don't directly use) is that $(\emptyset,\scr{P}_0^+)$ is a $(2,2,\FT(G_0^+))$-concentrated almost-partition of $G_0^+$.

    So the goal is to show that after taking the inverse of the raise, the partition remains `concentrated', in terms of cliques in the modified quasi-vortex, cliques of size at most $2$ in $G_0^+$, and facial triangles of $G_0^+$. The way we do this is by showing that raise isn't too `chaotic' in how it changes the parts.
    
    Specifically, for any pair of parts $P_1,P_2$ in $\scr{P}_0^+$, there are two concerns we have. The first are edges in the modified quasi-vortex $H'$ that go between $\tau^{-1}(P_1)$ and $\tau^{-1}(P_2)$. The second are edges in $G_0^+$ that go between $\tau^{-1}(P_1)$ and $\tau^{-1}(P_2)$ that don't have their endpoints in $P_1$ and $P_2$. The former could prevent the partition from being concentrated in the vortex, and the latter causes problems specifically on the boundary, both for being concentrated in the vortex and for keeping the facial triangles concentrated. We show that all of these problematic edges have at least one endpoint in a small set of vertices.

    The proof idea is relatively simple. In $G_0^+$, the non-crossing property (\cref{nonCrossing}) ensures that all the intersection points of $P_2$ with $B'$ lie between two points $a,b$ of $P_1\cap B'$ in the cyclic ordering of $V(U')$. We find that for any `problematic edge' as described above, at least one endpoint lies in $J_a'\cup J_b'$ (where $(J_v':v\in B')$ is the modified decomposition). There is a decent amount of case analysis required to reach this conclusion, but it fundamentally comes down to the fact that the raise requires $v\in V(H')$ to be mapped to $x\in B'$ with $v\in J_x'$. This immediately takes care of all the edges in the modified quasi-vortex, and `problematic' edges in $G_0^+$ get `blocked' by $P_1$ and $P_2$ using the non-crossing property.

    The only exception is that we don't have any control with the number of edges coming from $V(W)$ to $B'$, so we can't control the behaviour of bad edges between $V(W)$ and some other part $P_1$. So we have to avoid the case of $V(W)$ for now.
    
    More specifically, we have the following technical lemma.

    \begin{lemma}
        \label{liftConcentratedPrelim}
         Let $\Gamma$ be an almost-embedding of a graph $G$, let $\scr{PR}=(V_M,X,G',G_0^+,W,D,H',\scr{J}',\linebreak\phi,\scr{P}_0^+)$ be a partitioned reduction of $\Gamma$ of vortex-width at most $k\in \ds{N}$, let $\tau$ be a raise for $\scr{PR}$, and let $P_1,P_2\in \scr{P}_0^+$ be distinct. Then there exists $S\subseteq V(G')$ with $|S|\leq 2(k+1)$ such that if $u\in \tau^{-1}(P_1)$ and $v\in \tau^{-1}(P_2)$ are adjacent in $G'$ and satisfy either:
        \begin{enumerate}
            \item $u,v\in V(H')$ and are adjacent in $H'$, or,
            \item $u\notin P_1$ and $v\notin V(W)$.
        \end{enumerate}
        then $(u,v)$ intersects $S$.
    \end{lemma}

    \begin{proof}
        Let $G_0':=G_0(\Lambda(\scr{PR}))=G_0^+-V(W)$, $B':=B'(\scr{PR})$, and $U':=U'(\scr{PR})$. Recall that $D$ is $G_0'$-clean with $B(D,G_0')=B'$ and $U(D,G_0')=U'$. So $U'$ is a cycle or a path. Let $(J_x':x\in B'):=\scr{J}'$. Recall that $\scr{J}'$ is a planted $U'$-decomposition of $H'$.

        Observe that if $P_1$ or $P_2$ is empty, then taking $S=\emptyset$ suffices. So we may assume that neither $P_1$ nor $P_2$ is empty.

        Next, presume that for some $i\in \{1,2\}$, $P_i=V(W)$. By \cref{raiseInvW}, $\tau^{-1}(P_i)=P_i=V(W)$. Recall that $V(W)$ is disjoint to $H'$. So for any $z\in \tau^{-1}(P_i)=P_i=V(W)$, we have:
        \begin{enumerate}
            \item $z\notin V(H')$, and
            \item $z\in P_i=V(W)$.
        \end{enumerate}

        So there is no pair $u\in \tau^{-1}(P_1)$ and $u\in \tau^{-1}(P_2)$ that satisfy either:
        \begin{enumerate}
            \item $u,v\in V(H')$, or,
            \item $u\notin P_1$ and $v\notin V(W)$.
        \end{enumerate}

        So we can take $S=\emptyset$ in this case. So we can assume that neither $P_1$ nor $P_2$ is $V(W)$ from now on.
        
        Since $\scr{P}_0^+$ is a partition of $V(G_0^+)$ and $V(W)\in \scr{P}_0^+$, $P_1$ and $P_2$ are disjoint from $V(W)$. Thus, $P_1$ and $P_2$ are contained in $V(G_0')$. Recall that $D$ is $G_0'$-clean, and that $B(D,G_0')=B'$ and $U(D,G_0')=U'$. Further, since $\scr{P}_0^+$ is connected, $G_0^+[P_1]$ and $G_0^+[P_2]$ are connected. Since $P_1$ and $P_2$ are disjoint from $V(W)$ and recalling that $G_0'=G_0^+-V(W)$, we find that $H_1:=G_0'[P_1]$ and $G_0'[P_2]$ are connected.
        
        Since $\scr{P}_0^+$ is a partition of $G_0^+$ and $P_1,P_2\in \scr{P}_0^+$ are distinct, $P_1$ is disjoint to $P_2$. Let $H_2:=G_0'[N_{G_0'}[P_2]\setminus P_1]$. Since $G_0'[P_2]$ is connected and disjoint to $G_0'[P_1]=H_1$, $H_2$ is connected and disjoint to $H_1$. Note that $P_1=V(H_1)$ and $P_2\subseteq V(H_2)$.

        By \cref{raiseInv}, for each $i\in \{1,2\}$, $\tau^{-1}(P_i)\subseteq (P_i\setminus B')\cup \bigcup_{x\in P_1\cap B'}J_x'$. Thus, if $z\in \tau^{-1}(P_i)$, then either $z\in P_i\setminus B'$ or $z\in \bigcup_{x\in P_i\cap B'}J_x'$. Note that $P_i\setminus B'\subseteq V(G_0^+)\setminus B'$ is disjoint to $V(H')\supseteq \bigcup_{x\in P_i\cap B'}J_x'$. So either:

        \begin{enumerate}
            \item $z\in P_i\setminus B'\subseteq P_i$ and $z\notin V(H')$, or,
            \item $z\in \bigcup_{x\in P_i\cap B'}J_x'$.
        \end{enumerate}
        
        In particular, if $u\in \tau^{-1}(P_1)$ satisfies either:
        \begin{enumerate}
            \item $u\in V(H')$, or
            \item $u\notin P_1$,
        \end{enumerate}
        then $u\in \bigcup_{x\in P_1\cap B'}J_x'$. Since $\scr{PR}$ has vortex-width at most $k$, so does $\scr{R}$, and thus $\scr{J}'$ has width at most $k$. Thus, $\bigcup_{x\in P_1\cap B'}J_x'\leq (k+1)|P_1\cap B'|$. So if $|P_1\cap B'|\leq 2$, we can take $S=\bigcup_{x\in P_1\cap B'}$. Thus, we may assume that $|P_1\cap B'|\geq 3$.

        By \cref{nonCrossing}, there does not exist $x,y\in V(H_2)\cap B'$ and $a,b\in V(H_1)\cap B'$ such that $x\prec_{U'} a\prec_{U'} y\prec_{U'} b$. Thus and since $|V(H_1)\cap B'|=|P_1\cap B'|\geq 3$, there exists distinct $a,b\in V(H_1)\cap B'=P_1\cap B'$ such that for each $y\in V(H_2)\cap B'$ and $z\in P_1\cap B'$, $z\preceq_{U'} a\prec_{U'} y\prec_{U'} b$. Let $S:=J_a\cup J_b$. Observe that $|S|\leq 2(k+1)$ as $\scr{J}'$ has width at most $k$.

        \begin{claim}
            \label{claimIntervalIntersectsBlocks}
             Let $y\in V(H_2)\cap B'$, $z\in P_1\cap B'$, and let $I\subseteq U$ be connected with $y,z\in V(I)$. Then $V(I)$ intersects $\{a,b\}$.
        \end{claim}

        \begin{proofofclaim}
            Recall that $z\preceq_{U'} a\prec_{U'} y\prec_{U'} b$. If $z\in \{a,b\}$, then $V(I)$ intersects $z\in \{a,b\}$, as desired. So we may assume that $z\prec_{U'} a\prec_{U'} y\prec_{U'} b$. Since $U$ is a cycle or path, observe that $y$ and $z$ are in different connected components of $U-a-b$. Thus, since $I$ is connected and contains $y,z$, $V(I)$ intersects $\{a,b\}$, as desired.
        \end{proofofclaim}

        Fix $u\in \tau^{-1}(P_1)$ and $u\in \tau^{-1}(P_2)$ that are adjacent in $G'$ and satisfy either:
        \begin{enumerate}
            \item $u,v\in V(H')$ and are adjacent in $H'$, or,
            \item $u\notin P_1$ and $v\notin V(W)$.
        \end{enumerate}
        We show that $(u,v)$ intersects $S$.

        If $u,v\in V(H')$ and are adjacent in $H'$, then recall that $u\in \bigcup_{x\in P_1\cap B'}J_x'$ and $v\in \bigcup_{x\in P_2\cap B'}J_x'$. So there exists $z\in P_1\cap B'$ such that $u\in J_z'$, and $y\in P_2\cap B'$ such that $v\in J_y'$. For each $m\in \{u,v\}$, let $I_m$ be the subgraph of $U'$ induced by the vertices $x\in B'$ such that $m\in J_x'$. Recall that $I_u$ and $I_v$ are nonempty and connected. Further, since $u$ and $v$ are adjacent in $H'$, $I_u$ and $I_v$ intersect. Also, note that $z\in V(I_u)$ and $y\in V(I_v)$. So $I_u\cup I_v$ is connected and contains both $y$ and $z$. Thus, by \cref{claimIntervalIntersectsBlocks}, $I_u\cup I_v$ intersects $\{a,b\}$. So there exists $m\in \{u,v\}$ and $n\in \{a,b\}$ such that $n\in I_m$. So $m\in J_n'\subseteq S$, as desired.

        If $u\notin P_1$, recall that $u\in \bigcup_{x\in P_1\cap B'}J_x'\subseteq V(H')$. So there exists $z\in P_1\cap B'$ such that $u\in J_z'$. Note that $z\preceq_{U'} a\prec_{U'} b$.
        
        By the above case, we may assume that either $v\notin V(H')$, or $u,v$ are not adjacent in $H'$. Recall that $G'=G_0^+\cup H$, and that $u,v$ are adjacent in $G'$. Thus, we may assume that $u\in B'$, $v\in V(G_0^+)$, and that $u,v$ are adjacent in $G_0^+$. Since $u\in P_1$, $v\in P_2$, and since $P_1,P_2$ are disjoint to $V(W)$, we find that $u,v\in V(G_0')$ and that $u,v$ are adjacent in $G_0'$.
        
        Let $I_u$ be the subgraph of $U'$ induced by the vertices $x\in B'$ such that $u\in J_x'$. Recall that $I_u$ is nonempty and connected. Note that $z\in V(I_u)$. Further, since $\scr{J}'$ is planted, $u\in V(I_u)$.

        If $v\in P_2$, then observe that $u\in N_{G_0'}[P_2]$. Since $u\notin P_1$, we have $u\in V(H_2)$. So $u\in V(H_2)\cap B'$. Recall that $z\in P_1\cap B$ and $u,z\in V(I_u)$. Thus, by \cref{claimIntervalIntersectsBlocks} (with $y:=u$), $V(I_u)$ intersects $\{a,b\}$. So there exists $n\in \{a,b\}$ such that $u\in J_n'\subseteq S$, as desired. So we may assume that $v\notin P_2$
            
        Since $v\notin P_2$, recall that $v\in \bigcup_{x\in P_2\cap B'}J_x'\subseteq V(H')$. So $v\in B'$, and there exists $y\in P_2\cap B'$ such that $v\in J_y'$. Note that $y\in P_2\subseteq V(H_2)$, and thus recall that $a\prec_{U'} y\prec_{U'} b$.
        
        Let $I_v$ be the subgraph of $U'$ induced by the vertices $x\in B'$ such that $v\in J_x'$. Recall that $I_v$ is nonempty and connected. Note that $y\in V(I_v)$. Further, since $\scr{J}'$ is planted, $v\in V(I_v)$.

        If $v\in P_1$, then by \cref{claimIntervalIntersectsBlocks} (with $z:=v$), $V(I_v)$ intersects $\{a,b\}$. So there exists $n\in \{a,b\}$ such that $v\in J_n'\subseteq S$, as desired. Thus, we may assume that $v\notin P_1$.

        Thus, neither $u$ nor $v$ is in $P_1$. Since $u,v$ are adjacent in $V(G_0')$, $G_0'[\{u,v\}]$ and $H_1$ are connected and disjoint subgraphs of $G_0'$. Thus, by \cref{nonCrossing}, there does not exist $m,n\in P_1$ such that $u\prec_{U'} m\prec_{U'} v\prec_{U'} n$. In particular, neither $u\prec_{U'} a\prec_{U'} v\prec_{U'} b$ nor $u\prec_{U'} b\prec_{U'} v\prec_{U'} a$. Since $a,b\in P_1$ and neither $u$ nor $v$ is in $P_1$, we also have that $\{u,v\}$ is disjoint to $\{a,b\}$. Thus, either $u,v\prec_U a\prec_U b$, or $a\prec_U u,v\prec_U b$.
        
        If $u,v\prec_U a\prec_U b$, observe that $v\prec_U a\prec_U y\prec_U b$. Recall that $v,y\in V(I_v)$. By the same argument as that used in \cref{claimIntervalIntersectsBlocks}, $V(I_v)$ intersects $\{a,b\}$. As argued previously, this gives $v\in S$, as desired. Otherwise, if $a\prec_U u,v\prec_U b$, observe that $z\preceq_U a\prec_U u\prec_U b$. Recall that $u,z\in V(I_u)$. By the same argument, we find that $V(I_u)$ intersects $\{a,b\}$, and thus $u\in S$, as desired.

        So in all cases, $(u,v)$ intersects $S$. Since $|S|\leq 2(k+1)$, this completes the proof.
    \end{proof}

    For the purposes of the final proof, it is easier to talk about $(X,\scr{P}(\tau))$ rather than converting everything to the partition $(\tau^{-1}(P):P\in \scr{P}_0^+)$. So we want to convert \cref{liftConcentratedPrelim} to a result about the parts in $\scr{P}(\tau)$. To do this, we need to relate the inverse of the projection with the inverse of the raise, as given by the following lemma.

    \begin{lemma}
        \label{invProjInvRaise}
        Let $\Gamma$ be an almost-embedding of a graph $G$, let $\scr{PR}=(V_M,X,G',G_0^+,W,D,H',\scr{J}',\linebreak\phi,\scr{P}_0^+)$ be a partitioned reduction of $\Gamma$, and let $\tau$ be a raise of $\scr{PR}$. Then for each $P\in \scr{P}(\tau)$, if $C\in \scr{C}(\scr{R})$ is such that $P\in \scr{P}_C(\tau)$, then $\phi^{-1}(P)\subseteq \tau^{-1}((\psi_C(\tau))(P))$.
    \end{lemma}

    \begin{proof}
        Let $\psi_C:=\psi_C(\tau)$. Recall that $P=\phi(\tau^{-1}(\psi_C(P)))\cap V(C)$. So $P\subseteq \phi(\tau^{-1}(\psi_C(P)))$ and $P\subseteq V(C)$.

        We consider two cases.

        Case 1: $\psi_C(P)=V(W)$.

        Then by \cref{raiseInvW}, $\tau^{-1}(\psi_C(P))=\tau^{-1}(V(W))=V(W)$. So $P\subseteq \phi(\tau^{-1}(\psi_C(P)))=\phi(V(W))=V_M$. Thus, by \cref{mostlyIdenInv}, $\phi^{-1}(P)\subseteq \phi^{-1}(V_M)=V(W)=\tau^{-1}(\psi_C(P))$, as desired.

        Case 2: $\psi_C(P)\neq V(W)$.

        Recall that $V(W)\in \scr{P}_0^+$. As argued above, $\phi(\tau^{-1}(V(W)))=V_M$. Thus, note that $V_M\cap V(C)\in \scr{P}_C\subseteq \scr{P}$, and that $\psi_C^{-1}(V_M\cap V(C))=V(W)\neq \psi_C(P)$. Thus, $V_M\cap V(C)\neq P$. Since $\scr{P}$ is a partition of $G-X$ (by \cref{raiseAP}) and since $P,V_M\cap V(C)\in \scr{P}$, $P$ is disjoint to $V_M\cap V(C)$. Since $P\subseteq V(C)$, $P$ is therefore disjoint to $V_M$. Since $C$ is a connected component of $G-X$ and since $P\subseteq V(C)$, $P$ is disjoint to $V_M\cup X$. So $P\subseteq V(G)\setminus (V_M\cup X)$. By \cref{mostlyIdenInv}, $P\subseteq V(G')\setminus V(W)$, and hence $\phi^{-1}(P)=P$.
        
        Thus, for each $v\in \phi^{-1}(P)$, $v\in P\subseteq \phi(\tau^{-1}(\psi_C(P)))$. So there exists $x\in \tau^{-1}(\psi_C(P))$ such that $\phi(x)=v$. Since $\{v\}\subseteq P\subseteq V(G')\setminus V(W)$, by \cref{mostlyIdenInv}, $\phi^{-1}(v)=\{v\}$. Thus, $\{x\}=\{v\}=\phi^{-1}(v)$. So $v=x\in \tau^{-1}(\psi_C(P))$. Thus, $\phi^{-1}(P)\subseteq \tau^{-1}(\psi_C(P))$, as desired.

        This completes the proof of the lemma.
    \end{proof}
    
    This gives the following technical corollary.

    \begin{corollary}
        \label{liftConcentratedPrelimCoroll}
        Let $\Gamma$ be an almost-embedding of a graph $G$, let $\scr{PR}=(V_M,X,G',G_0^+,W,D,H',\scr{J}',\linebreak\phi,\scr{P}_0^+)$ be a partitioned reduction of $\Gamma$ of vortex-width at most $k$, let $\tau$ be a raise for $\scr{PR}$, and let $P_1,P_2\in \scr{P}$ be distinct and such that there exists $C\in \scr{C}(\scr{PR})$ with $P_1,P_2\in \scr{P}_C(\tau)$. Then there exists $S\subseteq V(G')$ with $|S|\leq 2(k+1)$ such that if $u\in \phi^{-1}(P_1)$ and $u\in \phi^{-1}(P_2)$ are adjacent in $G'$ and satisfy either:
        \begin{enumerate}
            \item $u,v\in V(H')$ and are adjacent in $H'$, or,
            \item $u\notin (\psi_C(\tau))(P_1)$ and $(\psi_C(\tau))(P_2)$ is not $V(W)$.
        \end{enumerate}
        then $(u,v)$ intersects $S$.
    \end{corollary}

    \begin{proof}
        Set $\psi_C:=\psi_C(\tau)$.
        
        By \cref{liftConcentratedPrelim} with $\psi_C(P_1)$ as $P_1$ and $\psi_C(P_2)$ as $P_2$, there exists $S\subseteq V(G')$ with $|S|\leq 2(k+1)$ such that if $u\in \tau^{-1}(\psi_C(P_1))$ and $v\in \tau^{-1}(\psi_C(P_2))$ are adjacent in $G'$ and satisfy either:
        \begin{enumerate}
            \item $u,v\in V(H')$ and are adjacent in $H'$, or,
            \item $u\notin \psi_C(P_1)$ and $v\notin V(W)$.
        \end{enumerate}
        then $(u,v)$ intersects $S$.

        By \cref{invProjInvRaise}, if $u\in \phi^{-1}(P_1)$ and $v\in \phi^{-1}(P_2)$, then $u\in \tau^{-1}(\psi_C(P_1))$ and $v\in \tau^{-1}(\psi_C(P_2))$. If $\psi_C(P_2)$ is not $V(W)$, then since $V(W),\psi_C(P_2)\in \scr{P}_0^+$, which is a partition of $G_0^+$, $\psi_C(P_2)$ is disjoint to $V(W)$. By \cref{raiseInvW}, $\tau^{-1}(V(W))=V(W)$. If $v\in \phi^{-1}(P_2)\subseteq \tau^{-1}(\psi_C(P_2))$, then $\tau(v)\neq V(W)$, and thus $v\notin V(W)$. 
        
        So if $u\in \phi^{-1}(P_1)$ and $u\in \phi^{-1}(P_2)$ are adjacent in $G'$ and satisfy either:
        \begin{enumerate}
            \item $u,v\in V(H')$ and are adjacent in $H'$, or,
            \item $u\notin \psi_C(P_1)$ and $\psi_C(P_2)$ is not $V(W)$.
        \end{enumerate}
        then $u\in \tau^{-1}(\psi_C(P_1))$ and $v\in \tau^{-1}(\psi_C(P_2))$ are adjacent in $G'$ and satisfy either:
        \begin{enumerate}
            \item $u,v\in V(H')$ and are adjacent in $H'$, or,
            \item $u\notin \psi_C(P_1)$ and $v\notin V(W)$.
        \end{enumerate}
        and thus $(u,v)$ intersects $S$.
        
        Since $|S|\leq 2(k+1)$, this completes the proof.
    \end{proof}

    We can now prove that the almost-partition is concentrated. This basically comes down to applying \cref{liftConcentratedPrelimCoroll} on each pair of parts in a given clique in the quotient. Since the treewidth is bounded (by \cref{raiseTw}), only a finite number of pairs need to be considered. We use this to show that the $(X,\scr{P}(\tau))$ is not much worse than the partition $\scr{P}_0^+$ of $G_0^+$, which was concentrated on the facial triangles.

    The only difficulty is dealing with $V(W)$, as we couldn't control the bad edges coming from $V(W)$. However, $V(W)$ doesn't affect the cliques in the vortex, so it isn't an issue for showing that the almost-partition is $(1,b(b+1)(k+1),\VC(\Gamma))$-concentrated. Further, any counterexample to the almost-partition being $(2,(6k+8)\binom{b+1}{3},\Attachable(\Gamma))$-concentrated needs to intersect at least $3$ parts, of which two don't correspond to $V(W)$. Throwing away the part corresponding to $V(W)$ (if it exists), controlling the number of bad edges between the remaining two parts is enough to prevent such a counterexample from arising.

    We can now complete the proof.
    \begin{lemma}
        \label{raiseConcentrated}
        Let $\Gamma$ be an almost-embedding of a graph $G$, let $\scr{PR}$ be a partitioned reduction of $\Gamma$ of vortex-width at most $k\in \ds{N}$ and treewidth at most $b\in \ds{N}$, and let $\tau$ be a raise of $\scr{PR}$. Then $(X,\scr{P}(\tau))$ is $(1,b(b+1)(k+1),\VC(\Gamma))$-concentrated and $(2,(6k+8)\binom{m}{3},\Attachable(\Gamma))$-concentrated.
    \end{lemma}

    \begin{proof}
        Let $(G,\Sigma,G_0,\scr{D},H,\scr{J},A):=\Gamma$, let $(V_M,X,G',G_0^+,W,D,H',\scr{J}',\phi,\scr{P}_0^+):=\scr{PR}$, and let $\scr{C}:=\scr{C}(\scr{PR})$. For each $C\in \scr{C}$, let $\psi_C:=\psi_C(\tau)$ and $\scr{P}_C:=\scr{P}_C(\tau)=(\phi(\tau^{-1}(P))\cap V(C):P\in \scr{P}_0^+)$. Recall that $\scr{P}:=\scr{P}(\tau)=\bigsqcup_{C\in \scr{C}}\scr{P}_C$. Also, for each $C\in \scr{C}$, recall that $\psi_C$ is a bijection from $\scr{P}_C$ to $\scr{P}_0^+$ such that $\phi(\tau^{-1}(\psi_C(P)))\cap V(C)=P$ for each $P\in \scr{P}_C$.
        
        By \cref{raiseAP}, $(X,\scr{P})$ is an almost-partition of $G$. By \cref{raiseTw}, $\scr{P}$ has treewidth at most $b$.
        
        Fix a clique $Q$ of $(G-X)$. Since $\scr{P}$ has treewidth at most $b$, note that $|Q|\leq b+1$. 
        
        By \cref{raiseEdges}, there exists $C\in \scr{C}$ such that $Q\subseteq \scr{P}_C$, and $Q':=\psi_C(Q)$ is a clique in $G_0^+/\scr{P}_0^+$. Since $\psi_C$ is a bijection, $|Q'|=|Q|\leq k+1$. 

        Let $\prec$ be an ordering of $Q$. If $V(W)\in Q'$ we require that $P\preceq \psi_C^{-1}(V(W))$ for each $P\in Q$. Other than this, $\prec$ can be arbitrary.

        For each $i\in \{2,3\}$, let $\scr{Q}_i=\binom{Q}{i}$. Observe that $|\scr{Q}_i|=\binom{|Q|}{i}\leq \binom{b+1}{i}$.

        Fix $E=\{P_1,P_2\}\in \scr{Q}_2$ with $P_1\prec P_2$. So $P_1,P_2$ are distinct. Let $P_1':=\psi_C(P_1)$ and $P_2':=\psi_C(P_2)$. So $P_1'$ and $P_2'$ are distinct. For each $i\in \{1,2\}$, by \cref{liftConcentratedPrelimCoroll} (with $P_i'$ as $P_1$ and $P_{3-i}'$ as $P_2$), there exists a set $S_{E,i}'\subseteq V(G)$ with $|S_{E,i}'|\leq 2(k+1)$ such that if $u\in \phi^{-1}(P_1)$ and $v\in \phi^{-1}(P_2)$ are adjacent in $G'$ and satisfy either:
        \begin{enumerate}
            \item $u,v\in V(H')$ and are adjacent in $H'$, or,
            \item $u\notin P_i'$ and $P_2'\neq V(W)$,
        \end{enumerate}
        then $(u,v)$ intersects $S_{E,i}'$. Let $S_{E,i}:=\phi(S_{E,i}')$. Note that $|S_{E,i}|\leq |S_{E,i}'|\leq 2(k+1)$.

        Fix $T=\{P_1,P_2,P_3\}\in \scr{Q}_3$ with $P_1\prec P_2\prec P_3$. So $P_1,P_2,P_3$ are pairwise distinct. Let $P_1':=\psi_C(P_1)$, $P_2':=\psi_C(P_2)$, and $P_3':=\psi_C(P_2)$. Note that by choice of $\prec$, $V(W)\notin \{P_1',P_2'\}$. Let $T':=\psi_C(T)=\{P_1',P_2',P_3'\}$. Since $P_1,P_2,P_3$ are pairwise distinct, $P_1',P_2',P_3'$ are pairwise distinct. Thus and since $P_1',P_2',P_3'\in \scr{P}_0^+$, which is a partition for $G_0^+$, $P_1',P_2',P_3'$ are pairwise disjoint.
        
        Observe that $T'$ is a clique of size $3$ in $G_0^+/\scr{P}_0^+$. By \cref{facialConcentrated} (since $\scr{P}_0^+$ is a connected partition of the plane graph $G_0^+$), there exists at most two cliques in $\FT(G_0^+)$ that intersect all three parts in $T'$. Thus, there exists a set $S'_{T,0}\subseteq V(G_0^+)\subseteq V(G')$ of size at most two such that each clique in $\FT(G_0^+)$ that intersects all three parts in $T'$ also intersects $S_{T,0}'$. Let $S_{T,0}:=\phi(S'_{T,0})$. So $S_{T,0}\subseteq V(G)$ and $|S_{T,0}|\leq |S'_{T,0}|\leq 2$.
        
        Let $E_{12}:=\{P_1,P_2\}$ and $E_{13}:=\{P_1,P_3\}$. Note that $E_{12},E_{13}\in \scr{Q}_2$. Let $S_{T,1}:=S_{E_{12},1}$, $S_{T,2}':=S_{E_{12},2}'$, $S_{T,3}':=S_{E_{13},2}'$, and $S_T':=\bigcup_{i=0}^3 S_{T,i}'$. Observe that $|S_T'|\leq 3(2(k+1))+2=6k+8$. Set also $j_1:=2$ and $j_2:=j_3:=1$. So for each $i\in \{1,2,3\}$, if $u\in \phi^{-1}(P_i)$ and $v\in \phi^{-1}(P_{j_i})$ are adjacent in $G'$ and satisfy either:
        \begin{enumerate}
            \item $u,v\in V(H')$ and are adjacent in $H'$, or,
            \item $u\notin P_i'$ and $P_j'$ is not $V(W)$,
        \end{enumerate}
        then $(u,v)$ intersects $S_{T,i}'\subseteq S_T'$. Set $S_T:=\phi(S_T')$. So $|S_T|\leq |S_T'|\leq 6k+8$.

        Let $S_1:=\bigcup_{E\in \scr{Q}_2}S_{E,1}$, and let $S_2:=\bigcup_{T\in \scr{Q}_3}S_T$. So $S_1,S_2\subseteq V(G)$. Observe that $|S_1|\leq 2(k+1)|\scr{Q}_2|\leq 2(k+1)\binom{b+1}{2}=b(b+1)(k+1)$, and that $|S_2|\leq (6k+8)|\scr{Q}_{Q,3}|\leq (6k+8)\binom{b+1}{3}$.

        Consider any $K\in \VC(\Gamma)$. Note that $K$ is a clique in $G$. Presume, for a contradiction, that $K\setminus S_1$ intersects at least two distinct parts $P_1,P_2\in Q$, where $P_1\prec P_2$. So there exists $x\in P_1\cap K\setminus S_1$ and $y\in P_2\cap K\setminus S_1$. Thus, $(x,y)$ is disjoint to $S_1$. Since $P_1$ is distinct to $P_2$, and since $P_1,P_2\in \scr{P}$, which is a partition of $G-X$, $P_1,P_2$ are disjoint. Thus, $x\neq y$.
        
        Since $\VC(\Gamma)$ is closed under subsets, observe that $\{x,y\}\in \VC(\Gamma)$. Thus, $x,y\in V(H)$, and are adjacent in $H$. Hence and since $\{x,y\}$ is disjoint to $X$, $x,y\in V(H)\setminus X\subseteq V(H')$, and $x,y$ are adjacent in $H-(X\cap V(H))\subseteq H'$. In particular, $x,y\in V(G')$ and are adjacent in $G'$.
        
        Recall that $H'$ is disjoint to $W$. Thus, $x,y\in V(H-X)\subseteq V(H')\subseteq V(G')\setminus V(W)$. So $x,y\in V(G')$, and by \cref{mostlyIdenInv}, $\phi^{-1}(x)=\{x\}$ and $\phi^{-1}(y)=\{y\}$. So $\phi(x)=x$ and $\phi(y)=y$. Recall that $x\in P_1$ and $y\in P_2$, thus $\phi^{-1}(x)\subseteq \phi^{-1}(P_1)$ and $\phi^{-1}(y)\subseteq \phi^{-1}(P_2)$. So $x\in \phi^{-1}(P_1)$ and $y\in \phi^{-1}(P_1)$ are adjacent in $G'$ and satisfy $x,y\in V(H')$ and $x,y$ are adjacent in $H'$.

        Let $E=(P_1,P_2)$. Note that $E\in \scr{Q}_2$, so $S_{E,1}\subseteq S_1$. By definition of $S_{E,1}'$ (and $S_{E,2}'$), $(x,y)$ intersects $S_{E,1}'$ (and $S_{E,2}'$). Thus, $(\phi(x),\phi(y))=(x,y)$ intersects $S_{E,1}\subseteq S_1$ (and $S_{E,2}$). But this contradicts the fact that $(x,y)$ is disjoint to $S_1$. Thus, $K\setminus S_1$ intersects at most one part in $Q$. Since $|S_1|\leq b(b+1)(k+1)$ (and recalling that $\VC(\Gamma)\subseteq 2^{V(G)}$), we find that $\scr{P}$ is $(1,b(b+1)(k+1),\VC(\Gamma))$-concentrated.

        Now, consider any $K\in \Attachable(\Gamma)$. Note that $K$ is a clique in $G$. Presume, for a contradiction, that $K\setminus S_2$ intersects at least three pairwise distinct parts $P_1,P_2,P_3$ in $Q$, where $P_1\prec P_2\prec P_3$. So for each $i\in \{1,2,3\}$, there exists $x_i\in P_i\cap K\setminus S_2$. So $(x_1,x_2,x_3)$ is disjoint to $S_2$. Since $P_1,P_2,P_3$ are pairwise distinct, and since $P_1,P_2,P_3\in \scr{P}$, which is a partition of $G-X$, $P_1,P_2,P_3$ are pairwise disjoint. So $x_1,x_2,x_3$ are pairwise distinct. Since $\Attachable(\Gamma)$ is closed under subsets, $\{x_1,x_2,x_3\}\in \Attachable(\Gamma)$. Also, note that $P_1,P_2,P_3$ are disjoint to $X$, so $\{x_1,x_2,x_3\}$ is disjoint to $X$ and $x_1,x_2,x_3\in V(G-X)$.
        
        Since $\{x_1,x_2,x_3\}\in \Attachable(\Gamma)$, either $\{x_1,x_2,x_3\}\setminus A\in \VC(\Gamma)$ or $\{x_1,x_2,x_3\}\setminus A\in \AEC(\Gamma)$. Since $\{x_1,x_2,x_3\}$ is disjoint to $X\supseteq A$, $\{x_1,x_2,x_3\}\setminus A=\{x_1,x_2,x_3\}$, so either $\{x_1,x_2,x_3\}\in \VC(\Gamma)$ or $\{x_1,x_2,x_3\}\in \AEC(\Gamma)$. In the former case, set $E:=E_{12}=\{P_1,P_2\}$. By repeating the previous argument (replacing $x$ with $x_1$ and $y$ with $x_2$), we find that $x_1,x_2\in V(G')$, that $\phi(x_1)=x_1$ and $\phi(x_2)=x_2$, and that $\{x_1,x_2\}$ intersects $S_{E,1}'=S_{T,1}'\subseteq S_T'$. So $\phi(\{x_1,x_2\})=\{x_1,x_2\}$ intersects $\phi(S_T')=S_T\subseteq S_2$, a contradiction. So we may assume that $\{x_1,x_2,x_3\}\in \AEC(\Gamma)$.
        
        Recall that $\AEC(\Gamma)\setminus \FT(\Gamma)=\AEC(\Gamma)\setminus\FT(G_0)$ are the cliques of $G_0$ of size at most $2$. Since $|\{x_1,x_2,x_3\}|=3$, $\{x_1,x_2,x_3\}\in \FT(G_0)$. Since $\{x_1,x_2,x_3\}$ is disjoint to $X$, there exists $K'\in \FT(G_0^+)$ with $\phi(K')=\{x_1,x_2,x_3\}$. Since each facial triangle of $G_0^+$ has size exactly $3$, we find that $K'=\{x_1',x_2',x_3'\}$, where $\phi(x_i')=x_i$ for each $i\in \{1,2,3\}$. Note that $x_i'\in \phi^{-1}(x_i)\subseteq \phi^{-1}(P_i)$ for each $i\in \{1,2,3\}$.

        Set $T:=\{P_1,P_2,P_3\}$ and $T':=\psi_C(T)=\{P_1',P_2',P_3'\}$, where $P_1':=\psi_C(P_i)$ for each $i\in \{1,2,3\}$. Note that $T\in \scr{Q}_3$, so $S_T\subseteq S_2$. Recall that $P_1',P_2',P_3'$ are pairwise distinct. By definition of $\prec$, note that neither $P_1'$ nor $P_2'$ is $V(W)$.
        
        Since $\{x_1,x_2,x_3\}$ is disjoint to $S_2\supseteq S_T=\phi(S_T')$, $\{x_1',x_2',x_3'\}$ is disjoint to $S_T'\supseteq S_{T,0}'$. Thus, by definition of $S_{T,0}'$, $\{x_1',x_2',x_3'\}$ does not intersect all three parts of $T'$. So there exists $i\in \{1,2,3\}$ such that $x_i'\notin P_i'$. Let $j:=1$ if $i\in \{2,3\}$ and $j:=2$ if $i=1$. Since $j\leq 2$, $P_j'\neq V(W)$. Since $K'$ is a clique, $x_i'$ is adjacent to $x_j'$ in $G'$. So $x_i'\in \phi^{-1}(P_i)$ is adjacent to $x_j'\in \phi^{-1}(P_j)$ in $G'$ and satisfies $x_i'\notin P_i'$ and $P_j'$ is not $V(W)$. So by definition of $S_{T,i}'$, $(x_i',x_j')$ intersects $S_{T,i}'\subseteq S_T'$. Thus, $(x_i,x_j)=(\phi(x_i'),\phi(x_j'))$ intersects $\phi(S_T')=S_T\subseteq S_2$. But this contradicts the fact that $\{x_i,x_j\}\subseteq \{x_1,x_2,x_3\}$ is disjoint to $S_2$. So we conclude that $K\setminus S_2$ intersects at most two parts of $Q$. Since $|S_2|\leq (6k+8)\binom{b+1}{3}$ (and recalling that $\Attachable(\Gamma)\subseteq 2^{V(G)}$), we find that $(X,\scr{P})$ is $(2,(6k+8)\binom{b+1}{3},\Attachable(\Gamma))$-concentrated.

        Thus, $(X,\scr{P})$ is $(1,b(b+1)(k+1),\VC(\Gamma))$-concentrated and $(2,(6k+8)\binom{b+1}{3},\Attachable(\Gamma))$-concentrated. This completes the proof.
    \end{proof}
    
    \subsection{Proof}

    We can now complete the proof of \cref{raiseToPartition}, which we now recall.

    \raiseToPartition*

    \begin{proof}
        Let $(V_M,X,G',G_0^+,W,D,H',\scr{J}',\phi,\scr{P}_0^+):=\scr{PR}$, and let $\scr{P}:=\scr{P}(\tau)$.

        By \cref{raiseAP}, $(X,\scr{P})$ is an almost-partition of $G$. By \cref{raiseTw}, $(X,\scr{P})$ has treewidth at most $b$. By \cref{raiseWidth}, $(X,\scr{P})$ has width at most $\max((k+1)w,w')$. Since $\scr{PR}$ has loss at most $q$, $\scr{R}$ has loss at most $q$, and thus $|X|\leq q$. So $(X,\scr{P})$ has loss at most $q$. By \cref{raiseConcentrated}, $(X,\scr{P})$ is $(1,b(b+1)(k+1),\VC(\Gamma))$-concentrated and $(2,(6k+8)\binom{b+1}{3},\Attachable(\Gamma))$-concentrated. 
        
        Thus, $(X,\scr{P})$ is an almost-partition of $G$ of treewidth at most $b$, width at most $\max((k+1)w,w')$, and loss at most $q$ that is $(1,b(b+1)(k+1),\VC(\Gamma))$-concentrated and $(2,(6k+8)\binom{b+1}{3},\Attachable(\Gamma))$-concentrated. This completes the proof.
    \end{proof}

    \section{Finding a raise}
    \label{SectionMerging}

    This section is devoted to the proof of \cref{findRaise}.

    \subsection{Manipulating a partitioned reduction}

    We finally reach the heart of the problem, finding a raise. We remind the reader that the goal is not to find a raise of any partitioned reduction, but rather to manipulate the partitioned reduction until we can find a raise. Specifically, we merge some parts together, and add some more vertices to the loss set. The former leads to the following definition.

    Let $\scr{P},\scr{P}'$ be partitions of a graph $G$. We say that $\scr{P}'$ is a \defn{merging} of $\scr{P}$ if, for each $P\in \scr{P}$, there exists $P'\in \scr{P}'$ such that $P\subseteq P'$. We remark that in other papers, this is often phrased in reverse and called a `refinement'. That is, with the same condition, we would say that $\scr{P}$ is a refinement of $\scr{P}'$. Given that we will be going from smaller parts to larger parts, we prefer to use phrasing that compares the larger to the smaller, rather than the smaller to the larger.

    We make use of the following result.

    \begin{restatable}{observation}{extensionMinor}
        \label{extensionMinor}
        Let $\scr{P},\scr{P}'$ be connected partitions of a graph $G$ such that $\scr{P}'$ is a merging of $\scr{P}$. Then $G/\scr{P}'$ can be obtained from a minor of $G/\scr{P}$ by adding isolated vertices.
    \end{restatable}

    \begin{proof}
        Let $\scr{P}_1,\scr{P}_1'$ be the nonempty parts of $\scr{P}$, $\scr{P}'$ respectively. Observe that $\scr{P}_1,\scr{P}_1'$ are connected partitions of $G$, and that $G/\scr{P}$, $G/\scr{P}'$ are obtained from $G/\scr{P}_1$, $G/\scr{P}_1'$ respectively by adding isolated vertices. We show that $G/\scr{P}_1'$ can be obtained from a minor of $G/\scr{P}_1$. The result follows.

        For each $P\in \scr{P}_1'$, let $\scr{P}_P$ be the set of parts in $\scr{P}_1$ contained in $P$. Since each part in $\scr{P}_1,\scr{P}_1'$ is nonempty and by definition of a merging, observe that $\scr{P}^*:=(\scr{P}_P:P\in \scr{P}_1')$ is a partition of $G/\scr{P}_1$. Further, since $\scr{P}_1'$ is a connected partition of $G$, observe that $\scr{P}^*$ is connected. Thus, $(G/\scr{P}_1)/\scr{P}^*$ is a minor of $G/\scr{P}_1$, and there is an obvious bijection $\phi:G/\scr{P}_1'\mapsto (G/\scr{P}_1)/\scr{P}^*$ that sends each $P\in \scr{P}_1'$ to $\scr{P}_P$.
        
        Observe that for any pair of adjacent $P,P'\in \scr{P}_1'$ in $G/\scr{P}_1'$, we can find $Q\in \scr{P}_P$ and $Q'\in \scr{P}_{P'}$ such that $Q,Q'$ are adjacent in $G/\scr{P}_1$. It follows that $\scr{P}_P$ and $\scr{P}_P'$ are adjacent in $(G/\scr{P}_1)/\scr{P}^*$. Thus, $\phi$ is an isomorphism from $G/\scr{P}_1'$ to a spanning subgraph of $(G/\scr{P}_1)/\scr{P}^*$. It follows that $G/\scr{P}_1'$ is a minor of $(G/\scr{P}_1)/\scr{P}^*$ and thus $G/\scr{P}_1$, as desired.
    \end{proof}

    We remark that the isolated vertices just come from adding extra empty parts to $\scr{P}'$.

    Using \cref{extensionMinor} we obtain the following.

    \begin{restatable}{observation}{PReductionExtends}
        \label{PReductionExtends}
        Let $\Gamma$ be an almost-embedding, let $\scr{PR}=(V_M,X,G',G_0^+,W,D,H',\scr{J}',\phi,\scr{P}_0^+)$ be a partitioned reduction for $\Gamma$, and let $\scr{P}_0^*$ be a connected partition of $G_0^+$ containing $V(W)$ that is a merging of $\scr{P}_0^+$. Then $\scr{PR}^*:=(\scr{R},\scr{P}_0^*)$ is a partitioned reduction of $\Gamma$. Further, the vortex-width, remainder-width, and loss of $\scr{PR}^*$ are the vortex-width, remainder-width, and loss of $\scr{PR}$ respectively, and the treewidth of $\scr{PR}^*$ is at most the treewidth of $\scr{PR}$.
    \end{restatable}

    \begin{proof}
        As $\scr{PR}$ is a partitioned reduction of $\Gamma$, $\scr{R}:=(V_M,X,G',G_0^+,W,D,H',\scr{J}',\phi)$ is a reduction of $\Gamma$. Thus, by definition of $\scr{P}_0^*$, $\scr{PR}^*$ is a partitioned reduction of $\Gamma$. Since the vortex-width, remainder-width and loss of $\scr{PR}^*$ and $\scr{PR}$ are determined completed by $\scr{R}$, the vortex-width, remainder-width and loss of $\scr{PR}^*$ are the vortex-width, remainder-width and loss of $\scr{PR}$ respectively.

        By \cref{extensionMinor}, $G_0^+/\scr{P}_0^*$ can be obtained from a minor of $G_0^+/\scr{P}_0^+$ by adding isolated vertices. Noting that $G_0^+/\scr{P}_0^+$ is nonempty (as $V(W)\in \scr{P}_0^+$), neither of these operations increase the treewidth. Thus, the treewidth of $\scr{P}_0^*$, and thus $\scr{PR}^*$, is at most the treewidth of $\scr{PR}$ (which is the treewidth of $\scr{P}_0^+$). This completes the proof.
    \end{proof}

    As for the latter, specifically, we want to take a set of vertices on the modified boundary, and all vertices in their bags to the loss set. This leads to the following definition.

    Let $\Gamma$ be an almost-embedding, and let $\scr{PR}=(V_M,X,G',G_0^+,W,D,H',\scr{J}',\phi,\scr{P}_0^+)$ be a partitioned reduction of $\Gamma$ with $\scr{J}'=(J_x':x\in B'(\scr{PR}))$. For $S\subseteq B'(\scr{PR})$, let $X_S:=\bigcup_{s\in S}J_s'$, and let \defn{$\scr{PR}-S$}$:=(V_M,X\cup X_S,G',G_0^+,W,D,H',\scr{J}',\phi,\scr{P}_0^+)$.

    \begin{restatable}{observation}{PReductionExtraLoss}
        \label{PReductionExtraLoss}
        Let $\Gamma$ be an almost-embedding, let $\scr{PR}$ be a partitioned reduction for $\Gamma$, and let $S\subseteq B'(\scr{PR})$. Then $\scr{PR}-S$ is a partitioned reduction of $\Gamma$ of the same treewidth, vortex-width, and near-width, and at most the same remainder-width. Further, if $\scr{PR}$ has vortex-width at most $k\in \ds{N}$ and loss at most $q\in \ds{R}_0^+$, then $\scr{PR}-S$ has loss at most $q+(k+1)|S|$.
    \end{restatable}

    \begin{proof}
        Let $(G,\Sigma,G_0,\scr{D},H,\scr{J},A):=\Gamma$ and $(V_M,X,G',G_0^+,W,D,H',\scr{J}',\phi,\scr{P}_0^+):=\scr{PR}$ with $\scr{J}'=(J_x':x\in B'(\scr{PR}))$. So $\scr{PR}-S=(V_M,X\cup X_S,G',G_0^+,W,D,H',\scr{J}',\phi,\scr{P}_0^+)$, where $X_S:=\bigcup_{s\in S}J_s'$.

        Recall that $\scr{R}:=\scr{R}(\scr{PR})=(V_M,X,G',G_0^+,W,D,H',\scr{J}',\phi)$ is a reduction of $\Gamma$. Since $(H',U'(\scr{PR}),\scr{J}')$ is a graph-decomposition, $X_S\subseteq V(H')\subseteq V(H)\subseteq V(G)$. Thus, by \cref{extendReduction}, $\scr{R}':=(V_M,X\cup X_S,G',G_0^+,W,D,H',\scr{J}',\phi)$ is a reduction of $\Gamma$ of the vortex-width and at most remainder-width. Since these reductions have the same plane subgraph and obstruction subgraph and by definition of $\scr{PR}$ (in particular, the definition of $\scr{P}_0^+$), it follows that $\scr{PR}-S=(V_M,X\cup X_S,G',G_0^+,W,D,\linebreak H',\scr{J}',\phi,\scr{P}_0^+)$ is a partitioned reduction of $\Gamma$ of the same treewidth, vortex-width, and near-width, and at most the same remainder-width.
        
        If $\scr{PR}$ has vortex-width at most $k\in \ds{N}$, then observe that $|X_S|\leq \bigcup_{s\in S}(k+1)=(k+1)|S|$. Thus and if $\scr{PR}$ has loss at most $q\in \ds{R}_0^+$ (so $\scr{R}$ has loss at most $q$), by \cref{extendReduction}, $\scr{R}'$ (and thus $\scr{PR}-S$) has loss at most $q+(k+1)|S|$. This completes the proof.
    \end{proof}

    \subsection{Proof idea}

    As hinted at in \cref{SecMainIdea}, the goal is to merge, and then find some `breakpoints' $S$ in the modified boundary $B'$ such that $\scr{PR}-S$ admits a raise. So for each vertex $v$ in the modified quasi-vortex not contained in any bag indexed by $S$, we assign at part in $\scr{P}_0^+$ to put $v$ into that doesn't create `new adjacencies' in the quotient. To decide on which part to assign $v$ to, we look only at the interval $C_v$ of the modified underlying cycle $U'$ induced by the vertices $x\in B'$ whose bag contains $v$. If $V(C_v)$, intersects $S$, then $v$ will be in the loss-set, and is thus of no concern. Otherwise, ideally, $V(C_v)$ is contained entirely in one part $P$, in which case we assign $v$ to $P$. However, we find that there is exactly one scenario where we cannot do this. In this scenario, we have one part $P$ that intersects $V(C_v)$ and `locally dominates' the area around $C_v$. Specifically, each other part $P'$ that intersects or is adjacent to $V(C_v)$ is adjacent to $P$. Further, $P$ is the `only' such part in the area, meaning that if $v\in B'$ and is adjacent to another $u\in B'$, then either $V(C_u)$ (where $C_u$ is the interval of $U'$ induced by the vertices $x\in B'$ whose bag contains $u$) is contained in a single part, or $V(C_u)$ intersects $P$. In this case, we instead assign $v$ to $P$, possibly moving it out of a part that already contained it (if $v\in B'$).

    So the first question is, how should we merge the parts? The idea is that we want to a variant of \cref{treeDeletions} to split up the plane subgraph $G_0^+$. Specifically, we want to approximate $G_0^+$ with a graph of bounded treewidth, using the following lemma of \citet{Dujmovic2017} (see also \citet{Eppstein2000} for an earlier $O(gr)$ bound).
    \begin{lemma}
        \label{twPlanar}
        Every plane graph of radius at most $r\in \ds{N}$ has treewidth at most $3r$.
    \end{lemma}
    We remark that \cref{twPlanar} is proven in the context of simple graphs, but adding parallel edges and loops changes neither the treewidth nor radius, so it also holds for non-simple graphs.

    We also use the following corollary of \cref{treeDeletions}.

    \begin{corollary}
        \label{twSplit}
        Let $d\in \ds{R}^+$, and let $G$ be a vertex-weighted graph of treewidth at most $b\in \ds{N}$ and total weight at most $n\in \ds{R}_0^+$. Then there exists $S\subseteq V(G)$ with $|S|\leq (b+1)n/d$ such that each connected component of $G-S$ has weight at most $d$.
    \end{corollary}

    \begin{proof}
        Let $\scr{J}=(J_t:t\in V(T))$ be a tree-decomposition of $G$ of width at most $b$. Fix $r\in V(T)$, and let $\scr{V}=(G,T,r,\scr{J})$. So $\scr{V}$ is a rooted tree-decomposition of $G$ of width at most $b$.

        For each $t\in V(T)$, let $J_t^-:=J_t^-(\scr{V})$. Recall that $(J_t^-:t\in V(T))$ is a partition of $G$.
        
        For each $t\in V(T)$, weight $t$ in $T$ by the total weight of the vertices in $J_t^-$. Since $(J_t^-:t\in V(T))$ is a partition of $G$, the total weight of $T$ is exactly the total weight of $G$, which is at most $n$. By \cref{treeDeletions} with $q:=\floor{n/d}$, there exists $Z\subseteq V(T)$ with $|Z|\leq q\leq n/d$ such that each connected component of $T-Z$ has total weight at most $n/(q+1)\leq d$.

        Let $S:=\bigcup_{z\in Z}J_z$. Since $\scr{J}$ has width at most $b$, $|S|\leq (b+1)|Z|\leq (b+1)n/d$. Observe that for each connected component $C$ of $G-S$, there exists a component $T_C$ of $T-Z$ such that $V(C)\subseteq V(G(\scr{V}[T_C]))=\bigcup_{t\in V(T_C)}J_t^-$. Further, note that the total weight of $T_C$ is exactly $|V(G(\scr{V}[T_C]))|$, which is at most $d$. Thus, $|V(C)|\leq |V(G(\scr{V}[T_C]))|\leq d$, as desired.
    \end{proof}

    Of course, the plane subgraph $G_0^+$ does not have bounded radius. However, we are only interested in the behaviour of $G_0^+$ near $B'$, and we are only interested in the behaviour of parts (the quotient), not individual vertices. So we contract $V(W)$ and each part that intersects $B'$ to a single vertex, and each connected component of the graph after deleting these parts into a single vertex. As $\Lambda(\scr{PR})$ is standardised, $B'\subseteq N_G(V(W))$, and thus the resulting graph $G'$ has radius $2$ after deleting isolated vertices. So \cref{twPlanar} and \cref{twSplit} can be applied. Specifically, we weight each vertex in $G'$ according to how many vertices were contracted into that vertex. Note also that we can view $G'$ as the quotient of a connected partition $\scr{P}_0'$ of $G_0^+$.

    Backtracking, we find that \cref{twSplit} would have us delete a `small' number of parts/components $S'$, and every remaining component $C$ of $G'-S'$ would have `small' size. So the union of all the parts (of $\scr{P}_0'$) in $C$ has a `small' number of vertices. If the total number of vertices in these parts/components were small, we would then immediately be done by making each of these components a part (noting that each original part is either a vertex in $G'$ or contained in a component-vertex of $G'$). However, there is absolutely no reason to expect this. Fortunately, we don't need to separate the entire graph, because, once again, we are only interested in the area around the $B'$. Even more specifically, we are interested in $U'$, and where the vertices in quasi-vortex `lie' above $U'$. We find that it suffices to only break up $U'$ (by introducing breakpoints), as this controls the behaviour of the vortex-vertices.

    This is where we can make some headway. Because the partition was connected, each of these parts/components are connected, and thus by the non-crossing property, their intersections with $U'$ end up `blocking off' parts of the underlying cycle from the rest of $G_0':=G_0'(\scr{PR})$. We pick an arbitrary starting point for $U'$ to form a linear ordering of $B'$, and for each part $P$ we want to `delete', we remove (add to the breakpoints $S$) just the first and last vertices of $P\cap B'$ under this linear ordering. This has the effect of restricting any interval of $U'-S$ to a small region of the graph that only contains one of these parts we wanted to delete. We also need to control the neighbourhoods around intervals, so we repeat the small process with the neighbourhood (with some exclusions) around each part and component we wanted to delete. The result is that every interval of $U'-S$ and its neighbourhood lies in a very restricted region of $G_0^+$. Specifically, we find that $S'$ is `locally dominating' within this region, as described earlier.

    The desired partition is obtained from merging all the parts (of $\scr{P}_0'$) contained in a component of $G'-S'$, and leaving all the parts (of $\scr{P}_0^+$) contained in part or component in $S'$ untouched. The parts in $S'$ are the `locally dominating' parts mentioned at the start of this explanation.

    In terms of a proof, the set $S'$ inspires the following definition.

    Let $\scr{PR}=(V_M,X,G',G_0^+,W,D,H',\scr{J}',\phi,\scr{P}_0^+)$ be a partitioned reduction of an almost-embedding $\Gamma$. Recall that $G_0':=G_0'(\scr{PR})=G_0^+-V(W)$. A \defn{break} for $\scr{PR}$ is a set $\scr{B}$ of pairwise disjoint connected subgraphs of $G_0'$ such that:
    \begin{enumerate}
        \item for each $M\in \scr{B}$ and $P\in \scr{P}_0^+$, if $P$ intersects $V(M)$, then $P\subseteq V(M)$,
        \item for each $M\in \scr{B}$, at least one of the following is true:
        \begin{enumerate}
            \item $M=G_0'[P]$ for some $P\in \scr{P}_0^+$ distinct from $V(W)$, or,
            \item $V(M)$ is disjoint to $B'(\scr{R})$, and,
        \end{enumerate}
        \item for each connected subgraph $C$ of $G_0'$, at least one of the following is true:
        \begin{enumerate}
            \item $C$ intersects some $M\in \scr{B}$, or,
            \item $V(C)\subseteq P$ for some $P\in \scr{P}_0^+$.
        \end{enumerate}
    \end{enumerate}
    Since $V(W)\in \scr{P}_0^+$, for each $P\in \scr{P}_0^+$ distinct from $V(W)$, $P\subseteq V(G_0')$. So $G_0'[P]$ is well-defined in condition (2a) for each $M\in \scr{B}$.

    For our proof, the parts and components in $S'$ will be the elements of the break.

    So the proof can then be divided into two steps, being merging parts to find a small break, and then using the small break to find a small set of breakpoints, which is then used to find a raise (after adding the breakpoints and the vertices in their bags to the loss-set). We start with the former step, as it is by far the easier.

    \subsection{Finding a break}

    \begin{lemma}
        \label{findBreak}
        Let $n\in \ds{N}$, let $\Gamma$ be an almost-embedding of an $n$-vertex graph $G$, and let $\scr{PR}$ be a partitioned reduction of $\Gamma$ of width at most $w\in \ds{R}_0^+$. Then for any $d\in \ds{R}^+$ with $d\geq w$, there exists a partitioned reduction $\scr{PR}'$ of $\Gamma$ of near-width at most $d$ that admits a break of size at most $7n/d$. Further, the treewidth, vortex-width, remainder-width, and loss of $\scr{PR}'$ are at most the treewidth, vortex-width, remainder-width, and loss of $\scr{PR}$ respectively.
    \end{lemma}

    \begin{proof}
        Let $\scr{PR}=(V_M,X',G',G_0^+,W,D,H',\scr{J}',\phi,\scr{P}_0^+)$, and $G_0':=G_0'(\scr{PR})=G_0(\Lambda(\scr{PR}))=G_0^+-V(W)$. Recall that $D$ is $G_0'$-clean. Let $B':=B'(\scr{PR})=B(D,G_0')=V(G_0^+\cap H')$. Recall that $B'\subseteq V(H')$ is disjoint to $V(W)$. Set $\Lambda:=\Lambda(\scr{R})=(G',G_0^+,W,D,H',\scr{J}')$. So $\Lambda$ is a standardised plane+quasi-vortex embedding.

        Let $\scr{P}_B:=(P\in \scr{P}_0^+:P\cap B'\neq \emptyset)$, and let $V_B:=\bigcup_{P\in \scr{P}_B}P$. Since $V(W)\in \scr{P}_0^+$, observe that $V_B$ is disjoint to $V(W)$. So $V_B\subseteq V(G_0')$. Note also that $B'\subseteq V_B$.
        
        Let $\scr{C}$ be the connected components of $G_0'-V_B$, and let $\scr{P}_C:=(V(C):C\in \scr{C})$. Since $\scr{P}_0^+$ is a connected partition of $G_0^+$, observe that $\scr{P}_0':=\scr{P}_B\sqcup \scr{P}_C\sqcup \{V(W)\}$ is a connected partition of $G_0^+$ that is a merging of $\scr{P}_0^+$.

        Let $G':=G/\scr{P}_0'$. Since $\scr{P}'$ is a connected partition of $G$, $G'$ is a minor of $G$. Thus, $G'$ is planar. Since $\Lambda$ is standardised, $B'\subseteq N_{G_0^+}(V(W))$. So each $P\in \scr{P}_B$ is adjacent to $V(W)$ in $G'$. Further, observe that each $P\in \scr{P}_C$ is either adjacent to some $P'\in \scr{P}_B$ in $G'$, or is isolated in $G'$. Thus, observe that $G'$ has one component containing $V(W)$ with radius at most $2$, and every other component is an isolated vertex. By \cref{twPlanar}, the component of $G'$ containing $V(W)$ has treewidth at most $6$. Thus $G'$ has treewidth at most $6$, as adding isolated vertices does not increase this bound. Hence, $G'-V(W)$ also has treewidth at most $6$.
        
        Fix $v\in V(G'-V(W))$. Note that $v=P\in \scr{P}_0'$. In $G'-V(W)$, weight $v=P$ by $|P|$. Since $\scr{P}_0'$ is a partition of $G_0^+$, the total weight on $G'-V(W)$ is exactly $|V(G_0')|\leq |V(G)|=n$. By \cref{twSplit}, there exists a set $S'\subseteq V(G')$ with $|S'|\leq 7n/d$ such that each connected component of $G'-S'-V(W)$ has total weight at most $d$.
        
        Let $\scr{C}'$ denote the connected components of $G'-V(W)-S'$. So each $C\in \scr{C}'$ has total weight at most $d$. Let $P_C:=\bigcup_{P\in V(C)}P$. Since $\scr{P}_0'$ is a partition of $G$, observe that $|P_C|$ is exactly the total weight of $C$, which is at most $d$. So $|P_C|\leq d$. Since $\scr{P}_0'$ is a connected partition of $G_0^+$ and since $C$ is connected, observe that $G_0^+[P_C]$ is connected.

        Let $S'_C:=S'\cap \scr{P}_C$ and $S'_B:=S'\cap \scr{P}_B$. Observe that $S'=S'_C\sqcup S'_P$. Let $\scr{P}_C'$ be the set of parts $P\in \scr{P}_C$ such that $C_P\in S'_C$. Observe that $\scr{P}_0^*:=(P_C:C\in \scr{C}')\sqcup \scr{P}_C'\sqcup S'_B\sqcup (V(W))$ is a connected partition of $G_0^+$ containing $V(W)$ that is a merging of $\scr{P}_0^+$. Thus, by \cref{PReductionExtends}, $\scr{PR}^*:=(\scr{R},\scr{P}_0^*)$ is a partitioned reduction of $\Gamma$, and the treewidth, vortex-width, remainder-width, and loss of $\scr{PR}^*$ are at most the treewidth, vortex-width, remainder-width, and loss of $\scr{PR}$ respectively.

        Observe that each $P\in \scr{P}_C'\sqcup S'_U$ is a part of $\scr{P}_0^+$ distinct from $V(W)$. Thus, since $\scr{PR}$ has near-width at most $w$, $|P|\leq w\leq d$. For each $C\in \scr{C}'$, recall that $|P_C|\leq d$. Thus, each part of $\scr{P}_0^*$ distinct from $V(W)$ has size at most $d$. So $\scr{PR}^*$ has near-width at most $d$.

        It remains only to find a break of size at most $7n/d$ for $\scr{PR}^*$. Since $\scr{P}_0'$ is a connected partition of $G_0^+$ and since $V(W)\notin S'$, observe that $\scr{B}:=\{G_0'[P]:P\in S'\}$ is a well-defined collection of pairwise disjoint connected subgraphs of $G_0'$. Note that $|\scr{B}|=|S'|\leq 7n/d$. We show that $\scr{B}$ is a break for $\scr{PR}^*$.

        Observe that $\scr{P}_0'$ is a merging of $\scr{P}_0^*$. Since $S'\subseteq \scr{P}_0'$, it follows that for each $P\in \scr{P}^*$ and each $P'\in S'$, either $P$ is disjoint to $P'$ (and thus $G[P']\in \scr{B}$), or is contained in $P'$. Since $\scr{B}=\{G[P]:P\in S'\}$, it follows that if $P$ intersects $V(M)$ for some $M\in \scr{B}$, then $P\subseteq V(M)$.

        For each $M\in \scr{M}$, recall that $M=G_0'[P]$ for some $P\in S'\subseteq V(G'-V(W))$. Thus, $P\in \scr{P}_B\sqcup \scr{P}_C$. If $P\in \scr{P}_B$, recall that $P\in \scr{P}_B\subseteq \scr{P}_0^*$. So $M=G_0'[P]$ where $P\in \scr{P}_0^*$. If $P\in \scr{P}_C$, recall that $P=V(C)$ for some connected component $C$ of $G_0'-V_B$. Observe that $B'\subseteq V_B$, and thus $P$ is disjoint to $B'$.

        Consider any connected subgraph $C'$ of $G_0'$. Since $C'$ is connected, observe that either $V(C')$ intersects some $P\in S'$, or $V(C')\subseteq P_C$ for some connected component $C$ of $G'-V(W)-S'$. In the former case, $C'$ intersects $H_P\in \scr{B}$, and in the latter case, $V(C')\subseteq P_C\in \scr{P}_0^*$.

        Thus, $\scr{B}$ is a break for $\scr{PR}^*$ of size $|S'|\leq 7n/d$. This completes the proof.
    \end{proof}

    \subsection{Using the break}

    We now move to the later step, using the break to find breakpoints, and then a raise. Really, this step is two steps, with one of the steps being `given a set of connected subgraphs that divide the graph, find a set of breakpoints that divide the underlying cycle appropriately'. Specifically, we have the following technical lemma.

    \begin{lemma}
        \label{findBreakpoints}
        Let $G$ be a plane graph, let $\scr{H}$ be a set of at most $m\in \ds{R}_0^+$ connected subgraphs of $G$, and let $D$ be a $G$-clean disc in the plane with $U:=U(D,G)$. Then there exists a set $S\subseteq V(U)$ with $|S|\leq 2m$ such that if $\scr{I}'$ denotes the set of connected subgraphs of $U-S$, and $H_1,H_2\in \scr{H}$ are disjoint, then the following hold:
        \begin{enumerate}
            \item each $I\in \scr{I}'$ intersects at most one of $H_1,H_2$, and,
            \item if $I_1,I_1',I_2,I_2'\in \scr{I}'$ are such that $I_1'\subseteq I_1, I_2'\subseteq I_2$, $I_1'\cup I_2'$ is disjoint to $H_1\cup H_2$, and $N_G[V(I_1')]$ intersects $V(I_2')$, then $I_1\cup I_2$ intersects at most one of $H_1,H_2$.
        \end{enumerate}
    \end{lemma}

    \begin{proof}
        Let $B:=B(D,G)$. Recall that $B=V(U)\subseteq V(G)$. If $|B|=0$, then the lemma is trivially true with $S=\emptyset$. So we may assume that $|B|\geq 1$. Fix any $r\in B$. Recall that $U$ is a cycle or a path on at most $2$ vertices. Let $\prec$ be the linear ordering of $V(U)$ obtained from the cyclic ordering $\prec_U$ by starting at $r$. Note that if $a,b,x,y\in B$ are such that $x\prec a\prec y\prec b$, then $x\prec_U a\prec_U y\prec_U b$ also.
    
        For each $H\in \scr{H}$, if $|V(H)\cap B|\leq 2$, let $S_H:=V(H)\cap B$. Otherwise, observe that there exists distinct $x_H,y_H\in V(H)\cap B$ such that $x_H\prec y_H$ and, for each $z\in V(H)\cap B$, $x_H\preceq z\preceq y_H$. Let $S_H:=\{x_H,y_H\}$ in this case. Observe that in either scenario, $|S_H|\leq 2$, and $S_H\subseteq V(H)\cap B$.

        Let $S:=\bigcup_{H\in \scr{H}}S_H$. Note that $|S|\leq 2m$, and that $S\subseteq V(U)$. Let $\scr{I}'$ be the set of connected subgraphs of $U-S$.

        We start by proving the first property. We need the following claim.

        \begin{claim}
            \label{claimExtermalVerts}
            Let $I\in \scr{I}'$ and $H\in \scr{H}$ intersect. Then there exists distinct $x,y\in V(H)\cap S$ such that for each $z\in V(I)$, $x\prec z\prec y$.
        \end{claim}

        \begin{proofofclaim}
            Since $I$ and $H$ intersect, there exists $u\in V(I\cap H)$.
        
            By definition of $I\in \scr{I}'$, $V(I)\subseteq V(U-S)=B\setminus S$. Thus, $V(H)$ intersects $B\setminus S$ (in particular, $u\in V(H)\cap (B\setminus S)$).

            First, assume, for a contradiction, that $|V(H)\cap B|\leq 2$. Recall then that $V(H)\cap B=S_H\subseteq S$. So $V(H)\cap B$ is contained in $S$, contradicting the fact that $V(H)$ intersects $B\setminus S$. So we can conclude that $|V(H)\cap B|\geq 3$.
            
            Thus, recall that $S_H:=\{x_H,y_H\}$, where $x_H,y_H\in V(H)\cap B$ are distinct and satisfy $x_H\preceq z\preceq y_H$ for each $z\in V(H)\cap B$. Further, recall that $\{x_H,y_H\}=S_H\subseteq S$. So $x_H,y_H\in V(H)\cap S$. Since $I$ is disjoint to $S$, we find that $\{x_H,y_H\}$ is disjoint to $V(I)$. Thus and since $u\in V(I\cap H)$, we have $x_H\prec u\prec y_H$.

            Since $U$ is a cycle or a path (and since $x_H\neq y_H$), observe that there is a connected component $C$ of $U-x_H-y_H$ such that $V(C)$ is exactly the vertices $z\in V(U)$ for which $x_H\prec z\prec y_H$. In particular, $u\in V(C)$. Since $I$ is connected and contains $u$, it follows that $x_H\prec z\prec y_H$ for each $z\in V(I)$. The claim follows with $x:=x_H$ and $y:=y_H$.
        \end{proofofclaim}

        We can now prove the first property. Presume otherwise. So $H_1,H_2\in \scr{H}$ are disjoint and $I\in \scr{I}'$ intersects both $H_1$ and $H_2$. Thus, there exists $u,v\in V(I)$ with $u\in V(H_1)$ and $v\in V(H_2)$. Note that since $H_1$ and $H_2$ are disjoint, $u\neq v$.

        By \cref{claimExtermalVerts} with $H:=H_1$, there exists distinct $x,y\in V(H_1)\cap S$ such that for each $z\in V(I)$, $x\prec z\prec y$. In particular, $x\prec u,v\prec y$. Similarly (by \cref{claimExtermalVerts} with $H:=H_2$), there exists distinct $a,b\in V(H_2)\cap S$ such that $a\prec u,v\prec b$. So $x,a\prec u,v\prec b,y$. Since $H_1$ and $H_2$ are disjoint, observe that $a,b,u,v,x,y$ are all pairwise disjoint.

        Recall that $x,y,u\in V(H_1)\cap B=V(H_1\cap U)$, and that $a,b,v\in V(H_2)\cap B=V(H_2\cap U)$. Since $H_1,H_2$ are connected and disjoint subgraphs of $G$, by \cref{nonCrossing}, we have neither $x\prec_U a\prec_U u\prec_U b$ nor $x\prec v\prec y\prec a$.
            
        Since $a\neq x$, either $x\prec a$ or $a\prec x$. If $x\prec a$, then we obtain $x\prec a\prec u\prec b$, and thus $x\prec_U a\prec_U u\prec_U b$, contradiction. If instead $a\prec x$, then we obtain $a\prec x\prec v\prec y$ and thus $a\prec_U x\prec_U v\prec_U y$. This gives $x\prec_U v\prec_U y\prec_U a$, another contradiction.
        
        So in either scenario, we have a contradiction. It follows that $I$ intersects at most one of $H_1,H_2$. This proves the first property.

        We now prove the second property. Presume otherwise. So $H_1,H_2\in \scr{H}$ are disjoint, and $I_1,I_1',I_2,I_2'\in \scr{I}'$ are such that $I_1'\subseteq I_1, I_2'\subseteq I_2$, $I_1'\cup I_2'$ is disjoint to $H_1\cup H_2$, $N_G[V(I_1')]$ intersects $V(I_2')$, and $I_1\cup I_2$ intersects both $H_1,H_2$.
        
        Observe that if $I_1\cup I_2$ is connected, then $I_1\cup I_2\in \scr{I}$, contradicting the first property. So we may assume that $I_1\cup I_2$ is disconnected. In particular, since $I_1,I_2$ are connected, this implies that $I_1$ and $I_2$ are disjoint. Thus, since $I_1'\subseteq I_1$ and $I_2'\subseteq I_2$, we have that $I_1'$ and $I_2'$ are disjoint.

        Since $N_G[V(I_1')]$ intersects $V(I_2')$ and since $I_1'$ is disjoint to $I_2'$, there exists distinct $u\in V(I_1')$ and $v\in V(I_2')$ such that $u,v$ are adjacent in $G$.
        
        By the first property, neither $I_1$ nor $I_2$ intersect both $H_1$ and $H_2$. Thus and since $I_1\cup I_2$ intersects both $H_1$ and $H_2$, without loss of generality (up to swapping $H_1$ and $H_2$), $I_1$ intersects $H_1$ and $I_2$ intersects $H_2$. By \cref{claimExtermalVerts} with $H:=H_1$ and $I:=I_1$, there exists distinct $x,y\in V(H_1)\cap S$ such that for each $z\in V(I_1)$, $x\prec z\prec y$. In particular, $x\prec u\prec y$. Similarly (by \cref{claimExtermalVerts} with $H:=H_2$ and $I:=I_2$), there exists distinct $a,b\in V(H_2)\cap S$ such that $a\prec v\prec b$. Since $H_1$, $H_2$, and $I_1'\cup I_2'$ are pairwise disjoint, observe that $a,b,u,v,x,y$ are all pairwise disjoint. Further, observe that $H_1,H_2,G[\{u,v\}]$ are pairwise disjoint connected subgraphs of $G$.

        Since $x\neq a$, observe that either $x\prec a$ or $a\prec x$. Without loss of generality (up to swapping $(H_1,I_1,I_1',u,x,y)$ and $(H_2,I_2,I_2',v,a,b)$), presume $x\prec a$.
            
        Since $u\neq a$, either $u\prec a$ or $a\prec u$. If $u\prec a$, we obtain $u\prec a\prec v\prec b$ and thus $u\prec_U a\prec_U v\prec_U b$. This contradicts \cref{nonCrossing} with $u,v\in G[\{u,v\}]$ and $a,b\in V(H_2)$. So we can assume $a\prec u$. Since $u\prec y$, this implies $a\prec y$.
            
        Since $u\neq b$, either $u\prec b$ or $b\prec u$. If $b\prec u$, then we obtain $a\prec v\prec b\prec u$ and thus $a\prec_U v\prec_U b\prec_U u$. This again contradicts \cref{nonCrossing}, this time with $a,b\in V(H_2)$ and $y,v\in G[\{u,v\}]$. So we can assume $u\prec b$.

        Since $b\neq y$, either $b\prec y$ or $y\prec b$. If $y\prec b$, then we obtain $x\prec a\prec y\prec b$ and thus $x\prec_U a\prec_U y\prec_U b$. This contradicts \cref{nonCrossing} with $x,y\in V(H_1)$ and $a,b\in V(H_2)$. So we can assume $b\prec y$.

        Thus, we obtain $x\prec a\prec u\prec b \prec y$. Recall that $I_1$ intersects $H_1$. So there exists $w\in V(I_1\cap H_1)$. Since $a,b\in S$ and since $U$ is a cycle or a path, observe that the component of $U-S$ containing $u$ (which exists since $u\in V(I_1')$, which is a subgraph of $U-S$) consists only of vertices $z\in V(U)\setminus S$ satisfying $a\prec z\prec b$. Since $I_1\supseteq I_1'$ is a connected subgraph of $U-S$ containing $w$ and $u$, it follows that $a\prec w\prec b$. So $x\prec a\prec w\prec b$, and thus $x\prec_U a\prec_U w\prec_U b$. This contradicts \cref{nonCrossing} with $x,w\in V(H_1)$ and $a,b\in V(H_2)$.

        Thus, we have reached a contradiction in all scenarios. So we can conclude that $I_1\cup I_2$ intersects at most one of $H_1,H_2$. This completes the proof of the second property, and the lemma.
    \end{proof}

    We remark that the subgraphs $\scr{H}$ we give to \cref{findBreakpoints} will be the graphs in the break, along with their neighbourhoods (minus any overlap with other elements of the break).
    
    We can use the break to find a raise.

    \begin{lemma}
        \label{useBreak}
        Let $\scr{PR}$ be a partitioned reduction of an almost-embedding $\Gamma$ that admits a break $\scr{B}$ of size at most $q\in \ds{R}_0^+$. Then there exists a set $S\subseteq B'(\scr{PR})$ with $|S|\leq 4q$ such that $\scr{PR}-S$ admits a raise $\tau$.
    \end{lemma}

    \begin{proof}
        Let $(V_M,X,G',G_0^+,W,D,H',\scr{J}',\phi,\scr{P}_0^+):=\scr{PR}$, $B':=B'(\scr{PR})=V(G_0^+\cap H')$, $U':=U'(\scr{PR})$, $G_0':=G_0'(\scr{PR})=G_0^+-V(W)$, $\Lambda:=\Lambda(\scr{PR})=(G',G_0^+,W,D,H',\scr{J}')$, and $(J_x':x\in B'):=\scr{J}'$. So $\Lambda$ is a standardised plane+quasi-vortex embedding.

        Let $\scr{P}_{\scr{B}}$ be the set of $P\in \scr{P}_0^+$ distinct from $V(W)$ such that $G_0'[P]\in \scr{B}$ (recall that $P\subseteq V(G_0')$ as $V(W)\in \scr{P}_0^+$). So by definition of $\scr{B}$, for each $M\in \scr{H}$, either $V(M)$ is disjoint to $B'$ or $M=G_0'[P]$ for some $P\in \scr{P}_{\scr{B}}$.

        For each $M\in \scr{B}$, recall that $M$ is a connected subgraph of $G_0'$. Let $M':=G_0'[V(M)\cup (N_{G_0'}(V(M))\setminus \bigcup_{M^*\in \scr{B}}V(M^*))]$. Since $V(M)\subseteq V(M')$, observe that $M'$ is also a connected subgraph of $G_0'$.

        Let $\scr{H}=\bigcup_{M\in \scr{B}}\{M\}\cup\{M'\}\}$. Observe that $|\scr{H}|\leq 2|\scr{B}|\leq 2q$, and that $\scr{B}\subseteq \scr{H}$.
        
        Recall that $D$ is $G_0'$-clean, and that $U'=U(D,G_0')\subseteq V(G_0')$. Thus, $G_0'$ is a plane graph, $D$ is a $G_0'$-clean disc with $U(D,G_0')=U'$, and $\scr{H}$ is a set of at most $2q$ connected subgraphs of $G_0'$. Thus, by \cref{findBreakpoints}, there exists a set $S\subseteq V(U')=B'$ with $|S|\leq 4q$ such that if $\scr{I}'$ is the set of connected subgraphs of $U'-S$, then for each pair of disjoint $M_1,M_2\in \scr{H}$, the following hold:
        \begin{enumerate}
            \item each $I\in \scr{I}'$ intersects at most one of $M_1,M_2$, and,
            \item if $I_1,I_1',I_2,I_2'\in \scr{I}'$ are such that $I_1'\subseteq I_1, I_2'\subseteq I_2$, $I_1'\cup I_2'$ is disjoint to $M_1\cup M_2$, and $N_{G_0'}[V(I_1')]$ intersects $V(I_2')$, then $I_1\cup I_2$ intersects at most one of $M_1,M_2$.
        \end{enumerate}

        Let $\scr{I}$ be the set of connected and nonempty subgraphs of $U'-S$ for which every connected and nonempty subgraph $I'$ of $I$ satisfies $G_0^+[V(I')]=G_0'[V(I')]$ is connected (noting that $V(U')\subseteq V(G_0')$). Observe that $\scr{I}$ is closed under taking subgraphs that are connected and nonempty, and that $\scr{I}\subseteq \scr{I}'$.

        We start with the following basic observation about $\scr{I}$.

        \begin{claim}
             \label{claimUnions}
             Let $I,I'\in \scr{I}$ intersect. Then $I\cup I'\in \scr{I}$.
        \end{claim}
            
        \begin{proofofclaim}
             Since $I,I'\in \scr{I}$, $I,I'$ are connected and nonempty subgraphs of $U'-S$. Thus and since $I$ intersects $I'$, $I\cup I'$ is a connected and nonempty subgraph of $U'-S$. Thus, it remains only to show that for each connected subgraph $I''$ of $I\cup I'$, $G_0'[V(I'')]$ is connected. Since $I\cup I'$ is connected, observe that it suffices to show that for each pair of adjacent $u,v\in V(I\cup I')$, $u,v$ are adjacent in $G_0'$. Observe that $u,v$ are adjacent in either $I$ and $I'$. Thus and since $\scr{I}$ is closed under taking subgraphs that are connected and nonempty, $U'[\{u,v\}]\in \scr{I}$. Thus, $G_0'[\{u,v\}]$ is connected. So $u$ and $v$ are adjacent in $G_0'$, as desired. The claim follows.
        \end{proofofclaim}

        The next claim allows us to assign to each $I\in \scr{I}$ a part $P_I\in \scr{P}_0^+$ obeying some specific properties, which we then use to construct the raise.

        \begin{claim}
            \label{claimOneIntersection}
            For each $I\in \scr{I}$, there exists exactly one $P_I\in \scr{P}_0^+$ for which at least one of the following is true.
            \begin{enumerate}
                \item $V(I)\subseteq P_I$, or,
                \item $P_I\in \scr{P}_{\scr{B}}$ and $V(I)$ intersects $P_I$.
            \end{enumerate}
        \end{claim}

        \begin{proofofclaim}
            We first show that such a $P_I$ exists. By definition of $\scr{I}$, $G_0'[V(I)]$ is a connected subgraph of $G_0'$. Thus, by definition of $\scr{B}$ either $G_0'[V(I)]$ intersects some $M\in \scr{B}$, or $V(I)\subseteq P$ for some $\scr{P}_0^+$. In the latter case, set $P_I:=P$. In the former case, $V(I)\subseteq B'$ intersects $V(M)$, so by definition of $\scr{B}$ (and $\scr{P}_{\scr{B}}$), $M=G_0'[P]$ for some $P\in \scr{P}_{\scr{B}}$. Set $P_I:=P$ in this case.

            Next, we prove uniqueness. Presume otherwise. So there exists distinct $P,P'\in \scr{P}_0^+$ satisfying at least one of the two conditions. Note that since $I$ is nonempty (by definition of $\scr{I}$), if $V(I)$ is contained in either $P$ or $P'$, then it intersects $P$ or $P'$ respectively. So in either scenario, $V(I)$ intersects both $P$ and $P'$. Thus, if $V(I)$ is contained in either $P$ or $P'$, then $P$ intersects $P'$, contradicting the fact that $P,P'$ are distinct parts of $\scr{P}_0^+$, which is a partition of $G_0^+$. Thus, we can conclude that $P,P'\in \scr{P}_{\scr{B}}$ and $V(I)$ intersects both $P$ and $P'$. By definition of $\scr{P}_{\scr{B}}$, we find that $G_0'[P],G_0'[P']\in \scr{B}\subseteq \scr{H}$ both intersect $I\in \scr{I}$. However, since $P,P'\in \scr{P}_0^+$ are distinct, they are disjoint, and thus $G_0'[P],G_0'[P']\in \scr{H}$ are disjoint. So we have disjoint elements of $\scr{H}$ that both intersect an element of $\scr{I}$, and contradiction (by choice of $S$ and definition of $\scr{I}\subseteq \scr{I}'$). Thus, $P_I$ is unique, as desired.
        \end{proofofclaim}

        The goal of the next few claims is to control how the neighbourhoods around $I\in \scr{I}$ and the corresponding $P_I\in \scr{P}_0^+$ interact.

        \begin{claim}
            \label{claimNeighInPI}
            For each $I\in \scr{I}$, if $P_I\in \scr{P}_0^+$ is from \cref{claimOneIntersection}, and $P\in \scr{P}_0^+$ is not $V(W)$ and is disjoint to $V(I)$, then $P_I$ is disjoint to $N_{G_0^+}(V(I)\setminus P_I)$.
        \end{claim}

        \begin{proofofclaim}
            Presume otherwise. Since $V(W)\in \scr{P}_0^+$, observe that $P\subseteq V(G_0')$ and $P$ intersects $N_{G_0'}(V(I)\setminus P_I)$. Thus, there exists a (nonempty) connected component $I'$ of $I-(P_I\cap V(I))$ such that $P$ intersects $N_{G_0'}(V(I'))$. Since $\scr{I}$ is closed under taking nonempty subgraphs, $I'\in \scr{I}$. So $G_0'[V(I')]$ is connected. Since $\scr{P}_0^+$ is a connected partition of $G_0^+$, $G_0^+[P]=G_0'[P]$ is connected. Thus and since $P$ intersects $N_{G_0'}(V(I'))$, observe that $G_0'[V(I')\cup P]$ is a connected subgraph of $G_0'$. Thus, by definition of $\scr{B}$, either $V(I')\cup P\subseteq P'$ for some $P'\in \scr{P}_0^+$, or $V(I')\cup P$ intersects $M$ for some $M\in \scr{B}$.

            Presume, for a contradiction, that the former holds. Since $P$ is nonempty (as $P$ intersects $N_{G_0'}(V(I'))$) and since $P,P'\in \scr{P}_0^+$, we have $P'=P$. Thus, $V(I')\subseteq P'$. However, recall that $I'$ is a nonempty subgraph of $I$. Thus, we find that $P$ intersects $V(I)$, a contradiction.
            
            Thus, we may assume that the latter scenario holds. So $G_0'[V(I')\cup P]$ intersects $M$ for some $M\in \scr{B}$. Thus, either $I'$ intersects $M$, or $G_0'[P]$ intersects $M$.

            Presume, for a contradiction, that the former holds. By choice of $I'\subseteq I-(P_I\cap V(I))$, $V(M)\neq P_I$. However, since $I'$ is nonempty, we have $V(I)\nsubseteq P_I$. Thus, by definition of $P_I$, $P_I\in \scr{P}_{\scr{B}}$ and $P_I$ intersects $V(I)$. So $G_0'[P_I]\in \scr{B}$, and $G_0'[P_I]$ intersects $I$. Since $V(M)$ intersects $V(I')\subseteq B'$, by definition of $\scr{B}$ (and $\scr{P}_{\scr{B}}$), $V(M)=G_0'[P']$ for some $P'\in \scr{P}_{\scr{B}}$. So $G_0'[P']\in \scr{B}$. Since $V(M)\neq P_I$, $P'\neq P_I$. Thus and since $P',P_I\in \scr{P}_{\scr{B}}\subseteq \scr{P}_0^+$, $P'$ and $P_I$ are disjoint. So $G_0'[P']$ is disjoint to $G_0'[P_I]$. Thus, $I$ intersects $G_0'[P_I]$ and $G_0'[P']$ (as $I'\subseteq I$), and $G_0'[P'],G_0'[P_I]\in \scr{B}\subseteq \scr{H}$ are disjoint. This contradicts the choice of $S$ (and definition of $\scr{I}\subseteq \scr{I}'$).

            Hence, we may assume that the latter scenario holds (and the former doesn't). So $P$ intersects $V(M)$. By definition of $\scr{B}$, we have $P\subseteq V(M)$. Thus, $N_{G_0'}(P)\setminus \bigcup_{M^*\in \scr{B}}V(M^*)\subseteq V(M')$. Since $P$ intersects $N_{G_0'}(V(I'))$, $N_{G_0'}(P)$ intersects $V(I')$. Since the former scenario was false, $I'$ does not intersect any $M^*\in \scr{B}$. Thus, $M'$ intersects $I'\subseteq I$.

            By the same argument as in the previous scenario, $P_I$ intersects $V(I)$, $P_I\in \scr{P}_{\scr{B}}\subseteq \scr{P}_0^+$, and $G_0'[P_I]\in \scr{B}\subseteq \scr{H}$. Since $P$ is disjoint to $V(I)$, we have that $P$ is distinct to $P_I$. Thus (and since $P,P_I\in \scr{P}_0^+$), $P$ is disjoint to $P_I$. Since $P$ intersects $V(M)$ and is disjoint to $V(G_0'[P_I])$, $M$ and $G_0'[P_I]$ are distinct subgraphs in $\scr{B}$. Thus and since the subgraphs in $\scr{B}$ are pairwise disjoint, $M$ and $G_0'[P_I]$ are pairwise disjoint. Since $G_0'[P_I]\in \scr{B}$, $P_I\subseteq \bigcup_{M^*\in \scr{B}}V(M^*)$. Thus and since $V(M')\setminus V(M)$ is disjoint to $\bigcup_{M^*\in \scr{B}}V(M^*)$, we find that $G_0'[P_I]$ is disjoint to $M'$. Thus, $I\in \scr{I}$ intersects $G_0'[P_I]\in \scr{B}\subseteq \scr{H}$ and $M'\in \scr{H}$, where $G_0'[P_I]$ and $M'$ are disjoint. This contradicts the choice of $S$ (and definition of $\scr{I}\subseteq \scr{I}'$).

            Hence, we arrive in a contradiction in all cases. This completes the proof of the claim.
        \end{proofofclaim}

        \begin{claim}
            \label{claimNeighMeetsPI}
            For each $I\in \scr{I}$, if $P_I\in \scr{P}_0^+$ is from \cref{claimOneIntersection}, then for each $P\in \scr{P}_0^+$ that intersects $N_{G_0^+}[V(I)]$, $P$ intersects $N_{G_0^+}[P_I]$.
        \end{claim}

        \begin{proofofclaim}
            Recall that $V(I)\subseteq B'\subseteq V(G_0^+)$. So $N_{G_0^+}(V(I))$ is well-defined.
            
            If $P=P_I$, then observe that $P=P_I$ is nonempty (as it intersects $N_{G_0^+}(V(I))$), so $P=P_I$ intersects $N_{G_0^+}[P_I]$. So we may assume that $P$ is distinct from $P_I$.
        
            We first consider the case when $P$ intersects $V(I)$. Since $P$ is distinct from $P_I$ and since $P_I$ is uniquely defined, observe that $P\notin \scr{P}_{\scr{B}}$ and $V(I)\subsetneq P$. Thus (and since $I$ is connected), we can find a connected and nonempty subgraph $I'$ of $I$ such that $|V(I')\geq 2|$ and $V(I')$ contains exactly one vertex $v\in V(G_0^+)\setminus P$. So $V(I')\cap P=V(I'-v)$ (as $V(I')\subseteq V(I)\subseteq B'\subseteq V(G_0^+)$) is nonempty.
            
            Since $\scr{I}$ is closed under taking subgraphs that are connected and nonempty, we have $I'\in \scr{I}$. Let $P_{I'}\in \scr{P}_0^+$ be from \cref{claimOneIntersection}. Since $v\in V(G_0^+)\setminus P$ and since $V(I')$ intersects $P$ (and since $\scr{P}_0^+$ is a partition of $G_0^+$), observe that $P_{I'}$ does not contain $V(I')$. Thus, $P_{I'}\in \scr{P}_{\scr{B}}$ and $P_{I'}$ intersects $I'\subseteq I$. Thus, by uniqueness of $P_I$, we have $P_I=P_{I'}$. So $P_I$ intersects $I'$. Since $P_I$ is distinct from $P$ and since $V(I'-v)\subseteq P$, we find that $v\in P_I$. Since $I'\in \scr{I'}$, $G_0'[V(I')]$ is connected. Thus and since $v\in P_I$ and since $V(I'-v)$ is a nonempty subset of $P$, $P$ intersects $N_{G_0'}[P_I]\subseteq N_{G_0^+}[P_I]$, as desired.

            Thus, we may assume that $P$ does not intersect $V(I)$. So $P$ intersects $N_{G_0^+}(V(I))$.
            
            If $P$ is $V(W)$, recall that $B'=B(\Lambda)\subseteq N_{G_0^+}(V(W))=N_{G_0^+}(P)$ as $\Lambda$ is standardised. Since $V(I)$ is nonempty (by definition of $\scr{I}$) and contained in $B'$, we obtain $V(I)\subseteq N_{G_0^+}(P)$. Since $V(I)$ intersects $P_I$ (as $I$ is nonempty), we find that $N_{G_0^+}(P)$ intersects $P_I$, and thus $N_{G_0^+}[P_I]$ intersects $P$, as desired.

            Otherwise, by \cref{claimNeighInPI}, $P$ is disjoint to $N_{G_0^+}(V(I)\setminus P_I)$. Since $P$ intersects $N_{G_0^+}(V(I))$, we find that $P$ intersects $N_{G_0^+}(P_I)$, as desired.

            Thus, in all cases, $P$ intersects $N_{G_0^+}(P_I)$. This completes the proof of the claim.
        \end{proofofclaim}

        \begin{claim}
            \label{claimNeighPI}
            For each pair $I,I'\in \scr{I}$ such that $N_{G_0^+}[V(I)]$ intersects $V(I')$, if $P_I,P_{I'}\in \scr{P}_0^+$ are from \cref{claimOneIntersection} with $I$ and $I'$ respectively, then $N_{G_0^+}[P_I]$ intersects $P_{I'}$.
        \end{claim}

        \begin{proofofclaim}
            If $N_{G_0^+}[V(I)]$ intersects $P_{I'}$, then by \cref{claimNeighMeetsPI}, $P_{I'}$ intersects $N_{G_0^+}[P_I]$, as desired. If $N_{G_0^+}[P_I]$ intersects $V(I')$, then $P_I$ intersects $N_{G_0^+}[V(I')]$. Thus, by \cref{claimNeighMeetsPI}, $P_I$ intersects $N_{G_0^+}[P_{I'}]$, and $P_{I'}$ intersects $N_{G_0^+}[P_I]$, as desired. Thus, it remains only to consider the case when $N_{G_0^+}[V(I)\setminus P_I]$ intersects $V(I')\setminus P_{I'}$.

            Presume, for a contradiction, that $P_I$ is distinct to $P_{I'}$. Since $N_{G_0^+}[V(I)\setminus P_I]$ intersects $V(I')\setminus P_{I'}$, both $V(I)\setminus P_I$ and $V(I')\setminus P_{I'}$ are nonempty. Thus, by definition of $P_I,P_{I'}$, we have that $P_I,P_{I'}\in \scr{P}_{\scr{B}}$, and that $V(I)$ intersects $P_I$ and $V(I')$ intersects $P_{I'}$. So $I\cup I'$ intersects both $G_0'[P_I]$ and $G_0'[P_{I'}]$. Since $P_I,P_{I'}\in \scr{P}_{\scr{B}}$, $G_0'[P_I],G_0'[P_{I'}]\in \scr{B}\subseteq \scr{H}$. By uniqueness of $P_I,P_{I'}$ and since $P_I,P_{I'}$ are distinct, $G_0'[P_I]$ does not intersect $I'$ and $G_0'[P_{I'}]$ does not intersect $I$.

            Since $N_{G_0^+}[V(I)\setminus P_I]$ intersects $V(I')\setminus P_{I'}$ and since $I,I'$ are connected, there exists (nonempty) connected components $I_0,I_0'$ of $I-(V(I)\cap P_I)$ and $I'-(V(I')\cap P_{I'})$ respectively such that $N_{G_0^+}[V(I_0)]$ intersects $V(I_0')$. Since $V(I_0),V(I_0')\subseteq B'\subseteq V(G_0')$, it follows that $N_{G_0'}[V(I_0)]$ intersects $V(I_0')$. By definition of $I_0,I_0'$, they are disjoint to $G_0'[P_I]$ and $G_0'[P_{I'}]$ respectively. Thus and since $G_0'[P_I]$ does not intersect $I'\supseteq I_0'$ and $G_0'[P_{I'}]$ does not intersect $I\supseteq I_0$, $I_0\cup I_0'$ is disjoint to $G_0'[P_I]\cup G_0'[P_{I'}]$.

            Since $P_I,P_{I'}\in \scr{P}_0^+$ are distinct, $G_0'[P_I],G_0'[P_{I'}]$ are disjoint. Since $\scr{I}'\subseteq \scr{I}$ is closed under taking subgraphs that are connected and nonempty, $I,I',I_0,I_0'\in \scr{I}'$. Thus, we have disjoint $G_0'[P_I],G_0'[P_{I'}]\in \scr{H}$ and $I,I',I_0,I_0'\in \scr{I}'$ such that $I_0\subseteq I$, $I_0'\subseteq I'$, $I_1'\cup I_2'$ is disjoint to $G_0'[P_I]\cup G_0'[P_{I'}]$, and $N_{G_0'}[V(I_0)]$ intersects $V(I_0')$, and $I\cup I'$ intersects both $G_0'[P_I]$ and $G_0'[P_{I'}]$. This contradicts the choice of $S$ (and definition of $\scr{I}\subseteq \scr{I}'$).

            Thus, we can conclude that $P_I$ is $P_{I'}$. Noting that $V(I),V(I')$ intersect $P_I=P_{I'}$ (as $I,I'$ are nonempty), we find that $P_I=P_{I'}$ is nonempty. Thus, $N_{G_0^+}[P_I]$ intersects $P_{I'}$, as desired. This completes the proof of the claim.
        \end{proofofclaim}

        We are now finally ready to construct the raise.
        
        Let $\scr{PR}':=\scr{PR}-S$. Recall that $\scr{PR}'=(\scr{R}',\scr{P}_0^+)$, where $\scr{R}'=(V_M,X\cup X_S,G',G_0^+,W,D,H',\scr{J}',\phi)$, and $X_S=\bigcup_{s\in S}J_s'$. Note that $X_S\subseteq V(H')$. Thus, observe that $V(G')=(V(G_0^+)\setminus B')\sqcup X_S\sqcup (V(H')\setminus X_S)$. By \cref{PReductionExtraLoss}, $\scr{PR}'$ is a partitioned reduction of $\Gamma$.

        For each $v\in V(H')\setminus X_S$, let $C_v$ be the connected and nonempty subgraph of $U'$ induced by the vertices $x\in B'$ such that $v\in J_x'$. Since $\scr{R}$ is a reduction of $\Gamma$, $\Lambda$ is standardised. Thus, $(H',U',\scr{J}')$ is smooth in $G_0'$ (recalling that $U(\Lambda(\scr{R})=U'(\scr{R})=U'$ and that $G_0(\Lambda)=G_0'$). So there exists a spanning tree $T_v$ of $C_v$ such that for each subtree $T$ of $T_v$, $G_0'[V(T)]$ is connected. Since $v\notin X_S$, observe that $V(T_v)$ is disjoint to $S$, and thus $T_v$ is a connected and nonempty subgraph of $U'-S$. Thus, observe that $T_v\in \scr{I}$. Let $P_v':=P_{T_v}$ be from \cref{claimOneIntersection}. So $P_v'\in \scr{P}_0^+$, and either $V(T_v)=V(C_v)\subseteq P_v'$, or $P_v'\in \scr{P}_B$ and $P_v'$ intersects $V(T_v)=V(C_v)$. Either way, since $C_v$ is nonempty, $V(C_v)=V(T_v)$ intersects $P_v'$. Thus, by definition of $C_v$, there exists $x_v\in P_v'\cap V(C_v)\subseteq P_v'\cap B'$ such that $v\in J_{x_v}'$.
        
        Since $\scr{P}_0^+$ is a partition of $G_0^+$, for each $x\in V(G_0^+)$, there exists exactly one $P_x\in \scr{P}_0^+$ such that $x\in P_x$.

        Observe that $V(G')=(V(G_0^+)\setminus B')\sqcup X_S\sqcup (V(H')\setminus X_S)$. Define a map $\tau:V(G')\mapsto \scr{P}_0^+$ as follows. For each $v\in V(G_0^+)\setminus B'$, set $\tau(v):=P_v$. For each $v\in X_S$, there exists $s\in S$ such that $v\in J_s'$. Set $\tau(v):=P_s$. For each $v\in V(H')\setminus X_S$, set $\tau(v):=P_v'$.

        If $v\in V(G_0^+)\setminus B'$, recall that $\tau(v)=P_v$ and $v\in P_v$. So $v\in \tau(v)$.
        
        If $v\in X_S$, recall that $\tau(v)=P_s$, where $s\in S\subseteq B'$, $v\in J_s'$, and $s\in P_s=\tau(v)$. If $v\in V(H')\setminus X_S$, recall that $\tau(v)=P_v'$, $x_v\in B'\cap P_v'=B'\cap \tau(v)$, and $v\in J_{x_v}'$. Thus, if $v\in V(H')$, there exists $x\in B'\cap \tau(v)$ such that $v\in J_x'$.

        Thus, to show that $\tau$ is a raise for $\scr{PR}-S$, it only remains to show that for each pair of adjacent $u,v\in V(G')$ such that $\phi(u),\phi(v)\in V(G-(X\cup X_S))$, $N_{G_0^+}[\tau(u)]$ intersects $\tau(v)$.

        First, consider the case when $u,v\in V(G_0^+)\setminus B'$. Then $\tau(u)=P_u$ and $\tau(v)=P_v$, where $u\in P_u$ and $v\in P_v$. Since $u,v\in V(G_0^+)\setminus B'$ and since $G'=G_0^+\cup H'$, observe that $u,v$ are adjacent in $G_0^+$. Thus, $v\in N_{G_0^+}(u)\subseteq N_{G_0^+}[P_u]=N_{G_0^+}[\tau(u)]$. Since $v\in P_v=\tau(v)$, we find that $N_{G_0^+}[\tau(u)]$ intersects $\tau(v)$, as desired.

        Next, consider the case when $u,v\in V(H')$ and are adjacent in $H'$. Since $W$ is disjoint to $H'$, observe that $u=\phi(u),v=\phi(v)\in V(G-(X\cup X_S))$. In particular, neither $u$ nor $v$ is in $X_S$. So $u,v\in V(H'-X_S)$ (as $X_S\subseteq H'$), and are adjacent in $H'-X_S$. Since $u,v\in V(H'-X_S)=V(H')\setminus X_S$, recall that $\tau(u)=P_u'=P_{T_v}$ and $\tau(v)=P_v'=P_{T_v}$. Since $u,v$ are adjacent in $H'-X_S\subseteq H'$, $C_u$ and $C_v$ intersect. Hence, $T_v$ and $T_v$ intersect. Thus, by \cref{claimUnions}, $T_{uv}:=T_v\cup T_v\in \scr{I}$. Let $P_{uv'}:=P_{T_{uv}}$ be from \cref{claimOneIntersection}. Recall that $\tau(u)=P_u'$ and $\tau(v)=P_v'$ intersect $V(T_v)$ and $V(T_v)$ respectively, and hence $\tau(u),\tau(v)$ intersect $T_{uv}$. Thus, by \cref{claimNeighMeetsPI}, $N_{G_0^+}[P_{uv}']$ intersects both $\tau(u),\tau(v)$. Hence, if either $P_{uv}'=\tau(u)$ or $P_{uv}'=\tau(v)$, we find that $N_{G_0^+}[\tau(u)]$ intersects $\tau(v)$, as desired. So we may assume that $P_{uv}'$ is distinct from both $P_u'=\tau(u)$ and $P_v'=\tau(v)$.
        
        Since $P_u'$ intersects $V(T_v)\subseteq V(T_{uv})$ and since $P_v'$ intersects $V(T_v)\subseteq V(T_{uv})$, by the uniqueness of $P_{uv}'$ (from \cref{claimOneIntersection}), neither $P_u'\in \scr{B}$ nor $P_v'\in \scr{B}$. Thus, by \cref{claimOneIntersection}, $V(T_v)\subseteq P_u'$ and $V(T_v)\subseteq P_v'$. Since $T_v$ and $T_v$ intersect, we obtain that $P_u',P_v'\in \scr{P}_0^+$ intersect. Thus, $P_u'=\tau(u)$ is $P_v'=\tau(v)$. Hence (and since $P_u'=P_v'$ is nonempty as $V(T_v),V(T_v)\subseteq P_u'=P_v'$ are nonempty), we find that $\tau(u)$ intersects $\tau(v)$. So $N_{G_0^+}[\tau(u)]$ intersects $\tau(v)$, as desired.

        Since $G'=G_0^+\cup H'$, it remains only to consider the case when $u,v\in V(G_0^+)$, are adjacent in $G_0^+$, and at least one of $u,v\in B'$. Without loss of generality, say $u\in B'$. As before, we find that $u\in V(H')\setminus X_S$. So $\tau(u)=P_u'=P_{T_v}$. Since $v\in N_{G_0^+}[T_v]$ and since $v\in P_v$, by \cref{claimNeighMeetsPI}, $N_{G_0^+}[P_u']$ intersects $P_v$. If $v\in G_0^+\setminus B$, $\tau(v)=P_v$, and thus $N_{G_0^+}[P_u']=N_{G_0^+}[\tau(u)]$ intersects $P_v=\tau(v)$, as desired. So it remains only to consider when $u,v\in B'$.

        In this case, we have $\tau(u)=P_u'=P_{T_v}$ and $\tau(v)=P_v'=P_{T_v}$. By \cref{claimNeighPI}, $N_{G_0^+}[P_u']=N_{G_0^+}[\tau(u)]$ intersects $P_v'=\tau(v)$, as desired.

        Hence, in all cases, $N_{G_0^+}[\tau(u)]$ intersects $\tau(v)$. Thus, $\tau$ is a raise for $\scr{PR}-S$, as desired. Since $|S|\leq 4q$, this completes the proof.
    \end{proof}

    We now prove \cref{findRaise}, which we restate here.
    
    \findRaise*

    \begin{proof}
        By \cref{findBreak}, there exists a partitioned reduction $\scr{PR}'$ of $\Gamma$ of treewidth at most $b$, vortex-width at most $k$, near-width at most $d$, remainder-width at most $w'$, and loss at most $q$ that admits a break of size at most $7n/d$. By \cref{useBreak} (with $q:=7n/d$), there exists $S\subseteq B'(\scr{R}')$ with $|S|\leq 28n/d$ such that $\scr{PR}'-S$ admits a raise. By \cref{PReductionExtraLoss}, $\scr{PR}'-S$ is a partitioned reduction of $\Gamma$ of treewidth at most $b$, vortex-width at most $k$, near-width at most $d$, remainder-width at most $w'$, and loss at most $(k+1)|S|+q\leq 28(k+1)n/d+q$.
    \end{proof}

    \section{Finding a reduction}
    \label{SecFindReduction}

    In this section, we prove \cref{findReductionModify}. This is a technical process. We remind the reader that we are only interested in nontrivial almost-embeddings (with disjoint discs). This is because of the following observation.

    \begin{observation}
        \label{embeddedNonempty}
        Let $\Gamma=(G,\Sigma,G_0,\scr{D},H,\scr{J},A)$ be a nontrivial almost-embedding. Then $G_0$ is nonempty.
    \end{observation}

    \begin{proof}
        Presume otherwise. Since $A\subseteq V(G)$ and since $A\neq V(G)$, $G-A=G_0\cup H$ is nonempty. Since $G_0$ is empty, $H$ is nonempty. However, $G_0\cap H=B(\Gamma)$ is empty. Recalling that $V(U(\Gamma))=B(\Gamma)$ and that $\scr{J}$ is a $U(\Gamma)$-decomposition of $H$, we find that $\scr{J}$. Since each $v\in V(H)$ is contained in at least one bag of $\scr{J}$, we find that $H$ is empty, a contradiction. It follows that $G_0$ is nonempty.
    \end{proof}

    Ensuring that the embedded subgraph is nonempty allows us to avoid some strange edge cases.

    \subsection{Standardising an almost-embedding}
    \label{secStandardised}

    We start with an almost-embedding with disjoint discs. The first goal is to `clean up' the almost-embedding to remove some minor annoyances. 
    
    First, we want the embedded subgraph $G_0$ to be connected, and further to be $2$-cell embedded. Since we allow parallel edges and loops, we can easily just add edges disjointedly to the discs to achieve this (since the embedded subgraph is nonempty).
    
    Second, we want the underlying cycles to be contained in $G_0$ (remember that we do not consider edges to be labelled). This is easy to do by just adding edges around the boundary of the disc.

    The next thing we want is for all the discs to be `close' in $G_0$. Specifically, we there to be a vertex $v\in V(G_0)$ such that $N_{G_0}[r]$ intersects the boundary vertices of each disc. We can delete any discs that do not intersect $G_0$, and then add handles to the surface near the discs and edges along the handles to achieve this. Of course, this ruins a $2$-cell embedding, so we have to do the aforementioned clean-up after attaching the handles.
    
    We also have to be careful not to destroy any facial triangles. Adding edges is safe, provided we are careful when adding loop-edges, as seen with the following lemma.

    \begin{restatable}{lemma}{facialSpanningWeak}
        \label{facialSpanningWeak}
        Let $G$ be an embedded graph, let $G'$ be a spanning embedded supergraph of $G$, and presume that no loop-edge in $E(G')\setminus E(G)$ is contained in a triangle-face of $G$. Then $\FT(G)\subseteq \FT(G')$.
    \end{restatable}

    \begin{proof}
        We show the following, slightly stronger, result.
        \begin{claim}
            \label{claimSpanningTriangleFaces}
            Let $G$ be an embedded graph, let $G'$ be a spanning embedded supergraph of $G$, and presume that no loop-edge in $E(G')\setminus E(G)$ is contained in a triangle-face of $G$. Then for each triangle-face $F$ of $G$, there exists a triangle-face $F'$ of $G'$ contained in $F$ with the same boundary vertices as $F$.
        \end{claim}

        \begin{proofofclaim}
            By induction on $|E(G')\setminus E(G)|$. In the base case, $G=G'$, so the lemma is trivial. So we proceed to the inductive step. Fix in $e\in E(G')\setminus E(G)$, and let $G''$ be the embedded graph $G'-e$. $G''$ is also a spanning embedded supergraph of $G$, and no loop-edge in $E(G'')\setminus E(G)\subseteq E(G')\setminus E(G)$ is contained in a triangle-face of $G$. So by induction, for each triangle-face $F$ of $G$, there exists a triangle-face $F''$ of $G''$ contained in $F$ with the same boundary vertices as $F$.

            Fix a triangle-face $F$ of $G$. By the above, there exists a triangle-face $F''$ of $G''$ contained in $F$ with the same boundary vertices as $F$ contained in $F$. If $e$ is not contained in $F''$, then $F''$ is also a triangle-face of $G'$, with the same boundary as in $G''$. So the boundary vertices of $F''$ in $G'$ are the boundary vertices of $F''$ in $G''$, which are the boundary vertices of $F$ in $G$. Hence, $F''$ is the desired face of $G'$. Thus, we may assume that $e$ is contained in $F''$.

            Since $e$ is contained in $F''$, it is contained in $F$. Since $F$ is a triangle-face, by assumption, $e$ is not a loop-edge. Thus, observe that $e$ splits $F''$ into two new faces (faces of $G'$), one of which is an oval-face, and the other is a triangle-face $F'$. Observe that the boundary vertices of $F'$ are the boundary vertices of $F''$, which are the boundary vertices of $F$. Further, $F'$ is contained in $F''$, which is contained in $F$. Hence, $F'$ is the desired face of $G'$. This completes the proof of the claim.
        \end{proofofclaim}
    
        Fix $K\in \FT(G)$. So there exists a triangle-face $F$ of $G$ whose boundary vertices are exactly $K$. By \cref{claimSpanningTriangleFaces}, there exists a triangle-face $F'$ of $G'$ whose boundary vertices are the same as the boundary vertices of $F$, being $K$. So $F'$ is a triangle-face of $G'$ whose boundary vertices are exactly $K$. Thus, $K\in \FT(G')$. It follows that $\FT(G)\subseteq \FT(G')$.
    \end{proof}

    However, attaching handles could destroy facial triangles, if the handles were attached to triangle-faces. We can circumvent this issue by adding extra edges to create new faces with two boundary vertices whose role is specifically to host the handles.

    We give a formal definition for this idea of `close' discs.

    For an embedded graph $G$ and $S\subseteq V(G)$, say that a set of $G$-clean discs $\scr{D}$ is \defn{$S$-anchored} in $G$ if:
    \begin{enumerate}
        \item for each $D\in \scr{D}$, $B(D,G)$ intersects $S$, and,
        \item for each pair of distinct $D,D'\in \scr{D}$, $D\cap D'\subseteq S$
    \end{enumerate}

    Note the following.
    \begin{observation}
        \label{anchoredStronglyClean}
        Let $G$ be an embedded graph, let $S\subseteq V(G)$, and let $\scr{D}$ be a $S$-anchored set of $G$-clean discs. Then $\scr{D}$ is $G$-pristine.
    \end{observation}
    \begin{proof}
        For each pair of distinct $D,D'\in \scr{D}$, $D\cap D'\subseteq S\subseteq V(G)$.
    \end{proof}
    We will sometimes use \cref{anchoredStronglyClean} implicitly (to define $B(\scr{D},G)$ and $U(\scr{D},G)$).
    
    For an embedded graph $G$, say that a set $\scr{D}$ of $G$-clean discs is \defn{$G$-localised} if there exists $r\in V(G)$ such that $\scr{D}$ is $N_G[r]$-anchored in $G_0$. Note that this implies $G$ is nonempty (even if $\scr{D}$ is empty).

    For now, the discs will be pairwise disjoint, so the second condition for `anchored' will be trivially true. However, we will need to consider intersecting discs in the future, so this second condition will be relevant later.

    Combining all of our desired properties leads to the following definition.

    Say that an almost-embedding $\Gamma=(G,\Sigma,G_0,\scr{D},H,\scr{J},A)$ is \defn{standardised} if all the following hold:
    \begin{enumerate}
        \item $G_0$ is $2$-cell embedded in $\Sigma$,
        \item $U(\scr{D},G_0)\subseteq G_0$, and
        \item $\scr{D}$ is $G_0$-localised.
    \end{enumerate}

    We remark that the third condition implies that $G_0$ is nonempty, and thus that $\Gamma$ is nontrivial. Also, the second condition implies that $(H,U(\scr{D},G_0),\scr{J})$ is smooth in $G_0$, as shown by the following observation.

    \begin{observation}
        \label{subgraphSmooth}
        Let $G$ be a graph, and let $\scr{GD}=(H,U,\scr{J})$ be a graph-decomposition with $U\subseteq G$. Then $\scr{GD}$ is smooth in $G$.
    \end{observation}

    \begin{proof}
        Observe that $V(U)\subseteq V(G)$. For each $v\in V(H)$, $C_v:=C_v(\scr{GD})$ is a connected and nonempty (induced) subgraph of $U$. Thus, there exists a spanning tree $T_v$ of $C_v$. For each subtree $T$ of $T_v$, we have $G[V(T)]\supseteq U[V(T)]=C_v[V(T)]\supseteq T_v(V[T])=T$ is connected. It follows that $\scr{GD}$ is smooth in $G$, as desired.
    \end{proof}

    Now, as to why we want $\scr{D}$ to be $G_0$-localised, it is for the purposes of constructing a nice layering of the graph $G$. After forcing all the discs to be close, we can add `phantom edges' to $G_0$ inside the disc so that all the boundary vertices of $\scr{D}$ are within distance $2$ of $r$. Now, if take the BFS layering of this new graph $G_0'$ starting at $r$, we find a layering of $G_0$ in which the boundary vertices of $\Gamma$ are contained within layers $0$, $1$, and $2$, and $r$ is the only vertex in layer $0$. We can convert this to a layering of $G$ by moving $r$ into layer $1$, and placing all the vertices in $V(H)\setminus V(G_0)$ into layer $1$. The following observations summarise this process.

    \begin{restatable}{observation}{layeringSupergraph}
        \label{layeringSupergraph}
        Let $G$ be a graph, let $G'$ be a spanning supergraph of $G$, and let $\scr{L}$ be a layering of $G'$. Then $\scr{L}$ is a layering of $G$.
    \end{restatable}
    \begin{proof}
        Since $V(G)=V(G')$, $\scr{L}$ is a partition of $G$. Let $(L_0,\dots,L_m):=\scr{L}$ (where $m:=|\scr{L}|-1$). For each $i\in \{0,\dots,m\}$, observe that $N_G[L_i]\subseteq N_{G'}[L_i]\subseteq L_{i-1}\cup L_i\cup L_{i+1}$ (where $L_{-1}:=L_{m+1}:=\emptyset$), as desired.
    \end{proof}

    \begin{restatable}{observation}{layeringEmbedToG}
        \label{layeringEmbedToG}
        Let $\Gamma=(G,\Sigma,G_0,\scr{D},H,\scr{J},A)$ be an almost-embedding of a graph $G$, let $m\in \ds{N}$ with $m\geq 2$, and let $\scr{L}=(L_0,L_1,\dots,L_m)$ be a layering a $G_0$ such that $B(\Gamma)\subseteq L_0\cup L_1\cup L_2$. Set $L_1':=L_0\cup L_1\cup (V(H)\setminus B(\Gamma))$. Then $\scr{L}':=(\emptyset,L_1',L_2,\dots,L_m)$ is a layering of $G-A$.
    \end{restatable}
    \begin{proof}
        Recall that $G-A=G_0\cup H$, and that $B:=B(\Gamma)=V(G_0\cap H)$. Set $L_0':=\emptyset$, and for each $i\in \{2,\dots,m\}$, set $L_i':=L_i$. So $\scr{L}'=(L_0',L_1',\dots,L_m')$. Note that $L_i\subseteq L_i'$ for each $i\in \{1,\dots,m\}$. Set also $L_{-1}:=L_{-1}':=L_{m+1}:=L_{m+1}':=\emptyset$.

        Observe that $V(G-A)=V(G_0)\sqcup (V(H)\setminus B(\Gamma))$ and that $L_0\sqcup L_1\sqcup (V(H)\setminus B)=L_0'\sqcup L_1'$. Since $\scr{L}$ is a partition of $G_0$, it follows that $\scr{L}'$ is a partition of $G-A$.

        Since $V(G_0\cap H)=B\subseteq L_0\cup L_1\cup L_2$, observe that for each $i\in \{3,\dots,m\}$, $L_i=L_i'$ is disjoint to $V(H)$.
        
        Fix $i\in \{0,\dots,m\}$. We must show that $N_G[L_i']\subseteq L_{i-1}'\cup L_i'\cup L_{i+1}'$. If $i=0$, then $L_i'=\emptyset$, and this is trivial. If $i\geq 3$, then recall that that $L_i'=L_i$ is disjoint to $V(H)$. Thus, $N_G[L_i']=N_{G_0}[L_i]\subseteq L_{i-1}\cup L_i\cup L_{i+1}\subseteq L_{i-1}'\cup L_i'\cup L_{i+1}'$ (as $i-1\geq 1$).
        
        If $i\in \{1,2\}$, then observe that $L_i'\subseteq L_0\cup L_i\cup V(H)\setminus B$. Thus, observe that $N_G[L_i]\subseteq N_{G_0}[L_0\cup L_i]\cup V(H)\subseteq L_0\cup L_1\cup L_2\cup L_{i+1}\cup V(H)$ (as $i,i-1\in \{0,1,2\}$). Since $B\subseteq L_0\cup L_1\cup L_2$, observe that $L_0\cup L_1\cup L_2\cup V(H)\subseteq L_1'\cup L_2=L_1'\cup L_2'\subseteq L_{i-1}'\cup L_i'\cup L_{i+1}'$ (as $1,2\in \{i-1,i,i+1\}$). Finally, observe that $L_{i+1}\subseteq L_{i+1}'$ as $i+1\geq 1$. Hence, $N_G[L_i]\subseteq L_{i-1}'\cup L_i'\cup L_{i+1}'$, as desired. This completes the proof.
    \end{proof}

    The main result of this subsection is that we can `convert' an almost-embedding into a standardised almost-embedding, in a way that preserves key properties (namely, facial triangles). We formalise this as follows.

    Let $\Gamma=(G,\Sigma,G_0,\scr{D},H,\scr{J},A)$ be an almost-embedding. We say that a tuple $\Gamma'=(G',\Sigma',G_0',\scr{D}',H,\linebreak\scr{J},A)$ \defn{extends} $\Gamma$ if:
    \begin{enumerate}
        \item $G'$ is a graph,
        \item $A\subseteq V(G')$,
        \item $G'-A=G_0'\cup H$.
        \item $E(G)\setminus E(G-A)=E(G')\setminus E(G'-A)$,
        \item $\Sigma'$ is a surface,
        \item $G_0'$ is a graph embedded in $\Sigma'$,
        \item $\scr{D}'$ is a $G_0'$-pristine set of discs,
        \item $G_0$ is a spanning subgraph of $G_0'$,
        \item $\FT(G_0)\subseteq \FT(G_0')$, and,
        \item $U(G_0,\scr{D})$ is a spanning subgraph of $U(G_0',\scr{D}')$.
    \end{enumerate}

    Observe the following.

    \begin{restatable}{observation}{extendsIsSpanning}
        \label{extendsIsSpanning}
        Let $\Gamma$ be an almost-embedding, and let $\Gamma'$ extend $\Gamma$. Then $\Gamma'$ is an almost-embedding. Further:
        \begin{enumerate}
            \item vortex-width and apex-count of $\Gamma'$ are the vortex-width and apex-count of $\Gamma$ respectively,
            \item $G'$ is a spanning supergraph of $G$, and,
            \item $\FT(\Gamma)\subseteq \FT(\Gamma')$.
        \end{enumerate}
    \end{restatable}

    \begin{proof}
        Let $\Gamma=(G,\Sigma,G_0,\scr{D}',H,\scr{J},A)$. Recall that $\Gamma'=(G',\Sigma',G_0',\scr{D}',H,\scr{J},A)$, where:
        \begin{enumerate}
            \item $G'$ is a graph,
            \item $A\subseteq V(G')$,
            \item $G'-A=G_0'\cup H$,
            \item $E(G)\setminus E(G-A)=E(G')\setminus E(G'-A)$,
            \item $\Sigma'$ is a surface,
            \item $G_0'$ is a graph embedded in $\Sigma'$,
            \item $\scr{D}'$ is a $G_0'$-pristine set of discs,
            \item $G_0$ is a spanning subgraph of $G_0'$,
            \item $\FT(G_0)\subseteq \FT(G_0')$, and
            \item $U(G_0,\scr{D})$ is a spanning subgraph of $U(G_0',\scr{D}')$.
        \end{enumerate}
        
        So $B(G_0',\scr{D}')=B(G_0,\scr{D})=V(G_0\cap H)$. By \cref{spanningSmooth}, $\scr{J}$ is a planted $U(G_0',\scr{D}')$-decomposition of $H$. Thus, $\Gamma'$ is an almost-embedding. Further, since $\Gamma'$ has the same decomposition and apices, the vortex-width and apex-count of $\Gamma'$ is the vortex-width and apex-count of $\Gamma$ respectively.
        
        Observe that $\FT(\Gamma)=\FT(G_0)\subseteq \FT(G_0')=\FT(\Gamma')$. Thus, $\FT(\Gamma)\subseteq \FT(\Gamma')$. Since $G_0'$ is a spanning supergraph of $G_0$, $G'-A=G_0'\cup H$ is a spanning supergraph of $G-A=G_0\cup H$. Thus and since $E(G)\setminus E(G-A)=E(G')\setminus E(G'-A)$, observe that $G$ is a spanning subgraph of $G'$. This completes the proof.
    \end{proof}

    We use the fact that $\Gamma'$ is an almost-embedding implicitly from now on.

    The main result of this subsection is the following.

    \begin{restatable}{lemma}{findStandard}
        \label{findStandard}
        Let $g,p,k,a\in \ds{N}$, and let $\Gamma$ be a nontrivial $(g,p,k,a)$-almost-embedding with disjoint discs. Then there exists a standardised $(g+2p,p,k,a)$-almost-embedding $\Gamma'$ with disjoint discs that extends $\Gamma$.
    \end{restatable}

    \cref{findStandard} is really a combination of three lemmas. The first add handles so that the discs are localised, and the second adds edges so that the embedded graph contains the underlying cycles, and the third makes the embedded graph $2$-cell embedded.

    We start with the second step, as it is the easiest (and is actually a repeat of part of the proof of \cref{nonCrossing}). We need the following observation.

    \begin{restatable}{observation}{underlyingExact}
        \label{underlyingExact}
        Let $G$ be an embedded graph, let $G'$ be a spanning embedded supergraph of $G$, and let $\scr{D}$ be a set of pairwise disjoint $G'$-clean discs. Then each disc in $\scr{D}$ is $G$-clean, $B(\scr{D},G)=B(\scr{D},G')$, and $U(\scr{D},G)=U(\scr{D},G')$. Further, if $\scr{D}$ is $G$-localised, then $\scr{D}$ is also $G'$-localised.
    \end{restatable}
    \begin{proof}
        Since $G$ is an embedded subgraph of $G'$, for each $D\in \scr{D}$, we have $G\cap D=(G'\cap D)\cap G=(V(G')\cap D)\cap G$. Since $G$ is spanning embedded subgraph, we have $V(G')\cap G=V(G')$. Thus, $G\cap D=V(G')\cap D=B(D,G')\subseteq V(G')=V(G)$. So $D$ is $G$-clean, and $B(D,G)=B(D,G')$. Further, we have $G\cap D=B(D,G')=G'\cap D$, which implies $\partial D\setminus G=\partial D\setminus G'$. So the arcs of $D$ in $G$ are the arcs of $D$ in $G'$. It follows that $U(\scr{D},G)=U(\scr{D},G')$.
        
        If $\scr{D}$ is $G$-localised, then there exists $r\in V(G)$ such that $\scr{D}$ is $N_G[r]$-anchored. So each $D\in \scr{D}$ intersects $N_G[r]$, and for each pair of distinct $D,D'\in \scr{D}$, $D\cap D'\subseteq N_G[r]$. Since $G$ is a subgraph of $G'$, we find that $N_G[r]\subseteq N_{G'}[r]$. Thus, each $D\in \scr{D}$ intersects $N_G[r]\subseteq N_{G'}[r]$, and for each pair of distinct $D,D'\in \scr{D}$, $D\cap D'\subseteq N_G[r]\subseteq N_{G'}[r]$. So $\scr{D}$ is $N_G[r]$-anchored, and $\scr{D}$ is $G'$-localised, as desired.
    \end{proof}

    \begin{lemma}
        \label{addUnderlying}
        Let $G$ be an embedded graph, and let $\scr{D}$ be a finite set of pairwise disjoint $G$-clean discs. Then there exists a spanning embedded supergraph $G'$ of $G$ such that:
        \begin{enumerate}
            \item $\FT(G)\subseteq \FT(G')$,
            \item for each $D\in \scr{D}$, $D$ is $G'$-clean, and,
            \item $U(\scr{D},G)\subseteq G'$.
        \end{enumerate}
    \end{lemma}

    \begin{proof}
        Let $B:=B(\scr{D},G)$ and $U:=U(\scr{D},G)$ (which are well-defined by \cref{disjointStronglyClean}). Since the discs in $\scr{D}$ are pairwise disjoint and $G$-clean, for each $D\in \scr{D}$, we can embed (non-parallel and non-loop) edges along the boundary of $D$ between consecutive vertices in $U(D,G)$ that are non-adjacent in $G$. Since $\scr{D}$ is a finite set of pairwise disjoint discs, this can be done in a way that is disjoint to each disc in $\scr{D}$, the embedding of $G$, and any other edges embedded in this manner.
        
        Observe that the resulting graph $G'$ is a spanning embedded supergraph of $G$, each disc in $\scr{D}$ is $G'$-clean, and $U(\scr{D},G')$ (which is well-defined, by \cref{disjointStronglyClean}) satisfies $U(\scr{D},G')\subseteq G'$ (ignoring edge labels). By \cref{underlyingExact}, $B(\scr{D},G')=B$, and $U(\scr{D},G')=U$. So $U\subseteq G'$. Since no loop-edges were added, by \cref{facialSpanningWeak}, $\FT(G)\subseteq \FT(G')$. This completes the proof.
    \end{proof}

    Next, we give the proof of the first step, adding handles.

    \begin{lemma}
        \label{addHandles}
        Let $G$ be a nonempty graph embedded in a surface $\Sigma$ of genus at most $g\in \ds{N}$, and let $\scr{D}$ be a $G$-pristine set of at most $p\in \ds{N}$ pairwise disjoint discs (in $\Sigma$). Then there exists a graph $G'$ embedded in a surface of genus at most $g+2p$ and a set $\scr{D}'$ of at most $p$ pairwise disjoint $G'$-clean discs such that:
        \begin{enumerate}
            \item $G$ is a spanning subgraph of $G'$,
            \item $\FT(G)\subseteq \FT(G')$,
            \item $B(\scr{D}',G')=B(\scr{D},G)$ and $U(\scr{D}',G')=U(\scr{D},G)$, and,
            \item $\scr{D}'$ is $G'$-localised.
        \end{enumerate}
    \end{lemma}

    \begin{proof}
        Let $\scr{D}'$ be the set of discs in $\scr{D}$ that intersect $G$. So $|\scr{D}'|\leq |\scr{D}|\leq p$. Observe that $B:=B(\scr{D}',G)=B(\scr{D},G)$, and that $U:=U(\scr{D}',G)=U(\scr{D},G)$.
        
        If $\scr{D}'$ is empty, fix any $r\in V(G)$ (which is possible, as $G$ is nonempty). $\scr{D}'$ is trivially $N_G[r]$-anchored, and thus $\scr{D}'$ is $G$-localised. So the lemma is true with $G':=G$. Thus, we may assume that $s:=|\scr{D}'|\geq 1$.

        Let $\{D_1,\dots,D_s\}:=\scr{D}'$. Observe that $s\leq |\scr{D}|\leq p$. For each $i\in \{1,\dots,s\}$, set $B_i:=B(D_i,G)$ and $U_i:=B(D_i,G)$. By choice of $\scr{D}'$, $|B_i|\geq 1$. Fix $v_i\in B_i$. Let $F_i$ be the (unique) face of $G$ that contains the interior of $D_i$.
        
        For each $i\in \{1,\dots,s\}$, if $F_i$ is not a triangle-face, set $F_i:=F_i'$. Otherwise, we can embed a new (possibly parallel, but not loop) edge inside $F_i$ between distinct vertices to create two new faces, one of which is a triangle-face and the other is an oval-face. Let $F_i'$ be the oval-face (which is not a triangle-face).
        
        These new edges can be embedded disjointedly to each disc in $\scr{D}'$ (as the discs in $\scr{D}$ are pairwise disjoint) and all other edges embedded in this manner. Thus, the resulting embedded graph $G^*$ is a spanning embedded supergraph of $G$ embedded in $\Sigma$, and each disc in $\scr{D}'$ is $G^*$-clean. Since we did not add any loops, by \cref{facialSpanningWeak}, $\FT(G)\subseteq \FT(G^*)$. By \cref{underlyingExact}, $B(\scr{D}',G^*)=B$, and $U(\scr{D}',G^*)=U$.

        For each $i\in \{2,\dots,s\}$, we can add a handle $H_i$ that connects $F_1'$ and $F_i'$. The handles $H_2,\dots,H_s$ can be added disjointedly (within $F_1'$). Let $\Sigma'$ be the resulting surface. Observe that $g(\Sigma')=g(\Sigma)+2\max(s-1,0)\leq g+2p$.
        
        Since $F_1',\dots,F_s'$ are faces of $G^*$, observe that $G^*$ is also a graph embedded in $\Sigma'$. To avoid confusion, we use $G^+$ to refer to the embedding of $G^*$ in $\Sigma'$. Since none of $F_1',\dots,F_s'$ are triangle-faces of $G^*$, observe that each triangle-face of $G^*$ is also a triangle-face of $G^+$. Thus, observe that $\FT(G^*)\subseteq \FT(G^+)$. So $\FT(G)\subseteq \FT(G^+)$.

        Since $F_1',\dots,F_s'$ are disjoint to each disc in $\scr{D}'$, observe that $\scr{D}'$ is a set $G_0^+$-clean discs in $\Sigma'$. Further, observe that $B(\scr{D}',G^+)=B(\scr{D}',G^*)=B$ and $U(\scr{D}',G^+)=U(\scr{D}',G^*)=U$.

        For each $i\in \{2,\dots,s\}$, if $v_1\neq v_i$, we can embed a (possibly parallel, but not loop) edge from $v_1$ to $v_i$ along $H_i$. Observe that these edges can be embedded disjointedly to each other, the embedding of $G^+$, and each disc $D\in \scr{D}'$. Thus, the resulting graph $G'$ is an embedded spanning supergraph of $G^+$, and each disc in $\scr{D}'$ is $G^+$-clean.
        
        By \cref{underlyingExact}, $B(\scr{D}',G')=B(\scr{D}',G^+)=B$ and $U(\scr{D}',G')=U(\scr{D}',G^+)=U$. Since we did not add any loops, by \cref{facialSpanningWeak}, $\FT(G^+)\subseteq \FT(G')$. So $\FT(G)\subseteq \FT(G')$. Further, since $G$ is a spanning subgraph of $G^*$, which is $G^+$, which is a spanning subgraph of $G'$, $G$ is a spanning subgraph of $G'$. Thus, to complete the proof, it remains only to show that $\scr{D}'$ is $G'$-localised.
        
        By construction, observe that for each $i\in \{1,\dots,s\}$, $v_i\in N_{G'}[v_1]$. By two application of \cref{underlyingExact} and since $G^+=G^*$, observe that $B_i=B(D_i,G)=B(D_i,G^*)=B(D_i,G^+)=B(D_i,G')$. Thus, $v_i\in B(D_i,G')=B_i$. So $N_{G'}[v_1]$ intersects $B(D,G')$ for each $D\in \scr{D}'=\{D_1,\dots,D_s\}$. As the discs in $\scr{D}'$ are pairwise disjoint, for each pair of distinct $D,D'\in \scr{D}'$, $D\cap D'=\emptyset\subseteq N_{G'}[v_1]$. Thus, $\scr{D}'$ is $N_{G'}[v_1]$-anchored in $G'$. Hence, $\scr{D}'$ is $G'$-localised. This completes the proof.
    \end{proof}

    As for the final step, making the embedding $2$-cell, we delay the proof. This is because the proof uses an algorithm common to several other lemmas in this paper, so we prefer to group them all under a common lemma. However, describing and proving this lemma requires a decent detour, so in the interest of staying planted on the current proof, we delay that proof until the end of this section (see \cref{SecEdges}).

    For now, we just give the statement of the lemma that allows us to find a $2$-cell embedding.

    \begin{restatable}{lemma}{makeTwoCell}
        \label{makeTwoCell}
        Let $G$ be a nonempty embedded graph, and let $\scr{D}$ be a finite set of pairwise disjoint $G$-clean discs. Then there exists a spanning $2$-cell embedded supergraph $G'$ of $G$ with $\FT(G)\subseteq \FT(G')$ such that each $D\in \scr{D}$ is $G'$-clean.
    \end{restatable}

    We remark that the requirement that $G$ is nonempty is necessary. This is why we needed the embedding to be nontrivial.

    We can now prove \cref{findStandard}, which we recall here.

    \findStandard*

    \begin{proof}
        Let $(G,\Sigma,G_0,\scr{D},H,\scr{J},A):=\Gamma$. By \cref{embeddedNonempty}, $G_0$ is nonempty. Let $B:=B(\scr{D},G_0)=V(G_0\cap H)$ and $U:=U(\scr{D},G_0)$.
        
        By \cref{addHandles} (as $G_0$ is nonempty), there exists a graph $G_0'$ embedded in a surface $\Sigma'$ of genus at most $g(\Sigma)+2p\leq g+2p$ and a set $\scr{D}'$ of at most $p<\infty$ pairwise disjoint $G_0'$-clean discs such that:
        \begin{enumerate}
            \item $G_0$ is a spanning subgraph of $G_0'$,
            \item $\FT(G_0)\subseteq \FT(G_0')$,
            \item $B=B(\scr{D}',G_0')$ and $U=U(\scr{D}',G_0')$, and,
            \item $\scr{D}'$ is $G_0'$-localised.
        \end{enumerate}

        By \cref{addUnderlying} (with $\scr{D}:=\scr{D}'$), there exists a spanning embedded supergraph $G_0''$ of $G_0'$ such that $U=U(\scr{D}',G_0')\subseteq G_0''$ and $\FT(G_0')\subseteq \FT(G_0'')$. By \cref{makeTwoCell} (as $G_0''\supseteq G_0'\supseteq G_0$ is nonempty), there exists a spanning $2$-cell embedded supergraph $G_0'''$ of $G_0''$ such that $\FT(G_0'')\subseteq \FT(G_0''')$ and each $D\in \scr{D}'$ is $G_0'''$-clean. So $G_0'''$ is a ($2$-cell) spanning embedded supergraph of $G_0'$, a spanning supergraph of $G_0$, and $\FT(G_0)\subseteq \FT(G_0''')$. Note also that $\scr{D}'$ is $G_0'''$-pristine by \cref{disjointStronglyClean}.
        
        By \cref{underlyingExact}, $U(\scr{D}',G_0''')=U(\scr{D}',G_0')=U$, and $\scr{D}'$ is $G_0'''$-localised. So $U(\scr{D}',G_0''')$ is a spanning supergraph of (and equal to) $U$. Note also that $U(\scr{D}',G_0''')=U\subseteq G_0''\subseteq G_0'''$.

        Since $G_0$ is a spanning subgraph of $G_0'''$, note that $V(G_0'''\cup H)=V(G_0\cup H)$. Since $G-A=G_0\cup H$, it follows that $A$ is disjoint to $G_0'''\cup H$, and that for each edge $e$ in $E(G)\setminus E(G-A)$, $e$ is not contained in $E(G_0'''\cup H)$ but has both endpoints in $V(G_0'''\cup H)\cup V(A)$.

        Let $G'$ be the graph obtained from $G_0'''\cup H$ by adding the vertices in $A$ and the edges in $E(G)\setminus E(G-A)$. By the above paragraph, $G'$ is well-defined. Further, $G'-A=G_0'''\cup H$, and $E(G)\setminus E(G-A)=E(G')\setminus E(G'-A)$.
        
        Let $\Gamma':=(G',\Sigma',G_0''',\scr{D}',H,\scr{J},A)$. By the definitions of $G',\Sigma,G_0',G_0'''$ and $\scr{D}'$, $\Gamma'$ extends $\Gamma$. By \cref{extendsIsSpanning}, $\Gamma'$ is an almost-embedding of vortex-width at most $k$ and apex-count at most $a$. In particular, $(H,U(\scr{D}',G_0'''),\scr{J})$ is a graph-decomposition.
        
        Since the discs in $\scr{D}'$ are pairwise disjoint, $\Gamma'$ has disjoint discs. Recalling that $g(\Sigma')\leq g+2p$, and that $|\scr{D}'|\leq p$, we find that $\Gamma'$ is a $(g+2p,p,k,a)$-almost-embedding. Further, since $G_0'''$ is $2$-cell embedded, since $U(\scr{D}',G_0''')\subseteq G_0'''$, and since $\scr{D}'$ is $G_0'''$-localised, $\Gamma'$ is standardised. This completes the proof.
    \end{proof}

    \subsection{Finding a cutting subgraph}

    The next goal is to reduce the genus of the standardised almost-embedding $\Gamma$. As mentioned before, \citet{rtwltwMCC} found a method to `cut' open a specific subgraph (the `cutting subgraph') of an embedded graph to obtain a planar graph (which was then used in \citet{Distel2022Surfaces}). Since this idea is buried inside their proofs, we have to formally isolate their idea and repeat the proofs. We do this in the appendix. Since it is relevant to us, we also take the chance to analysis how clean discs and facial triangles in the embedded graph can be converted to clean discs and facial triangles in the plane graph.

    We identify two key stages of their proof. (1) find the `cutting subgraph' and (2) `cut' along the cutting subgraph. Both of these stages have important extra properties that they can provide that we need to utilise. As such, we prefer to split the proof into two separate lemmas, which we now state.

    A \defn{rooted spanning tree} for a graph $G$ is a rooted tree $(T,r)$ such that $T$ is a spanning tree of $G$. A \defn{vertical path} $P$ in a rooted spanning tree is any path that is contained in the unique path from $r$ to some $v\in V(T)=V(G)$. So any pair of distinct vertices in $P$ are related.

    \begin{restatable}{lemma}{findCutting}
        \label{findCutting}
        Let $G$ be a graph $2$-cell embedded in a surface of genus $g\in \ds{N}$, let $T$ be a rooted spanning tree for $G$, and let $T'$ be a nonempty subtree of $T$. Then there exists a nonempty embedded subgraph $M$ of $G$ such that:
        \begin{enumerate}
            \item $M$ is the union of $T'$, at most $2g$ vertical paths in $T$, and exactly $g$ edges in $E(G)\setminus E(T)$,
            \item $|E(M)|=|V(M)|+g-1$, and
            \item $M$ has exactly one face.
        \end{enumerate}
    \end{restatable}

    \begin{restatable}{lemma}{cuttingLemmaSimple}
        \label{cuttingLemmaSimple}
        Let $G$ be a graph $2$-cell embedded in a surface of genus $g\in \ds{N}$, let $M$ be a nonempty embedded subgraph of $G$ with $|E(M)|=|V(M)|+g-1$ and exactly one face, and let $\scr{D}$ be a $V(M)$-anchored set of $G$-clean discs. Then there exists a $2$-cell embedded plane graph $G'$, a map $\phi:V(G')\mapsto V(G)$, a connected and nonempty plane subgraph $W$ of $G'$, a $V(W)$-anchored set $\scr{D}'$ of $G'$-clean discs, and a bijection $\beta:\scr{D}'\mapsto \scr{D}$ such that:
        \begin{enumerate}
            \item $G\cap G'=G-V(M)=G'-V(W)$ (as non-embedded graphs),
            \item $\phi$ is the identity on $V(G')\setminus V(W)$,
            \item $\phi(V(W))=V(M)$,
            \item for each $v\in V(G)$, $N_G[v]\subseteq \phi(N_{G'}[\phi^{-1}(v)])$,
            \item for each $K\in \FT(G)$, there exists $K'\in \FT(G')$ with $\phi(K')=K$, and,
            \item for each $D'\in \scr{D}'$, the restriction of $\phi$ to $B(D',G')$ is an isomorphism from $U(D',G')$ to $U(\beta(D'),G)$.
        \end{enumerate}
    \end{restatable}

    Since \cref{findCutting} and \cref{cuttingLemmaSimple} use essentially the same techniques as in \citet{rtwltwMCC} and \citet{Distel2022Surfaces} (albeit with a few extra details to maintain some extra properties), we leave the proofs of these results to the appendix. Collectively, proving these lemmas is the subject of \cref{SecCutting}. We remark that there are even more properties of the plane graph $G'$ in \cref{cuttingLemmaSimple} that we haven't stated here, as we have filtered out any unnecessary details (given that \cref{cuttingLemmaSimple} is already quite long). The full set of properties can be found in \cref{SecCutting} (see \cref{cuttingLemmaTrue}).

    We informally refer to the embedded graph $M$ in \cref{findCutting} and \cref{cuttingLemmaSimple} as the \defn{cutting subgraph}, and the embedded graph $W$ in \cref{cuttingLemmaSimple} as the \defn{plane cut-subgraph}. As a guide to the reader, when we get to defining reductions, $V_M$ will be $V(M)$, and $W$ and $G'$ from \cref{cuttingLemmaSimple} will be $W$ and $G_0^+$ in the definition of reductions respectively.

    The goal for this subsection is to find a cutting subgraph with `desirable' properties, which will require modifying (extending) the almost-embedding slightly. We leave `performing the cut' to the next subsection.

    The first thing we need to decide upon is the spanning tree to give \cref{findCutting}, as this decides the cutting subgraph.
    Recall that we want a layering such that $V(M)=V_M$ has a bounded size intersection with each layer. A natural layering of the embedded graph $G_0$ is the BFS layering. This can be `converted' to a spanning tree (provided $G_0$ is connected).

    A \defn{breath-first search spanning tree} (\defn{BFS spanning tree}) rooted at $r\in V(G)$ is a rooted spanning tree $(T,r)$ for which $\dist_G(v,r)=\dist_T(v,r)$ for each $v\in V(G)$.
    
    If we use a BFS spanning tree in \cref{findCutting}, from the construction of $M$ we find that each layer of the BFS layering of $G_0$ intersects at most $2g$ vertices of $V(M)\setminus V(T')$, which is good. However, this BFS layering is not a layering of the graph $G-A$ (nor can it be easily converted to one), as if the boundary vertices stretch across many layers, the vortex could make vertices in faraway layers adjacent. 
    
    However, as mentioned earlier, if we start with a standard almost-embedding, in which the discs are localised, and add `phantom' edges within the disc to create a new embedded graph $G_0''$ in which the every boundary vortex is within distance $2$ of some $r\in V(G_0')$, then we can convert the BFS layering of $G_0'$ from $r$ to a layering of $G$. So the idea is to take a BFS spanning tree of $G_0'$ rooted at $r$ and feed that into \cref{findCutting}. We find a cutting subgraph $M$ for $G_0''$ such that each layer of the corresponding layering of $G$ only intersects $2g$ vertices of $V(M)$, which is exactly what we need for \cref{layeringSeperation}.

    There is one major catch here. We found a cutting subgraph for $G_0''$, not $G_0$. It is not clear (from the algorithm) how we can use this cutting subgraph to cut up $G_0$. We need to add, at very least, the edges in $E(M)$ to $G_0$, so that the algorithm can function and perform the cut. Fortunately, it suffices to only add the edges in $E(M)$, as it is easy to check that the conditions for a `cutting subgraph' are entirely dependent on $M$ and not $G_0'$, provided that $M$ is an embedded subgraph. So $M$ is still a cutting subgraph of the embedded graph $G_0':=G_0\cup M$.

    However, we recall that these new edges all exists within the discs. So the discs are not $G_0'$-clean. However, each of these edges split one of the discs into two discs. Combining these splits together, if we started with $p$ discs and added $m$ edges, we end up a $G_0'$-pristine set $\scr{D}'$ of $p+m$ discs. We also find that $U(\scr{D}',G_0')$ is a spanning supergraph of $U(\scr{D},G_0)$, so we can still attach $H$ as a vortex. This is the reason we allow the discs to intersect in the definition of an almost-embedding.
    
    As for how many edges we need to add, we find that each vertical path includes at most one `phantom edge'. We can pick $T$ so that each phantom edge is included in $T$. So from the construction of $M$, we find that $m\leq 2g$.

    So if $\Gamma$ was a standardised $(g,p,k,a)$-almost-embedding, we can use this method to find a $(g,p+2g,k,a)$-almost-embedding that extends $\Gamma$ and has a `cutting subgraph' for the embedded subgraph.

    We would like for the discs to be $V(M)$ anchored. This is so that, post cut, the discs are still `close'. Any `split discs' will automatically satisfy this. But it is possible that we did not need to add any edges within a disc. Provided that we ensure $V(M)$ touches this disc, we are still safe. We can easily achieve this by taking the subtree $T'$ in \cref{findCutting} to be a tree of radius $1$ that intersects every disc (found using the definition of localised).

    We will also need to keep track of how many boundary vertices $V(M)$ intersects. We can pick $T'$ to have exactly $s:=|\scr{D}|$ intersections with the discs and contain the root. We can then pick $T$ so that each vertical path adds at most $2$ new intersections with the discs (noting that the root is already accounted for). From the construction of $M$, we then find that $M$ intersects at most $4g+s$ boundary vertices.

    We can also pick $T'$ so that it has at most $s+1$ vertices, and hence contributes at most $s+1$ vertices to the per-layer intersection. Combined with the one intersection per-layer from each of the $2g$ vertical paths, we find that $V(M)$ intersects at most $2g+s+1$ vertices per layer. This stays consistent even after updating the layering using \cref{layeringEmbedToG} (which only moves the root to a new layer).

    We now introduce some definitions to formally describe the ideas we have discussed thus far.

    We say that an embedded graph $M$ is a \defn{cutting subgraph} for an almost-embedding $\Gamma=(G,\Sigma,G_0,\linebreak\scr{D},H,\scr{J},A)$ if:

    \begin{enumerate}
        \item $M$ is a nonempty embedded subgraph of $G_0$,
        \item $|E(M)|=|V(M)|+g(\Sigma)-1$, and,
        \item $M$ has exactly one face.
    \end{enumerate}

    We say that the \defn{boundary intersection} of $M$ is $|V(M)\cap B(\Gamma)|$. We say that $M$ is \defn{anchoring} if $\scr{D}$ is $V(M)$-anchored.

    Now, recall that we will need to modify (extend) the almost-embedding to find the cutting subgraph. We need to preserve some `standard' properties while doing so. However, the discs may no longer be localised (because they have new intersection points). Fortunately, being standardised is stronger than what we need. Instead, we only need the embedding to be $2$-cell (which is needed to perform the cut), and to preserve `smoothness' (to ensure that the plane+quasi-vortex embedding we produce later is standardised). This leads to the following definition.

    Say that an almost-embedding $\Gamma=(G,\Sigma,G_0,\scr{D},H,\scr{J},A)$ is \defn{weakly-standardised} if all the following hold:
    \begin{enumerate}
        \item $G_0$ is $2$-cell embedded in $\Sigma$, and
        \item $(H,U(\scr{D},G_0),\scr{J})$ is smooth in $G_0$.
    \end{enumerate}

    The main result of this section is then as follows.

    \begin{restatable}{lemma}{findAECutting}
        \label{findAECutting}
        Let $g,p,k,a\in \ds{N}$, and let $\Gamma$ be a standardised $(g,p,k,a)$-almost-embedding with disjoint discs. Then there exists a weakly-standardised $(g,p+2g,k,a)$-almost-embedding $\Gamma'$, an anchoring cutting subgraph $M$ of $\Gamma'$, and a layering $\scr{L}$ of the graph $G'$ of $\Gamma'$ such that:
        \begin{enumerate}
            \item $\Gamma'$ extends $\Gamma$,
            \item each layer of $\scr{L}$ contains at most $2g+p+1$ vertices of $M$, and,
            \item $M$ has boundary intersection at most $4g+p$.
        \end{enumerate}
    \end{restatable}

    Before we can prove \cref{findAECutting}, we first need to prove that we can `split' the discs as described earlier. For this, we have the following lemma.

    \begin{lemma}
        \label{splitDiscs}
        Let $G$ be an embedded graph, let $\scr{D}$ be a set of pairwise disjoint $G$-clean discs, let $G'$ be an embedded spanning supergraph of $G$, and let $S\subseteq V(G')$ be such that:
        \begin{enumerate}
            \item each $e\in E(G')\setminus E(G)$ is contained in the interior of some $D\in \scr{D}$ and has both endpoints in $S$, and,
            \item each $D\in \scr{D}$ intersects $S$,
        \end{enumerate}
        Then there exists an $S$-anchored $G'$-pristine set of discs $\scr{D}'$ with $|\scr{D}'|=|\scr{D}|+|E(G')\setminus E(G)|$ such that $B(\scr{D}',G')=B(\scr{D},G)$ and $U(\scr{D}',G')$ is a spanning supergraph of $U(\scr{D},G)$.
    \end{lemma}

    \begin{proof}
        For each $e\in E(G')\setminus E(G)$, let $D_e\in \scr{D}$ be such that $e$ is contained in the interior of $D_e$. Since $D_e$ is $G$-clean, and since $G'$ is an embedded spanning supergraph of $G$, observe that the endpoints of $e$ lie on $\partial D_e$. Thus, we find that $e$ splits $D_e$ into two discs whose union is $D_e$ and whose intersection is $e$ and the endpoints of $e$.

        For each $D\in \scr{D}$, let $E_D:=\{e\in E(G')\setminus E(G):D_e=D\}$. Let $V_D$ be the set of vertices that are an endpoint of some $e\in E_D$. For each $e\in E_D$, recall that the endpoints of $e$ lie in $\partial D_e=\partial D$. So $V_D\subseteq \partial D$. Further, recall that each $e\in E_D\subseteq E(G')\setminus E(G)$ has both endpoints in $S$. Thus, $V_D\subseteq S$.
        
        Since the discs of $\scr{D}$ are pairwise disjoint, observe that $E(G')\setminus E(G)=\bigsqcup_{D\in \scr{D}}E_D$.

        Fix $D\in \scr{D}$. Since the discs of $\scr{D}$ are pairwise disjoint and since each $e\in E(G')\setminus E(G)$ is contained in a disc in $\scr{D}$, observe that $D$ is disjoint to each $e\in E(G')\setminus (E(G)\cup E_D)$. Since $D$ is $G$-clean and since $G'$ is an embedded spanning supergraph of $G$, we find that $D\cap G'=B(D,G)\cup \bigcup_{e\in E_D}e$. Further, we find that $(D\setminus \partial D)\cap G'=\bigcup_{e\in E_D}e$, and $\partial D\cap G'=B(D,G)$.
        
        Recall that each $e\in E_D$ splits $D$ into two discs. Following the same logic and since the embeddings of edges are disjoint, observe that $E_D$ splits $D_e$ into a set $\scr{D}_D''$ of exactly $|E_D|+1$ discs whose union is $D$, and such that the union of all pairwise intersections of distinct discs is $V_D\cup \bigcup_{e\in E_D}e$. So $\bigcup_{D''\in \scr{D}_D''}D''=D$ and $\bigcup_{D',D''\in \scr{D}_D'',D'\neq D''}D'\cap D''=V_D\cup \bigcup_{e\in E_D}e$. In particular, if $E_D=\emptyset$, then $\scr{D}_D''=\{D\}$. 
        
        Since $V_D\cup \bigcup_{e\in E_D}e$ does not contain any open ball of nonzero radius, we find that the discs in $\scr{D}''_D$ are pairwise internally disjoint. Further, if $|E_D|\geq 1$, we find that each $D''\in \scr{D}_D''$ is properly contained in $D$.

        Since the discs in $\scr{D}''_D$ are pairwise internally disjoint and since $\bigcup_{D''\in \scr{D}_D''}D''=D$, observe that $\bigcup_{D''\in \scr{D}_D''}\partial D''=\partial D\cup \bigcup_{D',D''\in \scr{D}_D'',D'\neq D''}D'\cap D''=\partial D\cup V_D\cup \bigcup_{e\in E_D}e$. Recall that $V_D\subseteq \partial D$. Thus, $\bigcup_{D''\in \scr{D}_D''}\partial D''\subseteq \partial D\cup \bigcup_{e\in E_D}e$.
        
        If $|E_D|\geq 1$, then recall that each $D''\in \scr{D}_D''$ is properly contained in $D$. Thus, $\partial D''\neq \partial D$. Hence, $\partial D''$ intersects some $e\in \bigcup_{e\in E_D}e$. Since the embedding of edges do not intersect or contain a closed curve, $\partial D''$ intersects the endpoints of $e$, and thus $V_D\subseteq S$. Otherwise, if $|E_D|=0$, then $\scr{D}_D''=\{D\}$. Since $D$ intersects $S$ and since $D$ is $G$-clean, $\partial D$ intersects $S$. Thus, in either case, observe that for each $D''\in \scr{D}_D''$, $\partial D''$ intersects $S$.
        
        Presume, for a contradiction, that for some $D''\in \scr{D}''_D$, the interior of $D''$ intersects $G'$. Since $\bigcup_{D'\in \scr{D}_D''}D'=D$ and since $(D\setminus \partial D)\cap G'=\bigcup_{e\in E_D}e$, the interior of $D''$ contains some $x\in e\in E_D$. Since the discs in $\scr{D}''_D$ are pairwise internally disjoint, no other disc in $\scr{D}''_D$ contains $x$. But $x\in \bigcup_{e\in E_D}e\subseteq \bigcup_{D',D'''\in \scr{D}_D'',D'\neq D'''}D'\cap D'''$, a contradiction. So the discs of $\scr{D}''_D$ are internally disjoint from $G'$.
        
        Since the discs of $\scr{D}''_D$ are internally disjoint from $G'$, we can shrink each $D''\in \scr{D}_D''$ slightly to obtain a $G'$-clean disc $D'=D'_{D''}$ contained in $D''$ (and thus $D$) with $\partial D'\cap G'=\partial D''\cap V(G)$. Note that we can perform this shrink such that $\partial D''\setminus (G\setminus V(G))\subseteq \partial D'$. Thus and since $D''$ intersects $S\subseteq V(G)$, $D'$ intersects $S$.
        
        Fix $D'''\in \scr{D}''_D\setminus \{D''\}$. Since $D'\subseteq D''$, we find $D'\cap D'''\subseteq D''\cap D'''\subseteq V_D\cup \bigcup_{e\in E_D}e\subseteq G'$. Thus, $D'\cap D'''\subseteq D'\cap G'$. Since $D'$ is $G'$-clean, $D'\cap G'=\partial D'\cap V(G')$. Thus, $D'\cap D'''\subseteq \partial D'\cap V(G')\subseteq V(G')$. Hence, $D'\cap D'''=V(G')\cap (D'\cap D''')\subseteq V(G)\cap (D''\cap D''')$. Recall that $D''\cap D'''\subseteq V_D\cup \bigcup_{e\in E_D}e$, and thus $V(G')\cap (D''\cap D''')\subseteq V_D\subseteq S$. Thus, $D'\cap D'''\subseteq S$.
        
        Let $\scr{D}_D':=\{D'_{D''}:D''\in \scr{D}''_D\}$. So $\scr{D}_D'$ is a set of $G'$-clean discs with $|\scr{D}_D'|=|\scr{D}_D''|$. Further, for any pair of distinct $D'_{D''},D'_{D'''}\in \scr{D}_D'$ (so $D'',D'''$ are distinct), since $D'_{D'''}\subseteq D'''$, we have $D'_{D''}\cap D'_{D'''}\subseteq D_{D''}'\cap D'''\subseteq S\subseteq V(G)$. In particular and since each $D'\in \scr{D}_D'$ is $G'$-clean, $\scr{D}_D'$ is $G'$-pristine.
        
        Observe that $\bigcup_{D'\in \scr{D}_D'}\partial D'\cap G=\bigcup_{D''\in \scr{D}_D'}\partial D''\cap V(G)=(B(D,G)\cup \bigcup_{e\in E_D}e)\cap V(G)=B(D,G)$. Thus, $B(\scr{D}_D',G')=B(D,G)$. As $\bigcup_{D''\in \scr{D}_D''}\partial D''=D$ and since $\partial D\cap G'\subseteq V(G')$, we find that $\partial D\subseteq \bigcup_{D''\in \scr{D}_D''}\partial D''\setminus (G\setminus V(G))\subseteq \bigcup_{D'\in \scr{D}_D'} \partial D'$. Thus, $\partial D\setminus V(G)=\partial D\setminus V(G')\subseteq \bigcup_{D'\in \scr{D}_D'} \partial D'\setminus V(G')$.
        
        Hence, for each arc $C$ of $D$ in $G$ (connected component of $\partial D\setminus V(G)$), $C\subseteq \bigcup_{D'\in \scr{D}_D'} \partial D'\setminus V(G')$. Since $\scr{D}'_D$ is $G'$-pristine, no point in $C$ is contained in $(\partial D'\setminus V(G'))\cap (\partial D'''\setminus V(G'))$ for distinct $D',D'''\in \scr{D}'_D$. Thus and since $\partial D'$ is closed for each $D'\in \scr{D}_D'$ (and since $|\scr{D}_D'|=|E_D|+1<\infty$ as $G'$ is finite), we find that $C$ intersects exactly one $D'\in \scr{D}'_D$. Thus, $C\subseteq \partial D'\setminus V(G')$. Since $\partial D'$ is closed, both endpoints of $C$ are contained in $\partial D'$, and thus $C$ is a component of $\partial D'\setminus V(G')$. Hence, $C$ is an arc of $D'$ in $G'$, and the endpoints of $C$ are adjacent in $U(D',G')\subseteq U(\scr{D}_D',G')$. It follows that $U(D,G)$ is a spanning subgraph of $U(\scr{D}_D',G')$.

        Recall that each $D'\in \scr{D}_D'$ intersects $S$, and that $D'\cap D'''\subseteq S$ for distinct $D',D'''\in \scr{D}_D'$. Thus, $\scr{D}_D'$ is $S$-anchored.

        Let $\scr{D}':=\bigsqcup_{D\in \scr{D}}\scr{D}_D'$. Recall that for each $D\in \scr{D}$, $\scr{D}_D'$ is $G'$-pristine, $S$-anchored, and each disc in $\scr{D}_D'$ is contained in $D$. Thus and since the discs in $\scr{D}$ are pairwise disjoint, observe that $\scr{D}'$ is $G'$-pristine and $S$-anchored. Recall that for each $D\in \scr{D}$, $B(\scr{D}_D',G')=B(D,G)$ and $U(D,G)$ is a spanning subgraph of $U(\scr{D}_D',G')$. Thus, observe that $B(\scr{D}_D,G')=B(D,G)$, and $U(\scr{D},G)$ is a spanning subgraph of $U(\scr{D}',G')$
        
        Finally, observe that $|\scr{D}'|=\sum_{D\in \scr{D}}|\scr{D}'_D|=\sum_{D\in \scr{D}}|E_D|+1=|\scr{D}|+|E(G')\setminus E(G)|$ (as $E(G')\setminus E(G)=\bigsqcup_{D\in \scr{D}}E_D$). This completes the proof.
    \end{proof}

    We also need the following basic observations.

    \begin{observation}
        \label{anchoredSuperset}
        Let $G$ be an embedded graph, let $S,S'\subseteq V(G)$ with $S'\subseteq S$, and let $\scr{D}$ be an $S'$-anchored and $G$-pristine set of discs. Then $\scr{D}$ is $S$-anchored.
    \end{observation}
    \begin{proof}
        For each $D\in \scr{D}$, $B(D,G)$ intersects $S'\subseteq S$, so $B(D,G)$ intersects $S$. For each pair of distinct $D,D'\in \scr{D}$, $D\cap D'\subseteq S'\subseteq S$. Thus, $\scr{D}$ is $S$-anchored, as desired.
    \end{proof}

    \begin{restatable}{observation}{spanning2Cell}
        \label{spanning2Cell}
        Let $G$ be a $2$-cell embedded graph, and let $G'$ be a spanning supergraph of $G$. Then $G'$ is $2$-cell embedded.
    \end{restatable}

    \begin{proof}
        Embedding an edge into a $2$-cell face splits the face into two new $2$-cell faces. The result follows from an easy inductive argument, embedding each $e\in E(G')\setminus E(G)$ in turn.
    \end{proof}

    We can now prove \cref{findAECutting}.

    \findAECutting*

    \begin{proof}
        Let $(G,\Sigma,G_0,\scr{D},H,\scr{J},A):=\Gamma$. Since $G$ is standardised, $G_0$ is $2$-cell embedded (in $\Sigma$). Since $\Gamma$ has disjoint discs, the discs in $\scr{D}$ are pairwise disjoint. Since $\Gamma$ has disc count at most $p$, $|\scr{D}|\leq p$. Set $g':=g(\Sigma)$. So $g'\leq g$. Set $B:=B(D,G_0)$ and $U:=U(D,G_0)$.
    
        Since $G$ is standardised, $\scr{D}$ is $G_0$-localised. So there exists $r\in V(G_0)$ such that $\scr{D}$ is $N_{G_0}[r]$-anchored. Hence, for each $D\in \scr{D}$, $B(D,G_0)$ intersects $N_{G_0}[r]$. Fix $x_D\in B\cap N_{G_0}[r]$, prioritising setting $x_D:=r$ if $r\in B$.

        For each $D\in \scr{D}$ and each $v\in B\setminus N_{G_0}[r]$, we can embed a new (possibly parallel) edge $E_v$ with endpoints $v,x_D$ contained in the interior of $D$. Call the resulting embedded graph $G_0''$. Observe that $G_0''$ is a spanning supergraph of $G_0$. Thus, by \cref{spanning2Cell}, $G_0''$ is $2$-cell embedded.

        For each $i\in \{1,2\}$, let $N_i:=\{v\in V(G_0''):\dist_{G_0''}(v,r)\leq i\}$. Note that $\{x_D:D\in \scr{D}\}\subseteq N_{G_0}[r]\subseteq N_1$, and thus $B\subseteq N_2$.
        
        Observe that $E(G_0'')\setminus E(G_0)=\{E_v:v\in B\setminus N_{G_0}[r]\}$. Thus, observe that each $e\in E(G_0'')\setminus E(G_0)$ has one endpoint in $N_1$ and the other in $N_2\setminus N_1$, both endpoints in $N_2$, and is contained in the interior of some $D\in \scr{D}$.

        Let $V_T':=\{r\}\cup \{x_D:D\in \scr{D}\}$. Observe that $G_0[V_T']$ is connected and has radius at most $1$, with every vertex being at distance at most $1$ from $r$. So we can find a subtree $T'$ of $G_0'$ of radius at most $1$, with every vertex being at distance at most $1$ from $r$, such that $V(T')=V_T'$. Note that $T'$ is nonempty as $r\in V_T'$, and that $V(T')\subseteq N_1$. Further, observe for each $v\in V(T')$, $\dist_G(v,r)=\dist_{T'}(v,r)\in \{0,1\}$. Finally, since the discs in $\scr{D}$ are pairwise disjoint and by choice of $\{x_D:D\in \scr{D}\}$ (in particular, since we chose $x_D=r$ if $D\in \scr{D}$ intersects $r$), observe that each $D\in \scr{D}$ intersects $V(T')$ exactly once, and that $B\cap V(T')=|\scr{D}|\leq p$.

        Since $G_0''$ is connected, $G_0''$ is $2$-cell embedded. Thus, we can find a BFS spanning tree $T$ for $G_0''$ rooted at $r$. Recall that $T'$ is a tree and that for each $v\in V(T')$, $\dist_G(v,r)=\dist_{T'}(v,r)$. Thus, we can pick $T$ such that $T'\subseteq T$. Further, since $V(T')\subseteq N_1$ and since each $e\in E(G_0'')\setminus E(G_0)$ has one endpoint in $N_1$ and the other in $N_2\setminus N_1$, we can further pick $T$ such that $E(G_0'')\setminus E(G_0)\subseteq E(T)$ (while maintaining that $T'\subseteq T$). 
        
        Since each $e\in E(G_0'')\setminus E(G_0)$ has one endpoint in $N_1$ and the other in $N_2\setminus N_1$, note that each vertical path in $T$ contains at most one edge in $E(G_0'')\setminus E(G_0)$.
        
        Recall that $T'$ is a nonempty subtree of $T$, that $G_0''$ is $2$-cell embedded (in $\Sigma$), and that $g'=g(\Sigma)$. By \cref{findCutting}, there exists a nonempty embedded subgraph $M$ of $G_0''$ such that:
        \begin{enumerate}
            \item $M$ is the union of $T'$, at most $2g'$ vertical paths in $T$, and exactly $g'$ edges in $E(G)\setminus E(T)$,
            \item $|E(M)|=|V(M)|+g'-1$, and
            \item $M$ has exactly one face.
        \end{enumerate}
        Since $E(G_0'')\setminus E(G_0)\subseteq E(T)$ and since each vertical path in $T$ contains at most one edge in $E(G_0'')\setminus E(G_0)$, observe that $M$ contains at most $2g'\leq 2g$ edges in $E(G_0'')\setminus E(G_0)$.

        Let $G_0'$ be the embedded graph $G_0\cup M$ (with the restriction of the embedding of $G_0''\supseteq G_0\cup M$). Note that $M$ is a nonempty embedded subgraph of $G_0'$. Since $M\subseteq G_0''$ and since $G_0''$ is a spanning embedded supergraph of $G_0$, $G_0'$ is a spanning embedded supergraph of $G_0$ and a spanning embedded subgraph of $G_0''$. Observe that $E(G_0')\setminus E(G_0)=(E(G_0'')\setminus E(G_0))\cap E(M)$.
        
        Since $M$ contains at most $2g$ edges in $E(G_0'')\setminus E(G_0)$, observe that $|E(G_0')\setminus E(G_0)|\leq 2g$. Recall that each $e\in E(G_0'')\setminus E(G_0)$ has both endpoints in $N_2$ and is contained in the interior of some $D\in \scr{D}$. Thus, each $e\in E(G_0')\setminus E(G_0)$ has both endpoints in $N_2\cap V(M)\subseteq V(G_0')$ and is contained in the interior of some $D\in \scr{D}$. As $T'\subseteq M$ and since $V(T')\subseteq N_1\subseteq N_2$, we find that $V(T')\subseteq N_2\cap V(M)$. Thus, since $V(T')$ intersects each $D\in \scr{D}$, so does $N_2\cap V(M)$. Finally, recall that the discs in $\scr{D}$ are pairwise disjoint, and that $|\scr{D}|\leq p$.
        
        Hence, by \cref{splitDiscs} (with $S:=N_2\cap V(M)$), there exists a $(N_2\cap V(M))$-anchored $G_0'$-pristine set of discs $\scr{D}'$ with $|\scr{D}'|=|\scr{D}|+|E(G')\setminus E(G)|\leq p+2g$ such that $B(\scr{D}',G_0')=B(\scr{D},G_0)=B$ and $U':=U(\scr{D}',G_0')$ is a spanning supergraph of $U(\scr{D},G_0)=U$.

        Presume, for a contradiction, that a loop-edge of $E(G_0')\setminus E(G_0)=E(M)\setminus E(G_0)$ is contained in a triangle-face of $G_0$. Since triangle-faces are $2$-cell, we find that this loop creates a loop-face in $M$, which in turn implies that $M$ has at least two faces. This contradicts the fact that $M$ has exactly one face. Thus, no loop-edge in $E(G_0')\setminus E(G_0)$ is contained in a triangle-face of $G_0$. Thus, by \cref{facialSpanningWeak}, $\FT(G_0)\subseteq \FT(G_0')$.

        Since $G_0$ is a spanning subgraph of $G_0'$, note that $V(G_0'\cup H)=V(G_0\cup H)$. Since $G-A=G_0\cup H$, it follows that $A$ is disjoint to $G_0'\cup H$, and that for each edge $e$ in $E(G)\setminus E(G-A)$, $e$ is not contained in $E(G_0'\cup H)$ but has both endpoints in $V(G_0'\cup H)\cup V(A)$. Let $G'$ be the graph obtained from $G_0'\cup H$ by adding the vertices in $A$ and the edges in $E(G)\setminus E(G-A)$. We find that $G'$ is well-defined. that $G'-A-G_0'\cup H$, and that $E(G)\setminus E(G-A)=E(G')\setminus E(G'-A)$.
        
        Let $\Gamma':=(G',\Sigma,G_0',\scr{D}',H,\scr{J},A)$. By the definitions of $G',G_0'$ and $\scr{D}'$, $\Gamma'$ extends $\Gamma$. By \cref{extendsIsSpanning}, $\Gamma'$ is an almost-embedding of vortex-width at most $k$ and apex-count at most $a$. By definition of $g'\leq g$, $\Gamma'$ has genus at most $g$. Since $|\scr{D}'|\leq p+2g$, $\Gamma'$ has disc-count at most $p+2g$. So $\Gamma$ is a $(g,p+2g,k,a)$-almost-embedding.
        
        Since $\Gamma$ is standardised, $U\subseteq G_0$. By \cref{subgraphSmooth}, $(H,U,\scr{J})$ is smooth in $G_0$. Thus, by \cref{spanningSmooth} $(H,U',\scr{J})$ is a graph-decomposition, and is smooth in $G_0'$. By \cref{spanning2Cell}, $G_0'$ is $2$-cell embedded. Thus, $\Gamma'$ is weakly-standardised. Since $M$ is a nonempty embedded subgraph of $G_0'$ (embedded in $\Sigma$) with $|E(M)|=|V(M)|+g'-1=|V(M)|+g(\Sigma)-1$ and only one face, $M$ is a cutting subgraph for $\Gamma'$. Since $V(M)\subseteq V(M)\cap V(T'')$, $\scr{D}'$ is $V(M)$-anchored by \cref{anchoredSuperset}. Thus, $M$ is anchoring.

        Recall that $V(T')\cap B=|\scr{D}|\leq p$. Since $r\in V(T')$ and since $B\subseteq N_2$, observe that for each vertical path $P$ of $T$, $|(V(P)\cap B)\setminus V(T')|\leq 2$. Observe that $V(M)$ is the union $V(T')$ and $V(P)$ for at a set of at most $2g'\leq 2g$ vertical paths $P$ in $T$. Thus, observe that $|V(M)\cap B|\leq 4g+p$. Since $B(\scr{D}',G_0')=B$, $M$ has boundary intersection at most $4g+p$.

        Let $\scr{L}'=(L_0',\dots,L_m')$ (for some $m\in \ds{N}$) be the BFS layering of $G_0''$ from $r$. By adding empty layers at the end, we may assume that $m\geq 2$. Observe that $L_0'=\{r\}$, $L_1'=N_1\setminus \{r\}$, and $L_2'=N_2\setminus N_1$.

        Since $G_0''$ is a spanning supergraph of $G_0'$, by \cref{layeringSupergraph}, $\scr{L}'$ is a layering of $G_0'$. Recall that $B\subseteq N_2=L_0'\cup L_1'\cup L_2'$. Thus, by \cref{layeringEmbedToG} (applied to $\Gamma'$), if $L_1:=L_0'\cup L_1'\cup (V(H)\setminus B(\Gamma'))$, $L_0:=\emptyset$, and $L_i:=L_i'$ for each $i\in \{2,\dots,m\}$, then $\scr{L}:=(\emptyset,L_1,L_2,\dots,L_m)$ is a layering of $G'-A$.
        
        Observe that $B(\Gamma')=V(G_0'\cap H)=V(G_0)\cap V(H)=B$ as $G_0'$ is a spanning supergraph of $G_0$. Thus and recalling that $L_0=\{r\}$, we find that $L_1=\{r\}\cup L_1'\cup (V(H)\setminus B)=\{r\}\cup L_1'\cup (V(H)\setminus V(G_0))$.

        Since $V(M)\subseteq V(G_0)$ and since $r\in V(T')$, observe that $V(M)\cap L_1\subseteq V(T')\cup L_1'$. Thus, observe that for each $i\in \{0,\dots,m\}$, $V(M)\cap L_i\subseteq V(T')\cup L_i'$.
        
        Since $T$ is a BFS tree rooted at $r$, for each vertical path $P$ in $T$ and each $i\in \{0,\dots,m\}$, $|L_i'\cap V(P)|\leq 1$. Recall that $V(M)$ is the union $V(T')$ and $V(P)$ for at a set of at most $2g'\leq 2g$ vertical paths $P$ in $T$. Thus, for each $i\in \{0,\dots,m\}$, recalling that $V(M)\cap L_i\subseteq V(T')\cup L_i'$, observe that $|V(M)\cap L_i|\leq |V(T')|+2g$. Since $V(T')=\{r\}\cup \{x_D:D\in \scr{D}\}$, observe that $|V(T')|\leq |\scr{D}|+1\leq p+1$. Thus, $|V(M)\cap L_i|\leq 2g+p+1$. This completes the proof.
    \end{proof}

    \subsection{Performing the cut}

    So now we actually have the perform the cut. This is mostly just a case of applying \cref{cuttingLemmaSimple}. However, one major problem occurs. If a boundary vertex lies on $M$, then that boundary vertex gets sent to multiple, possibly far-away, vertices in the plane graph. This destroys the decomposition of the vortex.

    The solution is to 'delete' any vertex of the vortex contained in the bag of a boundary vertex. We can then modify the original decomposition to be a decomposition of what's left of the vortex. There is one more technical hurdle. If a vertex on the boundary wasn't in $V(M)$, but was in a bag of a boundary vertex, we can't delete it because we need it to index a bag. Instead, we make it isolated in the new `modified quasi-vortex'. This leads to the following definition.

    Let $\scr{GD}=(H,U,\scr{J})$ be a planted graph-decomposition with $(J_x:x\in V(U)):=\scr{J}$, and let $S\subseteq V(U)$. Let $X:=\bigcup_{s\in S}J_s$. Note that $X\subseteq V(H)$. Let $H'$ be the graph obtained from $H-X$ by adding each vertex in $X\cap V(U-S)$ as an isolated vertex. For each $x\in V(U-S)$, set $J_x':=(J_x\setminus X)\cup \{x\}$, and then let $\scr{J}':=(J_x':x\in V(U-S))$. We define \defn{$\scr{GD}-S$}$:=(X,H',\scr{J}')$.

    \begin{lemma}
        \label{modifyVortex}
        $\scr{GD}=(H,U,\scr{J})$ be a planted graph-decomposition, let $S\subseteq V(U)$, and let $(X,H',\scr{J}'):=\scr{GD}-S$. Then all the following are true:
        \begin{enumerate}
            \item $S\subseteq X\subseteq V(H)$,
            \item $H-X\subseteq H'\subseteq H-S$,
            \item if $\scr{GD}$ has width at most $k$, then $|X|\leq (k+1)|S|$,
            \item $\scr{GD}':=(H',U-S,\scr{J}')$ is a planted graph-decomposition whose width is at most the width of $\scr{GD}$,
            \item if $G$ is a graph with $V(G\cap H)=V(U)$, then $V(G\cap H')=V(U-S)$, and
            \item if $\scr{GD}$ is smooth in a graph $G$, then so is $\scr{GD}'$.
        \end{enumerate}
    \end{lemma}

    \begin{proof}
        Let $(J_x:x\in V(U)):=\scr{J}$. For each $v\in V(H)$, let $C_v:=C_v(\scr{GD})$ be the nonempty and connected subgraph of $U$ induced by the vertices $x\in V(U)$ such that $v\in J_x$.
        
        Recall that $X=\bigcup_{s\in S}J_s$, $H'$ is the graph obtained from $H-X$ by adding each $x\in X\cap V(U-S)$ as an isolated vertex, and $\scr{J}'=(J_x':x\in V(U-S))$, where $J_x'=(J_x\setminus X)\cup \{x\}$ for each $x\in V(U-S)$. Since $\scr{J}$ is a $U$-decomposition of $H$, we have that $X\subseteq V(H)$. Thus, $X\cap V(U-S)\subseteq H-S$. Further, since $\scr{GD}$ is planted, $s\in J_s$ for each $s\in S$, and thus $S\subseteq X\subseteq V(H)$. It follows that $H-X\subseteq H'\subseteq H-S$. 

        Consider any $x\in V(U-S)$. Since $\scr{GD}$ is planted, $x\in J_x$. If $x\notin X$, observe that $J_x'=J_x\setminus X\subseteq V(H-X)\subseteq V(H')$. Otherwise, observe that $J_x\setminus X\subseteq V(H-X)\subseteq V(H')$ and $x\in X\cap V(U-S)\subseteq V(H')$ (where $x$ is an isolated vertex in $H'$). Either way, we find that $J_x'\subseteq V(H')$. It follows that $\scr{J}'\subseteq 2^{V(H')}$.

        For each $v\in V(H')$, let $C_v'$ be the subgraph of $U-S$ induced by the vertices $x\in V(U-S)$ such that $v\in J_x'$. We must show that $C_v'$ is nonempty and connected. If $v\in X$, then since $v\in V(H')$, we have $v\in V(U-S)\cap V(H)$. Thus, $v\in J_x'$ if and only if $x=v\in V(U-S)$. Hence, $V(C_v')=\{v\}$, which is trivially connected and nonempty. So we may assume that $v\notin X$.

        Observe that $v\in V(H')\subseteq V(H)$. So $C_v$ is well-defined. By definition of $X$ and since $v\notin X$, $V(C_v)\cap S=\emptyset$. Thus, $C_v$ is a connected and nonempty subgraph of $U-S$. Since $\scr{GD}$ is planted and since $v\notin X$, observe that $V(C_v')=V(C_v)$, and thus $C_v'=C_v$ is a connected and nonempty subgraph of $U-S$, as desired.
        
        Consider any pair of adjacent vertices $u,v\in V(H')$. Since each $v\in X\cap V(U-S)$ is isolated in $H'$, we have that $u,v\in V(H-X)$ and are adjacent in $H-X\subseteq H$. So $C_u$ and $C_v$ intersect. Since $u,v\in V(H-X)$, neither $u$ nor $v$ is in $X$. Thus, by the previous paragraph, $T_v=C_u$ intersects $C_v=C_v'$.

        Thus, $\scr{GD}'=(H',U-S,\scr{J}')$ is a graph-decomposition. Observe that $V(H')=V(H-X)\cup (X\cap V(U-S))=(V(H)\setminus X)\cup (X\cap V(U-S))$. Since $\scr{GD}$ is planted, $V(U)\subseteq V(H)$, and thus $V(U-S)\subseteq V(H')$. Since $x\in J_x'$ for each $x\in V(U-S)$ trivially, we find that $\scr{GD}'$ is planted. Further, since $x\in J_x$ for each $x\in V(U-S)$ (as $\scr{GD}$ is planted), we have $J_x'\subseteq J_x$. It follows that the width of $\scr{GD}'$ is at most the width of $\scr{GD}$.

        Set $B:=V(U)$ and $B':=V(U)\setminus S$. So $V(U-S)=B'$. Observe that $V(H')=(V(H)\setminus X)\cup (X\cap B')$. Since $X\subseteq V(H)$, we obtain $V(H')=V(H)\setminus (X\setminus B')$. Since $\scr{GD}$ is planted, we have $S\subseteq X$, and thus $(X\setminus B')\cap B=S$. So if $G$ is a graph with $V(G\cap H)=V(U)=B$, then $V(G\cap H')=B\setminus (X\setminus B')=B\setminus (B\cap (X\setminus B'))=B\setminus S=B'=V(U-S)$.

        If $\scr{GD}$ is smooth in a graph $G$, then $V(U-S)\subseteq V(U)\subseteq V(G)$, and for each $v\in V(H)\supseteq V(H')$, there exists a spanning tree $T_v$ of $C_v:=C_v(\scr{GD})$ such that for each subtree $T$ of $T_v$, $G[V(T)]$ is connected. If $v\notin X$, then $T_v$ is a spanning tree of $C_v(\scr{GD}')=C_v'=C_v$. If $v\in X$, then $C_v'$ is trivially a spanning tree of $C_v'=(U-S)[\{v\}]$. In this case, the only subtrees of $C_v'$ are $C_v'$ and the empty graph, both of whose vertices trivially induce connected subgraphs of $G$ (recalling that we consider the empty graph to be connected). It follows that $\scr{GD}'$ is smooth in $G$. This completes the proof.
    \end{proof}

    Now, since we ensured that the cutting subgraph was anchored, we find that the new discs are $V(W)$-anchored, where $W$ is the cut subgraph. $W$ itself isn't a part of the original graph $G$, and any vortex vertices that lay in a bag of a vertex of $V(W)$ have been deleted or made isolated, so $W$ contributes nothing towards trying to `reconstruct' an almost-embedding. Instead, it is more productive to pretend, for the purposes of reconstruction, that $W$ doesn't exist. Since $W$ is connected, we then find that all the discs are in the same face.

    This leads us to ask `can we merge the discs into a single disc'. The answer is yes. We can find a disc that contains the entirety of $V(W)$, then grows to contain most of each disc (after applying the mapping) in so that the original decomposition can be modified only slightly to become a decomposition of the modified quasi-vortex indexed by the underlying cycle of the new disc. This is then used to define the reduction, where the loss-set are the vertices in the vortex we `deleted', along with all the apices (as these are a pain, and we may as well remove them because there are only a finite number of them).

    The following lemma handles the disc `merging'.

    \begin{lemma}
        \label{mergeDiscs}
        Let $G$ be an embedded graph, let $W$ be a connected and nonempty embedded subgraph of $G$, and let $\scr{D}$ be a $V(W)$-anchored set, $G$-pristine set of discs. Then there exists a $(G,V(W))$-clean disc $D'$ such that $B(D',G,V(W))=B(\scr{D},G)\setminus V(W)$ and $U(D',G,V(W))$ is a spanning supergraph of $U(\scr{D},G)-(V(W)\cap B(\scr{D},G))$.
    \end{lemma}

    \begin{proof}
        Let $\{D_1,\dots,D_s\}:=\scr{D}$, with $s:=|\scr{D}|$, and let $B:=B(\scr{D},G)$ and $U:=U(\scr{D},G)$. For each $i\in \{1,\dots,s\}$, let $B_i:=B(D_i,G)$ and $U:=U(D_i,G)$. So $B=\bigcup_{i=1}^sB_i$ and $U=\bigcup_{i=1}^sU_i$.
        
        For each $i\in \{1,\dots,s\}$, let $\scr{C}_i$ be the arcs of $D_i$ in $G$, and let $\scr{C}_i^*$ be the subset of arcs in $\scr{C}_i$ whose endpoints are contained in $V(G)\setminus V(W)$. Since $\scr{D}$ is $V(W)$-anchored, $D_i$ intersects $V(W)$. Thus, $\scr{C}_i\setminus \scr{C}_i^*$ is nonempty. Observe that we can slightly shrink $D_i$ into a new $G$-clean disc $D_i'$ (contained in $D_i$) such that, if $B_i':=B(D_i',G)$, and $\scr{C}_i'$ are the arcs of $D_i$ in $G$, then:
        \begin{enumerate}
            \item $B_i'=B_i\setminus V(W)$ (so $D_i'$ is disjoint $W$),
            \item $\scr{C}_i^*\subseteq \scr{C}_i'$, and,
            \item there exists $C_i\in \scr{C}_i'\setminus \scr{C}_i^*$ and a path $P_i$ in the plane from $C_i$ to $V(W)$ whose interior is disjoint to $D_i'$ and is contained in the interior of $D_i$.
        \end{enumerate}
        Since $D_i'$ is disjoint to $W$, we find that $D_i'\cap (G-V(W))=D_i'\cap G=B_i'$. Let $U_i':=U(D_i',G)$. $B_i'=B_i\setminus V(W)$, since $\scr{C}_i^*\subseteq \scr{C}_i'$, and since $C_i\in \scr{C}_i'\setminus \scr{C}_i^*$, observe that $U_i'-C_i$ is a spanning supergraph of $U_i-(B_i\cap V(W))$ (where we consider $C_i$ to be an edge of $U_i$ with the same endpoints).
        
        Since $\scr{D}$ is $V(W)$-anchored, $D_1,\dots,D_s$ intersect only at vertices in $V(W)$. Thus, observe that $D_1',\dots,D_s'$ are pairwise disjoint. Note that $\bigcup_{i=1}^s B_i'=\bigcup_{i=1}^s B_i\setminus V(W)=B\setminus V(W)$, and that $\bigcup_{i=1}^s U_i'-C_i$ is a spanning supergraph of $U-(B\cap V(W))$.

        Since $W$ is connected, we can find a spanning tree $T$ for $W$. We can then `thicken' the embedding of this spanning tree slightly to find a disc $D_W$ that satisfies:
        \begin{enumerate}
            \item $T$ (and thus $V(W)$) is contained in the interior of $D_W$,
            \item $D_W$ is disjoint to $G-V(W)$ (and thus $D_W$ is $(G,V(W))$-clean, and $B(D_W,G,V(W))=\emptyset$), and,
            \item $D_W$ is disjoint to each of $D_1',\dots,D_s'$.
        \end{enumerate}
        Note that the latter is possible since $D_1',\dots,D_s'$ are each disjoint to $W$.
        
        Note that for each $i\in \{1,\dots,s\}$, there is a (closed) subpath $P_i'$ of $P_i$ contained in $D_W$ that includes the endpoint in $V(W)$. Observe that the boundary of $P_i\setminus P_i'$ is contained in the interior of $D_i$. Thus, the boundaries of $P_i\setminus P_i'$ are pairwise disjoint. It follows that we can `thicken' along each $P_i\setminus P_i'$ to obtain a disc $D'$ that contains $D_W$, is disjoint to $G$, and for each $i\in \{1,\dots,s\}$, satisfies:
        \begin{enumerate}
            \item $P_i\subseteq D'$,
            \item $D_i\cap D'$, and,
            \item $\partial D_i'\cap D'\subseteq C_i$.
        \end{enumerate}

        Recall that $D_1',\dots,D_s'$ are pairwise disjoint. It follows that $D':=D_W\cup \bigcup_{i=1}^s D_i'$ is a disc. Observe that $D'\cap (G-V(W))=\bigcup_{i=1}^s D_i'\cap (G-V(W))=\bigcup_{i=1}^s B_i'$. So $D'$ is $(G-V(W))$-clean, and $B(D',G-V(W))=\bigcup_{i=1}^sB_i'=\bigcup_{i=1}^sB_i'\setminus V(W)=B\setminus V(W)$.
        
        Further, observe that for each $i\in \{1,\dots,s\}$ and each arc $C$ of $D_i'$ in $G$ other than $C_i$, $C$ is an arc of $D'$ in $G$. Thus, if $\scr{C}$ are the arcs of $D'$ in $G$, then $\scr{C}\supseteq \bigcup_{i=1}^s \scr{C}_i'\setminus \{C_i\}$. Hence, $U(D',G-V(W))$ is a spanning supergraph of $\bigcup_{i=1}^s U_i'-C_i$, which is a spanning supergraph of $U-(B\cap V(W))$.

        Finally, since $D'$ contains $D_W$, which contains $V(W)$ in its interior, $D'$ contains $V(W)$ in its interior. So $D'$ is $(G,W)$-clean, $B(D',G,W)=B(D',G-V(W))=B\setminus V(W)$, and $U(D',G,W)=U(D',G-V(W))$ is a spanning supergraph of $U-(B\cap V(W))$. This completes the proof.
    \end{proof}

    The last obstacle is that we need the entire boundary to be in the neighbourhood of the plane subgraph (so that the plane+quasi-vortex embedding can be standardised). For this, we use the following lemma, which we prove later, along with \cref{makeTwoCell} (see \cref{SecEdges}).

    \begin{restatable}{lemma}{neighBoundary}
        \label{neighBoundary}
        Let $G$ be an embedded graph, let $S\subseteq V(G)$ be nonempty, and let $D$ be a $(G,S)$-clean disc. Then there exists a spanning embedded supergraph $G'$ of $G$ such that:
        \begin{enumerate}
            \item $G'-S$ is the plane graph $G-S$,
            \item $\FT(G)\subseteq \FT(G')$,
            \item $D$ is $(G',S)$-clean, and,
            \item $B(D,G,S)\subseteq N_{G'}(S)$
        \end{enumerate}
    \end{restatable}

    Finally, we need the following observation.

    \begin{observation}
        \label{smoothDelete}
        Let $\scr{GD}=(H,U,\scr{J})$ be a graph-decomposition that is smooth in a graph $G$, and let $X\subseteq V(G)\setminus V(U)$. Then $\scr{GD}$ is smooth in $G-X$.
    \end{observation}

    \begin{proof}
        Since $\scr{GD}$ is smooth in $G$, $V(U)\subseteq V(G)$. Thus and since $X\subseteq V(G)\setminus V(U)$, we have $V(U)\subseteq V(G-X)$. Since $\scr{GD}$ is smooth in $G$, for each $v\in V(H)$, there exists a spanning tree $T_v$ of $C_v:=C_v(\scr{GD})$ such that for each subtree $T'$ of $T_v$, $G[V(T)]$ is connected. Since $X$ is disjoint to $V(U)$, observe that $V(T)\subseteq V(G-X)$ and that $(G-X)[V(T)]=G[V(T)]$ is connected. It follows that $\scr{GD}$ is smooth in $G-X$, as desired.
    \end{proof}

    We can now get to the main result of this subsection.

    \begin{restatable}{lemma}{cutAE}
        \label{cutAE}
        Let $g,p,k,a,m\in \ds{N}$, and let $\Gamma$ be a weakly-standardised $(g,p,k,a)$-almost-embedding that admits an anchoring cutting subgraph $M$ of boundary intersection at most $m$. Then $\Gamma$ admits a reduction $\scr{R}$ of vortex-width at most $k$ and loss at most $(k+1)m+a$ whose split-vertices are $V(M)$.
    \end{restatable}

    \begin{proof}
        Let $(G,\Sigma,G_0,\scr{D},H,\scr{J},A):=\Gamma$, and let $B:=B(\Gamma)$ and $U:=U(\Gamma)$. Let $(J_v:v\in B):=\scr{J}$.
    
        Since $\Gamma$ is weakly-standardised, $G_0$ is $2$-cell embedded in $\Sigma$. Since $M$ is a cutting subgraph for $\Gamma$, $M$ is a nonempty embedded subgraph of $G_0$ with $|E(M)|=|V(M)|+g(\Sigma)-1$ and exactly one face. Since $\scr{D}$ is a $G_0$-pristine set of discs, $\scr{D}$ is a set of $G_0$-clean discs. Since $M$ is anchoring, $\scr{D}$ is $V(M)$-anchored. Thus, by \cref{cuttingLemmaSimple}, there exists a $2$-cell embedded plane graph $G_0^-$, a map $\phi':V(G_0^-)\mapsto V(G_0)$, a connected and nonempty plane subgraph $W$ of $G_0^-$, a $V(W)$-anchored set $\scr{D}'$ of $G_0^-$-clean discs, and a bijection $\beta:\scr{D}'\mapsto \scr{D}$ such that:
        \begin{enumerate}
            \item $G_0\cap G_0^-=G_0-V(M)=G_0^--V(W)$ (as non-embedded graphs),
            \item $\phi'$ is the identity on $V(G_0^-)\setminus V(W)$,
            \item $\phi'(V(W))=V(M)$,
            \item for each $v\in V(G_0)$, $N_{G_0}[v]\subseteq \phi'(N_{G_0^-}[(\phi')^{-1}(v)])$,
            \item for each $K\in \FT(G_0)$, there exists $K'\in \FT(G_0^-)$ with $\phi'(K')=K$, and,
            \item for each $D'\in \scr{D}'$, the restriction of $\phi$ to $B(D',G')$ is an isomorphism from $U(D',G')$ to $U(\beta(D'),G)$.
        \end{enumerate}

        Since $V(G_0^-)\cap V(G_0)=V(G_0^-)\setminus V(W)$, note that, up to relabelling $V(W)$, we can assume that $V(W)$ is disjoint to $V(G)$. In particular, we can assume that $V(W)$ is disjoint to $V(H)$ and $V(M)$.
        
        Let $G_0'$ be the plane graph $G_0^--V(W)$. So $G_0'=G_0-V(M)$ (as non-embedded graphs).

        Let $\{D_1',\dots,D_s'\}:=\scr{D}'$, with $s=|\scr{D}'|$, and for each $i\in \{1,\dots,s\}$, let $D_1:=\beta(D_i')$. Since $\beta$ is a bijection, $\scr{D}=\{D_1,\dots,D_s\}$. For each $i\in \{1,\dots,s\}$, let $B_i:=B(D,G)$, $B_i':=B(D',G')$, $U_i:=U(D,G)$, and $U_i':=U(D',G')$. So $B_i=\phi(B_i')$, and the restriction of $\phi$ to $B_i'$ is an isomorphism from $U_i'$ to $U_i$. Since $\phi$ is the identity on $V(G_0^-)\setminus V(W)=V(G_0)\setminus V(M)$, we find that $U_i-(V(M)\cap B_i)=U_i'-(V(W)\cap B_i')$. In particular, $B_i\setminus V(M)=B_i'\setminus V(W)$.
        
        Let $B:=B(\scr{D},G)=\bigcup_{i=1}^s B_i$, $B'':=B(\scr{D}',G')=\bigcup_{i=1}^s B_i'$, $U:=U(\scr{D},G)=\bigcup_{i=1}^s U_i$, and $U'':=U(\scr{D}',G')=\bigcup_{i=1}^s U_i'$. Observe that $U-(V(M)\cap B)=U''-(V(W)\cap B'')$, and $B\setminus V(M)=B''\setminus V(W)$.
        
        By \cref{mergeDiscs}, there exists a $(G_0^-,W)$-clean disc $D'$ such that $B':=B(D',G_0^-,W)=B(D',G_0')=B''\setminus V(W)=B\setminus V(M)$ and $U':=U(D',G_0^-,W)=U(D',G_0')$ is a spanning supergraph of $U''-(V(W)\cap B'')=U-(V(M)\cap B)$.

        Let $S:=V(M)\cap B$. So $U'$ is a spanning supergraph of $U-S$, and $B'=B\setminus S=V(U-S)$. Since $M$ has boundary-intersection at most $m$, $|S|\leq m$.
        
        By \cref{neighBoundary} (since $W$ is nonempty), there exists a spanning embedded (plane) supergraph $G_0^+$ of $G_0^-$ such that $G_0^+-V(W)$ is the embedded (plane) graph $G_0^--V(W)$ (which is $G_0'$), $\FT(G_0^)\subseteq \FT(G_0^+)$, $D'$ is $(G_0^+,V(W))$-clean, and $B(D,G_0^-,V(W))\subseteq N_{G_0^+}(V(W))$. Note that $B(D,G_0^+,V(W))=B(D,G_0')=B'=B\setminus S$. In particular, observe that $B(D,G_0^+,V(W))=B'=B(D,G_0^-,V(W))\subseteq N_{G_0^+}(V(W))$.

        Note that since $W$ is a connected and nonempty plane subgraph of $G_0^-$, and since $G_0^+$ is a plane supergraph of $G_0^-$, $W$ is a connected and nonempty plane subgraph of $G_0^+$.

        Recall that $(H,U,\scr{J})$ is a planted graph-decomposition, and that $V(G_0\cap H)=B$. Since $\Gamma$ is a $(g,p,k,a)$-almost-embedding, $(H,U,\scr{J})$ has width at most $k$. Further, since $\Gamma$ is weakly-standardised, $(H,U,\scr{J})$ is smooth in $G_0$.
        
        Let $(X',H',\scr{J}'):=(H,U,\scr{J})-S$. By \cref{modifyVortex} (with $G:=G_0$), all the following are true:
        \begin{enumerate}
            \item $S\subseteq X'\subseteq V(H)$,
            \item $|X'|\leq (k+1)|S|\leq (k+1)m$,
            \item $H-X'\subseteq H'\subseteq H-S$,
            \item $\scr{GD}':=(H',U-S,\scr{J}')$ is a planted graph-decomposition of width at most $k$,
            \item $V(G_0\cap H')=V(U-S)$, and
            \item $\scr{GD}'$ is smooth in $G_0$.
        \end{enumerate}

        Since $S=V(M)\cap B=V(M)\cap V(U)$, observe that $V(M)$ is disjoint to $V(U-S)$. Thus, $V(G_0'\cap H')=V((G_0-V(M))\cap H')=V(U-S)$. By \cref{smoothDelete} (with $X:=V(M)$), $\scr{GD}'$ is smooth in $G_0-V(M)=G_0'$. By \cref{spanningSmooth}, $\scr{GD}'':=(H',U',\scr{J}')=(H',U(D,G_0^+,V(W)),\scr{J}')$ is a planted graph-decomposition of width at most $k$, and is smooth in $G_0'$.

        By choice of labels for $W$, $V(H'\cap W)\subseteq V(H\cap W)=\emptyset$. So $W$ is disjoint to $H'$. Further, $V(G_0^+\cap H')=V((G_0^+-V(W)\cap H')=V(G_0'\cap H')=V(U-S)=B'$.

        Let $G':=G_0^+\cup H'$, and let $\Lambda:=(G',G_0^+,W,D',H',\scr{J}')$. By our previous observations, $\Lambda$ is a plane+quasi-vortex embedding. Further, observe that $G_0(\Lambda)=G_0'$ and $U(\Lambda)=U(D,G_0^+,V(W))=U'$. So $(H',U(\Lambda),\scr{J}')$ is smooth in $G_0(\Lambda)$. Recall also that $V(G_0^+\cap H')=B(D,G_0^+,V(W))\subseteq N_{G_0^+}(V(W))$. Thus, $\Lambda$ is standardised.

        Let $X:=A\cup X'$. Since $\Gamma$ is a $(g,p,k,a)$-almost-embedding, $|A|\leq a$. Thus and recalling that $|X'|\leq (k+1)m$, we find that $|X|\leq (k+1)m+a$. Observe that $X'\subseteq V(H)\subseteq V(G)$, and thus $A\subseteq X\subseteq V(G)$. Note that $X\cap V(H)=X'$, so $H-(X\cap V(H))=H-X'\subseteq H'\subseteq H-S\subseteq H$.

        Since $M\subseteq G_0$, observe that $V(M)\cap V(H)=V(M)\cap B=S$. Thus, $V(M\cap H')\subseteq V(M\cap (H-S))=\emptyset$. So $M$ is disjoint to $H'$. Recall that $G_0'=G_0\cap G_0^-$. Thus, $V(G_0')=V(G_0)\cap V(G_0^-)=V(G_0)\cap V(G_0^+)$ (since $G_0^+$ is a spanning supergraph of $G_0^-$). Since $M$ is disjoint to $G_0'=G_0-V(M)$, and since $M\subseteq G_0$, we have that $G_0^+$ is disjoint to $M$. Thus, $V(G')$ is disjoint to $V(M)$.

        Recall that $V(H')\subseteq V(H)\subseteq V(G)$. Thus, we can define a map $\phi:V(G')\mapsto V(G)$ that is $\phi'$ on $V(G_0^+)=V(G_0^-)$ and the identity on $V(H')\setminus V(G_0^+)$. As $W\subseteq G_0^+$, we have $\phi(V(W))=\phi'(V(W))=V(M)$. Recall that $W$ is disjoint to $H\supseteq H'$, and that $\phi'$ is the identity on $V(G_0^-)\setminus V(W)=V(G_0^+)\setminus V(W)$. Thus, $\phi$ is the identity on $V(G')\setminus V(W)$. In particular, since $W$ is disjoint to $H'$, $\phi$ is the identity on $V(H')$.

        For each $K\in \FT(G_0)$, recall that there exists $K'\in \FT(G_0^-)\subseteq \FT(G_0^+)$ such that $\phi'(K')=K$. Observe that $\phi'(K')=\phi(K')$ as $K\subseteq V(G_0^-)$. So there exists $K'\in \FT(G_0^+)$ with $\phi(K')=K$.

        Fix $v\in V(G-X)$. Since $A\subseteq X=A\cup X'$ and since $x\notin X$, we have $N_{G-X}[v]=N_{G_0\cup H}[v]\setminus X$. Observe that $N_{G_0\cup H}[v]=N_{G_0}[\{v\}\cap V(G_0)]\cup N_H[\{v\}\cap V(H)]$. Thus and since $X'\subseteq X$, observe that $N_{G-X}[v]\subseteq N_{G_0\cup H}[v]\setminus X\subseteq (N_{G_0}[\{v\}\cap V(G_0)])\cup (N_H[\{v\}\cap V(H)]\setminus X')$.
        
        Since $H-X'\subseteq H'\subseteq H$ and since $x\notin X\supseteq X'$, observe that $N_H[\{v\}\cap V(H)]\setminus X'\subseteq N_{H'}[\{v\}\cap V(H')]$. Recalling that $\phi$ is the identity on $V(H')$, we have $N_{H'}[\{v\}\cap V(H')]\subseteq \phi(N_{H'}[\phi^{-1}(v)\cap V(H')])$. So $N_H[\{v\}\cap V(H)]\setminus X'\subseteq \phi(N_{H'}[\phi^{-1}(v)\cap V(H')])$.

        Observe that $N_{G_0}[\{v\}\cap V(G_0)]\subseteq \phi'(N_{G_0^-}[(\phi')^{-1}(\{v\}\cap V(G_0))])$ (whether or not $v\in V(G_0)$). Since $\phi$ is $\phi'$ on $V(G_0^-)=V(G_0^+)$, observe that $(\phi')^{-1}(\{v\}\cap V(G_0))\subseteq \phi^{-1}(v)\cap V(G_0^+)$ and that $\phi'(N_{G_0^-}[\phi^{-1}(v)\cap V(G_0^+)])=\phi(N_{G_0^-}[\phi^{-1}(v)\cap V(G_0^+)])$. Thus, $N_{G_0}[\{v\}\cap V(G_0)]\subseteq \phi(N_{G_0^-}[\phi^{-1}(v)\cap V(G_0^+)])$. Since $G_0^+$ is a spanning subgraph of $G_0^-$, observe that $N_{G_0^-}[\phi^{-1}(v)]\subseteq N_{G_0^+}[\phi^{-1}(v)]$. Thus, $N_{G_0}[\{v\}\cap V(G_0)]\subseteq \phi(N_{G_0^+}[\phi^{-1}(v)\cap V(G_0^+)])$.

        Observe that $N_{G'}[\phi^{-1}(v)]=N_{G_0^+}[\phi^{-1}(v)\cap V(G_0^+)]\cup N_{H'}[\phi^{-1}(v)\cap V(H')]$. Thus, $\phi(N_{G'}[\phi^{-1}(v)])=\phi(N_{G_0^+}[\phi^{-1}(v)\cap V(G_0^+)]\cup N_{H'}[\phi^{-1}(v)\cap V(H')])=\phi(N_{G_0^+}[\phi^{-1}(v)\cap V(G_0^+)])\cup \phi(N_{H'}[\phi^{-1}(v)\cap V(H')])$. This gives $\phi(N_{G'}[\phi^{-1}(v)])\supseteq N_{G_0}[\{v\}\cap V(G_0)]\cup (N_H[\{v\}\cap V(H)]\setminus X')\supseteq N_{G-X}[v]$.

        Let $\scr{R}:=(V(M),X,G',G_0^+,W,D',H',\scr{J}',\phi)$. It follows from our observations thus far that $\scr{R}$ is a reduction of $\Gamma$ (whose split-vertices are $V(M)$). Note that $\Lambda(\scr{R})=\Lambda$.
        
        The loss of $\scr{R}$ is $|X|\leq (k+1)m$, and the vortex-width of $\scr{R}$ is the width of $\scr{J}'$, which is at most $k$. This completes the proof.
    \end{proof}

    \subsection{Embedding edges strategically}
    \label{SecEdges}

    We skipped two lemmas in this section, with the excuse being that they shared the same proof idea as several other lemmas. Specifically, we skipped \cref{makeTwoCell} and \cref{neighBoundary}. We now give the proofs of these lemmas, in the form of a general framework. This framework will be recycled for two future lemmas (\cref{makeTriangulation} and \cref{PQTContract}), so we really need to keep it general. The common idea is that we want to add new edges to an embedded graph that `respect' a (finite) set of clean discs and all facial triangles, until all faces satisfy a certain property. Beyond that, the specific properties we want vary somewhat, so we preserve as many properties as we can.

    Formally, we want to establish a `condition' that takes a graph and a collection of clean discs, and `decides' if the faces of the graph satisfy the condition. For example, we could ask `is this face $2$-cell' or `does this face intersect a disc'. This leads to the following definitions.

    Say that a \defn{decidable triple} is a triple $\scr{T}=(G,S,\scr{D})$, where $G$ is a nonempty embedded graph, $S\subseteq V(G)$, and $\scr{D}$ is a finite set of pairwise disjoint $(G,S)$-clean discs. Since each $D\in \scr{D}$ contains $S$ in its interior, we remark that either $S=\emptyset$ (in which case each disc in $\scr{D}$ is $G$-clean), or $|\scr{D}|=1$. Use \defn{$\scr{F}(\scr{T})$} to denote the faces of $G$.

    We define a \defn{facial criterion} $\chi$ to be an operation that, given a decidable triple $\scr{T}$, returns a function $\eta:\scr{F}(\scr{T})\mapsto \{0,1\}$. We think of the value $0$ as `false' and $1$ as `true'. For a face $F\in \scr{F}(\scr{T})$, we say that $F$ \defn{satisfies} $\chi$ in $\scr{T}$ if $\eta(F)=1$, and we say that $\scr{T}$ \defn{satisfies} $\chi$ if $\eta$ is the constant map to $1$ (so each $F\in \scr{F}(\scr{T})$ satisfies $\chi$ in $\scr{T}$).

    We do need the facial criterion to satisfy some basic properties. The goal is to add edges until each face satisfies the criterion, but this doesn't make sense if adding edges can cause a face to stop satisfying the criterion. So we require that the criterion is well-behaved under spanning supergraphs. This leads to the following definitions.

    If $\scr{T}=(G,S,\scr{D})$ is a decidable triple, say that a decidable triple $\scr{T}'$ \defn{extends} $\scr{T}$ if $\scr{T}'=(G',S,\scr{D})$, where $G'$ is a spanning embedded supergraph of $G$.

    Say that a facial criterion $\chi$ is \defn{consistent} if, for each pair of decidable triples $\scr{T},\scr{T}'$ such that $\scr{T}'$ extends $\scr{T}$, and any $F\in \scr{F}(\scr{T})\cap \scr{F}(\scr{T}')$ that satisfies $\chi$ in $\scr{T}$, $F$ satisfies $\chi$ in $\scr{T}'$.
    
    Finally, we need to ensure that the process of adding edges will eventually terminate (and satisfy the criterion when it does so). To do this, we need to establish some `base cases' for faces, where we guarantee that the criterion always holds true. This cases are inspired by the proof, and correspond to cases where we either can't find a way to embed a new edge (that avoids the discs), or where embedding new edges would not provide any inductive benefit. This leads to the following definition.
    
    Say that a facial criterion $\chi$ is \defn{reducible} if it is consistent and, for each decidable triple $\scr{T}=(G,S,\scr{D})$ and each $2$-cell $F\in \scr{F}(\scr{T})$, $F$ satisfies $\chi$ if either of the following are true:
    \begin{enumerate}
        \item the boundary walk of $F$ has length at most $3$, or
        \item $S=\emptyset$ and $\scr{D}\neq \emptyset$.
    \end{enumerate}
    In particular, if $\chi$ is reducible, every triangle-face satisfies $\chi$ in $\scr{T}$.

    The main technical result of this subsection is as follows.

    \begin{lemma}
        \label{facialCriteriaSpanning}
        Let $\chi$ be a reducible facial criterion, and let $\scr{T}=(G,S,\scr{D})$ be a decidable triple. Then there exists a spanning embedded supergraph $G'$ of $G$ such that: 
        \begin{enumerate}
            \item each $D\in \scr{D}$ is $(G',S)$-clean
            \item $(G',S,\scr{D})$ (is decidable and) satisfies $\chi$,
            \item $\FT(G)\subseteq \FT(G')$,
            \item each edge of $E(G')\setminus E(G)$ is contained in a face of $G$ that does not satisfy $\chi$ in $\scr{T}$, and,
            \item each edge of $E(G')\setminus E(G)$ with no endpoint in $S$ is contained in a face of $G-S$ that does not contain a vertex of $S$.
        \end{enumerate}
    \end{lemma}

    \begin{proof}
        Since $(G,S,\scr{D})$ is decidable, we have that $G$ is a nonempty embedded graph, $S\subseteq V(G)$, and $\scr{D}$ is a finite set of pairwise disjoint $(G,S)$-clean discs. Let $\Sigma$ be the surface $G$ is embedded in, and let $g:=g(\Sigma$).
    
        For some unknown $n\in \ds{N}$, we inductively create a sequence of embedded graphs $G_0,G_1,\dots,G_n$ (embedded in $\Sigma$), such that for any $i\in \{0,\dots,n\}$, $G_i$ satisfies all of the following:

        \begin{enumerate}
            \item $G_i$ is a spanning embedded supergraph of $G$,
            \item each $D\in \scr{D}$ is $(G_i,S)$-clean,
            \item $\FT(G)\subseteq \FT(G_i)$,
            \item each face of $G_i$ that does not satisfy $\chi$ in (the decidable triple) $\scr{T}_i:=(G_i,S,\scr{D})$ is contained in a face of $G$ that does not satisfy $\chi$ in $\scr{T}$,
            \item each face of $G_i$ with no boundary vertex in $S$ is contained in a face of $G-S$ that does not contain a vertex of $S$.
            \item each edge of $E(G_i)\setminus E(G)$ is contained in a face of $G$ that does not satisfy $\chi$ in $\scr{T}$, and
            \item each edge of $E(G_i)\setminus E(G)$ with no endpoint in $S$ is contained in a face of $G-S$ that does not contain a vertex of $S$.
        \end{enumerate}
        We remark that $S\subseteq V(G_i)$ as $V(G_i)=V(G)$ (since $G_i$ is a spanning supergraph of $G$), and that $\scr{T}_i$ is decidable as each $D\in \scr{D}$ is $(G_i-S)$-clean, and extends $\scr{T}$ as $G_i$ is a spanning embedded supergraph of $G$.

        Observe that we can set $G_0:=G$ (with $\scr{T}_0=\scr{T}$), as each face of $G$ with no boundary vertex in $S$ is a face of $G-S$, and trivially does not contain a vertex in $S$.
        
        We will ensure that $\scr{T}_n$ satisfies $\chi$. Thus, observe that the lemma will be satisfied with $G':=G_n$.
        
        For each $i\in \{1,\dots,n\}$, let $f_i$ be the number of faces of $G_i$. By Euler's formula, $|E(G_i)|\leq |V(G_i)|+f_i+g-2=|V(G)|+f_i+g-2$. We identify a pair $(e_i,F_i)$ such that:
        \begin{enumerate}
            \item $e_i\in E(G_i)$, and $G_i-e_i$ is the embedded graph $G_{i-1}$,
            \item $F_i$ is a face of $G_{i-1}$, contains $e$, and either:
            \begin{enumerate}
                \item $F_i$ is not $2$-cell, and $e_i$ does not split $F_i$ into two faces (so $F_i-e$ is a face of $G_i$), or,
                \item $F_i$ is $2$-cell, and $e$ splits $F_i$ into two $2$-cell faces (of $G_i$) with strictly shorter boundary walks.
            \end{enumerate}
        \end{enumerate}
        We refer to (2) as the `splitting requirements'. In the former case (the `not $2$-cell case'), note that $f_i=f_{i-1}$ and $|E(G_i)|=|E(G_i)|+1$. So Euler's formula implies that this case can only occur finitely many times. Likewise, since the boundary walks of the split faces are strictly shorter in the latter case (the `2-cell case'), it can also only occur finitely many times. This ensures that the process will terminate. That is, we can continue to inductively define $(G_{i+1},e_i,F_i)$ until at some point in time, we have that $\scr{T}_i$ satisfies $\chi$, at which point we take $i=n$ and $G'=G_n=G_i$ to satisfy the lemma.
        
        Thus, it suffices to assume that for some $i\in \ds{N}$, $G_{i-1}$ is defined, and $\scr{T}_{i-1}$ does not satisfy $\chi$. We must find $(G_i,e_i,F_i)$. We do this by finding a pair $(e,F)$ such that:
        \begin{enumerate}
            \item $e$ is an open path (in $\Sigma$) whose endpoints (boundary) are contained in $V(G_{i-1})$,
            \item $e$ either has an endpoint in $S$, or is disjoint to each $D\in \scr{D}$,
            \item $F$ is face of $G_{i-1}$ that does not satisfy $\chi$ in $\scr{T}_{i-1}$,
            \item $F$ contains $e$,
            \item $F$ is not a triangle-face,
            \item if $F$ has a boundary vertex in $S$, then $e$ has an endpoint in $S$,
            \item if $F$ is not $2$-cell, then $e$ does not split $F$ into two faces, and,
            \item if $F$ is $2$-cell, then $e$ splits $F$ into two $2$-cell faces with strictly shorter boundary walks.
        \end{enumerate}
        Note that since $e$ is an open path whose endpoints are in $V(G_{i-1})$, we can define a graph $H$ by adding $e$ as an edge to $G_{i-1}$ (whose endpoints in the graph are the endpoints of the open path), and embedding this edge as the open path. So when we talk about $e$ splitting (or not splitting) $F$ into faces, we mean faces of $H$.

        We now argue that we can take $(e_i,F_i,G_i):=(e,F,H)$. Note that we will have $\scr{T}_i=(H,S,\scr{D})$.
        
        It is immediate that $e\in E(H)$, and $H-e$ is the embedded graph $G_{i-1}$. By definition, $F$ is a face of $G_{i-1}$, and $e$ is contained in $F$. Further, $e$ and $F$ satisfy the `splitting requirements', with the second to last requirement on $(e,F)$ satisfying the `not $2$-cell case', and the last requirement satisfying the `$2$-cell case'.
        
        We now use the properties of $G_{i-1}$ to show that $H$ satisfies the properties required of $G_i$. Observe that $H$ is a spanning embedded supergraph of $G_{i-1}$. Since $G_{i-1}$ is a spanning supergraph of $G$, it follows that $H$ is also a spanning embedded supergraph of $G$.
        
        For each $D\in \scr{D}$, since $D$ is $G_{i-1}$-clean, $S$ is contained in the interior of $D$ and $D\cap (G_{i-1}-S)\subseteq V(G_{i-1}-S)$. If $e$ has an endpoint in $S$, then $H-S=G_{i-1}$, and $D\cap (H-S)\subseteq V(G_{i-1}-S)=V(H-S)$. Otherwise, $e$ is disjoint to $D$, so $D\cap (H-S)=D\cap (H-S)\setminus e=D\cap (H-e-S)=D\cap (G_{i-1}-S)\subseteq V(G_{i-1}-S)=V(H-S)$. So either way, $D$ is $(H,S)$-clean. Thus, each $D\in \scr{D}$ is $(H,S)$-clean. In particular, note that $\scr{T}':=(H,S,\scr{D})$($=\scr{T}_i$) is decidable, and extends $\scr{T}_{i-1}$.

        Observe that $E(H)\setminus E(G_{i-1})=\{e\}$. Since $e$ is contained in $F$ (and hence no other face of $G_{i-1}$), and since $F$ is not a triangle-face, no triangle-face of $G_{i-1}$ contains a loop-edge in $E(H)\setminus E(G_{i-1})$. Thus, by \cref{facialSpanningWeak}, $\FT(G_{i-1})\subseteq \FT(H)$. Recalling that $\FT(G)\subseteq \FT(G_{i-1})$, we obtain $\FT(G)\subseteq \FT(H)$.

        Observe that each face $F'$ of $H$ is either a face of $G_{i-1}$, or is contained in $F$ and has both endpoints of $e$ as boundary vertices.

        Let $F'$ be a face of $H$ that does not satisfy $\chi$ in $\scr{T}'$. If $F'$ is a face of $G_{i-1}$, then since $\chi$ is consistent (as it is reducible), $F'$ also does not satisfy $\chi$ in $\scr{T}_{i-1}$. So $F'$ is contained in a face $F'':=F'$ of $G_{i-1}$ that does not satisfy $\chi$ in $\scr{T}_{i-1}$. Otherwise, $F'$ is contained in $F$, which does not satisfy $\chi$ in $\scr{T}_i$. Either way, $F'$ is contained in a face $F''$ ($F'$ in the former case, $F$ in the latter case) of $G_{i-1}$ that does not satisfy $\chi$ in $\scr{T}_i$. By definition of $G_{i-1}$, $F''$, and thus $F'$, is contained in a face of $G$ that does not satisfy $\chi$ in $\scr{T}$.

        Let $F'$ be a face of $H$ that does not contain a boundary vertex in $S$. If $F'$ is a face of $G_{i-1}$, then by definition of $G_{i-1}$, $F'$ is contained in a face of $G-S$ that contains no vertices in $S$. Otherwise, $F'$ is contained in $F$, and has both endpoints of $e$ as boundary vertices. So $e$ does not have an endpoint in $S$. Thus, $F$ does not have a boundary vertex in $S$, and is contained in a face of $G-S$ that contains no vertices in $S$. Since $F'$ is contained in $F$, $F'$ is contained in a face of $G-S$ that contains no vertices in $S$.

        Let $e'$ be an edge in $E(H)\setminus E(G)$. Observe that either $e=e'$, or $e'\in E(G_{i-1})\setminus E(G)$. In the former case, $e'=e$ is contained in $F$, which is contained in a face of $G$ that does not satisfy $\chi$ in $\scr{T}$ (as $F$ does not satisfy $\chi$ in $\scr{T}_{i-1}$). So $e'=e$ is contained in a face of $G$ that does not satisfy $\chi$ in $\scr{T}$. In the latter case, by definition of $G_{i-1}$, $e'$ is contained in a face of $G$ that does not satisfy $\chi$ in $\scr{T}$.

        Presume further that $e'$ does not have an endpoint in $S$. If $e'\in E(G_{i-1})\setminus E(G)$, then recall that $e'$ is contained in a face of $G-S$ containing no vertices in $S$. Otherwise, if $e'=e$, then $F$ contains no boundary vertices, and thus is contained in a face of $G-S$ containing no vertices in $S$. Since $e'=e$ is contained in $F$, $e'$ is contained in a face of $G-S$ containing no vertices in $S$.

        Thus, we can take $(G_i,e_i,F_i):=(H,e,F)$. So now we just need to find $e,F$.

        Recall that $\scr{T}_{i-1}$ does not satisfy $\chi$. Thus, there exists a face $F$ of $G_{i-1}$ that does not satisfy $\chi$ in $\scr{T}_{i-1}$. Let $x$ be a boundary vertex of $F$ in $S$ if such a vertex exists. Otherwise, let $x$ be any boundary vertex.

        If $F$ is not $2$-cell, then $F$ is not a triangle-face. If we further have $x\in S$, observe that we can find an open path $e$ in $F$ with $x$ as one endpoint and the other in $V(G_{i-1})$ (possibly $x$ as well) that does not split $F$ into two faces. Otherwise, if $x\notin S$ (and thus $F$ has no boundary vertices in $S$), since $\scr{D}$ is a finite set of pairwise disjoint discs (and since $G$ is nonempty), we can find an open path $e$ in $F$ with both (possibly not distinct) endpoints in $V(G_{i-1})$ that is disjoint to each $D\in \scr{D}$ and does not split $F$ into two faces. Either way, $(e,F)$ satisfies all desired properties.

        If $F$ is $2$-cell, then since $\chi$ is reducible, either $S\neq \emptyset$ or $\scr{D}=\emptyset$. Let $x,v_1,\dots,v_n,x$ be the vertex boundary walk of $F$. So the length of the boundary walk is $n+1$. Since $\chi$ is reducible, $n\geq 3$ (the length of the boundary walk is at least $4$). So $F$ is not a triangle-face. Observe that we can find an open path $e$ whose endpoints are $x$ and $v_2$ (which may not be distinct) that splits $F$ into two $2$-cell faces $F^-$ and $F^+$ whose vertex boundary walks are $x,v_1,v_2,x$ and $x,v_2,\dots,v_n,x$. The length of $F^-$'s boundary walk is $3<n+1$, and the length of $F^+$'s boundary walk is $n<n+1$. So both $F^-$'s and $F^+$'s boundary walks are strictly shorter than $F$'s boundary walk.
        
        If $x\in S$, then $e$ has an endpoint in $S$. Otherwise, $F$ has no boundary vertices in $S$. In this case, observe that the entire boundary of $F$ is contained in $G_{i-1}-S$. Fix $D\in \scr{D}$. If such a disc exists, then $\scr{D}$ is nonempty, and thus $S=\emptyset$. Since $D$ is $(G_{i-1},S)$-clean, $D$ contains $S$ and $D$ is $(G_{i-1}-S)$-clean. From the latter, and since the boundary of $F$ is contained in $G_{i-1}-S$, we find that either $D$ is disjoint to $F$, or $F$ contains the interior of $D$. However, the interior of $D$ contains $S$, and $F$ is disjoint to $S$ (as $S\subseteq V(G_{i-1})$ and $F$ is a face of $G_{i-1}$). Since $S$ is nonempty, it follows that $F$ is disjoint to $D$. Since $e$ is contained in $F$, $e$ is therefore disjoint to each $D\in \scr{D}$. Thus (and in either scenario), $(e,F)$ satisfy all desired properties.

        Hence, we can always define $(e,F)$, and thus we can always construct the sequence $G=G_0,\dots,G_n=G'$. This completes the proof.
    \end{proof}

    We can now prove \cref{makeTwoCell}.

    \makeTwoCell*


    \begin{proof}
        For a decidable triple $\scr{T}$, say that a face $F\in \scr{F}(\scr{T})$ satisfies $\chi$ in $\scr{T}$ if $F$ is $2$-cell (homeomorphic to a disc). Observe that $\chi$ is a facial criterion. Since whether or not a face satisfies $\chi$ depends only on the face itself (and not its boundary/the graph it's a face of), observe that $\chi$ is consistent. Thus and since each $2$-cell face satisfies $\chi$, $\chi$ is reducible.
        
        Fix $D\in \scr{D}$. Since $D$ is $G$-clean, note that $D$ is $(G,\emptyset)$-clean. Thus (and since $G$ is nonempty), observe that $(G,\emptyset,\scr{D})$ is decidable. By \cref{facialCriteriaSpanning}, there exists a spanning embedded supergraph $G'$ of $G$ such that each $D\in \scr{D}$ is $(G',\emptyset)$-clean, $(G',S,\scr{D})$ (is decidable and) satisfies $\chi$, and $\FT(G)\subseteq \FT(G')$.

        Fix $D\in \scr{D}$. Since $D$ is $(G',\emptyset)$-clean, note that $D$ is $G$-clean.

        Since $\scr{T}':=(G',S,\scr{D})$ satisfies $\chi$, each face of $G'$ satisfies $\chi$ in $\scr{T}'$. Thus, each face of $G'$ is $2$-cell. So $G'$ is $2$-cell embedded. This completes the proof.
    \end{proof}

    Before we prove \cref{neighBoundary}, we introduce a bit more theory, to reduce repeating common ideas. We frequently want to reduce to the case where the `relevant' faces are $2$-cell embedded and their boundary vertices are a clique. Because we use this condition so frequently, it is helpful to have a common definition that `adds' this condition as an alternative to an existing condition, along with any basic conditions needed to be reducible. This leads to the following definition.

    Let $\chi$ be a facial criterion. The \defn{clique criteria extension} of $\chi$ is the facial criterion $\chi'$ that, for each reducible triple $\scr{T}=(G,S,\scr{P})$ and each face $F\in \scr{F}(\scr{T})$, sees $F$ satisfy $\chi'$ in $\scr{T}$ if $F$ satisfies $\chi$ in $\scr{T}$ or if $F$ is $2$-cell and any of the following hold:
    \begin{enumerate}
        \item the boundary vertices of $F$ are a clique in $G$, or,
        \item $S=\emptyset$, and $\scr{D}\neq \emptyset$.
    \end{enumerate}

    \begin{observation}
        \label{cliqueCriteria}
        Let $\chi$ be a consistent facial criterion, and let $\chi'$ be the clique criteria extension of $\chi$. Then $\chi'$ is reducible.
    \end{observation}

    \begin{proof}
        Observe that if a $2$-cell face $F$ of an embedded graph $G$ has a boundary walk of length at most $3$, then the boundary vertices of $F$ are a clique in $G$. It follows that if $\chi'$ is consistent, then it is reducible.
        
        Let $\scr{T},\scr{T}'$ be reducible triples such that $\scr{T}'$ extends $\scr{T}$. So if $\scr{T}=(G,S,\scr{D})$, then $\scr{T}'=(G',S,\scr{D})$ for some spanning embedded supergraph of $G$. Fix $F\in \scr{T}\cap \scr{T}'$. We must show that if $F$ satisfies $\chi'$ in $\scr{T}$, then $F$ also satisfies $\chi'$ in $\scr{T}'$. 
        
        If $F$ satisfies $\chi$ in $\scr{T}$, then $F$ also satisfies $\chi$ in $\scr{T}'$ as $\chi$ is consistent. Whether or not $F$ is $2$-cell, $S=\emptyset$, or $D\neq \emptyset$ remains the same when consider whether $F$ satisfies $\chi$ in $\scr{T}$ or in $\scr{T}'$. So it remains only to consider the case when $F$ is $2$-cell and the boundary vertices $B$ of $F$ in $G$ are a clique in $G$. In this case, since $G'$ is a spanning embedded supergraph of $G$ (yet has $F$ as a face), the boundary vertices of $F$ in $G'$ is $B$. Since $B$ is a clique in $G\subseteq G'$, $B$ is a clique in $G'$, as desired. It follows that $\chi'$ is consistent, and thus that $\chi'$ is reducible as observed above.
    \end{proof}

    We now prove a slightly more general lemma.
    \begin{lemma}
        \label{neighInFace}
        Let $G$ be an embedded graph, let $S\subseteq V(G)$ be nonempty, let $D$ be a $(G,S)$-clean disc, let $F$ be a face of $G-S$ containing the interior of $D$, and let $B_F$ be the boundary vertices of $F$ in $G-S$. Then there exists a spanning embedded supergraph $G'$ of $G$ such that:
         \begin{enumerate}
             \item $G'-S$ is the embedded graph $G-S$,
             \item $\FT(G)\subseteq \FT(G')$,
             \item $D$ is $(G',S)$-clean,
             \item for each face $F'$ of $G'$ contained in $F$, the boundary vertices of $F'$ are a clique in $G'$, and
             \item $B_F=N_{G'}(S)$.
         \end{enumerate}
    \end{lemma}
    
    \begin{proof}
        For a decidable triple $\scr{T}=(G^*,S,\scr{D})$, say that a face $F\in \scr{F}(\scr{T})$ satisfies $\chi$ in $\scr{T}$ if $F$ is disjoint to each $D\in \scr{D}$. Since this depends only on $F$ and the discs in $\scr{D}$, observe that $\chi$ is consistent. Let $\chi'$ be the clique criteria extension of $\chi$. By \cref{cliqueCriteria}, $\chi'$ is reducible.

        As $S$ is nonempty, $G$ is nonempty. Thus, observe that $\scr{T}:=(G,S,\{D\})$ is decidable. By \cref{facialCriteriaSpanning} (with $\chi:=\chi'$), there exists a spanning embedded supergraph $G'$ of $G$ such that: 
        \begin{enumerate}
            \item $D$ is $(G',S)$-clean
            \item $\scr{T}':=(G',S,\{D\})$ (is decidable and) satisfies $\chi'$,
            \item $\FT(G)\subseteq \FT(G')$,
            \item each edge of $E(G')\setminus E(G)$ is contained in a face of $G$ that does not satisfy $\chi$ in $\scr{T}$, and,
            \item each edge of $E(G')\setminus E(G)$ with no endpoint in $S$ is contained in a face of $G-S$ that does not contain a vertex of $S$.
        \end{enumerate}

        First, assume, for a contradiction, that $G'-S$ is not the plane graph $G-S$. Since $G'$ is a spanning embedded supergraph of $G$, it follows that there exists an edge $e\in E(G'-S)\setminus E(G-S)$. Note that $e\in E(G')\setminus E(G)$ and has no endpoint in $S$. Thus, $e$ is contained in a face $F'$ of $G$ that does not satisfy $\chi$ in $\scr{T}$, and a face $F''$ of $G-S$ that does not contain a vertex of $S$.
        
        Since $F'$ does not satisfy $\chi$ in $\scr{T}$, $F'$ intersects $D$. Since $D$ is $(G,S)$-clean (and is thus $(G-S)$-clean), its interior is contained in a face $F'''$ of $G-S$. Since $F'$ intersects $D$, observe that $F'$, and thus $e$, is also contained in $F'''$. Since $e$ is contained in $F''$, it follows that $F''=F'''$. However, since $S$ is nonempty and since the interior of $D$ contains $S$ (as $D$ is $(G,S)$-clean), $F'''=F''$ contains a vertex of $S$, a contradiction. It follows that $G'-S$ is the embedded graph $G-S$.

        Next, consider a face $F'$ in $G'$ contained in $F$. Let $B_F'$ be the boundary vertices of $F'$. Since $G'-S$ is the embedded graph $G-S$, observe that either $F'=F$ or $B_F'$ intersects $S$. However, since $F$ contains the interior of $D$, which contains $S$, which is nonempty, $F$ intersects $S$. $F'$ does not intersect $S$ (as it is a face of $G'$ and $S\subseteq V(G')$), so it follows that $F'\neq F$. So $B_F'$ intersects $S$. Since the interior of $D$ contains $S$, it follows that $D$ intersects $F'$.

        Since $\scr{T}'$ satisfies $\chi'$, $F'$ satisfies $\chi'$ in $\scr{T}'$. $F'$ intersects $D$, so $F'$ does not satisfy $\chi$. Thus and since $S\neq \emptyset$, it follows that $B_F'$ is a clique in $G'$.
        
        It remains only to show that $B_F=N_{G'}(S)$. Since $S$ is contained in the interior of $S$, which is contained in $F$ (a face of $G-S$ whose boundary vertices are $B_F$), and since $G'$ is an embedded supergraph of $G-S$, observe that $N_{G'}(S)\subseteq B_F$. So we only need to show that $B_F\subseteq N_{G'}(S)$.
        
        Fix $v\in B_F$. Since $F$ is a face of $G-S$, note that $v\notin S$. Since $G-S$ is the embedded graph $G'-S$, $F$ is a face of $G'-S$ (with $v$ as a boundary vertex). Thus, note that there is a face $F'$ of $G'$ contained in $F$ whose boundary vertices $B_F'$ contain $v$. By the same argument as above, $B_F'$ intersects $S$. As $F'$ is contained in $F$, $B_F'$ is a clique in $G'$. Thus and since $v\notin S$, we obtain $v\in N_{G'}(S)$. Hence, $B_F\subseteq N_{G'}(S)$.

        This completes the proof of the lemma.
    \end{proof}

    We can now prove \cref{neighBoundary} as an easy corollary.

    \neighBoundary*
    \begin{proof}
        Since $D$ is $(G,S)$-clean (and thus $(G-S)$-clean), observe that the interior of $D$ is contained in a face $F$ of $G-S$. Let $B_F$ be the boundary vertices of this face. Observe that $B(D,G,S)=B(D,G-S)\subseteq B_F$. By \cref{neighInFace}, there exists a spanning embedded supergraph $G'$ of $G$ such that:
        \begin{enumerate}
            \item $G'-S$ is the embedded graph $G-S$,
            \item $\FT(G)\subseteq \FT(G')$,
            \item $D$ is $(G',S)$-clean, and,
            \item $B(D,G,S)\subseteq B_F=N_{G'}(S)$.
        \end{enumerate}
        This completes the proof.
    \end{proof}

    \subsection{Backtracking}

    We can now complete the proof of \cref{findReductionModify}. It's important to remember that, starting from the original almost-embedding $\Gamma$, we had to extend $\Gamma$ twice before we could find a reduction. Thankfully, we can easily reverse this process.

    \begin{restatable}{observation}{extendExtend}
        \label{extendExtend}
        Let $\Gamma$ be an almost-embedding, let $\Gamma'$ extend $\Gamma$, and let $\Gamma''$ extend $\Gamma'$. Then $\Gamma''$ extends $\Gamma$.
    \end{restatable}

    \begin{proof}
        Let $\Gamma:=(G,\Sigma,G_0,\scr{D},H,\scr{J},A)$. Observe that $\Gamma'=(G',\Sigma',G_0',\scr{D}',H,\scr{J},A)$, where:
        \begin{enumerate}
            \item $G_0$ is a spanning subgraph of $G_0'$,
            \item $\FT(G_0)\subseteq \FT(G_0')$, and,
            \item $U(G_0,\scr{D})$ is a spanning subgraph of $U(G_0',\scr{D}')$.
        \end{enumerate}
        Thus, $\Gamma''=(G',\Sigma',G_0',\scr{D}',H,\scr{J},A)$, where:
        \begin{enumerate}
            \item $G_0'$ is a spanning subgraph of $G_0''$,
            \item $\FT(G_0')\subseteq \FT(G_0'')$, and,
            \item $U(G_0',\scr{D}')$ is a spanning subgraph of $U(G_0'',\scr{D}'')$.
        \end{enumerate}
        It follows that $G_0$ is a spanning subgraph of $G_0''$, that $\FT(G_0)\subseteq \FT(G_0'')$, and that $U(G_0,\scr{D})$ is a spanning subgraph of $U(G_0'',\scr{D}'')$. The result follows ($\Gamma''$ satisfies all other requirements automatically).
    \end{proof}

    \begin{restatable}{observation}{extendsWithReduction}
        \label{extendsWithReduction}
        Let $\Gamma$ be an almost-embedding, let $\Gamma'$ extend $\Gamma$, and let $\scr{R}$ be a reduction of $\Gamma'$. Then $\scr{R}$ is a reduction of $\Gamma$ of the same vortex-width and loss, and at most the same remainder-width.
    \end{restatable}

    \begin{proof}
        Let $\Gamma:=(G,\Sigma,G_0,\scr{D},H,\scr{J},A)$. Observe that $\Gamma'=(G',\Sigma',G_0',\scr{D}',H,\scr{J},A)$, where:
        \begin{enumerate}
            \item $G_0$ is a spanning subgraph of $G_0'$,
            \item $\FT(G_0)\subseteq \FT(G_0')$, and,
            \item $U(G_0,\scr{D})$ is a spanning subgraph of $U(G_0',\scr{D}')$.
        \end{enumerate}
        By \cref{extendsIsSpanning}, $G'$ is a spanning supergraph of $G$. In particular, note that for each $v\in V(G-X)=V(G'-X)$, $N_{G-X}[v]\subseteq N_{G'-X}$.

        Let $(V_M,X,G'',G_0^+,W,D,H'',\scr{J}'',\phi):=\scr{R}$. Using the above properties and the fact that $\scr{R}$ is a reduction of $\Gamma'$, all of the following are immediate.
        \item $V_M\subseteq V(G_0')=V(G_0)$,
        \item $A\subseteq X\subseteq V(G')=V(G)$,
        \item $(G'',G_0^+,W,D,H'',\scr{J}'')$ is a standardised plane+quasi-vortex embedding,
        \item $V_M$ is disjoint to $V(G'')$,
        \item $H-(X\cap V(H))\subseteq H''\subseteq H$, and,
        \item $\phi$ is a map from $V(G'')$ to $V(G')=V(G)$ such that:
        \begin{enumerate}
            \item for each $v\in V(G-X)=V(G'-X)$, $N_{G-X}[v]\subseteq N_{G'-X}[v]\subseteq \phi(N_{G''}[\phi^{-1}(v)])$,
            \item $\phi$ is the identity on $V(G')\setminus V(W)$,
            \item $\phi(V(W))=V_M$, and,
            \item for each $K\in \FT(\Gamma)=\FT(G_0)\subseteq \FT(G_0')$ disjoint to $X$, there exists $K'\in \FT(G_0^+)$ with $\phi(K')=K$.
        \end{enumerate}
        Thus, $\scr{R}$ is a reduction for $\Gamma$.

        The vortex-width and loss of $\scr{R}$ as a reduction of $\Gamma$ are the same as the vortex-width and loss respectively of $\scr{R}$ as a reduction of $\Gamma'$, being the width of $\scr{J}''$ and $|X|$ respectively. So it remains only to check the remainder-width.

        Let $w$ be the remainder-width of $\scr{R}$ as a reduction of $\Gamma'$. Observe that each connected component $C$ of $G-X$ is contained in a connected component $C'$ of $G'-X$ (as $G$ is a spanning subgraph of $G'$). Note that $|V_M\cap V(C)|\leq |V_M\cap V(C')|\leq w$. So the remainder-width of $\scr{R}$ as a reduction of $\Gamma$ is at most $w$, which is the remainder-width of $\scr{R}$ as a reduction of $\Gamma'$. This completes the proof.
    \end{proof}

    We can apply these observations, along with the results of all the previous subsections, to prove \cref{findReductionModify}.

    \findReductionModify*

    \begin{proof}
        By \cref{findStandard}, there exists a standardised $(g+2p,p,k,a)$-almost-embedding $\Gamma'$ with disjoint discs that extends $\Gamma$. By \cref{findAECutting} (with $\Gamma:=\Gamma'$ and $g:=g+2p$), there exists a weakly-standardised $(g,2g+5p,k,a)$-almost-embedding $\Gamma''$, an anchoring cutting subgraph $M$ of $\Gamma''$, and a layering $\scr{L}$ of the graph $G''$ of $\Gamma''$ such that:
        \begin{enumerate}
            \item $\Gamma''$ extends $\Gamma'$,
            \item each layer of $\scr{L}$ contains at most $2g+5p+1$ vertices of $M$, and,
            \item $M$ has boundary intersection at most $4g+9p$.
        \end{enumerate}

        By \cref{cutAE} (with $\Gamma:=\Gamma''$ and $m:=4g+9p$), $\Gamma''$ admits a reduction $\scr{R}$ of vortex-width at most $k$ and loss at most $(k+1)(4g+9p)+a$ whose split-vertices are $V(M)$.

        By \cref{extendExtend}, $\Gamma''$ extends $\Gamma$, and $G''$ is a spanning supergraph of $G$. By \cref{extendsWithReduction}, $\scr{R}$ is a reduction of $\Gamma$ of vortex-width at most $k$ and loss at most $(k+1)(4g+9p)+a$ (whose split-vertices $V_M$ are $V(M)$). By \cref{layeringSupergraph}, $\scr{L}$ is a layering for $G$. For each layer $L\in \scr{L}$, we have $|L\cap V_M|=|L\cap V(M)|\leq 2g+5p+1$. This completes the proof.
    \end{proof}

    \section{Finding and applying adjustments}
    \label{secAdjustmentLemmas}

    In this section, we prove \cref{applyAdjustment} and \cref{planePartitions}, starting with the former.

    \subsection{Applying an adjustment}
    \label{secApplyAdj}

    The proof of \cref{applyAdjustment} essentially boils down to showing that we can `modify' the quasi-vortex to account for the vertices (in $S$) that were deleted. The extra edges added to $G-S$ don't cause much extra difficulty. We do this using the extra same method as when we had to modify the vertex to avoid vertices in the cut in \cref{cutAE}. So the proof becomes a process of applying the same modification, and then checking that the extra edges don't cause issues (mainly, that facial triangles are preserved, which is accounted for by a condition of the adjustment).

    Before we get to the proof, we need the following observation.

    \begin{restatable}{observation}{underlyingPreserved}
        \label{underlyingPreserved}
        Let $G,G'$ be embedded graphs in the same surface, let $D$ be a disc that is $G$-clean and $G'$-clean, and let $S\subseteq V(G)$ be such that $G'$ is a spanning embedded supergraph of $G-S$. Then $B(D,G')=B(D,G)\setminus S$, and $U(D,G)-(B(D,G)\cap S)$ is a spanning subgraph of $U(D,G')$.
    \end{restatable}

    \begin{proof}
        Let $(B,U):=(B(\scr{D},G),U(\scr{D},G))$, $S_B:=B\cap S$, and $(B',U'):=(B(\scr{D},G'),U(\scr{D},G'))$. So $V(U)=B$ and $V(U')=B'$. We must show that $B'=B\setminus S=B\setminus S_B$, and that $U-S_B$ is a spanning subgraph of $U'$.

        Since $D$ is $G'$-clean, observe that $D\cap G'=\partial D\cap G'=\partial D\cap V(G')$. Since $G'$ is an embedded supergraph of $G'-S$, we have $\partial D\cap V(G')=\partial D\cap V(G-S)=(\partial D\cap V(G))\setminus S=B\setminus S$. Thus, $B'=B\setminus S$.

        Let $e$ be an edge of $U-S_B$. So there exists an arc $C$ of $D$ in $G$ whose endpoints are disjoint to $S_B$ (and thus $S$). Since $\partial D\cap V(G')=(\partial D\cap V(G))\setminus S$, observe that $C$ is an arc of $D$ in $G'$. Thus, $e$ is an edge of $U'$. It follows that $U'$ is a spanning supergraph of $U-S_B$, as desired.
    \end{proof}

    We are now ready to prove \cref{applyAdjustment}.

    \applyAdjustment*

    \begin{proof}
        Let $(G,\Sigma,G_0,\scr{D},H,\scr{J},A):=\Gamma$, let $(V_M,X,G',G_0^+,W,D,H',\scr{J}',\phi):=\scr{R}$, and let $G_0':=G_0'(\scr{R})$. So $G_0'$ is the plane graph $G_0^+-V(W)$, and $D$ is $G_0'$-clean. Also, $\scr{A}$ is an adjustment for the adjustable triple $(G_0^+,W,D)$. Let $\Lambda:=\Lambda(\scr{R})=(G',G_0^+,W,D,H',\scr{J}')$, let $B':=B'(\scr{R})=B(D,G_0^+,V(W))=B(D,G_0')=V(G_0^+\cap H')$, $B^-:=B'\setminus S$, and let $U':=U'(\scr{R})=U(D,G_0^+,V(W))=U(D,G_0')$.
        
        Let $S_B:=S\cap B'$. Since $S_B\subseteq B'=V(U')$, observe that $(X_H,H^-,\scr{J}^-):=(H',U',\scr{J}')-S_B$ is well-defined. By \cref{modifyVortex}, we have
        \begin{enumerate}
            \item $S_B\subseteq X_H\subseteq V(H')$,
            \item $H'-X_H\subseteq H^-\subseteq H'-S$,
            \item $|X_H|\leq (k+1)|S_B|$,
            \item $\scr{GD}':=(H^-,U'-S_B,\scr{J}^-)$ is a planted graph-decomposition whose width is at most the width of $(H',U',\scr{J}')$ (which is the vortex-width of $\scr{R}$),
            \item $V(G_0^+\cap H^-)=V(U'-S_B)=B'\setminus S_B$, and
            \item $\scr{GD}'$ is smooth in $G_0'$.
        \end{enumerate}
        
        Since $S\subseteq V(G_0^+)\subseteq V(G')$, $\phi(S)$ is well-defined. Set $X^+:=(X\cup X_H \cup \phi(S))$ and $G^:=G_0^-\cup H^-$. Since $V(G_0^-)=V(G_0^+-S)\subseteq V(G_0^+)\subseteq V(G')$ and since $V(H^-)\subseteq V(H'-S)\subseteq V(H')\subseteq V(G')$, observe that $V(G^-)\subseteq V(G')$. Thus, the restriction $\phi^-$ of $\phi$ onto $V(G^-)$ is well-defined. Set $\scr{R}':=(V_M,X^+,G^-,G_0^-,W,D,H^-,\scr{J}^-,\phi^-)$.
        
        We show that $\scr{R}'$ is the desired reduction. We start by showing that $\Lambda^-:=(G^-,G_0^-,W,D,H^-,\scr{J}^-)$ is a standardised plane+quasi-vortex embedding.

        Recall that $S\subseteq V(G_0^+)\setminus V(W)$, and that $G_0^-$ is a spanning plane supergraph of $G_0^+-S$. Thus, observe that $W$ is a connected and nonempty plane subgraph of $G_0^-$. Recall that $D'$ is $(G_0^-,V(W))$-clean. Since $W$ is disjoint to $H'\supseteq H^-$, $W$ is disjoint to $H^-$. Since $G_0^-$ is spanning supergraph of $G_0^+-S$, observe that $N_{G_0^-}(V(W))\supseteq N_{G_0^+}(V(W))\setminus S\supseteq B^-$.
        
        Observe that $V(G_0^-\cap H^-)=V((G_0^+-S)\cap H^-)=V(G_0^+\cap H^-)\setminus S=(B'\setminus S_B)\setminus S=B'\setminus S_B=B^-$ (as $S_B=S\cap B'$). So $V(G_0^-\cap H^-)=B^-\subseteq N_{G_0^-}(V(W))$.

        Let $G_0^*$ be the plane graph $G_0^--V(W)$. Observe that $G_0^*$ is a spanning plane supergraph of $G_0^+-V(W)-S=G_0'-S$. By \cref{underlyingPreserved}, $B(D,G_0^-,V(W))=B(D,G_0^*)=B(D,G_0')\setminus S=B^-=V(G_0^-\cap H^-)$, and $U^-:=U(D,G_0^-,V(W))=U(D,G_0^*)$ is a spanning supergraph of $U(D,G_0')-S_B=U'-S_B$. By \cref{smoothDelete} (noting that $X:=S$ is disjoint to $B^-=V(U'-S)$), $\scr{GD}'$ is smooth in $G_0'-S$. Thus, by \cref{spanningSmooth}, $\scr{GD}^-:=(H^-,U^-,\scr{J}^-)$ is a planted graph-decomposition that is smooth in $G_0^*$ and whose width is at most the vortex-width of $\scr{R}$.
        
        It follows that $\Lambda^-$ is a standardised plane+quasi-vortex embedding, with $G_0(\Lambda^-)=G_0^*$, $B(\Lambda^-)=B^-$, and $U(\Lambda^-)=U^-$.

        Recall that $V_M\subseteq V(G_0)$, and that $A\subseteq X\subseteq V(G)$. Since $X_H\subseteq V(H')\subseteq V(H)\subseteq V(G)$ and since $\phi$ is a map to $V(G)$, it follows that $A\subseteq X^+\subseteq V(G)$. Recall that $V_M$ is disjoint to $V(G)\supseteq V(G^-)$. So $V_M$ is disjoint to $V(G^-)$.
        
        Since $X_H\subseteq V(H')$, observe that $X^+\cap V(H)\supseteq X_H\cup (X\cap V(H'))$. Recall also that $H-(X\cap V(H))\subseteq H'\subseteq H$. Thus, observe that $H-(X^+\cap V(H))\subseteq H-((X\cap V(H))\cup X_H)\subseteq H'-X_H\subseteq H^-\subseteq H'\subseteq H$. So $H-(X^+\cap V(H))\subseteq H^-\subseteq H$.
        
        Since $\phi$ is a map from $V(G')$ to $V(G)$, $\phi^-$ is a map from $V(G^-)$ to $V(G)$. Since $V(W)\subseteq V(G_0^-)\subseteq V(G^-)$, $\phi^-(V(W))$ is well-defined and $\phi^-(V(W))=V(W)$. Since $V(G^-)\subseteq V(G)$, we find that $\phi^-$ is the identity on $V(G^-)\setminus V(W)\subseteq V(G')\setminus V(W)$.

        Let $K\in \FT(G_0)$ be disjoint to $X^+$. So $K$ is disjoint to $X$, and thus exists $K'\in \FT(G_0^+)$ such that $\phi(K')=K$. Since $\phi(S)\subseteq X^+$, $K'$ is disjoint to $S$. Thus (by definition of $\scr{A}$), $K'\in \FT(G_0^-)$. This implies $K'\subseteq V(G_0^-)\subseteq V(G^-)$, so $\phi^-(K')$ is well-defined, and $\phi^-(K')=\phi(K')=K$.

        Fix $v\in V(G-X^+)$. We must show that $N_{G-X^+}[v]\subseteq \phi^-(N_{G^-}[(\phi^-)^{-1}(v)])$. Since $X\subseteq X^+$, $v\in V(G-X)$. Thus, $N_{G-X}[v]\subseteq \phi(N_{G'}[\phi^{-1}(v)])$. Since $v\notin X^+$, observe that $N_{G-X^+}[v]\subseteq N_{G-X}[v]\setminus X^+\subseteq \phi(N_{G'}[\phi^{-1}(v)])\setminus X^+\subseteq \phi(N_{G'}[\phi^{-1}(v)])\setminus (X_H\cup \phi(S))$. So $N_{G-X^+}[v]\subseteq \phi(N_{G'}[\phi^{-1}(v)])\setminus (X_H\cup \phi(S))$.

        Since $X_H\subseteq V(H')$, which is disjoint to $V(W)$, $\phi(X_H)=X_H$. Thus, observe that $\phi(N_{G'}[\phi^{-1}(v)])\setminus (X_H\cup \phi(S))\subseteq \phi(N_{G'}[\phi^{-1}(v)]\setminus (X_H\cup S))$. Observe that $\phi^{-1}(v)$ is disjoint to $X_H\cup S$ as $v\notin \phi(X_H\cup S)=X_H\cup \phi(S)\subseteq X^+$. Thus, $N_{G'}[\phi^{-1}(v)]\subseteq N_{G'-(X_H\cup S)}[\phi^{-1}(v)]$. So $N_{G-X^+}[v]\subseteq \phi(N_{G'-(X_H\cup S)}[\phi^{-1}(v)])$.
        
        Observe that $G'-(X_H\cup S)\subseteq (G_0^+-S)\cup (H'-X_H)\subseteq G_0^-\cup H^-=H$. Thus, $N_{G'-(X_H\cup S)}[\phi^{-1}(v)]\subseteq N_{G^-}[\phi^{-1}(v)]$. So $N_{G-X^+}[v]\subseteq \phi(N_{G^-}[\phi^{-1}(v)])$. In particular, we find that $\phi^{-1}(v)\subseteq V(G^-)$. So $\phi^{-1}(v)=(\phi^-)^{-1}(v)$. Also, since $N_{G^-}[\phi^{-1}(v)]\subseteq V(G^-)$, $\phi(N_{G^-}[\phi^{-1}(v)])=\phi^-(N_{G^-}[\phi^{-1}(v)])$. Thus, $\phi(N_{G^-}[\phi^{-1}(v)])=\phi^-(N_{G^-}[(\phi^-)^{-1}(v)])$. Hence, $N_{G-X^+}[v]\subseteq \phi^-(N_{G^-}[(\phi^-)^{-1}(v)])$.
        
        It follows that $\scr{R}'$ is a reduction for $\Gamma$. Note that the plane subgraph of $\scr{R}'$ is $G_0^-$, and that the split vertices, obstruction subgraph, and disc of $\scr{R}'$ are $V_M,W,D$ respectively. The vortex-width of $\scr{R}'$ is the width of $\scr{J}^-$, which is at most $k$. Observe that each component $C^-$ of $G-X^+$ is contained in component $C$ of $G-X$ (as $X\subseteq X^+$), so $|V(C^-)\cap V_M|$ is at most the remainder-width of $\scr{R}$. Thus, the remainder-width of $\scr{R}'$ is at most the remainder-width of $\scr{R}$.

        Recalling that $|X_H|\leq (k+1)|S_B|=(k+1)(S\cap B')$ and that $S\cap B'=S_B\subseteq X_H$, we obtain $|S\cup X_H|\leq |S\setminus B'|+(k+1)|S\cap B'|\leq (k+1)|S|$ as $k\geq 0$ (since $k\in \ds{N}$). So $|X^+|\leq |X|+(k+1)|S|$. Since $\scr{R}$ has loss at most $q\in \ds{R}_0^+$, $|X|\leq q$, and since $\scr{A}$ has loss at most $q'\in \ds{R}_0^+$, $|S|\leq q'$. Thus, $|X^+|\leq q+(k+1)q'$. So $\scr{R}'$ has loss at most $q+(k+1)q'$. This completes the proof.
    \end{proof}

    \subsection{Finding an adjustment}

    The final main step is to prove \cref{planePartitions}, for which we need to find adjustments of treewidth $2$. As mentioned before, we use a very slight variation of \cref{tw2Surfaces}. The catch is that we need to be in a `triangulation-like' setting. Specifically, we have the following definition.

    Say that a plane graph $G$ is a \defn{quasi-triangulation} if for each face $F$ of $G$, the boundary vertices of $F$ are a clique in $G$. We remark (without proof) that if a plane quasi-triangulation $G$ is simple with $V(G)\geq 3$, then $G$ is a triangulation.

    The algorithm behind \cref{tw2Surfaces} is to repeatedly find small `separators' (see \cref{SecSeparators} for a definition) between two large parts, grow these separators by a bounded amount so that the remaining regions of the graph also admit small separators, and then repeat. These grown separators end up as the parts in the graph. In a quasi-triangulation, we find that the separators induce connected subgraphs (see \cref{sepCC}), and the method of `growing' doesn't destroy connectivity. So repeating the proof on a quasi-triangulation gives a connected partition. Because this really is just repeating the exact same proof, we leave it to the appendix. See \cref{conPlanarLemma}.
    
    The proof that these separators induce connected subgraphs is also a rather tedious detour. Since similar lemmas have appeared in other papers (for example, \citet{Hendrey2025}), and since this lemma is really just an observation about plane graphs, we leave it to the appendix as well. See \cref{SecSeparators}.

    The end result is the following lemma.

    \begin{restatable}{lemma}{conPlanar}
        \label{conPlanar}
        Let $G$ be a plane quasi-triangulation, and let $W$ be a connected and nonempty subgraph of $G$, and let $n:=|V(G)|-|V(W)|+1$. Then $G$ has admits a connected partition $\scr{P}$ of treewidth at most $2$ such that $V(W)\in \scr{P}$ and each part of $\scr{P}$ other than $V(W)$ has size at most $12\sqrt{3n}$.
    \end{restatable}

    Overall, proving \cref{conPlanar} is the subject of \cref{SecQT}.

    We now just need the ability to take an adjustable triple $(G,W,D)$ and `convert' $G$ to a plane quasi-triangulation while respecting $D$. More specifically, we have the following lemma.

    \begin{lemma}
        \label{makeTriangulation}
        Let $G$ be a plane graph, let $S\subseteq V(G)$ be nonempty, and let $D$ be a $(G,S)$-clean disc. Then there exists a spanning plane supergraph $G'$ of $G$ such that:
        \begin{enumerate}
            \item $G'$ is a plane quasi-triangulation,
            \item $D$ is $(G',S)$-clean, and,
            \item $\FT(G)\subseteq \FT(G')$.
        \end{enumerate}
    \end{lemma}
    
    \begin{proof}
        Let $\chi$ be the facial criterion such that, for each decidable triple $\scr{T}=(G^*,S,\scr{D})$, and each face $F\in \scr{F}(\scr{T})$, $F$ does not satisfy $\chi$ in $\scr{T}$. $\chi$ is trivially consistent. Let $\chi'$ be the clique criteria extension of $\chi$. By \cref{cliqueCriteria}, $\chi'$ is reducible.

        As $S$ is nonempty, $G$ is nonempty. Thus, observe that $\scr{T}:=(G,S,\{D\})$ is decidable. By \cref{facialCriteriaSpanning} (with $\chi:=\chi'$), there exists a spanning embedded (plane) supergraph $G'$ of $G$ such that: 
        \begin{enumerate}
            \item $D$ is $(G',S)$-clean
            \item $\scr{T}':=(G',S,\{D\})$ (is decidable and) satisfies $\chi'$, and,
            \item $\FT(G)\subseteq \FT(G')$.
        \end{enumerate}

        It remains only to show that $G'$ is a plane quasi-triangulation. Consider a face $F$ of $G'$. Since $\scr{T}'$ satisfies $\chi'$, $F$ satisfies $\chi'$ in $\scr{T}'$. $F$ does not satisfy $\chi$ in $\scr{T}'$ and $S\neq \emptyset$, it follows that the boundary vertices of $F$ are a clique in $G'$. So $G'$ is a plane quasi-triangulation. This completes the proof.
    \end{proof}

    We can now prove \cref{planePartitions}.

    \planePartitions*

    \begin{proof}

        Fix $n\in \ds{N}$, $d\in \ds{R}^+$ with $d\geq 12\sqrt{3n}$, and let $(G,W,D)$ be an adjustable triple with $|V(G)|\leq n+|V(W)|-1$. So $G$ is a plane graph, $W$ is a connected and nonempty plane subgraph of $G$, and $D$ is a $(G,V(W))$-clean disc. We must show that $(G,W,D)$ admits an adjustment treewidth at most $2$, near-width at most $d$, and loss at most $0$.
        
        By \cref{makeTriangulation} (with $S:=V(W)$, which is nonempty), there exists a spanning plane supergraph $G'$ of $G$ such that:
        \begin{enumerate}
            \item $G'$ is a plane quasi-triangulation,
            \item $\FT(G)\subseteq \FT(G')$, and,
            \item $D$ is $(G',V(W))$-clean.
        \end{enumerate}
        Note that $W$ is a connected and nonempty subgraph of $G'$. Let $n':=|V(G')|-|V(W)|+1$. Observe that $n'=|V(G)|-|V(W)|+1\leq n$.
        
        By \cref{conPlanar}, $G'$ admits a connected partition $\scr{P}$ of treewidth at most $2$ such that $V(W)\in \scr{P}$ and each part of $\scr{P}$ other than $V(W)$ has size at most $12\sqrt{3n'}\leq 12\sqrt{3n}\leq d$. Thus, $(G',\emptyset,\scr{P})$ is an adjustment of $(G,W,D)$ of treewidth at most $2$, near-width at most $d$, and loss at most $0$.
        
        It follows that $(f,g)$ generate adjustments of treewidth 2. This completes the proof.
    \end{proof}

    We remind the reader that apart from the proofs of \cref{findCutting}, \cref{cuttingLemmaSimple}, and \cref{conPlanar}, which we leave to the appendices, \cref{planePartitions} was the last lemma we needed to show. Thus, excluding the appendices, we have completed the proof of \cref{main}.

    \paragraph{Acknowledgements.} The author thanks David Wood for his supervision and suggestions to improve this paper.

    \bibliography{Ref.bib}

\begin{thebibliography}{15}
\providecommand{\natexlab}[1]{#1}
\providecommand{\url}[1]{\texttt{#1}}
\expandafter\ifx\csname urlstyle\endcsname\relax
  \providecommand{\doi}[1]{doi: #1}\else
  \providecommand{\doi}{doi: \begingroup \urlstyle{rm}\Url}\fi

\bibitem[Alon et~al.(1990)Alon, Seymour, and Thomas]{Alon1990}
Noga Alon, Paul Seymour, and Robin Thomas.
\newblock A separator theorem for nonplanar graphs.
\newblock \emph{J. Amer. Math. Soc.}, 3\penalty0 (4):\penalty0 801--808, 1990.

\bibitem[Diestel(2018)]{Diestel2018}
Reinhard Diestel.
\newblock \emph{Graph theory}, volume 173 of \emph{Graduate Texts in Mathematics}.
\newblock Springer, 5th edition, 2018.

\bibitem[Diestel et~al.(2012)Diestel, Kawarabayashi, M\"uller, and Wollan]{Diestel2012}
Reinhard Diestel, Ken-ichi Kawarabayashi, Theodor M\"uller, and Paul Wollan.
\newblock On the excluded minor structure theorem for graphs of large tree-width.
\newblock \emph{J. Combin. Theory Ser. B}, 102\penalty0 (6):\penalty0 1189--1210, 2012.
\newblock \doi{10.1016/j.jctb.2012.07.001}.

\bibitem[Distel et~al.(2022)Distel, Hickingbotham, Huynh, and Wood]{Distel2022Surfaces}
Marc Distel, Robert Hickingbotham, Tony Huynh, and David~R. Wood.
\newblock Improved product structure for graphs on surfaces.
\newblock \emph{Discrete Math. Theor. Comput. Sci.}, 24\penalty0 (2):\penalty0 Paper No. 6, 2022.

\bibitem[Distel et~al.(2024)Distel, Dujmovi\'c, Eppstein, Hickingbotham, Joret, Micek, Morin, Seweryn, and Wood]{Distel2024}
Marc Distel, Vida Dujmovi\'c, David Eppstein, Robert Hickingbotham, Gwena\"el Joret, Piotr Micek, Pat Morin, Micha\l~T. Seweryn, and David~R. Wood.
\newblock Product structure extension of the {A}lon--{S}eymour--{T}homas {T}heorem.
\newblock \emph{SIAM J. Discrete Math.}, 38\penalty0 (3):\penalty0 2095--2107, 2024.

\bibitem[Distel et~al.(2025)Distel, Dujmović, Joret, Micek, Morin, and Wood]{Distel2025}
Marc Distel, Vida Dujmović, Gwenaël Joret, Piotr Micek, Pat Morin, and David~R. Wood.
\newblock Planar graphs in blowups of fans, 2025.
\newblock arXiv:2407.05936.

\bibitem[Dujmovi\'{c} et~al.(2017)Dujmovi\'{c}, Morin, and Wood]{Dujmovic2017}
Vida Dujmovi\'{c}, Pat Morin, and David~R. Wood.
\newblock Layered separators in minor-closed graph classes with applications.
\newblock \emph{J. Combin. Theory Ser. B}, 127:\penalty0 111--147, 2017.

\bibitem[Dujmovi\'{c} et~al.(2020)Dujmovi\'{c}, Joret, Micek, Morin, Ueckerdt, and Wood]{rtwltwMCC}
Vida Dujmovi\'{c}, Gwena\"{e}l Joret, Piotr Micek, Pat Morin, Torsten Ueckerdt, and David~R. Wood.
\newblock Planar graphs have bounded queue-number.
\newblock \emph{J. ACM}, 67\penalty0 (4):\penalty0 Art. 22, 2020.

\bibitem[Eppstein(2000)]{Eppstein2000}
David Eppstein.
\newblock Diameter and treewidth in minor-closed graph families.
\newblock \emph{Algorithmica}, 27\penalty0 (3-4):\penalty0 275--291, 2000.

\bibitem[Hendrey and Wood(2025)]{Hendrey2025}
Kevin Hendrey and David~R. Wood.
\newblock Short paths in the planar graph product structure theorem, 2025.
\newblock arXiv:2502.01927.

\bibitem[Kostochka and Prince(2010)]{Kostochka2010}
A.~V. Kostochka and N.~Prince.
\newblock Dense graphs have {$K_{3,t}$} minors.
\newblock \emph{Discrete Math.}, 310\penalty0 (20):\penalty0 2637--2654, 2010.
\newblock \doi{10.1016/j.disc.2010.03.026}.

\bibitem[Linial et~al.(2008)Linial, Matou{\v{s}}ek, Sheffet, and Tardos]{Linial2008}
Nathan Linial, Ji{\v{r}}{\'i} Matou{\v{s}}ek, Or~Sheffet, and G\'abor Tardos.
\newblock Graph colouring with no large monochromatic components.
\newblock \emph{Combin. Probab. Comput.}, 17\penalty0 (4):\penalty0 577--589, 2008.

\bibitem[Lipton and Tarjan(1979)]{Lipton1979}
Richard~J. Lipton and Robert~E. Tarjan.
\newblock A separator theorem for planar graphs.
\newblock \emph{SIAM J. Appl. Math.}, 36\penalty0 (2):\penalty0 177--189, 1979.
\newblock \doi{10.1137/0136016}.

\bibitem[Mohar and Thomassen(2001)]{Mohar2001}
Bojan Mohar and Carsten Thomassen.
\newblock \emph{Graphs on surfaces}.
\newblock Johns Hopkins Studies in the Mathematical Sciences. Johns Hopkins University Press, Baltimore, MD, 2001.
\newblock ISBN 0-8018-6689-8.

\bibitem[Robertson and Seymour(2003)]{Robertson2003}
N.~Robertson and P.~D. Seymour.
\newblock Graph minors. {XVI}. {E}xcluding a non-planar graph.
\newblock \emph{J. Combin. Theory Ser. B}, 89\penalty0 (1):\penalty0 43--76, 2003.

\end{thebibliography}

    \appendix

    \section{Cutting on surfaces}
    \label{SecCutting}

    This section is devoted to the proofs of \cref{findCutting} and \cref{cuttingLemmaSimple}.

    \subsection{Finding a cutting subgraph}

    We first prove \cref{findCutting}. This is essentially the same proof as in \citet{rtwltwMCC} and \citet{Distel2022Surfaces}, but with a few minor extra details (such as the inclusion of the subtree $T'$).

    \findCutting*


    \begin{proof}
        Let $\Sigma$ be the surface $G$ is embedded in. So $g(\Sigma)=g$. Let $n,m,f$ be the number of vertices, edges, and faces of $G$ respectively. Since $G$ is $2$-cell embedded, by Euler's formula, $n-m+f=2-g$, and thus $m=n+f+g-2$.

        Let $D$ be the spanning subgraph of the dual graph of $G$ whose edge set is exactly the edges dual to the edges in $E(G)\setminus E(T)$. So $|V(D)|=f$ and $|E(D)|=|E(G)|-|E(T)|=m-(n-1)=f+g-1$.

        We show that $D$ is connected. Observe that it suffices to find, for each pair of faces $F_1,F_2$ in $D$, a path in $\Sigma$ from $F_1$ to $F_2$ disjoint to the embedding of $T$. The sequence of faces and edges that this path passes through then gives a walk in $D$ (provided the path does not intersect a vertex in $V(G)\setminus V(T)$ or intersect the same face on both sides of an intersection with an edge - the path can easily be modified so that neither of these scenarios occur).

        Since $\Sigma$ is path-connected, we can find paths $P_1,P_2$ from $F_1$ and $F_2$ respectively to $T$. Since $T$ is a tree, following the boundary of the embedding of $T$ forms a single closed loop. So we can link $P_1$ and $P_2$ by following the boundary of the embedding of $T$ to obtain the desired path. Thus, $D$ is connected, as claimed.

        Since $D$ is connected, we can find a spanning tree $T_D$ for $D$. Observe that $|E(T_D)|=|V(D)|-1=f-1$. Thus, $|E(D)\setminus E(T_D)|=(f+g-1)-(f-1)=g$. Let $E'$ be the set of edges in $G$ dual to the edges in $E(D)\setminus E(T_D)$. So $|E'|=|E(D)\setminus E(T_D)|=g$. Let $S$ be the set of endpoints of edges in $E'$. So $|Z|\leq 2|E'|\leq 2g$. For each $s\in S$, let $P_s$ be the vertical path in $T$ ending at $s$. Let $M$ be the embedded graph obtained from the union of $T'$, all the paths $P_s$ with $s\in S$, and all the edges in $E'$. Observe that $M$ is an embedded subgraph of $G$. Since $T'\cup \bigcup_{s\in S}P_s\subseteq T$ and since $|E'|=g$, observe that $|E(M)|=|V(M)|+g-1$.
        
        Next, we show that $M$ has exactly one face. Since $M\subseteq G$, for any two points $u,v\in \Sigma$ not in the embedding of $M$, we can find faces $F_u,F_v\in V(D)$ whose closures contain $u$ and $v$ respectively. Recall that $T_D$ is a spanning tree of $D$ whose edge-set is disjoint to the edges dual to $E(T)\cup E'\supseteq E(M)$. The path in $T_D$ from $F_u$ to $F_v$ therefore gives a path from $F_u$ to $F_v$ (and thus a path from $u$ to $v$) in $\Sigma$ that does not intersect $M$. Thus, the embedding of $M$ has exactly one face.
        
        Thus, $M$ has exactly one face, and $|E(M)|=|V(M)|+g-1$. So $M$ is a cutting subgraph for $G$. Recall that $M$ is an embedded subgraph of $G$, that $|E(M)|=|V(M)|+g-1$, and that $M$ is the union of $T'$, at most $2g$ vertical paths in $T$, and at most $g$ edges in $E(G)\setminus E(T)$. Since $T'$ is nonempty, so is $M$. This completes the proof.
    \end{proof}

    \subsection{Performing the cut}

    Next, we prove \cref{cuttingLemmaSimple}. This is more involved. We will give a more precise description of the list of properties we obtain from the `cutting' process, and then apply and filter these properties to obtain the properties we need.
    
    First, we need a precise way of describing the relationship between the original graph and the cut-graph.

    Let $G,G'$ be $2$-cell embedded graphs (on potentially different surfaces), and let $M,W$ be embedded subgraphs of $G,G'$ respectively. Let $\scr{F},\scr{F}'$ denote the faces of $G,G'$ respectively. We say that a triple $(\phi,\psi,\eta)$ is a \defn{projection} of $(G',W)$ onto $(G,M)$ if the following hold:
    \begin{enumerate}
        \item $\phi$ is a map $V(G')\mapsto V(G)$ such that:
        \begin{enumerate}
            \item $\phi$ is the identity on $V(G')\setminus V(W)$, and,
            \item $\phi(V(W))=V(M)$,
        \end{enumerate}
        \item $\psi$ is a map $E(G')\mapsto E(G)$ such that:
        \begin{enumerate}
            \item $\psi$ is the identity on $E(G')\setminus E(G'-V(W))$,
            \item the restriction of $\psi$ to $E(G')\setminus E(W)$ is a bijection between $E(G')\setminus E(W)$ and $E(G)\setminus E(M)$, and,
            \item for each $e\in E(G')$, if the endpoints of $e$ are $u',v'\in V(G')$, then the endpoints of $\psi(e)$ are $\phi(u')$ and $\phi(v')$, and
        \end{enumerate}
        \item $\eta$ is a map $\scr{F}\mapsto \scr{F}'$ such that:
        \begin{enumerate}
            \item $\eta$ is an injection, and,
            \item for each $F\in \scr{F}$, if the boundary walk of $\eta(F)$ is $e_1,v_1,\dots,e_n,v_n,e_1$, then the boundary walk of $F$ is $\psi(e_1),\phi(v_1),\dots,\psi(e_n),\phi(v_n),\psi(e_1)$.
        \end{enumerate}
    \end{enumerate}
    We remark that (3b) makes sense because both $G$ and $G'$ are $2$-cell embedded.

    The following lemma is then the `main idea' behind the cutting process. This proof is adapted from the proof used in \citet{rtwltwMCC} and \citet{Distel2022Surfaces}. We perform the proof more carefully.

    \begin{lemma}
        \label{cuttingLemmaTrue}
        Let $G$ be a graph $2$-cell embedded in a surface of genus $g\in \ds{N}$, and let $M$ be a embedded subgraph of $G$ with $|E(M)|=|V(M)|+g-1\geq 1$ and exactly one face. Then there exists a $2$-cell plane graph $G'$, a connected and nonempty plane subgraph $C$ of $G'$, and a projection $(\phi,\psi,\eta)$ of $(G',C)$ onto $(G,M)$ such that:
        \begin{enumerate}
            \item $G\cap G'=G-V(M)=G'-V(C)$ (as non-embedded graphs),
            \item $|V(G')|=|V(G)|+|V(M)|+2g-2$,
            \item $|E(G')|=|E(G)|+|E(M)|$,
            \item $|\phi^{-1}(v)|=\deg_M(v)$ for each $v\in V(M)$,
            \item $|\psi^{-1}(e)|=2$ for each $e\in V(M)$,
            \item $C$ is either a cycle, or consists of exactly two vertices and two parallel edges, and
            \item there is exactly one face $F'$ of $G'$ not in the image of $\eta$, and the boundary walk of $F'$ (when viewed as a graph) is $C$.
        \end{enumerate}
    \end{lemma}

    \begin{proof}
        For this proof, whenever we have a cyclic sequence of length $n$, indexed from $1$ to $n$, and we operate on an index in the sequence, we implicitly perform addition and subtraction mod $n$ (from $1$ to $n$).
    
        Note that, up to isomorphism (and subsequent relabelling of $V(G')$), we may assume that $V(G)\cap \ds{N}=\emptyset$.
        
        Recall that $G$ defines a (not unique) pair $((\pi_v:v\in V(G)), \lambda)$ that describes the embedding of $G$, where $\lambda:E(G)\mapsto \{-1,1\}$ are the edges signatures and, for each $v\in V(G)$, $\pi_v:E_v\mapsto E_v$ is a cyclic ordering of the edges $E_v$ incident to $v$ in $G$ (where we consider loops to be represented twice in $E_v$, once for each direction). If $\lambda_M$ is the restriction of $\lambda$ to $E(M)$ and, for each $v\in V(M)$, $\pi_{v,M}$ is the restriction of the cyclic sequence to the edges in $E(M)$, observe that $((\pi_{v,M}:v\in V(M)), \lambda_M)$ describes the embedding of $M$.
    
        Let $f$ be the number of faces of $M$. By Euler's formula, $|E(M)|\leq |V(M)|+f+g-2$, with equality if and only if $M$ is $2$-cell embedded. By assumption on $M$, $f=1$ and $|E(M)|=|V(M)|+g-1$. So $|V(M)|+f+g-2=|V(M)|+g-1=|E(M)|$. Thus, $M$ is $2$-cell embedded. 
        
        By assumption, there is exactly one face $F_M$ of $M$. Since $M$ is $2$-cell embedded, this face is homeomorphic to a disc. So $F_M$ has a boundary walk. Let $e_1,v_1,e_2,v_2,\dots,e_n,v_n,e_1$ be the boundary walk of $F_M$ (with $e_1,\dots,e_n\in E(M)$ and $v_1,\dots,v_n\in V(M)$). Since $F_M$ is the only face of $M$, each $e\in E(M)$ appears exactly twice. So $n=2|E(M)|=2|V(M)|+2g-2$. Thus and since $|E(M)|\geq 1$, each $v\in V(M)$ appears exactly $\deg_M(v)$ times. Note that if $\lambda_i:=\prod_{j=1}^i\lambda_M(e_j)=\prod_{j=1}^i \lambda(e_j)$, then $\lambda_n=1$ and, for each $i\in \{1,\dots,n\}$, $e_{i+1}=\pi^{\lambda_{i,M}}_{M,v_i}(e_i)=\pi^{\lambda_i}_{M,v_i}(e_i)$.

        Let $E_{G,1}=E(G)\setminus (E(G[V(M)])\cup E(G-V(M)))$, $E_{G,2}=E(G[V(M)])\setminus E(M)$. Observe that for each $i\in \{1,2\}$ and each $e\in E_i$, $e$ has exactly $i$ endpoints in $V(M)$. Note also that $E_{G,1}\sqcup E_{G,2}\sqcup E(G-V(M))=E(G)\setminus E(M)$.
        
        For each $i\in \{1,\dots,n\}$, let $E_i$ be the set of edges strictly between $e_i$ and $e_{i+1}$ in $\pi^{\lambda_i}_{v_i}$. Note that $E_i\subseteq E_{G,1}\cup E_{G,2}$. Further, if $E_i$ is nonempty, there exists (not necessarily distinct) $e_{i,1},e_{i,2}\in E_i$ such that $\pi_{v_i}^{\lambda_i}({e_i})=e_{i,1}$ and $\pi_{v_i}^{\lambda_i}({e_{i,2}})=e_{i+1}$.

        For each $e\in E_{G,1}$, observe that $e\in E_i$ for exactly one $i_e:=i\in \{1,\dots,n\}$. In this case, if $e$ has endpoints $u_e$ and $v_e$ with $u_e\in V(G)\setminus V(M)$ and $v_e\in V(M)$, then $v_e=v_{i_e}$. For each $e\in E_{G,2}$, observe that $e\in E_i$ for exactly two $i\in \{1,\dots,n\}$. Let $\{i_e,j_e\}$ be these indices. So if $e=\{u,v\}$, then (up to swapping $u$ and $v$), $u=v_{i_e}$ and $v=v_{j_e}$.
        
        We will construct a graph $G'$ whose vertex set is $\{1,\dots,n\}\cup (V(G)\setminus V(M))$, along with maps $\psi:E(G')\mapsto E(G)$ and $\psi':E(G)\setminus E(M)\mapsto E(G')$. We start by, for each $e\in E(G-V(M))$, adding $e$ to $G'$ and setting $\psi(e):=\psi'(e):=e$. Then, for each $e\in E_{G,1}$, we add an edge $e'$ to $G'$ with endpoints $u_e$ and $i_e$, and we let $\psi'(e):=e'$ and $\psi(e'):=e$. Similarly, for each $e\in E_{G,2}$, we add an edge $e'$ to $G'$ with endpoints $i_e$ and $j_e$, and we let $\psi'(e):=e'$ and $\psi(e')=e$. Since $E_{G,1}\sqcup E_{G,2}\sqcup E(G-V(M))=E(G)\setminus E(M)$, this completes the definition of $\psi'$. Observe that $\psi'$ is an injection. Finally, for each $i\in \{1,\dots,n\}$, we add an edge $e_i'$ with endpoints $i-1$ and $i$, and set $\psi(e_i'):=e_i$.
        
        Let $C$ be the graph with vertex set $\{1,\dots,n\}$ and edge set $\{e_i:i\in \{1,\dots,n\}\}$. If $n\geq 3$, observe that $C$ is a cycle. Otherwise, since $n=2|E(M)|$ and since $|E(M)|\geq 1$, observe that $n=2$ and $C$ consists of exactly two vertices and two parallel edges. In particular, $C$ is nonempty and connected. Note also that $C\subseteq G'$.

        Since $V(G)$ is disjoint to $\{1,\dots,n\}\subseteq \ds{N}$, observe that $G\cap G'=G-V(M)=G'-V(C)$ (as non-embedded graphs).

        Observe that the restriction of $\psi$ to $E(G')\setminus E(C)=E(G)\setminus E(C)$ is the inverse of $\psi'$. Thus, the restriction of $\psi$ to $E(G')\setminus E(C)$ is a bijection between $E(G')\setminus E(C)$ and $E(G)\setminus E(M)$. Note that an edge $e\in E(G')$ has no endpoints in $\{1,\dots,n\}=V(C)$ if and only if $e\in E(G-V(M))$. Thus, $E(G'-V(C))=E(G-V(M))$. Recalling that $\psi$ is the identity on $E(G-V(M))$, we find that $\psi$ is the identity on $E(G'-V(C))$. Since each $e\in E(M)$ appears exactly twice on the boundary walk of $F_M$, observe that $|\psi^{-1}(e)|=2$ for each $e\in E(M)$.
        
        Since $V(G)\cap \ds{N}=\emptyset$, observe that $V(G')=V(G)\setminus V(M)\sqcup \{1,\dots,n\}=(V(G)\setminus V(M))\sqcup V(C)$. Define a map $\phi:V(G')\mapsto V(G)$ via $\phi(v)=v$ for each $v\in V(G)\setminus V(M)$ and $\phi(i)=v_i$ for each $i\in \{1,\dots,n\}=V(C)$. By definition, $\phi$ is the identity of $V(G)\setminus V(M)$, and $\phi(V(C))=V(M)$. Since each $v\in V(M)$ appears exactly $\deg_M(v)$ times on the boundary walk of $F_M$, observe that $|\phi^{-1}(v)|=\deg_M(v)$ for each $v\in V(M)$.

        Observe that $E(G')=E(C)\sqcup \bigsqcup_{e\in E(G)\setminus E(M)}\psi'(e)$, and observe that $V(G')=\{1,\dots,n\}\sqcup (V(G)\setminus V(M))$. Thus, $|V(G')|=|V(G)|-|V(M)|+n$ and $|E(G')|=|E(C)|+|E(G)|-|E(M)|$. Note that $|E(C)|=n=2|E(M)|=2|V(M)|+2g-2$. Thus, $|V(G')|=|V(G)|+|V(M)|+2g-2$ and $|E(G')|=|E(G)|+|E(M)|$.

        Fix $e'\in E(G')$, let $u',v'$ be the endpoints of $e'$, and set $e:=\psi(e')$. If $e'\in \bigsqcup_{e''\in E(G)\setminus E(M)}\psi'(e'')$, then $e\in E(G)\setminus E(M)$ and $\psi'(e)=e'$. If $e\in E(G-V(M))$, then $e'=e$, and $u',v'\in V(G-V(M))$. Thus, $\phi(u')=u'$ and $\phi(v')=v'$. So $e=e'$ has endpoints $\phi(u')=u'$ and $\phi(v')=v'$. If $e\in E_{G,1}$, then without loss of generality $u'=u_e\in V(G)\setminus V(M)$ and $v'=i_e\in V(M)$. We then have $\phi(u')=u'$ and $\phi(v')=\phi(i_e)=v_{i_e}$. Recall that the endpoints of $e$ are $u_e=\phi(u')$ and $v_{i_e}=\phi(v')$. Otherwise, if $e\in E_{G,2}$, then without loss of generality $u'=i_e$ and $v'=j_e$. So $\phi(u')=\phi(i_e)=v_{i_e}$ and $\phi(u')=\phi(j_e)=v_{j_e}$. Recall that the endpoints of $e$ are $v_{i_e}=\phi(u')$ and $v_{j_e}=\phi(v')$. Otherwise, if $e'\in E(C)$, then there exists $i\in \{1,\dots,n\}$ such that the endpoints of $e'$ are $i$ and $i-1$. Recall that $\psi(e')=e_i$, and that the endpoints of $\psi(e')=e_i$ are exactly $v_{i-1}=\phi(i-1)$ and $v_i=\phi(i)$.
        
        Thus, for each $e\in E(G')$, if the endpoints of $e$ are $u',v'\in V(G')$, then the endpoints of $\psi(e)$ are $\phi(u')$ and $\phi(v')$.

        We now define a rotation system $((\tau_v:v\in V(G'),\delta)$ for $G'$. For each $e'\in E(G')$, let $\delta(e)=\lambda(\psi(e'))$. For each $v\in V(G')\setminus \{1,\dots,n\}=V(G)\setminus V(M)$, and each $e\in E(G)$ incident to $v$, set $\tau_v(\psi'(e)):=\psi'(\pi_v(e))$ (which is well-defined since both $e$ and $\pi_v(e)$ are in $E(G)\setminus E(M)$). For each $i\in \{1,\dots,n\}$, let $E_i'$ be the edges of $G'$ incident to $i$. We will define ${\tau_i}^{\lambda_i}$, from which $\tau_i$ of $E_i'$ can be uniquely obtained. 
        
        Observe that $E_i'$ is precisely the disjoint union of $(\psi'(e):e\in E_i)$ and $(e_i',e_{i+1}')$. Note that for each $e\in E_i\setminus \{e_{i,2}\}$, $\pi_{v_i}^{\lambda_i}(e)\in E_i$. So we can set ${\tau_i}^{\lambda_i}(\psi'(e)):=\psi'(\pi_{v_i}^{\lambda_i}(e))$. If $E_i\neq \emptyset$, we also set ${\tau_i}^{\lambda_i}(e_i'):=\psi'(e_{i,1})$, ${\tau_i}^{\lambda_i}(\psi'(e_{i,2})):=e_{i+1}'$, and ${\tau_i}^{\lambda_i}(e_{i+1}'):=e_i'$. Otherwise, set $\tau_i(e_i'):=e_{i+1}'$ and $\tau_i(e_{i+1}')=e_i'$. Note that $\tau^{\lambda_i}_i(e_i')=e_{i+1}'$ or $\tau^{-\lambda_i}_i(e_{i+1}')=e_i'$ if and only if $E_i=\emptyset$, which occurs if and only if $\pi^{\lambda_i}_{v_i}(e_i)=e_{i+1}$ and $\pi^{-\lambda_i}_{v_i}(e_{i+1})=e_i$.
        
        \begin{claim}
            \label{claimRotationSystem}
            Let $v'\in V(G')$ and $e'\in E(G')\setminus E(C)$ be incident, and let $\lambda\in \{-1,1\}$. Then $\psi({\tau_{v'}}^{\lambda}(e'))=\pi^{\lambda}_{\phi(v')}(\psi(e'))$.
        \end{claim}

        \begin{proofofclaim}
            Let $v:=\phi(v')$ and $e:=\psi'(e')$. Since $e'\notin E(C)$, $\psi'(e)$ is defined and $\psi'(e)=e'$. So we must show that $\psi({\tau_{v'}}^{\lambda}(\psi'(e)))=\pi^{\lambda}_v(e)$.
        
            If $v'\in V(G')\setminus \{1,\dots,n\}=V(G)\setminus V(M)$, then $v'=v$ and $\tau_v(e')=\psi'(\pi_v(e))$ by definition. Thus $\psi(\tau_v(e'))=\psi(\psi'(\pi_v(e)))=\pi_v(e)$. Applying the same argument to $(v,{\tau_{v}}^{-1}(e'))$ gives $\pi_v(\psi({\tau_{v'}}^{-1}(e')))=\psi(\tau_v({\tau_v}^{-1}(e')))=e$. So $\psi({\tau_{v'}}^{-1}(e'))=\pi_v^{-1}(e)$.
            
            If $v'=i\in \{1,\dots,n\}$, then $v=v_i$. Further, since $e'\notin E(C)$, $e\in E_i$. If $e\in E_i\setminus \{e_{i,2}\}$, then ${\tau_i}^{\lambda_i}(e')=\psi'(\pi_v^{\lambda_i}(e))$. So $\psi({\tau_i}^{\lambda_i}(e'))=\pi_v^{\lambda_i}(e)$, as desired. If $e\in E_i\setminus \{e_{i,1}\}$, then $\pi_v^{-\lambda_i}(e)\in E_i\setminus \{e_{i,2}\}$. Thus, by the previous argument, $\psi({\tau_i}^{\lambda_i}(\psi'(\pi_v^{-\lambda_i}(e)))=\pi_v^{\lambda_i}(\pi_v^{-\lambda_i}(e))=e$. So $\pi_v^{-\lambda_i}(e)=\psi({\tau_i}^{-\lambda_i}(e'))$.
            
            It remains only to consider the cases when $\lambda=\lambda_i$ and $e=e_{i,2}$, and when $\lambda=-\lambda_i$ and $e=e_{i,1}$. In the former case, ${\tau_i}^{\lambda_i}(e')=e_{i+1}'$, and $\pi^{\lambda_i}_v(e)=e_{i+1}=\psi(e_{i+1})$. So $\psi({\tau_i}^{\lambda_i}(e')=\pi^{\lambda_i}_v(e)$. In the latter case, ${\tau_i}^{-\lambda_i}(e')=e_i'$, and $\pi^{-\lambda_i}_v(e)=e_i=\psi(e_i')$. So $\psi({\tau_i}^{-\lambda_i}(e'))=\pi^{-\lambda_i}_v(e)$, as desired.
        \end{proofofclaim}

        $((\tau_v:v\in V(G')), \delta)$ defines a $2$-cell embedding of $G'$ into some surface $\Sigma'$. Note that this induces an embedding of $C$, so that $C$ becomes an embedded subgraph of $M$.
        
        Let $\scr{F},\scr{F}'$ denote the faces of $G,G'$ respectively.
        
        We first show that $e_1',1,\dots,e_n',n,e_1'$ is the reverse of the boundary walk of some $F^*\in \scr{F}'$. For each $i\in \{1,\dots,n\}$, set $\delta_i:=\prod_{j=1}^i \delta(e_j')$. We must show that $\delta_n=1$ and ${\tau_i}^{-\delta_i}(e_i')=e_{i+1}'$. Observe that, for each $i\in \{1,\dots,n\}$, $\delta(e_i')=\lambda(e_i)$, and thus $\delta_i=\lambda_i$. In particular, $\delta_n=\lambda_n=1$. Furthermore, for each $i\in \{1,\dots,n\}$, we have ${\tau_i}^{-\delta_i}(e_i')={\tau_i}^{-\lambda_i}(e_i')=e_{i+1}'$, as desired. So $e_1',1,\dots,e_n',n,e_1'$ is the reverse of the boundary walk of some $F^*\in \scr{F}'$. Note that this boundary walk (when viewed as a graph) is $C$.

        The next goals are to show that $G'$ is a plane graph (under the embedding given by the rotation system), and to define a bijection $\eta:\scr{F}\mapsto \scr{F}'\setminus \scr{F}^*$ such that, for each $F\in \scr{F}$, if the boundary walk of $\eta(F)$ is $f_1,x_1,\dots,f_n,x_n,f_1$, then the boundary walk of $F$ is $\psi(f_1),\phi(x_1),\dots,\psi(f_n),\phi(x_n),\psi(f_1)$. Observe that if we can do so, then the proof is complete.

        Firstly, if $M=G$, then $G=M$ has only one face $F_M$, whose boundary walk is $e_1,v_1,\dots,e_n,v_n,e_1$. Further, observe that $G'=C$. Since $C$ is either a cycle or has exactly $2$ vertices and $2$ parallel edges, any $2$-cell embedding of $C$ is plane, has exactly two faces, and the boundary walks of these faces are the reverse of each other. Thus, if we take $\eta(F_M)$ to be the unique face of $G'=C$ other than $F^*$, then the boundary walk of this face is $e_1',1,\dots,e_n',n,e_1'$. Observe that $\psi(e_1'),\phi(1),\dots,\psi(e_n'),\phi(n),\psi(e_1')=e_1,v_1,\dots,e_n,v_n,e_1$, which is precisely the boundary walk of $F_M$. This completes the proof in the case $M=G$. Thus and since $M\subseteq G$, we may assume that $M$ is a proper subgraph of $G$.

        Let $F\in \scr{F}'\setminus \{F^*\}$ have boundary walk $f_1',x_1',\dots,f_m',x_m',f_1'$ (which exists since $G'$ is $2$-cell embedded). For each $i\in \{1,\dots,n\}$, set $x_i:=\phi(x_i')$, $f_i:=\psi(f_i')$, $\delta_{i,F}:=\prod_{j=1}^i \delta(f_i')$. So $\delta_{m,F}=1$ and ${\tau_{x_i'}}^{\delta_i}(f_i')=f_{i+1}'$ (and ${\tau_{x_{i-1}'}}^{-\delta_{i-1}}(f_i')=f_{i+1}'$). We claim that there exists $\eta'(F)\in \scr{F}$ whose boundary walk is $f_1,x_1,\dots,f_m,x_m,f_1$. Note that, if true, $\eta'(F)$ is uniquely defined, as no two faces share a boundary walk.
        
        For each $i\in \{1,\dots,m\}$, set $\lambda_{i,F}:=\prod_{j=1}^i \lambda(f_i)$. We must show that $\lambda_{m,F}=1$ and $\pi^{\lambda_{i,F}}_{x_i}(f_i)=f_{i+1}$ (or, equivalently, that $\pi^{-\lambda_{i,F}}_{x_i}(f_{i+1})=f_i$). Observe that, for each $i\in \{1,\dots,m\}$, $\lambda(f_i)=\delta(f_i')$ and thus $\lambda_{i,F}=1=\delta_{i,F}$. By \cref{claimRotationSystem}, if $f_i'\in E(G')\setminus E(C)$, then $\pi^{\lambda_{i,F}}_{x_i}(f_i)=\psi({\tau_{x_i'}}^{\lambda_{i,F}}(f_i'))=f_{i+1}$, and $\pi^{-\lambda_{i-1,F}}_{x_{i-1}}(f_i)=\psi({\tau_{x_{i-1}'}}^{-\lambda_{i-1,F}}(f_i'))=f_{i-1}$. So it remains only need to check the case when $f_i',f_{i+1}'\in E(C)$. Thus $x_i'=j\in \{1,\dots,n\}$, $x_i=v_j$, and $\{f_i',f_{i+1}'\}=\{e_j',e_{j+1}'\}$. Note that since $F\neq F^*$ we have $(f_i',f_{i+1}',\lambda_{i,F})\notin \{(e_j',e_{j+1}',-\lambda_j),(e_{j+1}',e_j',\lambda_j)\}$. So $(f_i',f_{i+1}',\lambda_{i,F})\in \{(e_j',e_{j+1}',\lambda_j),(e_{j+1}',e_j',-\lambda_j)\}$. But then either 
        $\tau^{\lambda_j}_j(e_j')=e_{j+1}'$ or $\tau^{-\lambda_j}_j(e_{j+1}')=e_j'$, which we recall occurs if and only if $\pi^{\lambda_j}_{v_j}(e_j)=e_{j+1}$ and $\pi^{-\lambda_j}_{v_j}(e_{j+1})=e_j$. Thus, $\pi^{\lambda_{i,F}}_{x_i}(f_i)=f_{i+1}$, as desired. 

        It remains to check that the boundary walk is exactly $f_1,x_1,\dots,f_m,x_m,f_1$, instead of some positive number of repetitions of the boundary walk being $f_1,x_1,\dots,f_m,x_m,f_1$.
        
        We need the following claim.

        \begin{claim}
            \label{claimEdgesNotInM}
            If $M$ is a proper subgraph of $G$, then there does not exist $F\in \scr{F}$ whose boundary edges are all contained in $E(M)$.
        \end{claim}

        \begin{proofofclaim}
            Presume otherwise. Let $f_1,x_1,\dots,f_m,x_m,f_1$ be the boundary walk of $F$. So $f_1,\dots,f_m\in E(M)$ (and $x_1,\dots,x_m\in V(M)$). For each $i\in \{1,\dots,m\}$, set $\lambda_{i,F}:=\prod_{j=1}^i \lambda(f_i)$. So $\pi_{x_i}^{\lambda_{i,F}}(f_i)=f_{i+1}\in E(M)$. It follows that $\pi_{x_i,M}^{\lambda_{i,F}}(f_i)=f_{i+1}$. Since $\lambda_M$ is the restriction of $\lambda$ onto $E(M)$, we find that $f_1,x_1,\dots,f_m,x_m,f_1$ is the boundary walk of a face of $M$. But $F_M$ is the only face of $M$. Thus, $m=n$, and up to reordering, for each $i\in \{1,\dots,n\}$, $f_i=e_i$, $x_i=v_i$, and $\lambda_{i,F}=\lambda_i$. Thus, for each $i\in \{1,\dots,n\}$, we have $\pi_{v_i}^{\lambda_i}(e_i)=e_{i+1}$, and hence $E_i=\emptyset$.

            Since $G$ is $2$-cell embedded, $G$ is connected. Thus and since $M$ is a proper subgraph of $G$, there exists $e\in E(G)\setminus E(M)$ with at least one endpoint in $M$. Recall that $e\in E_i$ for at least one $i\in \{1,\dots,n\}$. In particular, $E_i\neq \emptyset$, a contradiction.
        \end{proofofclaim}
        
        We now prove that the boundary walk is exactly $f_1,x_1,\dots,f_m,x_m,f_1$. Presume otherwise. Then there exists some $m'<m$ such that $\lambda_{m',F}=1$ and $f_{m'+1}=f_1$. By \cref{claimEdgesNotInM}, we can assume that $f_1\in E(G)\setminus E(M)$. However, since $f_1',x_1',\dots,f_m',x_m',f_1'$ is the boundary walk of $F$, we either have $\lambda_{m',F}=\delta_{m',F}\neq 1$ or $f_{m'+1}'\neq f_1'$. We remark that since $f_1\in E(G)\setminus E(M)$, $f_1'\neq f_{m'}'$ implies $f_1\neq f_{m'}$. This gives the desired contradiction.

        Thus, $\eta'(F)$ is well-defined, and can be uniquely identified by its boundary walk. This defines a map $\eta':\scr{F}'\setminus \{F^*\}\mapsto \scr{F}$. We now show that $\eta'$ is surjective. Fix $F\in \scr{F}$, and let $f_1,x_1,\dots,f_m,x_m,f_1$ be the boundary walk of $F$, which exists since $G$ is $2$-cell embedded. As before, observe that $f_1,\dots,f_m\in E(M)$ only if $G=M$, a contradiction. Thus and without loss of generality, we can assume that $f_1\in E(G)\setminus E(M)$. Let $f_1':=\psi'(f_1)$. If $x_1\in V(G)\setminus V(M)=V(G')\setminus V(C)$, set $x_1':=x_1=\phi(x_1)$. Otherwise (if $x\in V(M)$), since $f_1\notin E(M)$, there exists $x_1'\in \phi^{-1}(x_1)$ such that $f_1\in E_i$ otherwise. In either case, observe that $x_1'$ is an endpoint of $f_1'$. There is a unique face $F'\in \scr{F}'$ whose facial walk can start with $f_1',x_1'$. Observe that the boundary walk of $\eta'(F')$ can start with $\psi(f_1')=f_1,\phi(x_1')=x_1$, which then uniquely identifies $\eta'(F')$ as $F$, as desired.

        Since $\eta'$ is surjective it follows that $|\scr{F'}|\geq |\scr{F}|+1$. Let $g'$ be the genus of the surface that $G'$ is embedded in. By Euler's formula and since $G'$ is $2$-cell embedded, $|V(G')|-|E(G')|+|\scr{F}'|=2-g'$. Recall that $|V(G')|=|V(G)|+|V(M)|+2g-2$ and $|E(G')|=|E(G)|+|E(M)|=|E(G)|+|V(M)|+g-1$. Further, since $G$ is $2$-cell embedded, $|V(G)|-|E(G)|+|\scr{F}|=2-g$. So $|V(G')|-|E(G')|=2-g-|\scr{F}|+g-1=1-|\scr{F}|$. Thus, $|\scr{F}'|=1-g'+|\scr{F}|\geq |\scr{F}|+1$. Since $g'\geq 0$, it follows that $g=0$ and $|\scr{F'}|=|\scr{F}|+1$. So $G'$ is a plane graph and $\eta'$ is a bijection. 
        
        Let $\eta$ be the inverse of $\eta'$. We can view $\eta$ as a injective map $\scr{F}\mapsto \scr{F}'$, where the only face not in the image of $\eta$ is $F^*$. By definition of $\eta$, for each $F\in \scr{F}$, $\eta'(\eta(F))=F$. Thus, if the boundary walk of $\eta(F)$ is $f_1,x_1,\dots,f_n,x_n,f_1$, then the boundary walk of $F$ is $\psi(f_1),\phi(x_1),\dots,\psi(f_n),\phi(x_n),\psi(f_1)$. This completes the proof.
    \end{proof}

    \subsection{Filtering details}

    We can now prove \cref{cuttingLemmaSimple}. The main effort comes from using $\eta$ to preserve facial triangles/discs.

    For a surface $\Sigma$ and a set $S\subseteq \Sigma$, use \defn{$\bar{S}$} to denote the closure of $S$ (in $\Sigma$).

    \cuttingLemmaSimple*
    \begin{proof}
    
        First, consider the case when $|E(M)|=0$. So $|E(M)|=|V(M)|+g-1=0$. Since $M$ is nonempty, observe that $g=0$ and $|V(M)|=1$. So $G$ is plane. Let $G'$ be obtained from $G$ by relabelling the unique vertex $v\in V(M)$ to some label $x\notin V(G)$ (and updating the edges accordingly). Let $W$ be the plane subgraph of $G'$ consisting only of the vertex $x$, and let $\phi:V(G')\mapsto V(G)$ be the identity on $V(G')\setminus \{x\}$ and send $x$ to $v$. Observe that the faces of $G$ are the faces of $G'$, so let $\eta$ be the identity map between them. Finally, let $\scr{D}':=\scr{D}$. It is easy to check that $(G',W,\scr{D}',\eta)$ then satisfy all the desired properties. So we may assume that $|E(M)|\geq 1$.

        By \cref{cuttingLemmaTrue}, there exists a $2$-cell embedded plane graph $G'$, a connected and nonempty plane subgraph $C$ of $G'$, and a projection $(\phi,\psi,\eta)$ of $(G',C)$ onto $(G,M)$ such that:
        \begin{enumerate}
             \item $G\cap G'=G-V(M)=G'-V(C)$ (as non-embedded graphs),
             \item $|V(G')|=|V(G)|+|V(M)|+2g-2$,
             \item $|E(G')|=|E(G)|+|E(M)|$,
             \item $|\phi^{-1}(v)|=\deg_M(v)$ for each $v\in V(M)$,
             \item $|\psi^{-1}(e)|=2$ for each $e\in V(M)$,
             \item $C$ is either a cycle, or consists of exactly two vertices and two parallel edges, and
             \item there is exactly one face $F'$ of $G'$ not in the image of $\eta$, and the boundary walk of $F'$ (when viewed as a graph) is $C$.
        \end{enumerate}
        Set $W:=C$. So $G\cap G'=G-V(M)=G'-V(W)$.

        By definition, $\phi$ is the identity on $V(G')\setminus V(W)=V(G')\setminus V(C)$, and $\phi(V(W))=V(M)$.

        For each $v\in V(G)\setminus V(M)$, observe that $v\in V(G')\setminus V(W)$ and $\phi(v)=v$. Thus, $v\in \phi(\phi^{-1}(v))\subseteq \phi(\phi^{-1}(N_G[v]))$.
        
        For each vertex $x\in N_G(v)$, there exists an edge $e\in E(G)$ with endpoints $v,x$. If $e\in E(G)\setminus E(M)$, then $|\psi^{-1}(e)|=1$ as restriction of $\psi$ to $E(G')\setminus E(W)$ is a bijection between $E(G')\setminus E(W)$ and $E(G)\setminus E(M)$. Otherwise, $e\in E(M)$, and thus $|\psi^{-1}(e)|=2$. Either way, there exists $e'\in \psi^{-1}(e)$. Let $v',x'$ be the endpoints of $e'$. Note that $x'\in N_{G'}(v')$ (and $v'\in N_{G'}(x')$).
        
        $\phi(v'),\phi(x')$ are precisely the endpoints of $\psi(e')=e$. Thus and without loss of generality, $\phi(v')=v$ and $\phi(x')=x$. Thus, $x'\in N_{G'}(\phi^{-1}(v))$, and $x\in \phi(N_{G'}(\phi^{-1}(v)))$. It follows that $N_G[v]\subseteq \phi(N_{G'}(\phi^{-1}[v]))$.

        For each $K\in \FT(G)$, there exists a triangle-face $F$ whose boundary vertices are exactly $K$. So $F$ is $2$-cell, and the boundary walk of $F$ is of the form $e_1,v_1,e_2,v_2,e_3,v_3,e_1$, where $K=\{v_1,v_2,v_3\}$.
        
        Let $e_1',v_1',\dots, e_n',v_n',e_1'$ be the boundary walk of $\eta(F)$. We find that $\psi(e_1'),\phi(v_1'),\dots,\psi(e_n'),\phi(v_n'),\psi(e_1')$ is exactly $e_1,v_1,e_2,v_2,e_3,v_3,e_1$. Thus, $n=3$, and without loss of generality, for each $i\in \{1,2,3\}$, $\psi(e_i')=e_i$ and $\phi(v_i')=v_i$. In particular and since $G'$ is $2$-cell embedded, $\eta(F)$ is a triangle face. The boundary vertices of $\eta(F)$ are exactly $\{v_1',v_2',v_3'\}$, so $\{v_1',v_2',v_3'\}\in \FT(G')$. Further, $\phi(\{v_1',v_2',v_3'\})=\{v_1,v_2,v_3\}=K$. Thus, $K':=\{v_1',v_2',v_3'\}\in \FT(G')$ satisfies $\phi(K')=K$.

        Fix $F\in \scr{F}$, and let $e_1,v_1,\dots,e_n,v_n,e_1$ be the boundary walk of $\eta(F)$. So the boundary walk of $F$ is $\psi(e_1),\phi(v_1),\dots,\psi(e_n),\phi(v_n),\phi(e_1)$. Let $C_F$ be obtained from $\bar{\eta(F)}$ by identifying $v_i$ with $v_j$ whenever $\phi(v_i)=\phi(v_j)$, and identifying the line $e_i$ with the line $e_j$ in the same direction (as given by the endpoints) whenever $\psi(e_i)=\psi(e_j)$. Let $q_F$ be the corresponding quotient map. It follows that $C_F$ is homeomorphic to $\bar{F}$ via a homeomorphic $h_F$ that sends each $v_i$ (as an equivalence class) to $\phi(v_i)$.

        Since $\scr{D}$ is $G$-pristine, for each $D\in \scr{D}$, $D$ is $G$-clean. Thus, there exists (a unique) face $F_D$ of $G$ containing the interior of $D$. Further, the closure of $F_D$ contains $D$. Let $F_D':=\eta(F_D)$, let $C_D:=C_{F_D}$, let $h_D:=h_{F_D}$, and let $q_D:=q_{F_D}$. So $q_D$ is a quotient map from $\bar{F_D'}$ to $C_D$, and $h_D$ is a homeomorphism from $C_D$ to $\bar{F_D}$. Let $B_{F,D}$ be the boundary vertices of $F_D$, let $B_{F,D}'$ be the boundary vertices of $F_D'$, and let $B_D:=B(D,G)$. Note that $B_D\subseteq B_{F,D}$. Further, for each $v\in B_{F,D}$, note that $(h_D\circ q_D)^{-1}(v)=\phi^{-1}(v)\cap B_{F,D}'\neq \emptyset$.
        
        Let $D'':=h_D^{-1}(D)$, and let $D'''$ be the set of points in the closure of $\eta(F_D)$ that get mapped to $D''$ under $q_D$. Since $D$ is a disc, observe that $D'''$ consists of a disc $D'_D:=D'$ plus some number of isolated points. Further, observe that the restriction of $h_D\circ q_D$ to $D'$ is a homeomorphism from $D'$ to $D$. Let $f_D$ denote this homeomorphism.
        
        Since $D$ is $G$-clean, observe that $D'\cap G'=\partial D'\cap G'=f_D^{-1}(\partial D\cap G)=f_D^{-1}(B_D)$. Note that the restriction of $f_D$ to $V(G')\cap D'$ is the restriction of $\phi$ to $V(G')\cap D'$, and that the image of $D'\setminus V(G')$ under $f_D$ is disjoint to $V(G)$. Thus, $D'\cap G'=f_D^{-1}(B_D)=\phi^{-1}(B_D)\cap D'$. In particular, note that $D'\cap G'\subseteq \phi^{-1}(B_D)\subseteq V(G')$. Hence, $D'$ is $G'$-clean.
        
        Let $B_D':=B(D',G')=D'\cap G'$. So $B_D'=f_D^{-1}(B_D)=\phi^{-1}(B_D)\cap D'$. Since $f_D$ is a bijection, we find that restriction of $\phi$ to $B_D'$ is a bijection from $B_D'$ to $B_D$. Call this restriction $\phi_D$.

        Since $f_D$ is a homeomorphism from $D'$ to $D$, for each arc of $D'$ in $G'$ (connected component of $\partial D'\setminus V(G')$) with endpoints $u',v'$, there is an arc of $D$ in $G$ whose endpoints are $f_D(u)=\phi_D(u)$ and $f_D(v)=\phi_D(v)$ respectively. The inverse is true for the arcs of $D$ in $G$. Thus, we find that $\phi_D$ is an isomorphism from $U(D',G')$ to $U(D,G)$.

        Since $D$ intersects $V(M)$, $D'$ intersects $\phi^{-1}(V(M))$. So there exists $x\in B_D'$ with $\phi(x)\in V(M)$. Recall that $V(G\cap G')=V(G)\setminus V(M)$. Thus, $\phi(x)\notin V(G')$, and thus $\phi(x)\neq x$. Since $\phi$ is the identity on $V(G')\setminus V(C)$, this gives $x\in V(C)$. Thus, $D'$ intersects $V(C)$.

        Let $\scr{D}':=\{D'_D:D\in \scr{D}\}$. Let $\beta:\scr{D}'\mapsto \scr{D}$ be the map that sends each $D'\in \scr{D}'$ to the $D\in \scr{D}$ such that $D'=D'_D$.
        
        To prove the lemma (with $W:=C$), it remains only to show that $\scr{D}'$ is $G'$-pristine, and further that $\scr{D}'$ is anchoring. Recalling that each $D'\in \scr{D}'$ is $G$-clean and intersects $V(C)$ (and since $V(C)\subseteq V(G')$), observe that it suffices to show that for each pair of distinct $D_1',D_2'\in \scr{D}'$ $D_1'\cap D_2'\subseteq V(C)$. Let $D_1,D_2\in \scr{D}$ be such that $D_1'=D'_{D_1}$ and $D_2'=D'_{D_2}$. So $D_1,D_2$ are distinct. For $i\in \{1,2\}$, set $(B_i,B_i',F_i,F_i',h_i,q_i,f_i):=(B_{D_i},B_{D_i}',F_{D_i},F_{D_i}',h_{D_i},q_{D_i},f_{D_i})$.

        First, presume, for a contradiction, that $D_1'\setminus B_1'$ intersects $D_2'\setminus B_2'$. Since $D_1',D_2'$ are $G'$-clean, we obtain $F_1'=F_2'$. Since $\eta$ is an injection, this implies $F_1=F_2$, and thus $(h_1\circ q_1)=(h_2\circ q_2)$. Recall that, for each $i\in \{1,2\}$, the restriction of $h_i\circ q_i$ to $D'i$ (which is $f_i$) is a homeomorphism to $D_i$ that sends $B_i'$ to $B_i$. Thus, we find that $D_1\setminus B_1$ and $D_2\setminus B_2$ intersect. This contradicts the fact that $\scr{D}$ is $G$-pristine. Thus, $D_1'\setminus B_1'$ is disjoint to $D_2'\setminus B_2'$. In particular, we find that $D_1'\cap D_2'\subseteq B_1'\cup B_2'\subseteq V(G')$. So $\scr{D}'$ is $G'$-pristine. Further, since $D_i\cap V(G')=B_i'$ for each $i\in \{1,2\}$, we find that $D_1'\cap D_2'=B_1'\cap B_2'$.
        
        Next, presume, for a contradiction, that $(B_1'\cap B_2')\setminus V(C)$ is nonempty. Thus, $B_1'\setminus V(C)$ intersects $B_2'\setminus V(C)$. Since $\phi$ is the identity on $V(G')\setminus V(C)=V(G)\setminus V(M)$, we find that $B_1\setminus V(M)$ intersects $B_2\setminus V(M)$. This contradicts the fact that $\scr{D}$ is $V(M)$-anchored. Thus, $D_1'\cap D_2'\subseteq B_1\cap B_2\subseteq V(C)$. This completes the proof.
    \end{proof}

    \section{Partitioning a plane quasi-triangulation}
    \label{SecQT}
    
    This section is devoted to the proof of \cref{conPlanar}. Before we can do so, we need to take a relatively substantial detour to study plane quasi-triangulations and minimal `separators' within them.

    A set $S\subseteq V(G)$ in a graph $G$ is a \defn{separator} between $X,Y\subseteq V(G)$ in $G$ if, for each connected component $C$ of $G-(X\cup Y)$, $V(C)$ intersects at most one of $X,Y$. A separator $S$ is \defn{minimal} if no proper subset of $S$ is a separator between $X$ and $Y$ in $G$. A \defn{separator} between $A,B\subseteq G$ in $G$ is a separator between $V(A)$ and $V(B)$ in $G$.
    
    \subsection{Separators in plane quasi-triangulations}
    \label{SecSeparators}

    The main goal of this subsection is to prove the following lemma.

    \begin{restatable}{lemma}{sepCC}
        \label{sepCC}
        Let $A,B$ be connected and disjoint subgraphs of a plane quasi-triangulation $G$, let $C$ be a connected component of $G-V(A)-V(B)$, and let $S$ be a minimal separator between $N_G(V(A))\cap V(C)$ and $N_G(V(B))\cap V(C)$ in $C$. Then $G[S]$ is connected.
    \end{restatable}

    To prove \cref{sepCC}, we need to perform a careful study of minimal separators in plane quasi-triangulations. We start by considering a minimal separator $S$ in a plane quasi-triangulation $G$ between two connected subgraphs $A,B$. We want to not only show that $G[S]$ is connected, but also that all the cut-vertices lie in $V(A\cup B)$. It is difficult to prove the general case directly, so we will start by having the separator $S$ be disjoint to $A$ and $B$. We will then work our way up, allowing $S$ to intersect one of $A,B$ and then both of $A,B$.
    
    \begin{lemma}
        \label{sepTNoOverlap}
        Let $A,B$ be connected subgraphs of a plane quasi-triangulation $G$, and let $S$ be a minimal separator between $A$ and $B$ in $G$ that is disjoint to $V(A)\cup V(B)$. Then $G[S]$ is connected contains no cut-vertices.
    \end{lemma}

    \begin{proof}
        Since $A$ is connected and since $S$ is disjoint to $V(A)$, there exists a connected component $C$ of $G-S$ containing $A$. Since $S$ is a separator between $A$ and $B$ in $G$, $C$ does not intersect $B$. 
        
        Let $E_S$ be the set of edges with exactly one endpoint in $V(C)$. By maximality of $C$, for each $e\in E_S$, the endpoint of $e$ not in $V(C)$ is in $S$. By minimality of $S$ (and since $S$ is disjoint to $V(A)$), for each $s\in S$ there is an $e_s\in E_S$ with $s$ as an endpoint.

        Consider the embedding of $C$ in the plane. Since $B$ is connected and disjoint to $C$, $B$ is contained in a face $F$ of $C$. Since $C$ is connected, $F$ is $2$-cell. Let $J$ be the closed curve obtained from following the boundary of $F$ from just inside $F$. Since $G$ is finite, up to a slight modification of $J$, we can assume that $J$ does not intersect $V(G)$, and that $J$ intersects each edge of $G$ at at most one point. Further, we can assume that the edges that $J$ intersects are exactly the edges in $E_S$. Finally, observe that (up to making $J$ hug the boundary of $C$ sufficiently tightly), $J$ separates $A$ from $B$ in the plane.

        We now adjust $J$ to form a new closed curve $J'$ by, whenever $J$ meets an $e\in E_S$, adjusting $J$ so that it instead follows $e$ up to the endpoint in $s$, and then follows $e$ again in reverse on the other side to restart the path. Up to hugging the edges sufficiently tightly, $J'$ intersects itself and $G$ only at vertices in $S$. Within $J'$, we can find a closed curve $J''$ that does not intersect itself and intersects $G$ only at vertices in $S$ (and not at edges in $G$). Since $S$ is disjoint to $V(B)$, observe that $J''$ also separates $A$ from $B$ in the plane.
        
        By minimality of $S$ (and since $S$ is disjoint to $V(A)\cup V(B)$), observe that $J''$ intersects each $s\in S$. Thus and since $J''$ is not self-intersecting, following $J''$ gives a cyclic ordering $v_1,\dots,v_n$ of the vertices in $S$, where $n=|S|$. Since $J''$ intersects $G$ only at vertices in $S$, observe that consecutive vertices in this cyclic ordering share a common face (that $J''$ passes through). Since $G$ is a quasi-triangulation, the boundary vertices of each face form a clique. Thus, $G[S]$ is connected and does not contain a cut-vertex, as desired.
    \end{proof}

    \begin{lemma}
        \label{sepTOneOverlap}
        Let $A,B$ be connected subgraphs of a plane quasi-triangulation $G$, and let $S$ be a minimal separator between $A$ and $B$ in $G$ that is disjoint to $V(B)$. Then $G[S]$ is connected, and every cut-vertex in $G[S]$ is in $V(A)$.
    \end{lemma}

    \begin{proof}
        If $V(A)\cap S=\emptyset$, the result follows immediately by \cref{sepTNoOverlap}. So we may assume that $V(A)\cap S\neq \emptyset$.

        Let $\scr{C}$ be the set of connected components of $G-S$ that intersect $A$. For each $C\in \scr{C}$, let $S_C:=N_G(V(C))\cap S$. Since $S$ is a separator between $A$ and $B$ in $G$, observe that $B$ does not intersect $C$. Thus, $S_C$ is a separator between $C$ and $B$ in $G$. Further, by minimality of $S$, for each $s\in S_C$, there exists a path from $A$ to $B$ in $G$ whose intersection with $S$ is $\{s\}$. Since $s$ is adjacent to $V(C)$ in $G$ by definition, there exists a path from $C$ to $B$ in $G$ whose intersection with $S$ is $\{s\}$. Therefore, $S_C$ is a minimal separator between $C$ and $B$ in $G$. Since $C,B$ are connected, and since $S_C\subseteq S$ is disjoint to both $C$ and $B$, by \cref{sepTNoOverlap}, $G[S_C]$ is connected and does not contain a cut-vertex.

        Let $S_A:=S\cap V(A)$ and let $S':=S\setminus V(A)$. Fix any label $s'$ that is either in $S'$ or not in $V(G)$. We now show that $S_A$ is contained within a connected component of $G[S\setminus \{s'\}]$. Fix any two vertices $s_1,s_2$ in $S_A$. Since $A$ is connected, there exists a path $P$ between them in $A$. Let $P'$ be a connected component of $P-(V(P)\cap S)$, and let $u,v$ be the vertices of $S\cap V(P)$ adjacent to $P'$ (of which there are exactly two, since $s_1,s_2\in V(A)$). So $u,v\in S_A$. In particular, $s'\notin \{u,v\}$. Observe that $P'\subseteq C$ for some $C\in \scr{C}$. Therefore, $u,v\in S_C$. Since $G[S_C]$ is connected and does not contain a cut-vertex, $u$ and $v$ are in the same connected component of $G[S_C\setminus \{s'\}]\subseteq G[S\setminus \{s'\}]$. Applying this argument consecutively, it follows that $s_1$ and $s_2$ are in the same connected component of $G[S\setminus \{s'\}]$. Thus, $S_A$ is contained in a connected component of $G[S\setminus \{s'\}]$.

        In particular, we find that $S_A$ is contained in a connected component of $G[S]$ (even if $S'=\emptyset$, by fixing $s'\notin V(G)$), and that for each $s'\in S'$, $S_A$ is contained in a connected component of $G[S]-s'$.

        We are now ready to show that $G[S]$ is connected, and that every cut-vertex of $G[S]$ is in $V(A)$. Equivalently, $G[S]$ is connected, and further, for each $s'\in S'$, $G[S]-s'$ is connected. By the above, it suffices to show that for each $s'\in S'$ and each $s\in S\setminus (S_A\cup \{s'\})=S'\setminus \{s'\}$, there exists $s_A\in S_A$ in the same connected component of $G[S]-s'$ as $s$. By minimality of $S$ and since $s\notin V(A)$, observe that $s\in S_C$ for some $C\in \scr{C}$. Since $V(A)\cap S\neq \emptyset$ and since $A$ is connected, observe that there exists $s_A\in V(A)\cap S_C\subseteq S_A$. So $s,s_A\in S_C$. Since $G[S_C]$ is connected and has no cut-vertices, $G[S_C\setminus \{s'\}]$ is connected. In particular, $s$ and $s_A$ are in the same connected component of $G[S_C\setminus \{s'\}]\subseteq G[S]-s'$, as desired. This completes the proof.
    \end{proof}

    \begin{lemma}
        \label{sepT}
        Let $A,B$ be connected subgraphs of a plane quasi-triangulation $G$, and let $S$ be a minimal separator between $A$ and $B$ in $G$. Then $G[S]$ is connected, and every cut-vertex in $G[S]$ is in $V(A)\cup V(B)$.
    \end{lemma}

    \begin{proof}
        The proof is very similar to \cref{sepTOneOverlap}, with the only differences being that \cref{sepTOneOverlap} is used in place of \cref{sepTNoOverlap} to prove intermediate results, and that we now take $S':=S\setminus (V(A)\cup V(B))$. For ease of reference, we include the proof full below.

        If $V(B)\cap S=\emptyset$, the result follows immediately by \cref{sepTOneOverlap}. So we may assume that $V(B)\cap S\neq \emptyset$.

        Let $\scr{C}$ be the set of connected components of $G-S$ that intersect $B$. For each $C\in \scr{C}$. Let $S_C$ be the subset of $S$ consisting of the vertices $s\in S$ that are adjacent to $V(C)$ in $G$. Since $S$ is a separator between $A$ and $B$ in $G$, observe that $A$ does not intersect $C$. Thus, $S_C$ is a separator between $C$ and $A$ in $G$. Further, by minimality of $S$, for each $s\in S_C$, there exists a path from $A$ to $B$ in $G$ whose intersection with $S$ is $\{s\}$. Since $s$ is adjacent to $V(C)$ in $G$ by definition, there exists a path from $C$ to $A$ in $G$ whose intersection with $S$ is $\{s\}$. Therefore, $S_C$ is a minimal separator between $C$ and $A$ in $G$. Since $C,A$ are connected and since $S_C\subseteq S$ is disjoint to $C$, by \cref{sepTOneOverlap}, $G[S_C]$ is connected and each cut-vertex of $G[S_C]$ is in $V(A)$.

        Let $S_B:=S\cap V(B)$ and let $S':=S\setminus (V(A)\cup V(B))$. Fix any label $s'$ that is either in $S'$ or not in $V(G)$. We now show that $S_B$ is contained within a connected component of $G[S\setminus \{s'\}]$. Fix any two vertices $s_1,s_2$ in $S_B$. Since $B$ is connected, there exists a path $P$ between them in $B$. Let $P'$ be a connected component of $P-(V(P)\cap S)$, and let $u,v$ be the vertices of $S\cap V(P)$ adjacent to $P'$ (of which there are exactly two, since $s_1,s_2\in V(B)$). So $u,v\in S_B$. In particular, $s'\notin \{u,v\}$. Observe that $P'\subseteq C$ for some $C\in \scr{C}$. Therefore, $u,v\in S_C$. Since $s'\notin V(A)$, $G[S_C\setminus \{s'\}]$ is connected. Thus, $u$ and $v$ are in the same connected component of $G[S_C\setminus \{s'\}]\subseteq G[S\setminus \{s'\}]$. Applying this argument consecutively, it follows that $s_1$ and $s_2$ are in the same connected component of $G[S\setminus \{s'\}]$. Thus, $S_B$ is contained in a connected component of $G[S\setminus \{s'\}]$.

        In particular, we find that $S_B$ is contained in a connected component of $G[S]$ (even if $S'=\emptyset$, by fixing $s'\notin V(G)$), and that for each $s'\in S'$, $S_B$ is contained in a connected component of $G[S]-s'$.

        We are now ready to show that $G[S]$ is connected, and that every cut-vertex in $V(A)\cup V(B)$. Equivalently, $G[S]$ is connected, and for each $s'\in S'$, $G[S]-s'$ is connected. By the above, it suffices to show that for each $s'\in S$ and each $s\in S\setminus (S_B\cup \{s'\})$, there exists $s_B\in S_B$ in the same connected component of $G[S]-s'$ as $s$. By minimality of $S$ and since $s\notin V(B)$, observe that $s\in S_C$ for some $C\in \scr{C}$ (even if $s\in V(A)$). Since $V(B)\cap S\neq \emptyset$ and since $B$ is connected, observe that there exists $s_B\in V(B)\cap S_C\subseteq S_B$. So $s,s_B\in S_C$. Since $s'\notin V(A)$, $G[S_C\setminus \{s'\}]$ is connected. Thus, $s$ and $s_B$ are in the same connected component of $G[S_C\setminus \{s'\}]\subseteq G[S]-s'$, as desired. This completes the proof.
    \end{proof}

    The next goal is to show that, even after deleting a connected subgraph $C$ of a plane quasi-triangulation $G$, a minimal separator between connected subgraphs $A,B$ of $G-V(C)$ still induces a connected subgraph of $S$. To do this, we need a way of `contracting' in a plane quasi-triangulation that avoids `destroying' the quasi-triangulation. We `contract' $C$ to a point $c$, show that either the original separator $S$ or $S\cup \{c\}$ is a minimal separator between $A$ and $B$ in the new plane quasi-triangulation, and then use the fact that $c$ won't be a cut-vertex (as it isn't in $V(A)\cup V(B)$) to show that $G[S]$ is connected.
    
    As for how we `contract', we have the following lemma.

    \begin{lemma}
        \label{PQTContract}
        Let $G$ be a plane quasi-triangulation, and let $A$ be a connected and nonempty subgraph of $G$. Then there exists a plane quasi-triangulation $G'$ and a vertex $a\in V(G')$ such that $G'-a$ is the plane graph $G-V(A)$, and $N_{G'}(a)=N_{G}(V(A))$.
    \end{lemma}

    \begin{proof}
        Since $A$ is connected, there exists a face $F_A$ of $G-V(A)$ containing $A$. Let $B_A$ be the boundary vertices of $F_A$ (in $G-V(A)$). Since $A$ is contained in $F_A$, observe that $N_G(V(A))\subseteq B_A$.

        Presume, for a contradiction, that $N_G(V(A))\neq B_A$. So there exists some $v\in B_A$ not contained in $N_G(V(A))$. Note that $v\in V(G-V(A))$, and thus $v\notin N_G[V(A)]$. Observe that there is a face $F$ of $G$ contained in $F_A$ with $v$ as a boundary vertex. Let $B_F$ be the boundary vertices of this face (in $G$). So $v\in B_F$. Since $G$ is a plane quasi-triangulation, $B_F$ is a clique in $G$. Thus and since $v\in B_F\setminus N_G[V(A)]$, $B_F$ does not intersect $V(A)$. Thus, $F$ is a face of $G-V(A)$. It follows that $F=F_A$. Since $F_A=F$ contains $A$ and is a face of $G$, we find that $A$ is empty, a contradiction. Therefore, $B_A=N_G(V(A))$.

        Embed a new vertex $a$ inside $F_A$ to make a new plane graph $G''$. So $G''-a$ is the plane graph $G-V(A)$. Observe that we can find a disc $D$ in the plane such that the interior of $D$ contains $a$, but $D$ is disjoint to $G''-a=G-V(A)$. So $D$ is $(G'',\{a\})$-clean, and the interior of $D$ is contained in $F_A$ (which is a face of $G''-a=G-V(A)$).

        By \cref{neighInFace} (with $S:=\{a\}$, $\scr{D}:=\{D\}$, and $F:=F_A$, noting that $B_F:=B_A$), there exists a spanning embedded (plane) supergraph $G'$ of $G''$ such that:
        \begin{enumerate}
            \item $G'-a$ is the embedded (plane) graph $G''-a$ (which is $G-V(A)$),
            \item for each face $F$ of $G'$ contained in $F_A$, the boundary vertices of $F$ are a clique in $G'$, and,
            \item $N_{G'}(S)=B_A=N_{G}(V(A))$.
        \end{enumerate}
        
        It remains to show that $G'$ is a plane quasi-triangulation. Let $F$ be a face of $G'$ with boundary vertices $B_F$. Since $G'-a=G-V(A)$, observe that either $F$ is a face of $G$ (with the same boundary vertices), or $F$ has $a$ as a boundary vertex. In the former case, as $G$ is a plane quasi-triangulation, $B_F$ is a clique in $G$. $B_F$ is also disjoint to $V(A)$ (as $B_F\subseteq V(G')$ is disjoint to $V(A)$), so $B_F$ is a clique in $G-V(A)=G'-a$ and thus $G'$. In the latter case, since $a$ is contained in $F_A$, observe that $F$ is contained in $F$. Thus, the boundary vertices of $F'$ are a clique in $G'$ (by definition of $G'$). So in either scenario, $B_F$ is a clique in $G'$. Hence, $G'$ is a plane quasi-triangulation. This completes the proof.
    \end{proof}

    We then obtain the following result.

    \begin{lemma}
        \label{sepNT}
        Let $C$ be a connected subgraph of a plane quasi-triangulation $G$, let $A,B$ be connected subgraphs of $G-V(C)$, and let $S$ be a minimal separator between $A$ and $B$ in $G-V(C)$. Then $G[S]$ is connected.
    \end{lemma}

    \begin{proof}
        If $V(C)=\emptyset$, then the lemma is true by \cref{sepT}. So we may assume that $C$ is nonempty.
    
        By \cref{PQTContract}, there exists a plane quasi-triangulation $G'$ and $c\in V(G')$ such that $G'-c$ is the plane graph $G-V(C)$. Since $A,B$ are connected subgraph of $G$ disjoint to $C$, observe that $A,B$ are connected subgraphs of $G'$. Further, observe that $S\cup \{c\}$ is a separator between $A$ and $B$ in $G'$, and that (by minimality of $S$) any subset of $S\cup \{c\}$ that does not contain $S$ is not a separator between $A$ and $B$ in $G'$. Thus, there exists $S'\in \{S,S\cup \{c\}\}$ that is a minimal separator between $A$ and $B$ in $G'$.
        
        By \cref{sepT}, $G'[S']$ is connected, and each cut-vertex of $G'[S']$ is in $V(A)\cup V(B)$. Since $c\notin V(A)\cup V(B)$, $G'[S'\setminus \{c\}]$ is connected. Observe that $S'\setminus \{c\}=S$, and thus $G'[S'\setminus \{c\}]=G'[S]=G[S]$ is connected, as desired.
    \end{proof}

    From here, we can finally prove \cref{sepCC}.
    
    \sepCC*

    \begin{proof}
        Note that if $N_G[V(C)]$ does not intersect both $V(A)$ and $V(B)$, then $S$ is empty and the lemma is trivial. So we may assume that $N_G[V(C)]$ intersects $V(A)$ and $V(B)$.
    
        Note that $A,B,C$ are all pairwise disjoint. So $A,C$ are connected subgraphs of $G-V(B)$ and $B,C$ are connected subgraphs of $G-V(A)$. Let $A':=G[N_G(V(A))\cap V(C)]$ and $B':=G[N_G(V(B))\cap V(C)]$. Observe that $V(A')$ is a minimal separator between $A$ and $C$ in $G-V(B)$, and that $V(B')$ is a minimal separator between $B$ and $C$ in $G-V(A)$. Thus, by two applications of \cref{sepNT} (with $(A,B,C,S):=(A,C,B,V(A'))$ and $(B,C,A,V(B'))$ respectively), $G[V(A')]=A'$ and $G[V(B')]=B'$ are connected. 
        
        If $N_G(V(A))$ intersects $V(B)$, let $D:=G[V(A)\cup V(B)]$. Since $A$ and $B$ are connected, $D$ is connected. Observe that $A',B'$ are disjoint from $D$, and are thus connected subgraphs of $G-V(D)$. Observe that $S$ is a minimal separator between $A'$ and $B'$ in $G-V(D)$ (as $C$ is a connected component of $G-V(D)$). Thus, by \cref{sepNT} (with $(A,B,C,S):=(A',B',D,S)$), $G[S]$ is connected, as desired.

        If $N_G(V(A))$ is disjoint to $V(B)$, observe that $V(G)\setminus (V(A)\cup V(B))$ is a separator between $A$ and $B$ in $G$ (as $A$ and $B$ are also disjoint). Thus, there is some $S'\subseteq V(G)\setminus (V(A)\cup V(B))$ that is a minimal separator between $A$ and $B$ in $G$. By \cref{sepT}, $G[S']$ is connected. Hence, $S'$ intersects at most one connected component of $G-V(A)-V(B)$. However, observe that $S'$ intersects each component $C'$ of $G-V(A)-V(B)$ such that $N_G[V(C')]$ intersects $V(A)$ and $V(B)$. Thus, there is at most one such component $C'$. Since $N_G[V(C)]$ intersects $V(A)$ and $V(B)$, $C$ is the only such component. Recalling that $N_G[V(A)]$ is disjoint to $V(B)$, it follows that any separator between $A'$ and $B'$ in $C$ is a separator between $A$ and $B$ in $G$. Thus, $S$ is a (minimal) separator between $A$ and $B$ in $G$. By \cref{sepT}, $G[S]$ is connected, as desired.
    \end{proof}

    \subsection{Preliminary lemma}

    Since we follow the same algorithm as \citet{Distel2024}, we need to introduce one of their intermediate lemmas.

    For $t,s\in \ds{N}$, \defn{$K_{s,t}^*$} is the graph obtained from $K_{s,t}$ by turning the $s$ side of the bipartition into a clique.

    Let $G$ be a graph, let $X,Y\subseteq V(G)$ be disjoint, and let $C$ be a subgraph of $G-X-Y$. Define \defn{$\kappa_G(X,C,Y)$} to be the number of pairwise vertex-disjoint paths in $C$ with one endpoint in $N_G(X)\cap V(C)$ and the other in $N_G(Y)\cap V(C)$.

    We use the following lemma, which is a very slight variant of \citet[Lemma~12]{Distel2024}.
    \begin{lemma}
    \label{smallSep}
        Let $n,t\in \ds{N}$, let $G$ be a simple $n$-vertex $K^*_{3,t}$-minor-free graph, and let $\alpha\in \ds{R}_0^+$ be such that $|E(G)|\leq \alpha tn$. Let $X$ and $Y$ be disjoint nonempty sets of vertices in $G$ such that $G[X]$,  $G[Y]$ and $G[X\cup Y]$ are connected. 
        Then there is a set $S\subseteq V(G-X-Y)$ such that:
        \begin{enumerate}
            \item $|N_G[S]| \leq t\sqrt{3\alpha n}$, 
            \item $\kappa_G(X,C,Y) \leq t\sqrt{3\alpha n}$ for every component $C$ of $G-X-Y-S$, and
            \item $G[X\cup S]$ and $G[Y\cup S]$ are connected.
        \end{enumerate}
    \end{lemma}
    In \citet[Lemma~12]{Distel2024}, $\alpha$ is a fixed constant from \citet{Kostochka2010} such that every $K^*_{3,t}$-minor-free simple graph $G$ has at most $\alpha t|V(G)|$ edges. However, a very brief scan of the proof confirms that $\alpha$ is only used to bound the number of edges of $G$ (and not a subgraph or minor of $G$), so the proof works with any $\alpha$ that satisfies $|E(G)|\leq \alpha t|V(G)|=\alpha tn$. This is useful, because the original $\alpha$ is a poor bound for small $t$, whereas, in our case (with planar graphs), we can take $\alpha=1$. We also prefer having an explicit constant in our proofs.

    We can then obtain the following as an easy corollary.
    \begin{corollary}
        \label{planarSmallSep}
        Let $n\in \ds{N}$, let $G$ be an $n$-vertex plane graph, and let $A,B$ be disjoint connected subgraphs of $G$ such that $N_G(V(A))$ intersects $V(B)$. Then there is a set $S\subseteq V(G-V(A)-V(B))$ such that:
        \begin{enumerate}
            \item $|N_G[S]| \leq 3\sqrt{3n}$, 
            \item $\kappa_G(X,C,Y) \leq 3\sqrt{3n}$ for every component $C$ of $G-V(A)-V(B)-S$, and
            \item $G[V(A)\cup S]$ and $G[V(B)\cup S]$ are connected.
        \end{enumerate}
    \end{corollary}

    \begin{proof}
        Delete all parallel and loop edges from $G$ to obtain a simple plane graph $G'$. So $|V(G')|=n$. Let $X:=V(A)$ and $Y:=V(B)$. Since $A$ and $B$ are connected and since $N_G(V(A))$ intersects $V(B)$, $X$ and $Y$ are nonempty, and $G[X],G[Y],G[X\cup Y]$ are all connected. Thus, since $V(G')=V(G)$ and since $E(G)\setminus E(G')$ consists only of parallel and loop edges, $G'[X],G'[Y],G'[X\cup Y]$ are all well-defined and connected. 
        
        It is well-known that any simple $n$-vertex plane graph has at most $3n-6\leq 3n$ edges. So $|E(G')|\leq 3n$. Further, since $G'$ is plane, it is $K_{3,3}$-minor-free, and thus $K_{3,3}^*$-minor-free.

        Thus, we can apply \cref{smallSep} with $\alpha=1$ and $t=3$ to $G'$. We find that there exists a set $S\subseteq V(G'-X-Y)=V(G-V(A)-V(B))$ such that:
        \begin{enumerate}
            \item $|N_{G'}[S]| \leq 3\sqrt{3n}$, 
            \item $\kappa_{G'}(V(A),C',V(B)) \leq 3\sqrt{3n}$ for every component $C'$ of $G'-V(A)-V(B)-S$, and
            \item $G'[V(A)\cup S]$ and $G'[V(B)\cup S]$ are connected.
        \end{enumerate}

        Observe that for each connected component $C$ of $G-V(A)-V(B)-S$, deleting parallel and loop edges from $C$ gives a connected component $C'$ of $G'-V(A)-V(B)-S$. Since $\kappa$ counts vertex-disjoint (not edge-disjoint) paths, observe that $\kappa_G(V(A),C,V(B))=\kappa_{G'}(V(A),C',V(B))\leq 3\sqrt{3n}$.

        Observe that $N_{G'}[S]=N_G[S]$. Thus, $|N_G[S]|=|N_{G'}[S]|\leq 3\sqrt{3n}$. Further, observe that $G[V(A)\cup S]\supseteq G'[V(A)\cup S]$ is connected, and that $G[V(B)\cup S]\supseteq G'[V(B)\cup S]$ is connected. This completes the proof.
    \end{proof}

    \subsection{Proof}

    This subsection proves a technical lemma, from which \cref{conPlanar} follows easily. We will need the following observation.
    \begin{observation}
        \label{PQT1Con}
        Every plane quasi-triangulation $G$ is connected.
    \end{observation}

    \begin{proof}
        Between any set of vertices $u,v$ in $G$, we can find a path $P$ in the plane between them. We can adjust $P$ to a find a new path $P'$ that intersects $G$ only at vertices in $G$. Let $u=x_0,x_1,\dots,x_n=v$ be the sequence of vertices of $G$ that $P'$ intersects. Observe that consecutive vertices in this sequence are on the boundary of a common face. Since $G$ is a plane quasi-triangulation, the boundary vertices of each face is a clique. Thus, there is a path $P$ from $u$ to $v$ whose vertices are contained in $x_0,x_1,\dots,x_n$. Hence, $G$ is connected, as desired.
    \end{proof}
    
    \begin{lemma}
        \label{conPlanarLemma}
        Let $G$ be a plane quasi-triangulation on $n$ vertices, and let $\scr{H}$ be a set of at most $2$ pairwise disjoint connected subgraphs of $G$ such that if $|\scr{H}|=2$ with $\scr{H}=\{A,B\}$, then $N_G(V(A))$ intersects $V(B)$, and $\kappa_G(V(A),C,V(B))\leq 6\sqrt{3n}$ for every component $C$ of $G-V(A)-V(B)$. Then $G$ has admits a connected partition $\scr{P}$ of treewidth at most $2$ such that: 
        \begin{enumerate}
            \item for each $H\in \scr{H}$, $V(H)\in \scr{P}$, and,
            \item for each $P\in \scr{P}$, either $P=V(H)$ for some $H\in \scr{H}$, or $|P|\leq 12\sqrt{3n}$.
        \end{enumerate}
    \end{lemma}

    \begin{proof}
        Let $X=\bigcup_{H\in \scr{H}}V(H)$. We proceed by induction, primarily on $|V(G)\setminus X|$ and secondarily on $|\scr{H}|$.

        The base case is when $V(G)=X$. Since every graph on at most two vertices has treewidth at most $1$, and since $G(V(H))\supseteq H$ is connected for each $H\in \scr{H}$, $\scr{P}=(V(H):H\in \scr{H})$ is trivially the desired connected partition of $G$. So we may assume $V(G)\neq X$, and that the result holds for smaller values of $|V(G)\setminus X|$ (and consequently, larger values of $|X|$ for the same $G$). Note that this also means that we can assume $n\geq 1$.

        If $\scr{H}=\emptyset$, then $X=\emptyset$. Pick any $v\in V(G)$ (which exists, since $n\geq 1$). Since $|\{v\}|\geq |X|=0$, we can apply induction on $G$ with $\scr{H}:=\{G[\{v\}]\}$. The result follows since $G[\{v\}]$ is connected and $|\{v\}|=1\leq 12\sqrt{3n}$ (as $n\geq 1$). So we may assume that $\scr{H}\neq\emptyset$.

        Next, consider the case where $|\scr{H}|=1$. So $\scr{H}=\{A\}$ for some connected subgraph $A$ of $G$. If $V(A)$ is empty, apply induction with $\scr{H}:=\emptyset$, and add an extra copy of $\emptyset$ to the resulting part. Since we consider the empty graph to be connected, and since this new part is an isolated vertex in the quotient (which does not increase the treewidth beyond $0$), this gives the desired partition of $G$. So we may assume that $V(A)$ is nonempty.
        
        By \cref{PQT1Con}, $G$ is connected. Thus and since $V(A)$ is nonempty and since $V(G)\neq X=V(A)$, there exists $v\in N_G(V(A))$. By \cref{planarSmallSep} applied to $G$ with $B:=G[\{v\}]$, 
        there is a set $S\subseteq V(G-V(A)-v)$ such that:
        \begin{enumerate}
            \item $|N_G[S]|\leq 3\sqrt{3n}$, 
            \item $\kappa_G(V(A),C,v)\leq 3\sqrt{3n}$ for each component $C$ of $G-V(A)-v-S$, and 
            \item $G[S\cup\{v\}]$ is connected. 
        \end{enumerate}
        Let $B:=G[S\cup \{v\}]$. So $B$ is connected, disjoint to $A$, and $V(B)$ intersects $N_G(V(A))$. Further, for each connected component $C$ of $G-V(A)-V(B)=G-V(A)-v-S$, we have $\kappa_{G}(V(A),C,V(B)) \leq
        \kappa_{G}(V(A),C,v) + |N_G[S]| \leq 6\sqrt{3n}$. Also, observe that $|V(A)\cup V(B)|\geq |V(A)\cup \{v\}|>|V(A)|$. Thus, we can apply induction to $G$ with $\scr{H}:=\{A,B\}$. We obtain a connected treewidth $2$ partition $\scr{P}$ of $G$ with $V(A),V(B)\in \scr{P}$ such that, for each $P\in \scr{P}$, either $P\in \{V(A),V(B)\}$, or $|P|\leq 12\sqrt{3n}$. Since $|S|\leq |N_G[S]|\leq 3\sqrt{3n}$, $|V(B)|\leq |S|+1\leq 12\sqrt{3n}$ (as $n\geq 1$). Thus, $\scr{P}$ is the desired connected partition of $G$.

        Finally, consider the case when $|\scr{H}|=2$. So $\scr{H}=\{A,B\}$, where $A,B$ are disjoint connected subgraphs of $G$ that are disjoint, such that $N_G(V(A))$ intersects $V(B)$, and $\kappa_G(V(A),C,V(B))\leq 6\sqrt{3n}$ for every component $C$ of $G-V(A)-V(B)$. Note that since $N_G(V(A))$ intersects $V(B)$, both $A$ and $B$ are nonempty.

        First, suppose that either $N_G(V(A))\subseteq V(B)$ or $N_G(V(B))\subseteq V(A)$. Without loss of generality, say $N_G(V(B))\subseteq V(A)$. Let $A':=G[V(A)\cup V(B)]$. Since $A$ and $B$ are connected and $N_G(V(A))$ intersects $V(B)$, $A'$ is a connected subgraph of $G$. As $V(A')=V(A)\cup V(B)=X$ and since $|\{A'\}|<|\{A,B\}$, we may apply induction with $\scr{H}:=\{A'\}$. We obtain a connected treewidth $2$ partition $\scr{P}'$ with $V(A')\in \scr{P}'$ such that, for each $P\in \scr{P}'$, either $P=V(A')$ or $|P|\leq 12\sqrt{3n}$.

        Since $A,B$ are nonempty, so is $A'$. Thus $V(A')$ appears exactly once in $\scr{P}'$. Let $\scr{P}$ be obtained from $\scr{P}'$ by deleting $V(A')$ and adding both $V(A)$ and $V(B)$. Since $A$ and $B$ are disjoint, $V(A')=V(A)\sqcup V(B)$, and thus $\scr{P}$ is a partition of $G$. Since $G[V(A)]\supseteq A$ and $G[V(B)]\supseteq B$ are connected, $\scr{P}$ is a connected. Finally, since $N_G(V(B))\subseteq V(A)$, and since $N_G(V(A))$ intersects $V(B)$, observe that $G/\scr{P}$ is obtained from $G/\scr{P}'$ by relabelling $V(A')$ to $V(A)$ and adding a new vertex $V(B)$ adjacent to $V(A)$. Since adding a degree 1 vertex to a graph does not increase the treewidth beyond $1$, $\scr{P}$ has treewidth at most $2$. By construction, $V(A),V(B)\in \scr{P}$. Thus and by definition of $\scr{P}'$, $\scr{P}$ is the desired partition.

        So we may now assume that $N_G(V(A))\setminus V(B)\neq \emptyset$ and $N_G(V(B))\setminus V(A)\neq \emptyset$.

        Let $\scr{C}$ be the set of connected components of $G-V(A)-V(B)$. Consider any $C\in \scr{C}$. By definition of $A,B$, we have that $\kappa_G(V(A),C,V(B))\leq 6\sqrt{3n}$. By Menger's theorem, there exists $S_C\subseteq V(C)$ such that $|S_C|\leq 6\sqrt{3n}$ and $S_C$ is a separator between $N_G(V(A)) \cap V(C)$ and $N_G(V(B))\cap V(C)$ in $C$. In particular, we can pick $S_C$ so that $S_C$ is a minimal separator. Since $A,B$ are connected and disjoint subgraphs of $G$, by \cref{sepCC}, $G[S_C]$ is connected.

        Let $S=\bigcup_{C\in \scr{C}}S_C$. Note that $S$ is disjoint to $V(A)\cup V(B)$.
        
        Let $\scr{C}'$ be the set of connected components of $G-V(A)-V(B)-S$. Observe that each $C'\in \scr{C}'$ is contained in some $C\in \scr{C}$. Thus and since $S_C\subseteq S$ is a separator between $N_G(V(A)) \cap V(C)$ and $N_G(V(B))\cap V(C)$ in $C$, observe that $V(C')$ intersects at most one of $N_G(V(A))$ and $N_G(V(B))$.
        
        Let $\scr{C}'_A$ be the set of components $C\in \scr{C}'$ such that $V(C)$ intersects $N_G(V(A))$, $\scr{C}'_B$ be the set of components $C\in \scr{C}'$ such that $V(C)$ intersects $N_G(V(B))$, and $\scr{C}'_D$ be the set of components $C\in \scr{C}'$ such that $V(C)$ does not intersect $N_G(V(A))\cup N_G(V(B))$. So $\scr{C}'=\scr{C}'_A\sqcup \scr{C}'_B\sqcup \scr{C}'_D$.

        For each $H\in \{A,B\}$, let $C_H:=\bigcup_{C'\in \scr{C}_H'}C'$. Note that $C_H$ is disjoint to $A$ and $B$, and $G[V(C_H\cup H)]$ is connected.

        Since $S_C$ is minimal for each $C\in \scr{C}$, observe that $S\subseteq N_G(V(C_A\cup A))\cap N_G(V(C_B\cup B))$. Thus, $G[V(C_A\cup A)\cup S]$ and $G[V(C_B\cup B)\cup S]$ are connected.

        Let $C_D:=\bigcup_{C'\in \scr{C}_D'}C'$. Note that for each connected component $C$ of $C_D$, $N_G(V(C))$ intersects $S$. Thus, $G[V(C_A\cup A\cup C_D)\cup S]$ and $G[V(C_B\cup B\cup C_D)\cup S]$ are connected.
        
        Since $\scr{C}'=\scr{C}'_A\sqcup \scr{C}'_B\sqcup \scr{C}'_D$, observe that $(V(A),V(B),S,V(C_A),V(C_B),V(C_D))$ is a partition of $G$.

        Fix $H\in \{A,B\}$. Define $H^*$ be the unique element of $\{A,B\}\setminus \{H\}$, and set $H':=G[V(C_H\cup H\cup C_D)\cup S]$. As observed earlier, $H'$ is connected. Since $N_G(V(H))$ is not contained in $V(H^*)$, observe that either $S$ or $C_H$ is nonempty. Thus, $|V(H')|>|V(H)|$.

        For each $H\in \{A,B\}$, since $(V(A),V(B),S,V(C_A),V(C_B),V(C_D))$ is a partition of $G$, observe that $V(G)\setminus (V(H')\cup V(H^*))=V(C_{H^*})$. Thus, observe that $G-(V(H')\cup V(H^*))=C_{H^*}$.
        
        Since $H$ and $H^*$ are disjoint subgraphs of $G$, and since $S\cup V(C_H)$ is disjoint to $V(A)\cup V(B)$, observe that $H'$ and $H^*$ are disjoint and connected subgraphs of $G$. Further, $N_G(V(H'))\supseteq N_G(V(H))$ intersects $V(H^*)$. Thus, we can apply \cref{planarSmallSep}. We obtain a set $S_H\subseteq V(G-V(H')-V(H^*))=V(C_{H^*})$ such that:
        \begin{enumerate}
            \item $|N_{G}[S_H]|\leq 3\sqrt{3n}$, 
            \item $\kappa_G(V(H'),C,V(H^*))\leq 3\sqrt{3n}$ for each component $C$ of $G-V(H')-V(H^*)-S_H=C_{H^*}-S_H$,
            \item $G[S_H\cup V(H')]$ is connected, and
            \item $G[S_H\cup V(H^*)]$ is connected. 
        \end{enumerate}

        Let $H'':=G[V(H')\cup S_H]$. Observe that $H''$ is disjoint to $H^*$, and that $N_G(V(H''))\supseteq N_G(V(H'))$ intersects $V(H^*)$. By property (3), $H''$ is connected. For each connected component $C$ of $G-V(H'')-V(H^*)=G-V(H')-V(H^*)-S_H=C_{H^*}-S_H$, observe that $\kappa_G(V(H''),C,V(H^*))\leq \kappa_G(V(H'),C,V(H^*))+|N_G[S_H]|\leq 6\sqrt{3n}$ by properties (1) and (2). Finally, observe that $|V(H'')\cup V(H^*)|\geq |V(H')\cup V(H^*)|>|V(H)\cup V(H^*)|$ as $|V(H')|>|V(H)|$ and since $H'$ and $H^*$ are disjoint. 
        
        So we may apply induction with $\scr{H}:=\{H'',H^*\}$. We find a connected partition $\scr{P}_H$ of $G$ of treewidth at most $2$ with $V(H''),V(H^*)\in \scr{P}_H$ such that for each $P\in \scr{P}_H$ with $P\notin \{V(H''),V(H^*)\}$, $|P|\leq 12\sqrt{3n}$. Note that $V(H''),V(H^*)$ each appear exactly once in $\scr{P}$ as $H'',H^*$ are nonempty (since $H$ is nonempty). Let $\scr{P}_H'$ be the result of removing $V(H'')$ and $V(H^*)$ from $\scr{P}_H$. Observe that $\scr{P}_H'$ is a connected partition of $G-V(H'')-V(H^*)=C_{H^*}-S_H$ of width at most $12\sqrt{3n}$.

        Let $D$ be the empty graph, and let $D':=G-V(C_D)=G[V(A\cup B\cup C_A\cup C_B)\cup S]$. Note that $G-V(D')=C_D$. Since $G[V(A\cup C_A)\cup S]$ and $G[V(B\cup C_B)\cup S]$ are connected and since $N_G(V(A))$ intersects $V(B)$, $D'$ is connected. For each $H\in \{A,B\}$, recall that either $S$ or $C_H$ is nonempty. Thus, either $S$ is nonempty, or both $C_A$ and $C_B$ are nonempty. Hence, $|V(D')|>|V(A)\cup V(B)|$. 
        
        Therefore, we may apply induction with $\scr{H}=\{D'\}$. We find a connected partition $\scr{P}_D$ of $G$ of treewidth at most $2$ with $V(D')\in \scr{P}_D$ such that for each $P\in \scr{P}_D$ with $P\neq V(D')$, $|P|\leq 12\sqrt{3n}$. Note that $V(D')$ appears exactly once in $\scr{P}$ as $D'$ is nonempty (since $A,B$ are nonempty). Let $\scr{P}_D'$ be the result of removing $V(D')$ from $\scr{P}_D$. Observe that $\scr{P}_D'$ is a connected partition of $G-V(D')=C_D$ of width at most $12\sqrt{3n}$.

        For each $C\in \scr{C}$, let $S_{A,C}:=S_A\cap V(C)$, $S_{B,C}:=S_B\cap V(C)$, and let $S_C':=S_C\cup S_{A,C}\cup S_{B,C}$. Observe that $|S_C'|\leq |N_G[S_A]|+|N_G[S_B]|+|S_C|\leq 12\sqrt{3n}$. Further, since $S_A\cup S_B\subseteq V(C_A)\cup V(C_B)$, which is disjoint to $V(A)\cup V(B)$, observe that $\bigsqcup_{C\in \scr{C}}S_C'=S\cup S_A\cup S_B$. Thus, $\scr{P}':=(V(A),V(B),V(C_A-S_A),V(C_B-S_B),V(C_D))\sqcup \bigsqcup_{C\in \scr{C}}S_C'$ is a partition of $G$.
        
        Fix $H\in \{A,B\}$ and $C\in \scr{C}$. Since $H''=[V(H')\cup S_H]$ is connected, for each $s\in S_{H,C}$, there exists a path from $s$ to $V(H')$ contained in $H''$. More specifically, each vertex of this path other than the endpoint in $V(H')$ is contained in $S_H$. Since $s\in V(C)$ and since $S_H\subseteq V(C_{H^*})$ (which is disjoint to $V(A)\cup V(B)$), each vertex of this path other than the endpoint in $V(H')$ is contained in $V(C)$ and thus $S_{H,C}$. Recall that $S_H\subseteq C_{H^*}$ and observe that $N_G(V(C_{H*}))\subseteq V(H^*)\cup S$. Since $H^*$ is disjoint to $H''$, it follows that the endpoint in $V(H')$ is contained in $S$ and thus $S_C$. Recalling that $G[S_C]$ is connected, it follows that $G[S_C\cup S_{H,C}]$ is connected. Repeating this argument for $H^*$, we find that $G[S_C']$ is connected.
        
        Recall that $G[S_H\cup V(H^*)]$ is connected. By a similar argument to the above, we find that $G[S_{H,C}\cup V(H^*)]$ is connected. Thus (and since $S_C\subseteq S$ is disjoint to $V(H^*)$, we find that $N_G(S_C')$ intersects $V(H^*)$. Repeating this argument for $H^*$, we find that $N_G(S_C')$ intersects $V(A)$ and $V(B)$.

        Recall that $\scr{P}'=(V(A),V(B),V(C_A-S_A),V(C_B-S_B),V(C_D))\sqcup \bigsqcup_{C\in \scr{C}}S_C'$ is a partition of $G$, that $\scr{P}_D$ is a connected partition of $C_D$ of width at most $12\sqrt{3n}$, and that, for each $H\in \{A,B\}$, $\scr{P}_H$ is a connected partition of $C_H-S_H$ of width at most $12\sqrt{3n}$. Additionally, recall that for each $C\in \scr{C}$, $|S_C'|\leq 12\sqrt{3n}$. It follows that $\scr{P}:=(V(A),V(B))\sqcup \scr{P}_A\sqcup \scr{P}_B\sqcup \scr{P}_D\sqcup (S_C':C\in \scr{C})$ is a connected partition of $G$ containing $V(A),V(B)$, and satisfying $|P|\leq 12\sqrt{3n}$ for each $P\in \scr{P}$ with $P\notin \{V(A),V(B)\}$. It remains only to show that $G/\scr{P}$ has treewidth at most $2$.
        
        Observe that in $G/\scr{P}'$, $V(C_A-S_A),V(C_B-S_B)$ and $V(C_D)$ are pairwise non-adjacent. Thus, if $P,P'\in \scr{P}_A\sqcup \scr{P}_B\sqcup \scr{P}_D$ are adjacent in $G/\scr{P}$, then $P,P'\in \scr{P}_H$ for some $H\in \{A,B,D\}$. Further, $P,P'$ are adjacent in $G/\scr{P}$ if and only if they are adjacent in $G/\scr{P}_H'$.
        
        Next, observe that in $G'/\scr{P}$ $V(C_A-S_A)$ is non-adjacent to $V(B)$, $V(C_B-S_B)$ is non-adjacent to $V(A)$, and $V(C_D)$ is non-adjacent to $V(A)$ and $V(B)$. Thus, if $H_1\in \{A,B,D\}$ and $H_2\in \{A,B\}$ are such that $P\in \scr{P}_{H_1}$ is adjacent to $V(H_2)$ in $G/\scr{P}$, then $H_1\neq D$ and $H_2=H_1^*$ (as $P\subseteq V(C_{H_1^*})$). Set $H:=H_1$. Observe further that, $P$ is adjacent to $V(H^*)$ in $G/\scr{P}$ if and only if $P$ is adjacent to $V(H^*)$ in $G/\scr{P}_H'$.

        For each $P\in \scr{P}_A\sqcup \scr{P}_B\sqcup \scr{P}_D$, observe that $P\subseteq V(C_P)$ for some $C\in \scr{C}$. Observe that in $G/\scr{P}$, $P$ is non-adjacent in $G/\scr{P}$ to $S_{C'}'$ for any $C'\in \scr{C}'\setminus \{C_P\}$, and any $P'\in \scr{P}_A\sqcup \scr{P}_B\sqcup \scr{P}_D$ with $C_{P'}\neq C_P$. Further, if $P\in \scr{P}_H$ for some $H\in \{A,B\}$, then $P$ is adjacent to $S_{C_P}'$ in $G/\scr{P}$ if and only if $P$ is adjacent to $V(H'')$ in $G/\scr{P}_H'$. Otherwise, if $P\in \scr{P}_D$, then $P$ is adjacent to $S_{C_P}'$ in $G/\scr{P}$ if and only if $P$ is adjacent to $V(D')$ in $G/\scr{P}_D'$.

        Observe that $(S_C':C\in \scr{C})$ is an independent set in $G/\scr{P}'$ (and thus $G/\scr{P}$). Also, for each $C\in \scr{C}$, note that $S_C'$ is adjacent to $V(A)$ and $V(B)$ in $G/\scr{P}$. Observe that $V(A'')$ is adjacent to $V(B)$ in $G/\scr{P}_A'$, and that $V(B'')$ is adjacent to $V(A)$ in $G/\scr{P}_B'$.
        
        Thus, we can construct $G/\scr{P}$ as follows.
        
        For each $C\in \scr{C}$, take a disjoint copy of $G/\scr{P}_A'$, $G/\scr{P}_B'$, and $G/\scr{P}_D'$. Delete any parts other than $V(A),V(B),V(A''),V(B'')$ and $V(D')$ that are not contained in $V(C)$. Relabel $V(A''),V(B''),V(D')$ to $S_C'$. Call the resulting graphs $G_{A,C},G_{B,C}$ and $G_{D,C}$ respectively. Observe that each of these graphs have treewidth at most $2$. Further, $S_C'$ is adjacent to $V(B)$ in $G_{A,C}$, and adjacent to $V(A)$ in $G_{A,C}$.

        To $G_{A,C}$, add the vertex $V(A)$ adjacent to $S_C'$ and $V(B)$, and to $G_{B,C}$, add the vertex $V(B)$ adjacent to $S_C'$ and $V(A)$. Observe that the resulting graphs $G_{A,C}'$ and $G_{B,C}'$ have treewidth at most $2$, as the neighbourhood of the new vertex is a clique of size $2$. To $G_{D,C}$, add the vertices $V(A)$ and $V(B)$ adjacent to each other and $S_C'$ to obtain a new graph $G'_{D,C}$. Observe that $G'_{D,C}$ also has treewidth at most $2$.

        Let $G_C$ be the result of identifying $G'_{A,C}$, $G'_{B,C}$, and $G'{D,C}$ on their shared vertices $\{V(A),V(B),S_C'\}$. Observe that $G_C$ also has treewidth at most $2$, as we identified cliques of size $3$. Finally, $G/\scr{P}$ is obtained from identifying each $G_C$ along the common vertices $V(A),V(B)$. This identifies cliques of size $2$, so $G/\scr{P}$ also has treewidth at most $2$. Thus, $\scr{P}$ has treewidth at most $2$, which completes the proof.
    \end{proof}

    We can now prove \cref{conPlanar}, which we restate here.

    \conPlanar*

    \begin{proof}
        Since $W$ is nonempty, by \cref{PQTContract}, there exists a plane quasi-triangulation $G'$ and a vertex $w\in V(G')$ such that $G'-w$ is the plane graph $G-V(W)$ and $N_{G'}(w)=N_G(V(W))$. Observe that $|V(G')|=|V(G)|-|V(W)|+1=n$. Apply \cref{conPlanarLemma} on $G'$ with $\scr{H}:=\{G'[\{w\}]\}$ to obtain a connected partition $\scr{P}'$ of $G'$ of treewidth at most $2$ with $\{w\}\in \scr{P}'$ and such that each part of $\scr{P}'$ other than $\{w\}$ has size at most $12\sqrt{3n}$. Observe that $\{w\}$ appears exactly once in $\scr{P}'$.
        
        Define $\scr{P}$ to be the result of replacing $\{w\}$ with $V(W)$ in $\scr{P}'$. Since $G[V(W)]\supseteq W$ is connected and since $G'-w=G-V(W)$, observe that $\scr{P}$ is a connected partition of $G$. Further, $V(W)\in \scr{P}$, and each part of $\scr{P}$ other than $V(W)$ has size at most $12\sqrt{3n}$. Finally, since $N_{G'}(w)=N_G(V(W))$, observe that $G/\scr{P}$ is obtained from $G'/\scr{P}'$ by relabelling $w$ to $V(W)$. Thus, $\scr{P}$ has treewidth at most $2$. This completes the proof.
    \end{proof}
\end{document}